\documentclass[twoside,11pt,reqno]{amsart}
\usepackage{amsmath,amssymb,amscd,mathrsfs,epic,wasysym,latexsym,tikz,mathrsfs,cite,hyperref}
\usepackage{stmaryrd}
\usepackage{pb-diagram}
\usepackage[matrix,arrow]{xy}
\usepackage{commath}
\usepackage{tikz-cd}
\usepackage{ytableau}
\usepackage{float}

\DeclareMathAlphabet{\mathpzc}{OT1}{pzc}{m}{it}

\makeatletter

\numberwithin{equation}{section}
\newtheorem{Proposition}[equation]{Proposition}
\newtheorem{Lemma}[equation]{Lemma}
\newtheorem{Theorem}[equation]{Theorem}
\newtheorem{Corollary}[equation]{Corollary}
\theoremstyle{definition}  %% makes all of the theorem environments which follow appear in \rm
\newtheorem{Definition}[equation]{Definition}
\newtheorem{Remark}[equation]{Remark}

\newtheorem{Example}[equation]{Example}

\newtheorem{Hypothesis}[equation]{Hypothesis}

\newtheorem{Question}[equation]{Question}

\let\<\langle
\let\>\rangle

\newcommand\Comment[2][\relax]{\space\par\medskip\noindent%
   \fbox{\begin{minipage}{\textwidth}\textbf{Comment\ifx\relax#1\else---#1\fi}\newline%
        #2\end{minipage}}\medskip
}

\def\bi{\text{\boldmath$i$}}

\def\b1{\text{\boldmath$1$}}

\newcommand{\swr}{\wr_{\super}}
\newcommand{\super}{\mathop{\tt s}\nolimits}
\newcommand{\nth}{\mathop{\rm th}\nolimits}
\newcommand{\nst}{\mathop{\rm st}\nolimits}

\newcommand{\sM}{\mathop{\rm sMor}\nolimits}
\newcommand{\prj}{\mathop{\rm Prj}\nolimits}

\newcommand{\sh}{\mathop{\rm sh}\nolimits}
\newcommand{\Id}{\mathop{\rm Id}\nolimits}
\newcommand{\Br}{\operatorname{Br}}
\newcommand{\cC} {\mathcal{C}}

\newcommand{\cT} {\mathcal{T}}
\newcommand{\ct} {\mathpzc{t}}

\newcommand{\Blo} {B}

\def\pmod#1{\text{ }(\text{\rm mod } #1)\,}

\newcommand{\Hom}{\operatorname{Hom}}

\newcommand{\End}{\operatorname{End}}

\def\core{{\operatorname{core}}}

\newcommand{\Z}{\mathbb{Z}}
\newcommand{\K}{\mathbb{K}}
\newcommand{\F}{\mathbb{F}}

\newcommand{\N}{\mathbb{N}}
\newcommand{\NN}{\mathbb{Z}_{> 0}}
\newcommand{\0}{{\bar 0}}
\renewcommand{\1}{{\bar 1}}
\def\eps{{\varepsilon}}
\def\phi{{\varphi}}

\newcommand{\cO}{\mathcal O}

\newcommand{\cR}{\mathcal R}

\newcommand{\Grot}{{\mathcal G}_0}

\newcommand{\cM}{\mathcal M}
\newcommand{\cQ}{\mathcal Q}
\newcommand{\cY}{\mathcal Y}

\newcommand{\la}{\lambda}
\newcommand{\La}{\Lambda}
\newcommand{\al}{\alpha}

\def\Si{\mathsf{S}}
\def\Ai{\mathsf{A}}
\def\Gi{\mathsf{G}}
\def\Hi{\mathsf{H}}
\def\Li{\mathsf{L}}
\def\Ni{\mathsf{N}}
\def\Di{\mathsf{D}}

\def\tSi{\tilde{\Si}}
\def\tAi{\tilde{\Ai}}

\def\ei{\mathsf{e}}
\def\bi{\mathsf{b}}
\def\ci{\mathsf{c}}
\def\fid{\mathsf{f}}

\newcommand{\si}{\sigma}

\newcommand{\Om}{\Omega}

\newcommand{\de}{\delta}

\newcommand{\ka}{\kappa}

\def\Mtype{\mathtt{M}}
\def\Qtype{\mathtt{Q}}

\newcommand{\Irr}{{\mathrm {Irr}}}

\newcommand{\Ind}{{\mathrm {Ind}}}

\newcommand{\Mor}{{\mathrm {Mor}}}

\newcommand{\Tr}{{\mathrm {Tr}}}
\newcommand{\tr}{{\mathrm {tr}}}

\newcommand{\rad}{{\mathrm {rad}\,}}

\newcommand{\Res}{{\mathrm {Res}}}

\newcommand{\C}{{\mathbb C}}

\newcommand{\R}{{\mathbb R}}

\renewcommand{\mod}{\bmod \,}

\def\K{\mathbb K}

\newcommand{\Zig}{{\mathtt{A}}}
\newcommand{\Zag}{{\mathtt{B}}}

\newcommand{\sR}{{\mathtt{R}}}
\newcommand{\sJ}{{\mathtt{P}}}

\newcommand{\sX}{{\mathtt{X}}}
\newcommand{\sY}{{\mathtt{Y}}}

\newcommand{\gP}{P}
\newcommand{\ttP}{\mathtt{P}}

\newcommand{\gQ}{Q}

\newcommand{\gR}{R}

\newcommand{\bM}{\mathbf{M}}
\newcommand{\gM}{M}
\newcommand{\ttM}{\mathtt{M}}

\newcommand{\gN}{N}

\newcommand{\bU}{\mathbf{U}}
\newcommand{\gU}{U}

\newcommand{\gV}{V}

\newcommand{\gW}{W}

\newcommand{\bX}{\mathbf{X}}
\newcommand{\bY}{\mathbf{Y}}

\newcommand{\gZ}{Z}

\def\col{{\operatorname{col}}}

\newcommand{\IBr}{{\mathrm {IBr}}}

\def\Par{{\mathscr P}}

\def\ud{{\underline{d}}}
\def\ude{{\underline{\delta}}}

\def\b{\mathfrak{b}}
\def\k{\Bbbk}

\def\wt{{\operatorname{wt}}}

\def\op{{\mathrm{op}}}
\def\sop{{\mathrm{sop}}}

\def\into{{\hookrightarrow}}

\def\mod#1{#1\!\operatorname{-mod}}

\def\underlinesmod#1{#1\!\operatorname{-\underline{smod}}}

\def\iso{\stackrel{\sim}{\longrightarrow}}

\def\sh{{\mathsf{sh}}}

\def\Ab{{\mathtt{Ab}}}
\def\T{{\mathtt T}}

\def\sfs{{\mathfrak s}}

\newcommand{\cc} {\mathpzc{c}}

\makeatletter
\newcommand*{\rom}[1]{\expandafter\@slowromancap\romannumeral #1@}
\makeatother

{\catcode`\|=\active
  \gdef\set#1{\mathinner{\lbrace\,{\mathcode`\|"8000%
  \let|\midvert #1}\,\rbrace}}
}
\def\midvert{\egroup\mid\bgroup}

{\catcode`\|=\active
  \gdef\set#1{\mathinner{\lbrace\,{\mathcode`\|"8000%
  \let|\midvert #1}\,\rbrace}}
}
\def\midvert{\egroup\mid\bgroup}

\colorlet{darkgreen}{green!50!black}
\tikzset{dots/.style={very thick,loosely dotted},
         greendot/.style={fill,circle,color=darkgreen,inner sep=1.5pt,outer sep=0},
         blackdot/.style={fill,circle,color=black,inner sep=1.1pt,outer sep=0},
         graydot/.style={fill,circle,color=gray,inner sep=1.1pt,outer sep=0},
         reddot/.style={fill,circle,color=red,inner sep=1.1pt,outer sep=0},
         bluedot/.style={fill,circle,color=blue,inner sep=1.1pt,outer sep=0}
}
\def\greendot(#1,#2){\node[greendot] at(#1,#2){}}
\def\blackdot(#1,#2){\node[blackdot] at(#1,#2){}}
\def\graydot(#1,#2){\node[graydot] at(#1,#2){}}
\def\reddot(#1,#2){\node[reddot] at(#1,#2){}}
\def\bluedot(#1,#2){\node[bluedot] at(#1,#2){}}

\newenvironment{braid}{% sets defaults for the braid diagrams
  \begin{tikzpicture}[baseline=6mm,black,line width=.7pt, scale=0.32,
                      draw/.append style={rounded corners},
                      every node/.append style={font=\fontsize{5}{5}\selectfont}]%
  }{\end{tikzpicture}
}

\def\Grid(#1,#2){%  draws a coordinate grid inside a braid diagram
  \draw[very thin,gray,step=2mm] (0,0)grid(#1,#2);
  \draw[very thin,darkgreen,step=10mm] (0,0)grid(#1,#2);
}

\newcommand\Tableau[2][\relax]{
  \begin{tikzpicture}[scale=0.5,draw/.append style={thick,black}]
    \ifx\relax#1\relax%
    \else % shade the boxes in #1
      \foreach\box in {#1} { \filldraw[blue!30]\box+(-.5,-.5)rectangle++(.5,.5); }
    \fi
    \newcount\row\newcount\col
    \row=0
    \foreach \Row in {#2} {
       \col=1
       \foreach\k in \Row {
          \draw(\the\col,\the\row)+(-.5,-.5)rectangle++(.5,.5);
          \draw(\the\col,\the\row)node{\k};
          \global\advance\col by 1
       }
       \global\advance\row by -1
    }
  \end{tikzpicture}
}

\newcommand\YoungDiagram[2][\relax]{
  \begin{tikzpicture}[scale=0.5,draw/.append style={thick,black}]
    \ifx\relax#1\relax%
    \else % shade the boxes in #1
    \foreach\box in {#1} {
      \filldraw[blue!30]\box rectangle ++(1,1);
    }
    \fi
    \newcount\row
    \row=0
    \foreach \col in {#2} {
       \draw(1,\the\row)grid ++(\col,1);
       \global\advance\row by -1
    }
  \end{tikzpicture}
}

\newenvironment{Young}{\begingroup
       \def\vr{\vrule height0.89\hoogte width\dikte depth 0.2\hoogte}
       \def\fbox##1{\vbox{\offinterlineskip
                    \hrule height\dikte
                    \hbox to \breedte{\vr\hfill##1\hfill\vr}
                    \hrule height\dikte}}
       \vbox\bgroup \offinterlineskip \tabskip=-\dikte \lineskip=-\dikte
            \halign\bgroup &\fbox{##\unskip}\unskip  \crcr }
       {\egroup\egroup\endgroup}
\def\diagram#1{\relax\ifmmode\vcenter{\,\begin{Young}#1\end{Young}\,}\else%
              $\vcenter{\,\begin{Young}#1\end{Young}\,}$\fi}

\makeindex

\begin{document}

%\frontmatter

\title[RoCK blocks for double covers of symmetric groups over a  DVR]{{\bf RoCK blocks for double covers of symmetric groups over a complete discrete valuation ring}}

\author{\sc Alexander Kleshchev}
\address{Department of Mathematics\\ University of Oregon\\ Eugene\\ OR 97403, USA}
\email{klesh@uoregon.edu}

\author{\sc Michael Livesey}
\address{School of Mathematics, University of Manchester, Manchester, M139PL, U.K.}
\email{michael.livesey@manchester.ac.uk}


\thanks{The first author was supported by the NSF grant DMS-2101791 and 
Charles Simonyi Endowment at the Institute for Advanced Study. 
The second author was supported by the EPSRC (grant no EP/T004606/1). 
Both authors would like to thank the Isaac Newton Institute for Mathematical Sciences, Cambridge, for support and hospitality during the programme `Groups, representations and applications: new perspectives' where part of the work on this paper was undertaken. 
}

\begin{abstract}
Recently the authors proved the existence of RoCK blocks for double covers of symmetric groups over an algebraically closed field of odd characteristic. In this paper we prove that these blocks lift to RoCK blocks over a suitably defined discrete valuation ring. Such a lift is even splendidly derived equivalent to its Brauer correspondent. We note that the techniques used in the current article are almost completely independent from those previously used by the authors. In particular, we do not make use of quiver Hecke superalgebras and the main result is proved using methods solely from the theory of representations of finite groups. Therefore, this paper much more resembles the work of Chuang and Kessar, where RoCK blocks for symmetric groups were constructed.
\end{abstract}

\maketitle

\tableofcontents

\section{Introduction}

Throughout this article $p$ will denote an odd prime and $(\K,\cO,\F)$ a $p$-modular system. So $\K$ is a field of characteristic zero, $\cO$ a complete discrete valuation ring with field of fractions $\K$ and residue field $\F$, which, in turn, is an algebraically closed field of characteristic $p$. We denote by $\overline{\phantom{x}}:\cO \to \F$ the natural projection.

For the remainder of the introduction, when dealing with a finite group $G$, we will assume that $\K$ (and consequently $\cO$) contains a primitive $|G|^{\nth}$ root of unity. This ensures that $\K G$ is split semisimple. For example, in Brou\'e's conjecture we assume $\K$ contains a primitive $|G|^{\nth}$ root of unity. Similarly, in Theorems A and B we assume $\K$ contains a primitive $(2n!)^{\nth}$ root of unity, where $\Blo^{\rho,d}$ is a block of $\cO \tSi_n$. Elsewhere in the article we will make it explicit whenever we are making such an assumption on $\K$.

%\textcolor{blue}{Is this better now? I emphasise, when RoCK blocks are introduced in $\S$5.2, that, since we only ever deal with double covers of $\Si_4$ and bigger, it is actually enough to assume $\K$ contains a primitive $(2n!)^{\nth}$ root of unity. (Having square roots of $-1$ and $2$ is then automatic.) I've also added a remark right at the end of Section 8 you might want to take a look at.}

This article is primarily concerned with Brou\'e's long-standing {\em abelian defect group conjecture}:

\vspace{2mm}
\noindent
{\bf Brou\'e's abelian defect group conjecture. }\index{Brou\'e's abelian defect group conjecture}
{\em 
For any finite group $G$, a block $B$ of $G$ with abelian defect group $D$ is derived equivalent to its Brauer correspondent~$b$.}

\vspace{2mm}

One can state the conjecture for blocks defined over $\F$ or $\cO$. Our focus is blocks defined over $\cO$.

For $n \in \N$, $\Si_n$ will denote the symmetric group on $n$ letters and $\tSi_n$ one of the two double covers of $\Si_n$. In addition, $\Ai_n$ will denote the corresponding alternating group and $\tAi_n$ its double cover. There is a one-to-one, isomorphic correspondence between the blocks of $\cO \Si_n$ and those of $\cO \tSi_n (1 + z)/2$. In particular, Brou\'e's conjecture is known for the blocks of $\cO \tSi_n (1 + z)/2$, as it is known to hold for blocks of symmetric groups (see \cite{CK} and \cite{CR}). We, therefore, restrict our attention in this article to blocks of $\cT_n := \cO \tSi_n (1 - z)/2$, also known as the {\em spin blocks of the symmetric group}.

Throughout the article it will be important to consider $\cO \tSi_n$ as a superalgebra via:
$$
(\cO \tSi_n)_{\0}:=\cO \tAi_n, \qquad (\cO \tSi_n)_{\1}:=\cO (\tSi_n\setminus\tAi_n).
$$

We recall that the {\em spin superblocks of $\cO \tSi_n$} are labeled by pairs $(\rho,d)$, where $\rho$ is a ${\bar p}$-core and $d \in \N$ (referred to as the {\em weight of the block}) such that $n = |\rho|+dp$. We denote such a superblock by $\Blo^{\rho,d}$. 
We note that the case $d=0$ corresponds to the defect zero situations, so we will often assume that $d>0$. In that case 
 $\Blo^{\rho,d}$ is in fact a spin block (and not just a spin superblock) of $\cO \tSi_n$. In this case the defect group $D$ of $\Blo^{\rho,d}$ is abelian if and only if $d<p$. If it is abelian, then $D \cong (C_p)^{\times d}$.

In \cite{KL} the authors defined the notion of a {\em $d$-Rouquier ${\bar p}$-core}, for $d \in \N$. For the remainder of the introduction we assume $0 < d < p$ and that $\rho$ is a $d$-Rouquier ${\bar p}$-core.

Throughout the article `$\otimes$' will denote a {\em tensor product of superalgebras}. For a superalgebra $A$, a {\em twisted wreath superproduct} $A \swr \cT_d$ is defined in $\S$\ref{sec:wreath}. For the definition of a {\em Morita superequivalence}, see $\S$\ref{SMoritaSuper}. Our main result (see Theorem \ref{thm:main}) is as follows:

\vspace{2mm}
\noindent
{\bf Theorem A.}
{\em
Let $0<d<p$ and $\rho$ be a $d$-Rouquier $\bar p$-core. Then 
$\Blo^{\rho,d}$ is Morita superequivalent to $\Blo^{\rho,0} \otimes (\Blo^{\varnothing,1} \swr \cT_d)$.
}

\vspace{2mm}

The analogous result for blocks defined over $\F$ was proved in \cite{KL}. Indeed, that paper proved that we have Morita superequivalences in a more general setting. Namely, from specific (RoCK) blocks of quiver Hecke superalgebras to `local' objects. We also use the term {\em RoCK block} to describe the block $\Blo^{\rho,d}$ in Theorem A. 
Moreover, the corresponding purely even subalgebra $\Blo^{\rho,d}_\0$ is a spin block of $\cO\tAi_{n}$, which is also referred to as a {\em RoCK block}.

Using Theorem A, we then go on to prove that $\Blo^{\rho,d}$ is not just derived equivalent but splendidly derived equivalent to its Brauer correspondent (see Corollary \ref{cor:spld}):

\vspace{2mm}
\noindent
{\bf Theorem B.}
{\em 
Let $0<d<p$ and $\rho$ be a $d$-Rouquier $\bar p$-core. Then 
$\Blo^{\rho,d}$ and $\Blo^{\rho,d}_{\0}$ are splendidly derived equivalent to their respective Brauer correspondents. In particular, Brou\'e's abelian defect group conjecture holds for RoCK blocks $\Blo^{\rho,d}$ and $\Blo^{\rho,d}_{\0}$ of $\cO\tSi_n$ and $\cO\tAi_n$ respectively.
}

\vspace{2mm}

This is an improvement over the derived equivalence constructed for the corresponding blocks defined over $\F$ in \cite{KL}.

Ebert, Lauda and Vera in \cite{ELV} as well as Brundan and the first author in \cite{BKBroue} 
independently proved that any spin block of an $\F \tSi_n$ with weight $d$ is derived equivalent to some RoCK block $\Blo^{\rho,d}$. Together with \cite{KL}, this completed the proof of Brou\'e's conjecture for the spin blocks of $\F \tSi_n$. Currently, the conjecture remains open for blocks defined over $\cO$.

The article is organized as follows: Section \ref{sec:prelim} consists of various preliminaries concerning combinatorics and general algebras. Section \ref{sec:grdd_supalg} states all the relevant results on superalgebras and supermodules. Section \ref{sec:spin_blocks} introduces the double covers of the symmetric group and their block theory. Section \ref{sec:RoCK} is where we first introduce RoCK blocks and in Section \ref{sec:weight_one} we analyse weight one RoCK blocks in detail. In Section \ref{sec:X_Y} the bisupermodule $\bX$, which will ultimately induce our desired Morita superequivalence, and the related bisupermodule $\bY$ are defined. Theorem A is proved in Section \ref{sec:main}, while Theorem B is proved in Section \ref{sec:vert_source}.

%\mainmatter

\section{Preliminaries}\label{sec:prelim}

\subsection{Generalities}
We denote $\N:=\Z_{\geq 0}$. Throughout this article $p$ is an odd prime, 
\begin{equation}\label{EEll}
\ell := (p-1)/2\quad\text{and}\quad I:=\{0,1,\dots,\ell\},
\end{equation}
and $(\K,\cO,\F)$ is a $p$-modular system. So $\K$ is a field of characteristic zero, $\cO$ a complete discrete valuation ring with field of fractions $\K$ and residue field $\F$, which is an algebraically closed field of characteristic $p$. We denote by $\overline{\phantom{x}}:\cO \to \F$ the natural projection.

%For the entire article we set .

Throughout the paper, we denote by $\cR$ either $\K$ or $\cO$ (we will rarely work over $\F$). 
We will at several points need the fact that $2$ is invertible is $\cR$ (certainly it is invertible in $\F$ and hence in $\cO$.) 
For $\cR$-modules $M$ and $\gN$, we set $\gM \otimes \gN:=\gM \otimes_{\cR} \gN$. (Starting with \S\ref{SSbM} we will use bold letters to distinguish specific fixed bi(super)modules from general ones.)

All $\cR$-algebras are assumed to be finitely generated as $\cR$-modules. A subalgebra of an $\cR$-algebra $A$ is always assumed to be unital.
 Since by assumption $A$ is finitely generated as an $\cR$-module, so is a subalgebra of $A$ (as $\cR$ is noetherian). 
Most algebras in this article will be $\cO$-algebras. Therefore, when we do not specify the underlying ring, it will be assumed that {\em algebra} refers to an $\cO$-algebra. If $A$ and $B$ are algebras, then we write $A \sim_{\Mor} B$ if $A$ and $B$ are Morita equivalent.

If $A$ is an $\cR$-algebra, then an $A$-module will always refer to a finitely generated, left $A$-module. Such modules satisfy the Krull-Schmidt property, see \cite[Corollary 4.6.7]{Lin5}. 
Since all $A$-modules are assumed to be finitely generated they are also finitely generated as $\cR$-modules and so a submodule of an $A$-module is finitely generated as an $\cR$-module. 
For $\cR$-algebras $A$ and $B$, we identify $(A,B)$-bimodules and $(A\otimes B^\op)$-modules; in particular, all $(A,B)$-bimodules are  assumed to be finitely generated. 
If $\gM$ is isomorphic to a direct summand of $\gN$, for (bi)modules $\gM$ and $\gN$, we write $\gM \mid \gN$.

A {\em projective cover} $(\gP,\varphi)$ of an $A$-module $\gM$ is a projective $A$-module $\gP$ and a surjective $A$-module homomorphism $\varphi:\gP \to \gM$ such that no proper direct summand of $\gP$ surjects onto $\gM$. Our conditions on $A$ ensure that it is a semi-perfect ring, hence projective covers of (finitely generated) $A$-modules always exist and are unique up to isomorphism. We set $\Omega_A(\gM)$ to be the {\em Heller translate} of $\gM$. That is, the kernel of $\varphi$. We inductively define powers of the Heller translate $\Omega_A^n$, for $n \in \NN$. Finally, $\Omega_A^0(\gM)$ is defined as the direct sum of all the non-projective, indecomposable summands of $\gM$.

For finite groups $G$ and $H$, any $(\cR G,\cR H)$-bimodule $\gM$ can be thought of as an $\cR(G \times H)$-module via $(g,h).m = g.m.h^{-1}$, for all $g\in G$, $h\in H$ and $m\in \gM$. Similarly, any $\cR(G \times H)$-module $\gM$ can be viewed as an $(\cR G,\cR H)$-bimodule via $g.m.h = (g,h^{-1}).m$, for all $g\in G$, $h\in H$ and $m\in \gM$. We will use these two concepts completely interchangeably. For example, a {\em projective} $(\cR G,\cR H)$-bimodule means that it is projective as  an $\cR(G \times H)$-module, a vertex of an indecomposable $(\cR G,\cR H)$-bimodule means its vertex as an $\cR(G \times H)$-module, etc.

\subsection{Split semisimple algebras}\label{sec:SSA}

Let $A$ be a (finite dimensional) semisimple $\K$-algebra such that every irreducible $A$-module is absolutely irreducible. We say $A$ is {\em split semisimple}. In other words, $A$ is isomorphic to a finite direct sum of matrix algebras defined over $\K$. Throughout this article we are primarily concerned with blocks of finite groups and Brauer tree algebras, both of which satisfy the above property, provided $\K$ contains enough roots of unity.

We set $\Grot(A)$ to be the Grothendieck group of (finite dimensional) $A$-modules and $\Grot^{+}(A)$ the subset of classes in $\Grot(A)$ that represent actual $A$-modules.
So
\begin{align*}
\Grot(A) = \Big\{\sum_{\gM} a_{\gM}[\gM]\mid a_{\gM} \in \Z \Big\} \quad \text{and} \quad \Grot^{+}(A) = \Big\{\sum_{\gM} a_{\gM}[\gM]\mid a_{\gM} \in \N \Big\},
\end{align*}
where $\gM$ runs over the isomorphism classes of irreducible $A$-modules. We refer to the elements of $\Grot^{+}(A)$ as the {\em characters of $A$} and call $[\gM]$ an {\em irreducible character} if $\gM$ is an irreducible $A$-module. We denote by $\Irr(A)$ the set of irreducible characters of $A$. We will also often use $\Z\Irr(A)$ and $\N\Irr(A)$ to mean $\Grot(A)$ and $\Grot^{+}(A)$ respectively.

Let $B$ be an $\cO$-free $\cO$-algebra such that $\K B := \K \otimes_{\cO} B$ is split semisimple. We say $B$ is {\em $\K$-split semisimple}. Set $\Irr(B) := \Irr(\K B)$ and
\begin{align*}
\prj(B) &:= \{[\K \otimes_{\cO} \gM]\in \Grot^+(\K B)\mid  \gM \text{ is an indecomposable projective }B\text{-module}\},
\\
\N\prj(B) &:= \{[\K \otimes_{\cO} \gM]\in \Grot^+(\K B)\mid  \gM \text{ is a projective }B\text{-module}\}.
\end{align*}
%Note that $\prj(B)$ and $\N\prj(B)$ lie in $\Grot^{+}(\K B)$.
For $\chi,\psi \in \N\Irr(B)$, we write 
$$\chi \geq_{{}_{\Irr(B)}} \psi\quad\text{and}\quad \chi \geq_{{}_{\prj(A)}} \psi
$$ 
to mean $\chi - \psi \in \N\Irr(B)$ and $\chi - \psi \in \N\prj(B)$, respectively.

Let $B,C$ be  $\K$-split semisimple algebras and $\gN$ a $(B,C)$-bimodule. Then the functor $\gN \otimes_C ?$ induces a ($\Z$-)linear function $\Grot(\K C) \to \Grot(\K B)$ that, by an abuse of notation, we also denote $\gN \otimes_C ?$. 
If $C$ is a subalgebra of $B$, %that is also $\K$-split semisimple,
then we denote by
\begin{align*}
\downarrow^B_C: \Grot(\K B) \to \Grot(\K C) \qquad \text{and} \qquad \uparrow_C^B: \Grot(\K C) \to \Grot(\K B)
\end{align*}
the linear functions induced by the functors $\Res^B_C$ and $\Ind_C^B$ respectively. Similar notation applies for algebras over $\K$---for example if $B$ and $C$ are split semisimple $\K$-algebras and $\gN$ a $(B,C)$-bimodule, then $\gN \otimes_C ?$ induces a ($\Z$-)linear function $\Grot(C) \to \Grot(B)$ that we also denote $\gN \otimes_C ?$.

If $G$ is a finite group and $H$ a subgroup, we write
\begin{align*}
\downarrow^G_H: \Grot(\K G) \to \Grot(\K H) \qquad \text{and} \qquad \uparrow_H^G: \Grot(\K H) \to \Grot(\K G)
\end{align*}
for the functions induced by the functors $\Res^G_H$ and $\Ind_H^G$ respectively. Furthermore, if $b$ is a block idempotent of $\cO G$ and $c$ of $\cO H$, then
\begin{align*}
\downarrow^{G,b}_{H,c}: \Grot(\K Gb) \to \Grot(\K Hc) \qquad \text{and} \qquad \uparrow_{H,c}^{G,b}: \Grot(\K Hc) \to \Grot(\K Gb)
\end{align*}
will denote the functions induced by truncated restriction $\Res^{G,b}_{H,c}$ and truncated induction $\Ind_{H,c}^{G,b}$ respectively. That is, the functions induced by $c\cO Gb \otimes_{\cO Gb} ?$ and $b\cO Gc \otimes_{\cO Hc} ?$.

\subsection{Group algebras}\label{sec:grp_alg}
Throughout the article we assume standard facts about vertices and sources, see, for example, \cite[Chapter 5]{Lin5}. 

At several stages %throughout this article 
we will use the Brauer homomorphism, see \cite[\S5.4]{Lin5}. Let $G$ be a finite group and $P$ a $p$-subgroup. The {\em Brauer homomorphism} $\Br_P$ is defined by
\begin{align*}
\Br_P: Z(\cO G) \to Z(\F C_G(P)), \ 
\sum_{g \in G}\al_g g  \mapsto \sum_{g \in C_G(P)}{\bar \al_g} g.
\end{align*}

For any finite group $G$, when we refer to a {\em block} $\cO Gb$, we implicitly mean that $b$ is a block idempotent of $\cO G$. We will use Alperin's definition of {\em defect group}, see \cite[Chapter IV]{Alp}. Namely, $D$ is the defect group of a block $\cO Gb$ if the block has vertex $\Delta D$ when considered as an $\cO (G \times G)$-module.

Let $D$ be a fixed finite $p$-group. For any subgroup $P \leq D$ and group monomorphism $\varphi:P \hookrightarrow D$, we set 
\begin{equation}\label{EDePhi}
\Delta_\varphi P := \{(x,\varphi(x))\mid x\in P\} \leq D \times D.
\end{equation}
 If $\varphi$ is the identity on $P$, we just write $\Delta P$.

\begin{Lemma}\label{lem:bimod_vert}
Let $G$, $H$ and $J$ be finite groups with a common $p$-subgroup $D$, and $\gM$ be an indecomposable $(\cO G,\cO H)$-bimodule with vertex $\Delta_\varphi P$ for some $P \leq D$ and $\phi:P \hookrightarrow D$. 
\begin{enumerate}
%\item Let $M$ be an $\cO G$-$\cO H$-bimodule with vertex $\Delta_\varphi Q$, for some $Q \leq D$ and $\phi:Q \to D$. Then $M \otimes_{\cO H}V$ is a direct sum of $\cO G$-modules with vertices contained in $Q$, for all $\cO H$-modules $V$.
%\item
%\begin{enumerate}
\item If\, $\gV$ is an indecomposable $\cO H$-module with vertex $Q$, then $\gM \otimes_{\cO H}\gV$ is a direct sum of indecomposable $\cO G$-modules each with vertex contained in $P \cap \varphi^{-1}({}^h Q \cap D)$, for some $h\in H$. 
%(Here, $\varphi^{-1}({}^h Q)$ means $\varphi^{-1}({}^h Q \cap D)$.)

\item If\, $\gN$ is an indecomposable $(\cO H,\cO J)$-bimodule with vertex $\Delta_\psi Q$ for some $Q \leq D$ and $\psi:Q \hookrightarrow D$, then $\gM \otimes_{\cO H}\gN$ is a direct sum of indecomposable $(\cO G,\cO J)$-bimodules each with vertex of the form $\Delta_\vartheta R$, for some $R \leq P$ and $\vartheta: R \hookrightarrow D$, with $\vartheta(R) \leq \psi(Q)$.
\end{enumerate}
\end{Lemma}

\begin{proof}
(i) We have $\gM\mid \Ind_{\Delta_\varphi P}^{G \times H}\gU$ for some indecomposable $\cO\Delta_\varphi P$-module $\gU$, and $\gV\mid \Ind_Q^H\gW$ for some indecomposable $\cO Q$-module $\gW$. By Mackey's Theorem, we have that $\Res^H_{\varphi(P)}\Ind_Q^H\gW$ is a direct sum of indecomposable $\cO \varphi(P)$-modules each with vertex contained in $\varphi(P) \cap {}^h Q$ for some $h\in H$. 

We now prove that we have $\cO G$-module isomorphisms 
\begin{align}
\begin{split}\label{algn:mod_isom}
(\Ind_{\Delta_\varphi P}^{G \times \varphi(P)}\gU) \otimes_{\cO \varphi(P)} \Ind_{\varphi(R)}^{\varphi(P)}\gZ & \cong (\Ind_{\Delta_\varphi P}^{G \times \varphi(P)}\gU) \otimes_{\cO \varphi(R)} \gZ \cong \Ind_R^G(\gU \otimes \gZ),
\end{split}
\end{align}
for any $R \leq P$ and $\cO \varphi(R)$-module $\gZ$, where $\gU \otimes \gZ$ is given the structure of an $\cO R$-module via
\begin{align*}
r\cdot(u \otimes z) := ru\varphi(r)^{-1} \otimes \varphi(r) z
\qquad(\text{for $r\in R$, $u\in \gU$, $z\in \gZ$}).
\end{align*}
%for all $r\in R$, $u\in \gU$ and $z\in \gZ$. 
The first isomorphism in (\ref{algn:mod_isom}) is immediate. For the second, we need only observe that we have an $\cO R$-module isomorphism given by
\begin{align*}
(\Ind_{\Delta_\varphi R}^{R \times \varphi(R)}\Res^{\Delta_\varphi P}_{\Delta_\varphi R}\gU) \otimes_{\cO \varphi(R)} \gZ & \to \gU \otimes \gZ\\
(r_1 \otimes u \otimes \varphi(r_2)) \otimes z & \mapsto r_1 u \varphi(r_1)^{-1} \otimes \varphi(r_1 r_2)z \\
u \otimes z & \mapsfrom u \otimes z,
%\qquad\qquad\qquad\qquad\qquad
%(\text{for $r_1,r_2\in R$, $u\in \gU$, $z\in \gZ$}),
\end{align*}
%for all $r_1,r_2\in R$, $u\in \gU$ and $z\in \gZ$ 
and an $\cO (G \times \varphi(R))$-module isomorphism
\begin{align*}
\Ind_{\Delta_\varphi R}^{G \times \varphi(R)}\Res^{\Delta_\varphi P}_{\Delta_\varphi R}\gU \cong \Res^{G \times \varphi(P)}_{G \times \varphi(R)}\Ind_{\Delta_\varphi P}^{G \times \varphi(P)}\gU,
\end{align*}
given by the Mackey decomposition formula.

We have that
\begin{align*}
\gM \otimes_{\cO H} \gV\mid \ & (\Ind_{\Delta_\varphi P}^{G \times H}\gU) \otimes_{\cO H} \Ind_Q^H\gW \\
\cong\ & (\Ind_{\Delta_\varphi P}^{G \times \varphi(P)}\gU) \otimes_{\cO \varphi(P)} \Res^H_{\varphi(P)}\Ind_Q^H\gW,
\end{align*}
which, by (\ref{algn:mod_isom}) and the comments preceding it, is a direct sum of modules each with vertex contained in
\begin{align*}
\varphi^{-1}(\varphi(P) \cap {}^h Q) = P \cap \varphi^{-1}({}^h Q),
\end{align*}
for some $h\in H$.

(ii) Suppose $\gN$ lies in the block $\cO Jb$ when considered as a right $\cO J$ module. Certainly $\cO Jb$ is naturally an $(\cO J,\cO J)$-bimodule. We consider the $(\cO(G \times J),\cO(H \times J))$-bimodule $\gM \otimes \cO Jb$. Say $\cO Jb$ has defect group $S \leq J$. %(Certainly $N$ is relatively $(H \times S)$-projective and so we may choose $S$ such that $\psi(Q) \leq S$.)
Then $\cO Jb$ has vertex $\Delta S$ and $\gM \otimes \cO Jb$ has vertex $\Delta_{\varphi \times \Id_S}(P \times S) \leq (D \times S) \times (D \times S)$.

We have the following isomorphism of $\cO(G \times J)$-modules
\begin{align*}
(\gM \otimes \cO Jb) \otimes_{\cO(H \times J)} \gN &\cong \gM \otimes_{\cO H} \gN\\
(m \otimes jb) \otimes n &\mapsto m \otimes nbj^{-1}\\
(m \otimes b) \otimes n &\mapsfrom m \otimes n.
\end{align*}
%for all $m \in \gM$, $n \in \gN$ and $j \in J$. 
Therefore, by part (i), $\gM \otimes_{\cO H} \gN$ is a direct sum of indecomposable $(\cO G,\cO J)$-bimodules each with vertex contained in
\begin{align*}
(P \times S) \cap (\varphi \times \Id_S)^{-1}({}^{(h,j)} \Delta_\psi Q),
\end{align*}
for some $h\in H$ and $j \in J$. Since $\gM \otimes_{\cO H} \gN$ is a right $\cO J$-module each vertex is contained in
\begin{align*}
&{}^{(1,j^{-1})}[(P \times S) \cap (\varphi \times \Id_S)^{-1}({}^{(h,j)} \Delta_\psi Q)] \\
= & (P \times {}^{j^{-1}}S) \cap (\varphi \times \Id_{{}^{j^{-1}}S})^{-1}({}^{(h,1)} \Delta_\psi Q),
\end{align*}
for some $h\in H$ and $j \in J$. Setting $R'=P \cap \varphi^{-1}({}^hQ)$ and $\vartheta: R' \hookrightarrow D$, $x \mapsto \psi({}^{h^{-1}}\phi(x))$ we now have
\begin{align*}
(P \times {}^{j^{-1}}S) \cap (\varphi \times \Id_{{}^{j^{-1}}S})^{-1}({}^{(h,1)} \Delta_\psi Q) \leq \Delta_{\vartheta'}R'.
\end{align*}
Since every subgroup of $\Delta_{\vartheta'}R'$ is of the form $\Delta_\vartheta R$, for suitably defined $R$ and $\vartheta$, the claim follows.
\end{proof}

In the next lemma $\Tr^G_H:Z(\cO H)\to Z(\cO G)$ is the relative trace, see e.g. \cite[\S2.5]{Lin5}.

\begin{Lemma}\label{lem:unstable_block}
Let $G$ be a finite group, $H$ a normal subgroup and $\cO He$ a block with defect group $D$. If $C_G(e)=H$, then $\cO Gf$ is Morita equivalent to $\cO He$ with defect group $D$, where $f := \Tr_H^G(e)$. In particular, $f$ is a block idempotent of $\cO G$.
\end{Lemma}

\begin{proof}
The Morita equivalence is just a special case of \cite[Theorem C]{Ku1}. That $\cO Gf$ has defect group $D$ follows from the following isomorphisms of $\cO(G \times G)$-modules and $\cO(H \times H)$-modules respectively
$$
\Ind_{H \times H}^{G \times G}(\cO He) \cong \cO Gf, \qquad \Res^{G \times G}_{H \times H}(\cO Gf) \cong \bigoplus_{g_1,g_2 \in G/H} g_1\cO He g_2.
$$
The first isomorphism implies that $\cO Gf$ is relatively $\Delta D$-projective, while the second dictates that the vertex of $\cO Gf$ cannot be any smaller that $\Delta D$.
\end{proof}

\begin{Lemma}\label{lem:ind_M_H^G}
Let $G$ be a finite group, $H$ a normal subgroup and $\gM$ an indecomposable $\cO H$-module with vertex $D$. If $N_G(D) \leq H$, then $\Ind_H^G \gM$ is indecomposable.
\end{Lemma}

\begin{proof}
Consider the decomposition
\begin{align}\label{algn:decomp_M}
\Res^G_H\Ind_H^G \gM \cong \bigoplus_{g \in G/H}g\gM
\end{align}
of $\Res^G_H\Ind_H^G\gM$ into indecomposable $\cO H$-modules. Since $M$ has vertex $D$, $g\gM$ has vertex ${}^g D$. Let $g_1,g_2 \in G$ and suppose ${}^{g_1} D$ is conjuagte to ${}^{g_2} D$ in $H$. In other words ${}^{g_1} D = {}^{h g_1} D$, for some $h \in H$. Then $g_1^{-1}h g_2 \in N_G(D)\leq H$ and so $g_1^{-1} g_2 \in H$. We have now proved that all summands in (\ref{algn:decomp_M}) are pairwise non-isomorphic. The claim now follows from \cite[\S5, Propositon 2]{Ward}.
\end{proof}

We now examine a specific application of Lemma \ref{lem:ind_M_H^G}. If $G$ is a finite group and $H$ a normal subgroup, we define the subgroup
\begin{align}\label{algn:marcus_not}
(G \times G)_{G/H} := \{(g_1,g_2) \in G \times G\mid g_1 H = g_2 H\} \leq G \times G.
\end{align}

\begin{Lemma}\label{lem:vert_blocks}
Let $G$ be a finite group, $H$ a normal subgroup and $\gM$ an indecomposable $\cO(H \times H)$-module, with vertex $\Delta D$. 
\begin{enumerate}
\item If $C_G(D) \leq H$, then $\Ind_{H \times H}^{G \times H}\gM$ is an indecomposable $\cO(G \times H)$-module.
\item If $\gM$ extends to an $\cO (G \times G)_{G/H}$-module, then
\begin{align*}
\Res^{G \times G}_{G \times H} \Ind_{(G \times G)_{G/H}}^{G \times G} \gM \cong \Ind_{H \times H}^{G \times H}\gM.
\end{align*}
\end{enumerate}
In particular, if the hypotheses of both\, {\rm (i)} and\, {\rm (ii)} hold, then $\Ind_{(G \times G)_{G/H}}^{G \times G}\gM$ is an indecomposable $\cO(G \times G)$-module.

\begin{enumerate}
\item[(iii)] If $e \in Z(\cO G)$ is a block idempotent of $\cO H$ such that $\cO He$ has defect group $D$ and $C_G(D) \leq H$, then $e$ is a block idempotent of $\cO G$.
\end{enumerate}
\end{Lemma}

\begin{proof}
(i) Suppose $(g,h) \in G \times H$ normalizes $\Delta D$. Then
\begin{align*}
\Delta D = {}^{(g,h)}\Delta D = \{({}^g x,{}^h x)\mid x \in D\} = {}^{(g,g)}\{(x,{}^{g^{-1}h} x)\mid x \in D\}.
\end{align*}
In particular, $g^{-1}h \in C_G(D) \leq H$ and so $N_{G \times H}(\Delta D) \leq H \times H$. The claim now follows from Lemma \ref{lem:ind_M_H^G}.

(ii) This is just the Mackey decomposition formula once we note that
\begin{align*}
(G \times H) \cap (G \times G)_{G/H} = H \times H
\end{align*}
and
\begin{align*}
|(G \times H)\backslash(G \times G)/(G \times G)_{G/H}|=1.
\end{align*}

(iii) This just follows from (i) and (ii), as $\cO He$ is certainly an $\cO(G \times G)_{G/H}$-module such that
$
\cO Ge \cong \Ind_{(G \times G)_{G/H}}^{G \times G}(\cO He).
$
\end{proof}

\subsection{Brauer trees}\label{sec:Brauer_trees}

We assume the reader is familiar with Brauer trees, in particular the fact that the basic algebra of a block of a finite group with cyclic defect is isomorphic to an appropriately constructed Brauer tree algebra (see \cite[Proposition 3.10]{Lin}). We refer the reader to the same result for the definition of the Brauer tree algebra defined over $\cO$. There, it is stated only for basic algebras of blocks with cyclic defect but the definition works for arbitrary Brauer trees. We briefly state some key facts used in this article. A good reference is \cite[$\S$11.1]{Lin6}.

Let $\mathscr{T}$ be a Brauer tree with associated Brauer tree algebra $A$ defined over $\cO$. It is implicitly assumed that the nodes of $\mathscr{T}$ are labeled by $\Irr(A)$. Similarly, the edges of $\mathscr{T}$ are labeled by the isomorphism classes of irreducible $\F A$-modules. A node with multiplicity greater than one (of which there will be at most one) will label a number of irreducible characters equal to said multiplicity. Finally, the projective cover (as an $A$-module) of an irreducible $\F A$-module has character equal to the sum of characters associated to the nodes at either end of the appropriate edge of $\mathscr{T}$.

It is our convention with Brauer trees that a node has multiplicity one unless it has some number of rings around it. In this case the multiplicity is always one more than the number of rings. In this article we will only consider Brauer trees with multiplicities at most two.

Recall the notation (\ref{EEll}). 
Consider the following Brauer tree:

\begin{align}\label{Zig_Brauer_tree}
\begin{braid}\tikzset{baseline=0mm}
\coordinate (1) at (0,0);
\coordinate (2) at (3,0);
\coordinate (3) at (6,0);
\coordinate (4) at (9,0);
\coordinate (5) at (12,0);
\coordinate (6) at (15,0);
\coordinate (7) at (18,0);
\coordinate (8) at (21,0);
\coordinate (9) at (24,0);
\coordinate (10) at (27,0);
\coordinate (11) at (30,0);
\draw[thin] (1) -- (2);
\draw[thin] (2) -- (3);
\draw(7.5,0) node {$\cdots$};
\draw[thin] (4) -- (5);
\draw[thin] (5) -- (6);
\draw[thin] (6) -- (7);
\draw[thin] (7) -- (8);
\draw(22.5,0) node {$\cdots$};
\draw[thin] (9) -- (10);
\draw[thin] (10) -- (11);
\node at (1)[circle,fill,inner sep=1.5pt]{};
\node at (1) [above] {$\ell^+$};
\node at (2)[circle,fill,inner sep=1.5pt]{};
\node at (2) [above] {$(\ell-1)^+$};
\node at (5)[circle,fill,inner sep=1.5pt]{};
\node at (5) [above] {$1^+$};
\node at (6)[circle,fill,inner sep=1.5pt]{};
\node at (6) [above] {$0$};
\node at (7)[circle,fill,inner sep=1.5pt]{};
\node at (7) [above] {$1^-$};
\node at (10)[circle,fill,inner sep=1.5pt]{};
\node at (10) [above] {$(\ell-1)^-$};
\node at (11)[circle,fill,inner sep=1.5pt]{};
\node at (11) [above] {$\ell^-$};
\end{braid}
\end{align}

\vspace{5mm}
\noindent
We use $\Zag_\ell$ to signify the corresponding Brauer tree algebra defined over $\cO$. Note that $\F \otimes_{\cO} \Zag_{\ell}$ is the algebra denoted $\Zag_\ell$ in \cite{KL}. We denote by $\chi_{i^{(\pm)}} \in \Irr(\Zag_\ell)$ the element corresponding to the node $i^{(\pm)}$, for $i\in I$.
% We denote by $\psi_{i^{\pm}} \in \Grot_{\F}(\Zag_\ell)$ the element corresponding to the edge $(i^{(\pm)},(i+1)^{(\pm)})$.

Next, consider the following Brauer tree:
\vspace{2mm}
\begin{align}\label{Zag_Brauer_tree}
\begin{braid}\tikzset{baseline=0mm}
\coordinate (1) at (0,0);
\coordinate (A) at (0,0.4);
\coordinate (2) at (3,0);
\coordinate (3) at (6,0);
\coordinate (4) at (9,0);
\coordinate (5) at (12,0);
\coordinate (6) at (15,0);
\draw[thin] (1) -- (2);
\draw[thin] (2) -- (3);
\draw(7.5,0) node {$\cdots$};
\draw[thin] (4) -- (5);
\draw[thin] (5) -- (6);
\node at (1)[circle,fill,inner sep=1.5pt]{};
\node at (A) [above] {$0$};
\node at (2)[circle,fill,inner sep=1.5pt]{};
\node at (2) [above] {$1$};
\node at (5)[circle,fill,inner sep=1.5pt]{};
\node at (5) [above] {$\ell-1$};
\node at (6)[circle,fill,inner sep=1.5pt]{};
\node at (6) [above] {$\ell$};
\draw (1) circle (6mm);
\end{braid}
\end{align}

\vspace{5mm}
\noindent
We use $\Zig_\ell$ to signify the corresponding Brauer tree algebra defined over $\cO$. Note that $\F \otimes_{\cO} \Zig_{\ell}$ is isomorphic to the algebra denoted $\Zig_\ell$ in \cite{KL}. We denote by $\chi_i\in \Irr(\Zig_\ell)$ the element corresponding to the node $i$, for $1\leq i\leq \ell$ and by $\chi_{0^+},\chi_{0^-}\in \Irr(\Zig_\ell)$ the pair corresponding to the node $0$. %We denote by $\psi_i \in \Grot_{\F}(\Zig_\ell)$ the element corresponding to the edge $(i,(i+1))$.

Recall from $\S$\ref{sec:SSA} that, if $\gN$ is an $(A,B)$-bimodule, then we use $\gN \otimes_B ?$ to denote the corresponding function $\Grot(\K B) \to \Grot(\K A)$.

%\textcolor{green}{Changed the lemma below slightly. We actually need $3\ell$ rather than $\ell$.}

\begin{Lemma}\label{lem:Zig_Zag}We have the following relationships:
\begin{enumerate}
\item If  $0\leq n\leq 4 \ell-1$ then  
\begin{align*}
\Omega_{\Zag_{\ell} \otimes \Zag_{\ell}^{\op}}^n(\Zag_{\ell}) \otimes_{\Zag_{\ell}} \chi_{\ell^+} \geq_{{}_{\Irr(\Zag_{\ell})}}
\begin{cases}
\chi_{(\ell-n)^+} & \text{if } 0\leq n \leq \ell-1,\\
\chi_0 & \text{if } n = \ell,\\
\chi_{(n-\ell)^-} & \text{if } \ell+1\leq n \leq 2\ell,\\
\chi_{(3\ell-n)^-} & \text{if } 2\ell+1\leq n \leq 3\ell-1,\\
\chi_0 & \text{if } n = 3\ell,\\
\chi_{(n-3\ell)^+} & \text{if } 3\ell+1 \leq n \leq 4\ell-1.
\end{cases}
\end{align*}

Moreover, if $1 \leq i\leq \ell$ then
\begin{align*}
\Omega_{\Zag_{\ell} \otimes \Zag_{\ell}^{\op}}^{3\ell}(\Zag_{\ell}) \otimes_{\Zag_{\ell}} \chi_{i^+} \geq_{{}_{\prj(\Zag_{\ell})}}
\begin{cases}
\chi_{(\ell-i)^+} \text{ and } \chi_{(\ell-i)^-} & \text{if } 1\leq i \leq \ell-1,\\
\chi_0 & \text{if } i = \ell,
\end{cases}
\end{align*}
\begin{align*}
\Omega_{\Zag_{\ell} \otimes \Zag_{\ell}^{\op}}^{3\ell}(\Zag_{\ell}) \otimes_{\Zag_{\ell}} \chi_{i^-} \geq_{{}_{\prj(\Zag_{\ell})}}
\begin{cases}
\chi_{(\ell-i)^+} \text{ and } \chi_{(\ell-i)^-} & \text{if } 1\leq i \leq \ell-1,\\
\chi_0 & \text{if } i = \ell,
\end{cases}
\end{align*}
and
\begin{align*}
\Omega_{\Zag_{\ell} \otimes \Zag_{\ell}^{\op}}^{3\ell}(\Zag_{\ell}) \otimes_{\Zag_{\ell}} \chi_0 \geq_{{}_{\prj(\Zag_{\ell})}} \chi_{\ell^+} \text{ and }\chi_{\ell^-}.
\end{align*}

\item If $0\leq n\leq 2 \ell-1$ then
\begin{align*}
\Omega_{\Zig_{\ell} \otimes \Zig_{\ell}^{\op}}^n(\Zig_{\ell}) \otimes_{\Zig_{\ell}} \chi_\ell \geq_{{}_{\Irr(\Zig_{\ell})}}
\begin{cases}
\chi_{\ell-n} & \text{if } 0\leq n \leq \ell-1,\\
\chi_{0^+} + \chi_{0^-} & \text{if } n = \ell,\\
\chi_{n-\ell} & \text{if } \ell+1\leq n \leq 2\ell-1.
\end{cases}
\end{align*}
Moreover, if $1 \leq i\leq \ell$ then
\begin{align*}
\Omega_{\Zig_{\ell} \otimes \Zig_{\ell}^{\op}}^{\ell}(\Zig_{\ell}) \otimes_{\Zig_{\ell}} \chi_i \geq_{{}_{\prj(\Zig_{\ell})}}
\begin{cases}
\chi_{\ell-i} & \text{if } 1\leq i \leq \ell-1,\\
\chi_{0^+} + \chi_{0^-} & \text{if } i = \ell,
\end{cases}
\end{align*}
and
\begin{align*}
\Omega_{\Zig_{\ell} \otimes \Zig_{\ell}^{\op}}^{\ell}(\Zig_{\ell}) \otimes_{\Zig_{\ell}} (\chi_{0^+} + \chi_{0^-}) \geq_{{}_{\prj(\Zig_{\ell})}} \chi_{\ell}.
\end{align*}
\end{enumerate}
\end{Lemma}

\begin{proof}
%We focus on part (i). Part (ii) is proved similarly.
(i) By \cite[Theorem 2]{Green}, there exists a sequence $v_0, \dots, v_{4\ell-1}$ of vertices of the Brauer tree of $\Zag_{\ell}$ (possibly with repeats) and $\cO$-free $\Zag_{\ell}$-modules $\ttM_0, \dots, \ttM_{4\ell-1}$ such that $[\ttM_i] = \chi_{v_i}$, for all $0\leq i\leq 4\ell - 1$, with the following properties: $(v_i,v_{i+1})$ is an edge of the Brauer tree, for all $i \in \Z$ (where the subscripts are considered modulo $4\ell$), each edge occurs exactly twice in the sequence
\begin{align*}
(v_0,v_1),(v_1,v_2),\dots,(v_{4\ell-2},v_{4\ell-1}), (v_{4\ell-1},v_0)
\end{align*}
and $\Omega_{\Zag_{\ell}} \ttM_i \cong \ttM_{i+1}$ for all $i \in \Z$ (where again the subscripts are considered modulo $4\ell$).

Since each edge must occur exactly twice, it is clear that our sequence must be some cyclic permutation of
\begin{align*}
\ell^+,(\ell-1)^+, \dots ,1^+,0,1^-, \dots ,(\ell-1)^-, \ell^-, (\ell-1)^-, \dots, 1^-,0,1^+, \dots,(\ell-1)^+.
\end{align*}
In fact, by shifting, we can assume our sequence is exactly as above. In other words, the sequence walks from one end of the Brauer tree to the other and then back again.

Next note that, for all $\Zag_{\ell}$-modules $\gM$ and $n \in \N$,
\begin{align*}
\Omega^n_{\Zag_{\ell} \otimes \Zag_{\ell}^{\op}}(\Zag_{\ell})\otimes_{\Zag_{\ell}} \gM \cong \Omega^n_{\Zag_{\ell}} \gM \oplus \gP,
\end{align*}
for some projective $\Zag_{\ell}$-module $\gP$. This can be seen easily by taking a projective resolution of $\Zag_{\ell}$ as a $\Zag_{\ell} \otimes \Zag_{\ell}^{\op}$-module and then applying $? \otimes_{\Zag_{\ell}} \gM$. Therefore, for all $i,n \in \N$,
\begin{align}\label{algn:Brauer_walk}
\Omega^n_{\Zag_{\ell} \otimes \Zag_{\ell}^{\op}}(\Zag_{\ell})\otimes_{\Zag_{\ell}} \ttM_i \cong \ttM_{i+n} \oplus \ttP_{i,n},
\end{align}
for some projective $\Zag_{\ell}$-module $\ttP_{i,n}$. In particular, for all $n \in \N$,
\begin{align*}
\Omega^n_{\Zag_{\ell} \otimes \Zag_{\ell}^{\op}}(\Zag_{\ell})\otimes_{\Zag_{\ell}} \ttM_0 \cong \ttM_n \oplus \ttP_{0,n}.
\end{align*}
The first claim now follows. The second claim also follows immediately from (\ref{algn:Brauer_walk}) once we have noted that all the nodes that are not $\ell^{\pm}$ occur twice in our sequence. For example, for all $1\leq i \leq \ell-1$, since $v_{\ell-i} = v_{3\ell + i} = i^+$ and $v_{4\ell-i} = (\ell - i)^+, v_{6\ell+i} = v_i = (\ell - i)^-$, we have
\begin{align*}
\Omega_{\Zag_{\ell} \otimes \Zag_{\ell}^{\op}}^{3\ell}(\Zag_{\ell}) \otimes_{\Zag_{\ell}} \chi_{i^+} \geq_{{}_{\prj(\Zag_{\ell})}} \chi_{(\ell-i)^+} \quad \text{and} \quad \chi_{(\ell-i)^-}.
\end{align*}

(ii) The application of \cite[Theorem 2]{Green} gives a sequence $\ttM_0, \dots, \ttM_{2\ell-1}$ of $\cO$-free $\Zig_{\ell}$-modules such that
\begin{align*}
[\ttM_i] =
\begin{cases}
\chi_{\ell-i}&\text{if }0\leq i\leq \ell-1\\
\chi_{0^+} + \chi_{0^-} &\text{if }i = \ell\\
\chi_{i-\ell}&\text{if }\ell+1\leq i\leq 2\ell-1
\end{cases}
\end{align*}
and $\Omega_{\Zig_{\ell}} \ttM_i \cong \ttM_{i+1}$ for all $i \in \Z$ (where the subscripts are considered modulo $2\ell$). The proof now proceeds as in part (i).
\end{proof}

\subsection{Combinatorics}\label{sec:comb}

Let $n \in \N$. We use $[n]$ to signify the set $\{1,\dots,n\}$. We denote by $\Par(n)$ the set of partitions of $n$ and by $\Par_0(n)$ the set of strict partitions of $n$, i.e. partitions of $n$ without repeated parts. In addition, we set $\Par := \bigsqcup_{n\in \N}\Par(n)$ and $\Par_0 := \bigsqcup_{n\in \N}\Par_0(n)$. For any partition $\la$, we set $h(\la):=\max\{k\mid \la_k>0\}$ to be the {\em length} of $\la$. If, in addition, $\mu \in \Par$, we write $\mu \subseteq \la$ if $h(\mu) \leq h(\la)$ and $\mu_i \leq \la_i$, for all $1\leq i \leq h(\mu)$.

For a partition $\la$, we denote by $[\la]$ the {\em Young diagram} of $\la$, which consists of the {\em boxes} $(i,j)\in \Z_{>0}\times \Z_{>0}$ satisfying the following conditions:
$$
[\la]:=\{(i,j) \in \Z_{>0}\times \Z_{>0}\mid i \leq h(\la) \text{ and }1 \leq j \leq \la_i\}.
$$
For a {\em strict}\, partition $\la$ it is often more natural to work with its {\em shifted diagram}, which consists of the boxes satisfying the following conditions: 
\begin{align*}
\sh[\la] := \{(i,j) \in \Z_{>0}\times \Z_{>0}\mid i \leq h(\la)
\ \text{and}\ 
i\leq j \leq \la_i + i -1 
\}.
\end{align*}
For example, if $\la = (6,4,2,1) \in \Par_0(13)$, then the shifted diagram of $\la$ is
\vspace{2mm}
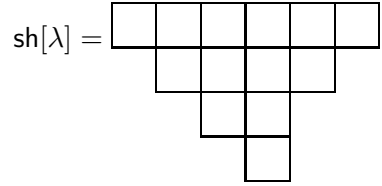
\begin{align*}
\sh[\la]=\begin{ytableau}
\, & & & & & \cr 
\none & & & & \cr
\none & \none & & \cr
\none & \none & \none & \cr
\end{ytableau}
\end{align*}

\vspace{2mm}
Let $\la\in\Par(n)$. 
We define $K_\la$ to be the number of bijections $T:[n] \to [\la]$ such that $T([m])$, for $1\leq m \leq n$, is always the Young diagram of a partition. 
We also denote by $\Par(\la)^{+1}$ the set of all $\mu \in \Par(n+1)$ such that $[\mu]$ is obtained by adding a box to $[\la]$. Similarly, we denote by $\Par(\la)^{-1}$ the set of all $\mu \in \Par(n-1)$ such that $[\mu]$ is obtained by removing a box from $[\la]$.

Let $\la \in \Par_0(n)$. We define $K_\la'$ to be the number of bijections $T:[n] \to \sh[\la]$ such that $T([m])$, for $1 \leq m \leq n$, is always the shifted diagram of a strict partition. 
We also denote by $\Par_0(\la)^{+1}$ the set of $\mu \in \Par_0(n+1)$ such that $\sh[\mu]$ is obtained by adding a box to $\sh[\la]$. Similarly, $\Par_0(\la)^{-1}$ will denote the set of $\mu \in \Par_0(n-1)$ such that $\sh[\mu]$ is obtained by removing a box from $\sh[\la]$.

The following lemma is well known and is easy to see:

\begin{Lemma}\label{lem:count_tab}
Let $n\in\NN$. We have:
\begin{enumerate}
\item if $\la \in \Par(n)$ then
$
\sum_{\mu \in \Par(\la)^{-1}}K_\mu = K_\la;
$
\item if $\la \in \Par_0(n)$ then
$
\sum_{\mu \in \Par_0(\la)^{-1}}K_\mu' = K_\la'.
$
\end{enumerate}
\end{Lemma}

The next lemma is proved in \cite[Lemma 4.4.13]{KL}.

\begin{Lemma}\label{lem:dim_sqd}
Let $n\in\N$. Then
$$
\sum_{\lambda\in\Par(n)}K_\lambda^2=n!=\sum_{\lambda\in\Par_0(n)}2^{n-h(\lambda)}(K_\lambda')^2.
$$
\end{Lemma}

We will use the $\bar p$-abacus notation for strict partitions introduced in~\cite{MorYas}. For $\la \in \Par_0(n)$, $\Ab_\la$ will signify its $\bar p$-abacus display, see \cite[\S2.3b]{KL} for more details on this. Our convention is that the $0^{\nth}$ position on the abacus display is always unoccupied, i.e. we always use $h(\la)$ beads in $\Ab_\la$. 
%\color{red}{\bf Question from Sasha:} Are you sure we want this requirement? In \cite{KL} it was convenient to allow $N$ to be bigger that $h(\la)$. Now you always want to use exactly $h(\la)$ beads? Is there a good reason for this? I am worried that this is almost certainly going to be too restrictive. For example this condition might get violated after slides up. \color{black}
An {\em elementary slide} up on some abacus display means simply moving a bead from position $r$ to position $r-p$, for some $r \geq p$. For $r=p$, this means removing the bead in position $p$ entirely. 
There is, of course, the analogous concept of an elementary slide down on an abacus display. In particular, inserting a bead in the empty $p^{\nth}$ position in an abacus display is considered to be an elementary slide down.

Following \cite{Morris}, we can associate to every $\la\in\Par_0$ its {\em $\bar p$-core} $\core(\la)\in\Par_0$. We can also define a non-negative integer called the {\em $\bar p$-weight} of $\la$:
\begin{align*}
\wt(\la):=(|\la|-|\core(\la)|)/p\in\N.
\end{align*}
If $\rho$ is a ${\bar p}$-core and $d \in \N$, we define
\begin{align*}
\Par_0(\rho,d):=\{\la\in\Par_0 \mid \core(\la)=\rho,\ \wt(\la)=d\}.
\end{align*}

If $\la \in \Par_0(\rho,d)$, then, for $i\in I$, we set $\la^{(i)}$ to be the {\em $i^{\nth}$ quotient of $\la$}, as defined in \cite[p.27]{MorYas}. So the $\la^{(i)}$'s are partitions, $\la^{(0)}$ being strict, with $\sum_{i=0}^\ell |\la^{(i)}| = d$. We set
\begin{equation}\label{EK(la)}
K(\la):=K_{\la^{(0)}}' K_{\la^{(1)}}\dots K_{\la^{(\ell)}}.
\end{equation}

We denote by $\La(I,d)$ the set of all tuples $\ud=(d_1,\dots,d_\ell)$ of non-negative integers such that  $d_0+\dots+d_\ell=d$. 
Given a tuple $\ud=(d_1,\dots,d_\ell)\in\La(I,d)$, we set
\begin{equation}\label{EDUnderline}
\Par_0(\rho,\ud)=\{\la\in \Par_0(\rho,d)\mid |\la^{(i)}|=d_i\ \text{for all $i\in I$}\}. 
\end{equation}
We also define the multinomial coefficient
\begin{equation}\label{EBinom}
\binom{d}{\ud}:=\binom{d}{d_0,\dots,d_\ell}.
\end{equation}

For $\la \in \Par_0(n-p)$ and $j \in I$, $\Par_0^j(\la)^+$ will signify the set of $\mu \in \Par_0(n)$ such that $\Ab_\mu$ is obtained from $\Ab_\la$ by an elementary slide down on the runner $j$ or $p-j$ or by inserting beads at the top of the runners $j$ and $p-j$. It is well known that this is equivalent to $\la$ and $\mu$ having the same $\bar p$-core, $\la^{(i)} = \mu^{(i)}$, for all $i \neq j$, and $\mu^{(0)} \in \Par_0(\la^{(0)})^{+1}$, if $j=0$ or $\mu^{(j)} \in \Par(\la^{(j)})^{+1}$, if $j\neq0$. We set $\Par_0^{\leq j}(\la)^+:=\bigsqcup_{i=0}^j \Par_0^i(\la)^+$. Conversely, for $\la \in \Par_0(n)$, $\Par_0^j(\la)^-$ is the set of $\mu \in \Par_0(n-p)$ such that $\la \in \Par_0^j(\mu)^+$. Again, we set $\Par_0^{\leq j}(\la)^-:=\bigsqcup_{i=0}^j \Par_0^i(\la)^-$.

\section{Superalgebras and supermodules}\label{sec:grdd_supalg}

\subsection{Superspaces and superalgebras}
Most of the definitions and results in the first part of this section come directly from \cite[$\S$2.2]{KL}. However, we restate a lot of the content of that article as all the superalgebras in \cite{KL} are defined over fields and we are primarily concerned with $\cO$-superalgebras.

Recal that $\cR$ denotes $\K$ or $\cO$. By an {\em $\cR$-superspace} we mean a finitely generated $\cR$-module $V$ with decomposition $V = V_{\0} \oplus V_{\1}$. For $\eps\in\Z/2$, we refer to the elements of $V_\eps$ as the homogeneous elements of {\em parity} $\eps$, and write $|v|=\eps$ for $v\in V_\eps$. In fact, whenever we write $|v|$, for some $v \in V$, it is always assumed that $v$ is homogeneous.

An $\cR$-subsuperspace $W\subseteq V$ is an $\cR$-submodule of $V$ such that $W = (W \cap V_{\0}) \oplus (W \cap V_{\1})$. We can, of course, treat $W$ as an $\cR$-subsuperspace in its own right.

Let $V,W$ be $\cR$-superspaces. The direct sum $V \oplus W$ and tensor product $V\otimes W$ are both considered as $\cR$-superspaces in the obvious way. A {\em homomorphism} $f:V\to W$ of $\cR$-superspaces is an $\cR$-linear map satisfying $f(V_\eps)\subseteq W_\eps$, for $\eps \in \Z/2$. 
%If, in addition, $f$ is bijective, then we say it is an {\em isomorphism}.

An {\em $\cR$-superalgebra} is an $\cR$-superspace $A$ that is also an $\cR$-algebra such that $A_{\eps}A_{\de}\subseteq A_{\eps+\de}$, for all $\eps,\de \in \Z/2$. 
Recalling our conventions that all $\cR$-algebras are assumed to be finitely generated as $\cR$-modules and {\em algebra} refers to an $\cO$-algebra, the same thus applies to superalgebras: they are all assumed to be  finitely generated as $\cR$-modules and 
{\em superalgebra} refers to an $\cO$-superalgebra. 

A {\em subsuperalgebra} of an $\cR$-superalgebra $A$ is a (unital) subalgebra that is also a subsuperspace.
A {\em homomorphism} $f:A\to B$ of $\cR$-superalgebras is an algebra homomorphism that is also a homomorphism of $\cR$-superspaces.

We define the $\cR$-superalgebra homomorphism
\begin{equation}\label{ESigma}
\si=\si_A:A\to A,\ a\mapsto (-1)^{|a|}a.
\end{equation}
Note that the superstructure (i.e. $\Z/2$-grading) on $A$ is completely determined by $\si_A$.

The tensor product $A\otimes B$ of $\cR$-superalgebras $A$ and $B$ is considered to be an $\cR$-superalgebra via
$$
(a\otimes b)(a'\otimes b') = (-1)^{|b| |a'|}aa'\otimes bb'.
$$
For $a \in A$ and $1\leq r\leq n$, we set
\begin{align}\label{algn:vr_def}
a_r := 1^{\otimes (r-1)}\otimes a\otimes 1^{\otimes(n-r)} \in A^{\otimes n}.
\end{align}

We define the superalgebra $A^{\sop}$ to be equal to $A$ as an $\cR$-superspace but with multiplication given by $a.b=(-1)^{|a||b|}ba$, for all $a,b\in A$. We define the superalgebra $A^{\op}$ to be equal to $A$ as an $\cR$-superspace but with multiplication given by $a.b=ba$, for all $a,b\in A$.

A superalgebra $A$ is called {\em supersymmetric} if we have an $\cR$-linear symmetrizing form $\tr: A \to \cR$ such that $\tr(A_{\1}) = \{0\}$. In this case we refer to $\tr$ as a {\em supersymmetrizing form}.

Let $A$ be a superalgebra. 
We denote by $A^\times$ the set of units in $A$. 
If $A_{\1}\cap A^\times \neq \varnothing$, we call $A$ a {\em superalgebra with superunit}, and any $u\in A_{\1}\cap A^\times$ is called a {\em superunit}.

\begin{Example}\label{ExCl}
We define the rank $n$ {\em Clifford superalgebra $\cC_n$} to be the $\cR$-superalgebra given by odd generators $\cc_1,\dots,\cc_n$ subject to the relations $\cc_r^2=1$ for $r=1,\dots,n$ and $\cc_r\cc_s=-\cc_s\cc_r$ for all $1\leq r\neq s\leq n$. We will use the well-known fact that $\cC_n\cong \cC_1^{\otimes n}$.  
\end{Example}

\begin{Example}
We denote by $\cM_{m \times n}(\cR)$ the set of $m \times n$ matrices with entries in $\cR$. 
%In particular, if $m=n$, we have the $\cR$-algebra $\cM_{n \times n}(\cR)$. % will denote the relevant $\cR$-algebra.
Denote by $\cM_{m|n}(\cR)$ the superalgebra $\cM_{(m+n) \times (m+n)}(\cR)$ with the usual matrix multiplication and %the superstructure given by
\begin{align*}
\cM_{m|n}(\cR)_{\0} & = \bigg\{\begin{pmatrix} W & 0 \\ 0 & Z \end{pmatrix}\mid W \in \cM_{m \times m}(\cR), Z \in \cM_{n \times n}(\cR)\bigg\}, \\
\cM_{m|n}(\cR)_{\1} & = \bigg\{\begin{pmatrix} 0 & X \\ Y & 0 \end{pmatrix}\mid X \in \cM_{m \times n}(\cR), Y \in \cM_{n \times m}(\cR)\bigg\}.
\end{align*}
\end{Example}

\begin{Example}
The superalgebra $\cQ_n(\cR)$ equals $\cM_{n \times n}(\cR) \oplus \cM_{n \times n}(\cR)$ as an algebra, with $\si_{\cQ_n(\cR)}:(x,y) \mapsto (y,x)$, for all $x,y \in \cM_{n \times n}(\cR)$, cf. (\ref{ESigma}). 
\end{Example}

\begin{Example}\label{ExTn}
The {\em twisted group superalgebra} $\cT_n$ over $\cR$ of the symmetric group $\Si_n$ is the $\cR$-superalgebra given by odd generators $\ct_1,\dots,\ct_{n-1}$ subject to the relations
\begin{align*}
\ct_r^2=1,\quad \ct_r\ct_s = -\ct_s\ct_r\text{ if }|r-s|>1,\quad (\ct_r\ct_{r+1})^3=1.
\end{align*}
Choosing, for each $w\in \Si_n$, a reduced expression $w=s_{r_1}\cdots s_{r_k}$ in terms of simple transpositions, we define $\ct_w:=\ct_{r_1}\cdots \ct_{r_k}$ (which depends up to a sign on the choice of a reduced expression). It is well-known that   $\{\ct_w\mid w\in \Si_n\}$ is an $\cR$-basis of $\cT_n$ (one way to see that is to use one of the isomorphisms (\ref{algn:Sn_Tn_isom}) below).
\end{Example}

\subsection{Supermodules and bisupermodules}
\label{SSMod}

Let $A$ be an $\cR$-superalgebra. By an {\em $A$-supermodule} $\gM$ we mean an $A$-module which is also an $\cR$-superspace such that $A_{\eps}\gM_{\de}\subseteq \gM_{\eps+\de}$ for all $\eps,\de \in \Z/2$. A {\em subsupermodule} of $\gM$ is a submodule that is also a subsuperspace. The $A$-supermodule $\gM$ is called {\em irreducible} if it has exactly two  subsupermodules.

If $B$ is another $\cR$-superalgebra, an {\em $(A,B)$-bisupermodule} $\gM$ is an $(A,B)$-bimodule which is also an $\cR$-superspace such that $A_{\eps}\gM_{\de}\subseteq \gM_{\eps+\de}$ and $\gM_{\de}B_{\eps}\subseteq \gM_{\eps+\de}$ for all $\eps,\de \in \Z/2$. We note that we can also view the $(A,B)$-bisupermodule $\gM$ as an $(A \otimes B^{\sop})$-supermodule via
\begin{align}\label{bisupmodAB}
(a\otimes b).m=(-1)^{|b||m|}amb,
\end{align}
for all $a\in A$, $b\in B$ and $m\in \gM$. In this way, we identify the notions of an $(A,B)$-bisupermodule and an $(A \otimes B^{\sop})$-supermodule.

If $\gM$ is an $A$-supermodule, we denote by $\Pi M$ the module $\gM$ with parity swapped, i.e. $(\Pi M)_\eps=M_{\eps+1}$ for $\eps\in\Z/2\Z$, and the $A$-supermodule structure on $\Pi M$ is given by $a. m=am$.

A {\em homomorphism} $f:\gM\to \gN$ of $A$-supermodules is an $A$-module homomorphism that is also a homomorphism of $\cR$-superspaces, i.e. we are working in the ($\Pi$-)category $\underlinesmod{A}$ of finitely generated $A$-supermodules and even $A$-supermodule homomorphisms, cf. \cite[Definition 1.6]{BE}. So an isomorphism of $A$-supermodules is also always assumed even, and we use the notation $\gM\simeq \gN$ to signify that the $A$-supermodules $\gM$ and $\gN$ are (evenly) isomorphic.  
Note, for example, the regular supermodule for a superalgebra $A$ with superunit $u$ satisfies $A\simeq\Pi A$ via the isomorphism $a\mapsto au$. 
We write $\hom_A(\gM, \gN)$ for the $\cR$-space of all homomorphisms between $A$-supermodules $\gM$ and $\gN$. We also need the $\cR$-superspace 
$$
\Hom_A(\gM,\gN)=\Hom_A(\gM,\gN)_\0\oplus \Hom_A(\gM,\gN)_\1,
$$
where
$$
\Hom_A(\gM,\gN)_\0:=\hom_A(\gM, \gN)\quad\text{and}\quad  \Hom_A(\gM,\gN)_\1:=\hom_A(\gM, \Pi\gN).
$$

If $V$ is an $\cR$-superspace we denote by $|V|$ the $\cR$-space which is $V$ with grading forgotten. In particular, if  
$A$ is an $\cR$-superalgebra and $\gM$ is an $A$-supermodule, we have the $\cR$-algebra $|A|$ and  an 
 $|A|$-module $|\gM|$. Similarly, 
if $B$ is another superalgebra and $\gM$ is an $(A,B)$-bisupermodule, we have an $(|A|,|B|)$-bisupermodule $|\gM|$. 
%Now, $|A|\otimes |B|$ is the usual tensor product of the algebras $|A|$ and $|B|$ (i.e. no sign) but considered as a superalgebra in the obvious way. 
Note that 
$$|\Hom_A(\gM,\gN)|=\Hom_{|A|}(|\gM|,|\gN|)
$$
for all $\gM,\gN\in\underlinesmod{A}$. For the purely even sub(super)algebra $A_\0$ we do not distinguish between $|A_\0|$ and $A_\0$.

%When $\cR = \K$, we denote by $\Irr_{\super}(A)$ the set of isomorphism classes of irreducible $A$-supermodules. If $\cR = \cO$, we set $\Irr_{\super}(A) := \Irr_{\super}(\K A)$.

The category $\underlinesmod{A}$ is isomorphic to the category $\mod{\hat A}$, where
$$
\hat A := \langle A, e_{\0}, e_{\1}\mid  e_\eps e_\de = \de_{\eps,\de}e_\eps,\ e_{\0} + e_{\1} = 1,\ a e_\eps = e_{\eps+|a|} a,\  \text{ for all }\, \eps,\de \in \Z/2,\, a \in A,\rangle_{\cR}.
$$
Here, $e_\eps$ is understood to correspond to the projection $\gM \twoheadrightarrow \gM_\eps$, for an $A$-supermodule $\gM$. In particular, we have the Krull-Schmidt property for $\underlinesmod{A}$, see also \cite[Corollary~3.7.5]{NO}.

Using the above identification of $(A,B)$-bisupermodules with $(A \otimes B^{\sop})$-supermodules we can also define $(A,B)$-bisupermodule homomorphisms and isomorphisms. 
In particular for $(A,B)$-bisupermodules $\gM$ and $\gN$ we have the $\cR$-space $\hom_{A \otimes B^{\sop}}(\gM,\gN)$ and the $\cR$-superspace $\Hom_{A \otimes B^{\sop}}(\gM,\gN)$. 
As with ordinary (bi)modules, we again use the notation $\gM \mid \gN$ to signify that $\gM$ is (evenly) isomorphic to a direct summand of $\gN$, as (bi)supermodules.

An $A$-supermodule $\gM$ is called {\em absolutely indecomposable} if $|\gM|$ is indecomposable as an $|A|$-module.
An $(A,B)$-bisupermodule $\gM$ is called {\em absolutely indecomposable} if $|\gM|$ indecomposable as an $(|A|,|B|)$-bimodule.

For an $|A|$-module $\gM$, we define $\gM^{\si}$ to be the $|A|$-module equal to $\gM$ as an $\cR$-module but with the action given via $a.m := (-1)^{|a|}am$, for all $a \in A$ and $m \in \gM$. We call $\gM$ {\em self-associate} when $\gM \cong \gM^{\si}$. Otherwise, we say $\gM$ is {\em non-self-associate}.

Let $A,B,C$ and $D$ be $\cR$-superalgebras, $\gM$ an $(A,C)$-bisupermodule and $\gN$ a $(B,D)$-bisupermodule. We define the $(A\otimes B,C\otimes D)$-bisupermodule $\gM\boxtimes \gN$ to be the $\cR$-superspace $\gM\otimes \gN$ with the action 
$$
(a\otimes b)(m\otimes n)=(-1)^{|b||m|}(am\otimes bn),\qquad
(m\otimes n)(c\otimes d)=(-1)^{|c||n|}(mc\otimes nd),
$$
for all $a \in A$, $b \in B$, $c \in C$, $d \in D$, $m \in \gM$ and $n \in \gN$. In particular, given an $A$-supermodule $\gM$ and a $B$-supermodule $\gN$, we have the $(A\otimes B)$-supermodule $\gM\boxtimes \gN$ with action 
$
(a\otimes b)(m\otimes n)=(-1)^{|b||m|}(am\otimes bn).
$

The following Remark, taken from \cite[Remark 2.2.10]{KL}, allows us to apply results concerning tensor products of supermodules to tensor products of bisupermodules.

\begin{Remark}\label{rem:bisupmod}
Using the above notation, the $(A \otimes C^{\sop}) \otimes (B \otimes D^{\sop})$-supermodule $\gM\boxtimes \gN$ can be identified with the $(A\otimes B, C\otimes D)$-bisupermodule $\gM\boxtimes \gN$ using (\ref{bisupmodAB}) and the superalgebra isomorphism
\begin{align*}
(A \otimes C^{\sop}) \otimes (B \otimes D^{\sop}) &\to (A \otimes B) \otimes (C \otimes D)^{\sop}\\
(a \otimes c) \otimes (b \otimes d) &\mapsto (-1)^{|b||c|}(a \otimes b) \otimes (c \otimes d).
\end{align*}
\end{Remark}

\begin{Lemma}\label{lem:tensor_bisupmod}
Let $A_1$, $A_2$, $B_1$, $B_2$, $C_1$ and $C_2$ be $\cR$-superalgebras, $\gM_1$ an $(A_1,B_1)$-bisupermodule, $\gM_2$ an $(A_2,B_2)$-bisupermodule, $\gN_1$ a $(B_1,C_1)$-bisupermodule and $\gN_2$ a $(B_2,C_2)$-bisupermodule. Then
\begin{align*}
(\gM_1 \boxtimes \gM_2) \otimes_{B_1 \otimes B_2} (\gN_1 \boxtimes \gN_2) & \to (\gM_1 \otimes_{B_1} \gN_1) \boxtimes (\gM_2 \otimes_{B_2} \gN_2)\\
(m_1 \otimes m_2) \otimes (n_1 \otimes n_2) & \mapsto (-1)^{|m_2||n_1|} (m_1 \otimes n_1) \otimes (m_2 \otimes n_2),
\end{align*}
for all $m_1 \in \gM_1$, $m_2 \in \gM_2$, $n_1 \in \gN_1$ and $n_2 \in \gN_2$, is an isomorphism of $(A_1 \otimes A_2,C_1 \otimes C_2)$-bisupermodules.
\end{Lemma}

\begin{proof}
This is just a fairly quick checking exercise.
\end{proof}

\subsection{Morita superequivalences}\label{SMoritaSuper}

Let $A$ and $B$ be superalgebras. A {\em Morita superequivalence} between $A$ and $B$ is a Morita equivalence between $A$ and $B$ induced by an $(A,B)$-bisupermodule $\gM$ and a $(B,A)$-bisupermodule $\gN$, i.e. $\gM\otimes_B \gN\simeq A$ as $(A,A)$-bisupermodules and $\gN\otimes_A \gM\simeq B$ as $(B,B)$-bisupermodules. 
We write $A\sim_{\sM}B$.

A {\em stable superequivalence of Morita type} between $A$ and $B$ is a stable equivalence of Morita type between $A$ and $B$ induced by an $(A,B)$-bisupermodule $\gM$ and a $(B,A)$-bisupermodule $\gN$. That is, there exist bisupermodule isomorphisms $\gM\otimes_B \gN\simeq A \oplus \gP$ and $\gN\otimes_A \gM\simeq B \oplus \gQ$, where $\gP$ (resp. $\gQ$) is an $(A,A)$-bisupermodule (resp. $(B,B)$-bisupermodule) such that  $|\gP|$ (resp. $|\gQ|$) is projective as an $(|A|,|A|)$-bimodule (resp. $(|B|,|B|)$-bimodule).

Many of the following results are taken from \cite[$\S\S$2.2c,d]{KL}, where it is assumed that the superalgebras are all defined over a field. However, all the proofs run through for $\cO$-superalgebras in exactly the same way.

\begin{Lemma}\label{lem:idmpt_Mor}
Let $A$ be a superalgebra and $e\in A_{\0}$ an idempotent such that $AeA=A$. Then 
$A$ and $eAe$ are Morita superequivalent via 
the $(eAe,A)$-bisupermodule 
$eA$ and 
the $(A,eAe)$-bisupermodule
$Ae$.
\end{Lemma}
\begin{proof}
\cite[Lemma 2.2.14]{KL}.
\end{proof}

\begin{Lemma}\label{lem:MSE}
Let $A_1,A_2,B_1,B_2$ be superalgebras. If 
$A_i$ and $B_i$ are Morita superequivalent via the $(A_i,B_i)$-bisupermodule $\gM_i$ and the $(B_i,A_i)$-bisupermodule $\gN_i$, for $i=1,2$, then $A_1\otimes A_2$ and $B_1\otimes B_2$ are Morita superequivalent via the $(A_1\otimes A_2,B_1\otimes B_2)$-bisupermodule $\gM_1\boxtimes \gM_2$ and the $(B_1\otimes B_2,A_1\otimes A_2)$-bisupermodule $\gN_1\boxtimes \gN_2$.
\end{Lemma}

\begin{proof}
\cite[Lemma 2.2.17]{KL}.
\end{proof}

%The following is proved in \cite[Lemma 2.2.19]{KL}.

\begin{Lemma}\label{lem:Mor_A0}
Let $A$ and $B$ be superalgebras with superunit. If the $(A,B)$-bisupermodule $\gM$ induces a Morita superequiavlence between $A$ and $B$, then $\gM_{\0}$ induces a Morita equiavlence between $A_{\0}$ and $B_{\0}$.
\end{Lemma}
\begin{proof}
\cite[Lemma 2.2.19]{KL}.
\end{proof}

Let $A$ be a superalgebra, $\gM$ be an $A_{\0}$-module and 
$\gN$ be an $|A|$-module. 
We give $\Ind_{A_{\0}}^A(\gM)$ the structure of an $A$-supermodule via $|a \otimes m| := |a|$, for all $a \in A$ and $m \in \gM$. 
If, in addition, $A$ has a superunit $u$, we denote by ${}^u\gM$ the $A_\0$-module which is $\gM$ as a $\cO$-module but the $A_\0$-action is given by $a.m=(u^{-1}au)m$.
If $\gM\subseteq \Res^{|A|}_{A_\0}\gN$, we have the $A_{\0}$-submodule $u\gM$ of $\Res^{|A|}_{A_{\0}}\gN$, and $u\gM\cong {}^u\gM$.

\begin{Lemma}\label{lem_red_to_A0}
Let $A$ be a superalgebra with superunit $u$.
\begin{enumerate}
\item If\, $\gM$ is an $A$-supermodule, then $\gM \simeq \Ind_{A_{\0}}^A\gM_{\0}$.

\item If\, $\gN$ is an indecomposable $|A|$-module, 
then $\Ind_{A_{\0}}^{|A|}\Res^{|A|}_{A_{\0}}\gN \cong \gN \oplus \gN^{\si}$. Moreover, either $\Res^{|A|}_{A_{\0}}\gN$ is indecomposable or $\Res^{|A|}_{A_{\0}}\gN = \gM \oplus u\gM$ for indecomposable $A_{\0}$-modules $\gM \ncong u\gM$.

\end{enumerate}
\end{Lemma}

\begin{proof}
(i) The $A$-supermodule homomorphism
$
\Ind_{A_{\0}}^A \gM_{\0} \to \gM,\ 
a \otimes m \mapsto am
$
is an isomorphism, as $\gM = \gM_{\0} \oplus u \gM_{\0}$.

(ii) We can give $\gN \oplus \gN^{\si}$ the structure of an $A$-supermodule via
\begin{align*}
(\gN \oplus \gN^{\si})_{\0} := \{(n,n)\mid n \in \gN\}\qquad\text{and}\qquad (\gN \oplus \gN^{\si})_{\1} := \{(n,-n)\mid n \in \gN\}.
\end{align*}
Since $2$ is invertible in $\cO$, $\gN \oplus \gN^{\si}$ is indeed a direct sum of $(\gN \oplus \gN^{\si})_{\0}$ and $(\gN \oplus \gN^{\si})_{\1}$. Moreover, $\Res^{|A|}_{A_{\0}}\gN \cong (\gN \oplus \gN^{\si})_{\0}$, as $A_{\0}$-modules, and so, by part (i), we have that 
\begin{equation}\label{E071222}
\Ind_{A_{\0}}^A\Res^{|A|}_{A_{\0}}\gN \simeq \gN \oplus \gN^{\si}
\end{equation}
as $A$-supermodules. The fact that $\Res^{|A|}_{A_{\0}}(\gN)$ is the direct sum of at most two indecomposable $A_{\0}$-modules follows immediately.

Suppose $\Res^{|A|}_{A_{\0}} \gN = \gM \oplus \gM'$, for indecomposable $A_{\0}$-modules $\gM,\gM'$. By (\ref{E071222}),  
$$\gN \oplus \gN^{\si}\cong \Ind_{A_{\0}}^{|A|}\gM\,\oplus\, 
\Ind_{A_{\0}}^{|A|} \gM',$$
as $|A|$-modules. Since $\gN$ is indecomposable, we may assume that 
$\gN\cong \Ind_{A_{\0}}^{|A|}\gM$ and $\gN^\si\cong  \Ind_{A_{\0}}^{|A|}\gM'$. 
Now,
$$
\gM\oplus \gM' \cong \Res^{|A|}_{A_{\0}}\gN\cong \Res^{|A|}_{A_{\0}}\Ind_{A_{\0}}^{|A|} \gM \cong \gM \oplus u \gM
$$
and so $\gM' \cong u \gM$. In particular,
$$
\Ind_{A_{\0}}^A \gM' \cong \Ind_{A_{\0}}^A u\gM \cong \Ind_{A_{\0}}^A \gM \cong \gN,
$$
as $|A|$-modules, proving all the claims except that $\gM \ncong u \gM$.

Suppose $\varphi:\gM \to u \gM$ is an $A_{\0}$-module isomorphism. Since, $\gN \cong \Ind_{A_{0}}^{|A|} \gM$, we may identify $\gN$ with $\gM \oplus u \gM$ and consider the $|A|$-module automorphism 
$
\psi: \gN \to \gN,\ 
m + u m' \mapsto \varphi(m) + u\varphi(m'),
$ 
for all $m,m' \in \gM$.

Certainly $\psi^2$ is an $|A|$-module automorphism of $\gN$. Since $\gN$ is indecomposable, there is some $c \in \cO^\times$ such that $c.\Id_{\gN} - \psi^2 \in \rad(\End_{|A|}(\gN))$. Let $k \in \cO^\times$ such that ${\bar k}^2 = {\bar c}^{-1}$ in $\F$. Now, the images of $(\Id_{\gN} + k.\psi)/2$ and $(\Id_{\gN} - k.\psi)/2$ in $\End_{|A|}(\gN)/\rad(\End_{|A|}(\gN))$ are orthogonal idempotents. Once we have shown that they are also non-zero we will have contradicted the indecomposability of $\gN$, proving the non-existence of $\varphi$.

One can quickly check that $\End_{|A|}(\gN)$ has an automorphism given by conjugation by the $\cO$-linear map
\begin{align*}
\gN  \to \gN, \ 
m + u m'  \mapsto m - u m',
\end{align*}
for all $m,m' \in \gM$. Note that, under this automorphism, $\psi$ and $-\psi$ get swapped and hence $(\Id_{\gN} + k.\psi)/2$ and $(\Id_{\gN} - k.\psi)/2$ get swapped. Therefore, $(\Id_{\gN} + k.\psi)/2$ is zero in $\End_{|A|}(\gN)/\rad(\End_{|A|}(\gN))$ if and only if $(\Id_{\gN} - k.\psi)/2$ is. Since they sum to $\Id_N$, we have shown that they are both non-zero and the proof is complete.
%(iii) The first part just follows from part (i) once we've identified $\gM$ with its corresponding $(A \otimes B^{\op})$-supermodule. For the second part we apply part (ii), where we give $(A,B)_{C_2}$ the structure of a superalgebra by setting $((A,B)_{C_2})_{\0} = A_{\0} \otimes B_{\0}$ and demanding that $u_A \otimes u_B^{-1}$ is odd.
%(iv) This is just part (iii) applied to situation $A=B=\gM$.
\end{proof}

\subsection{Projective supermodules}\label{sec:proj_sup_mod}
Let $A$ be a superalgebra with a superunit. In particular, for the left regular supermodule $A$ we have $A\simeq\Pi A$ as noted in \S\ref{SSMod}. A standard argument then shows that a  supermodule $\gP\in\underlinesmod{A}$ is projective if and only if $\gP\mid A^{\oplus n}$ for some $n$. 
% (this comes by a standard argument using the isomorphism $A\simeq \Pi A$).

\begin{Lemma} \label{LNewProj}
Let $A$ be a superalgebra with a superunit $u$ and 
$\gP\in\underlinesmod{A}$. 
%$\gP$ be a (finitely generated) $A$-supermodule. 
Then the following are equivalent:
\begin{enumerate}
\item[{\rm (i)}] $\gP$ is a projective $A$-supermodule;
\item[{\rm (ii)}] $|\gP|$ is a projective $|A|$-module;
\item[{\rm (iii)}] $\gP_\0$ is a projective $A_\0$-module;
\item[{\rm (iv)}] $\gP\simeq\Ind^A_{A_\0}\gQ$ for a projective $A_\0$-module $\gQ$. 
\end{enumerate}
\end{Lemma}
\begin{proof}
(iv)$\implies$(i) If $\gQ$ is a projective $A_{\0}$-module, then $\gQ\mid A_{\0}^{\oplus n}$, for some $n$, and so $\Ind_{A_{\0}}^A \gQ\mid \Ind_{A_{\0}}^A A_{\0}^{\oplus n} \simeq A^{\oplus n}$, hence $\Ind_{A_{\0}}^A \gQ$ is projective. 

(i)$\implies$(iii),(iv) Let $\gP$ be a projective $A$-supermodule. Then $\gP \simeq \Ind_{A_{\0}}^A\gP_{\0}$, by Lemma \ref{lem_red_to_A0}(i). Moreover, $\gP\mid A^{\oplus n}$, for some $n$, and so
$$
\gP_{\0}\mid \Res^{A}_{A_{\0}}\gP\mid \Res^{A}_{A_{\0}} A^{\oplus n} \cong A_{\0}^{\oplus n} \oplus A_{\1}^{\oplus n},
$$
as $A_{\0}$-modules. It remains to note that $A_{\0} \cong A_{\1}$, as $A_{\0}$-modules, via $a \mapsto au$, and so $\gP_{\0}$ is a projective $A_{\0}$-module, as desired.

(iii)$\implies$(i)  By Lemma \ref{lem_red_to_A0}(i), we have $\gP \simeq \Ind_{A_{\0}}^A\gP_{\0}$, so (i) comes from the implication (iv)$\implies$(i).

(i)$\implies$(ii) is clear since $\gP\mid A^{\oplus n}$ implies $|\gP| \,\, \big{|} \,\, |A|^{\oplus n}$. 

(ii)$\implies$(i) Let $\gP$ be an $A$-supermodule with $|\gP|$ being projective as an $|A|$-module. Then 
$\gP \simeq \Ind_{A_{\0}}^A\gP_{\0}$, by Lemma \ref{lem_red_to_A0}(i). Moreover, $|\gP| \,\, \big{|} \,\,  |A|^{\oplus n}$,   for some $n$, and so
$$
\gP_{\0} \,\, \big{|} \,\, \Res^{|A|}_{A_{\0}}|\gP| \,\, \big{|} \,\, \Res^{|A|}_{A_{\0}} |A|^{\oplus n} \cong A_{\0}^{\oplus n} \oplus A_{\1}^{\oplus n},
$$
as $A_{\0}$-modules. But the $A_\0$-modules $A_{\0}$ and $ A_{\1}$ are isomorphic, as noticed above. So $\gP_\0$ is projective, hence $\gP$ is projective by the implication (iii)$\implies$(i). 
\end{proof}

A {\em projective cover} of an $A$-supermodule is defined exactly like for modules. Our conditions on the superalgebra $A$ ensure that a projective cover of a (finitely generated) $A$-supermodule $\gM$  exists and is unique up to isomorphism. So we can define {\em Heller translate} $\Omega_A(\gM)$ of $\gM$ just like for modules. 

%We set $\Omega_A(\gM)$ to be the {\em Heller translate} of $\gM$. 

\begin{Lemma} \label{LProjCov}
Let $A$ be a superalgebra with a superunit and 
$\gM\in\underlinesmod{A}$. If $(\gP_{\gM_\0},\varphi_{\gM_\0})$ is a projective cover of\, $\gM_{\0}$ as an $A_{\0}$-module, 
then $(\Ind_{A_{\0}}^A\gP_{\gM_\0}, \Ind_{A_{\0}}^A\varphi_{\gM_\0})$ is a projective cover of\, $\gM$ as an $A$-supermodule and 
%$(|\Ind_{A_{\0}}^A\gP_{\gM_\0}|, \Ind_{A_{\0}}^A\varphi_{\gM_\0})$ is a projective cover 
of $|\gM|$ as an $|A|$-module. 
\end{Lemma}
\begin{proof}
By Lemma~\ref{LNewProj}, $\Ind_{A_{\0}}^A\gP_{\gM_\0}$ is projective as an $A$-supermodule and as an $|A|$-module. Moreover, no proper summand of $\Ind_{A_{\0}}^A\gP_{\gM_{\0}}$ will suffice as a projective cover of $\gM$ or $|\gM|$, since $\Res^A_{A_{\0}}\Ind_{A_{\0}}^A\gP_{\gM_{\0}} \cong \gP_{\gM_{\0}} \oplus {}^u\gP_{\gM_{\0}}$ is a projective cover of $\Res^A_{A_{\0}}\gM \cong \gM_{\0} \oplus u\gM_{\0}$.
\end{proof}

\begin{Corollary} \label{COm0} %{\rm \cite{}}%{\bf ()}
Let $A$ be a superalgebra with a superunit and $\gM\in\underlinesmod{A}$. Then $\Om_A(\gM)\simeq \Ind_{A_{\0}}^A\Om_{A_\0}(\gM_\0)$ and $|\Om_A(\gM)|\cong \Om_{|A|}(|\gM|)$. 
\end{Corollary}
\begin{proof}
Follows from Lemma~\ref{LProjCov} and the exactness of the functor $\Ind_{A_{\0}}^A$. 
\end{proof}

\begin{Corollary} \label{C?}
Let $A,B$ be superalgebras with superunits, 
$\gM$ an $A$-supermodule and $\gN$ a $B$-supermodule.  
If $(\gP_{\gM_{\0}},\varphi_{\gM_{\0}})$ is a projective cover of $\gM_{\0}$ as an $A_{\0}$-module and 
$(\gP_{\gN_{\0}},\varphi_{\gN_{\0}})$ is a projective cover of $\gN_{\0}$ as a $B_{\0}$-module then 
\begin{align*}
\left(\Ind_{A_{\0} \otimes B_{\0}}^{A \otimes B} (\gP_{\gM_{\0}} \boxtimes \gP_{\gN_{\0}}),\,\Ind_{A_{\0} \otimes B_{\0}}^{A \otimes B} (\varphi_{\gM_{\0}} \otimes \varphi_{\gN_{\0}}) \right)
\end{align*}
is a projective cover of $\gM \boxtimes \gN$ as an $(A \otimes B)$-supermodule.
\end{Corollary}
\begin{proof}
This follows similarly to Lemma \ref{LProjCov} once we have observed that $(\gP_{\gM_{\0}}\boxtimes \gP_{\gN_{\0}},\varphi_{\gM_{\0}}\otimes \varphi_{\gN_{\0}})$ is a projective cover of $\gM_{\0}\boxtimes \gN_{\0}$ as an $A_\0\otimes B_{\0}$-module.
\end{proof}

\begin{Lemma}\label{lem:super_Heller}
Let $A$ and $B$ be two superalgebras with superunits, $\gM\in\underlinesmod{A}$ and $\gN\in\underlinesmod{B}$. Then there exists a canonically defined monomorphism
$$
\Omega_A(\gM) \boxtimes \Omega_B(\gN) \hookrightarrow \Omega_{A \otimes B}(\gM \boxtimes \gN)
$$
of $(A \otimes B)$-supermodules. Furthermore, through this monomorphism,
\begin{align*}
\Omega_{A \otimes B}(\gM \boxtimes \gN)/(\Omega_A(\gM) \boxtimes \Omega_B(\gN)) \simeq (\Omega_A(\gM) \boxtimes \gN) \oplus (\gM \boxtimes \Omega_B(\gN)),
\end{align*}
as $(A \otimes B)$-supermodules.
\end{Lemma}

\begin{proof}
If $(\gP_\gM,\phi_\gM)$ is a projective cover of $\gM$ and $(\gP_\gN,\phi_\gN)$ is a projective cover of  $\gN$ then, using for example Lemma~\ref{LProjCov} and Corollary~\ref{C?}, %\textcolor{green}{why do you say `for example'?} 
%\textcolor{blue}{1. I have added reference to Lemma~\ref{LProjCov}, which I think is also needed; 2. I said `for example' because this was not the only way of seeing this, another is a direct argument as in the classical situation, but I do not mind to remove the words `for example'.}
we can see that  $(\gP_\gM\boxtimes \gP_\gN,\phi_\gM\otimes \phi_\gN)$ is a projective cover of $\gM\boxtimes \gN$. 
Now, 
$$\Om_{A\otimes B}(\gM\boxtimes\gN)=\ker (\phi_\gM\otimes \phi_\gN)=(\Om_A(\gM)\boxtimes \gP_\gN) + (\gP_\gM\boxtimes \Om_B(\gN)) \subseteq \gP_\gM \boxtimes \gP_\gN$$
and 
$$
\Big((\Om_A(\gM)\boxtimes \gP_\gN) + (\gP_\gM\boxtimes \Om_B(\gN))\Big)/\Omega_A(\gM) \boxtimes \Omega_B(\gN)\simeq
(\Omega_A(\gM) \boxtimes \gN) \oplus (\gM \boxtimes \Omega_B(\gN)),
$$
which implies the result.
\end{proof}

\subsection{Crossed superproducts}\label{sec:cross_super}

Let $G$ be a finite group. A {\em $G$-graded crossed superproduct} will refer to a superalgebra $A$ with a decomposition $\bigoplus_{g\in G} A_g$ into subsuperspaces such that $A_gA_h \subseteq A_{gh}$, for all $g,h\in G$, and such that, for all $g\in G$, we have $A_g\cap A^{\times} \neq \varnothing$. % either $A_g\cap A_{\0}\neq \varnothing$ or $A_g\cap A_{\1}\neq \varnothing$.
%If, in addition, $A$ is a superalgebra and each $A_g$ is a subsuperspace, then we refer to $A$ as a {\em $G$-graded crossed superproduct}.
If $A$ is a $G$-graded crossed superproduct, then so is $A^{\sop}$ by defining $(A^{\sop})_g=A_{g^{-1}}$, for all $g\in G$.
Note that $A_{1_G}$ is always a subsuperalgebra of $A$, and $(A_{1_G})^\sop=(A^\sop)_{1_G}$.

A $G$-graded crossed superproduct $A$ is called {\em supersymmetric} if it has a supersymmetrizing form $\tr: A \to \cO$ that turns $A$ into a supersymmetric algebra and $\tr(A_g)=0$, for all $g \in G\setminus\{1_G\}$. In particular, $\tr$ gives $A_{1_G}$ the structure of a supersymmetric algebra.

If $A$ and $B$ are $G$-graded crossed superproducts, we define
\begin{equation}\label{EDiag}
(A,B)_{G}:= \sum_{g\in G}A_g\otimes B_{g^{-1}}
=
\sum_{g\in G}A_g\otimes (B^{\sop})_g  \subseteq A \otimes B^{\sop}.
\end{equation}
The definition of $G$-graded crossed superproduct ensures that $$A_{1_G}\otimes B_{1_G}^\sop\subseteq (A,B)_{G}\subseteq A \otimes B^{\sop}$$ 
are subsuperalgebras. In particular, if $A$ and $B$ are both superalgebras with super units, we consider $A$ and $B$  as $C_2$-graded crossed superproducts in the natural way, and
\begin{align*}
(A,B)_{C_2}= (A_{\0} \otimes B_{\0}) \oplus (A_{\1} \otimes B_{\1}) \subseteq A \otimes B^{\sop}.
\end{align*}

The following Proposition is a super version of \cite[Theorem 3.4(a)]{Mar}, proved in \cite[Proposition 2.2.22]{KL} for superalgebras defined over a field. However, the proof is no more complicated for superalgebras defined over $\cO$. As long as the superalgebras are supersymmetric the proof of \cite[Proposition 2.2.22]{KL} runs through unaltered (see \cite[Remark 3.2(e)]{Mar}). Recall that one can view an $(A,B)$-bisupermodule as an $A\otimes B^{\sop}$-supermodule via (\ref{bisupmodAB}).

\begin{Proposition}\label{prop:ext_marcus}
Let $G$ be a finite group, and $A,\, B$ be supersymmetric $G$-graded crossed superproducts. Suppose $\gM$ is an $(A_{1_G}\otimes B_{1_G}^\sop)$-supermodule inducing a Morita superequivalence between $A_{1_G}$ and $B_{1_G}$. If $\gM$ extends to an $(A,B)_{G}$-supermodule, then $\Ind_{(A,B)_{G}}^{A\otimes B^{\sop}}(\gM)$ induces a Morita superequivalence between $A$ and $B$.
\end{Proposition}

\subsection{Dual supermodules}\label{sec:dual}

Let $A$ be a superalgebra and $\gM\in\underlinesmod{A}$. We define the {\em dual} $\gM^*$ of $\gM$ to be $\Hom_{\cO}(\gM, \cO)$ that we give the structure of a superspace through
\begin{align*}
\Hom_{\cO}(\gM, \cO)_{\eps} &:= \{f \in \Hom_{\cO}(\gM, \cO)\mid f(\gM_{\eps+\1})=0\}\qquad(\eps\in\Z/2).
%, \\\Hom_{\cO}(\gM, \cO)_{\1} &:= \{f \in \Hom_{\cO}(\gM, \cO)\mid f(\gM_{\0})=0\}.
\end{align*}
We treat $\gM^*$ as a right $A$-supermodule via $(f.a)(m):=f(am)$, 
%an $A^{\sop}$-module via $(a.f)(m):=(-1)^{|a||f|}f(am)$, 
for all $a \in A$, $f \in \gM^*$ and $m \in \gM$, hence also as an $A^{\sop}$-module via $(a.f)(m):=(-1)^{|a||f|}f(am)$.
% or as a right $A$-supermodule via $(f.a)(m):=f(am)$.

If $B$ is a another superalgebra and $\gM$ an $(A,B)$-bisupermodule, we can view $\gM^*$ as a $(B,A)$-bisupermodule via $(b.f.a)(m) := f(amb)$. Note that this is not always isomorphic to the bimodule obtained by considering $\gM$ as an $(A \otimes B^{\sop})$-supermodule and hence $\gM^*$ as an $(A \otimes B^{\sop})^{\sop}$-supermodule (which can also be thought of as a $(B,A)$-bisupermodule). The reason we have chosen to define $\gM^*$ using the former definition is that our main reason for introducing $\gM^*$ is Lemma \ref{lem:HomAM_M*}, which does not hold using this alternative definition of $\gM^*$.

\begin{Lemma}\label{lem:dual_comp}
Let $A$, $B$ and $C$ be superalgebras, $\gM$ an $(A,B)$-bisupermodule and $\gN$ a $(B,C)$-bisupermodule. Then there is an isomorphism of $(C,A)$-bisupermodules
\begin{align*}
\gN^* \otimes_B \gM^*  \iso (\gM \otimes_B \gN)^*, \ 
g \otimes f  \mapsto (m \otimes n \mapsto f(m)g(n)),
\end{align*}
%for all $f \in \gM^*$, $g \in \gN^*$, $m \in \gM$ and $n \in \gN$, 
\end{Lemma}

\begin{proof}
This is a standard check.
\end{proof}

\begin{Lemma}\label{lem:ind_dual}
Let $G,H$ be finite groups, 
$A$ a $G$-graded crossed superproduct with $a_g \in A_g \cap A^\times$ for all $g \in G$, and 
$B$ an $H$-graded crossed superproduct with $b_h \in B_h \cap B^\times$ for all $h \in H$. 
If $\gM$ is an $(A_{1_G},B_{1_H})$-bisupermodule then there is an  isomorphism of $(B,A)$-bisupermodules
\begin{align*}
B \otimes_{B_{1_H}} \gM^* \otimes_{A_{1_G}} A & \iso (A \otimes_{A_{1_G}} \gM \otimes_{B_{1_H}} B)^*, \\
b_h \otimes f \otimes a_g & \mapsto \bigg(a_{g'} \otimes m \otimes b_{h'} \mapsto
\begin{cases}
f(a_ga_{g'}mb_{h'}b_h) & \text{if }g'=g^{-1}, h'=h^{-1}, \\
0 & \text{otherwise.}
\end{cases}\bigg).
\end{align*}
%for all $g,g' \in G$, $h,h' \in H$, $f \in \gM^*$ and $m \in \gM$.
\end{Lemma}

\begin{proof}
The map is certainly a bijection as the set of $b_h \otimes f \otimes a_g$, as $f$ runs over $\gM^*$, is precisely the set of $(A \otimes_{A_{1_G}}\gM \otimes_{B_{1_H}} B)^*$ that is zero on all $a_{g'} \otimes \gM \otimes b_{h'}$, for all $g',h'$ with $g' \neq g^{-1}$ or $h' \neq h^{-1}$. One can readily check that this bijection is a homomorphism of $(B,A)$-bisupermodules.
\end{proof}

\begin{Lemma}\label{lem:dual_prod}
Let $A_1,A_2,B_1,B_2$ be superalgebras, $\gM$ an $(A_1,B_1)$-bisupermodule and $\gN$ an $(A_2,B_2)$-bisupermodule. Then there  is an isomorphism of $(B_1 \otimes B_2,A_1 \otimes A_2)$-bisupermodules
\begin{align*}
\gM^* \boxtimes \gN^* \iso (\gM \boxtimes \gN)^*, \ 
f \otimes g & \mapsto \big(m \otimes n \mapsto (-1)^{|g||m|} f(m)g(n)\big).
\end{align*}
%for all $f \in \gM^*$, $g \in \gN^*$, $m \in \gM$ and $n \in \gN$ is an isomorphism of $(B_1 \otimes B_2,A_1 \otimes A_2)$-bisupermodules.
\end{Lemma}

\begin{proof}
Again, this is a standard check.
\end{proof}

\begin{Lemma}\label{lem:HomAM_M*}
Let $A$ and $B$ be superalgebras, with $A$ being supersymmetric. If $\gM$ is an $(A,B)$-bisupermodule such that $|\gM| \otimes_{|B|} ?$ induces a Morita equivalence between $|B|$ and $|A|$, then $\gM$ and $\gM^*$ induce a Morita superequivalence between $B$ and $A$.
\end{Lemma}

\begin{proof}
By assumption, $\End_{A}(\gM) \cong B^{\op}$ as superalgebras.
By Morita theory, $|\gM|$ and $\Hom_{|A|}(|\gM|,|A|)$ induce a Morita equivalence between $|B|$ and $|A|$. More precisely,
\begin{align*}
|\gM \otimes_{B} \Hom_{A}(\gM,A)|=
|\gM| \otimes_{|B|} \Hom_{|A|}(|\gM|,|A|) \to |A|, \ 
m \otimes \varphi \mapsto \varphi(m)
\end{align*}
is an isomorphism of $(|A|,|A|)$-bimodules and
\begin{align*}
|\Hom_{A}(\gM,A) \otimes_{A} \gM|=
\Hom_{|A|}(|\gM|,|A|) \otimes_{|A|} |\gM| &\to \End_{|A|}(|\gM|) \cong |B|, \\ 
\varphi \otimes m &\mapsto (m' \mapsto \varphi(m')m)
\end{align*}
is an isomorphism of $(|B|,|B|)$-bimodules. Now, the first  isomorphism is easily checked to be an isomorphism of $(A,A)$-bisupermodules $\gM \otimes_{B} \Hom_{A}(\gM,A)$ and $A$, and the second  isomorphism is easily checked to be an isomorphism 
of $(B,B)$-bisupermodules $\Hom_{A}(\gM,A) \otimes_{A} \gM$ and $B$. In other words, $\gM$ and $\Hom_{A}(\gM,A)$ induce a Morita superequivalence between $B$ and $A$.

Let $\tr:A \to \cO$ be a supersymmetrizing form on $A$. Then, by \cite[Corollary 2.12.2]{Lin5},
\begin{align*}
|\Hom_{A}(\gM,A)|=
\Hom_{|A|}(|\gM|,|A|) \to |\gM^*|, \ 
f \mapsto \tr \circ f
\end{align*}
is an isomorphism of $(|B|,|A|)$-bimodules. Once again, this is easily checked to be an isomorphism of $(B,A)$-bisupermodules $\Hom_{A}(\gM,A)$ and $\gM^*$. 
\end{proof}

\begin{Lemma}\label{lem:sup_alg_idempt_dual}
Let $A$ be a supersymmetric superalgebra 
with supersymmtrizing form $\tr:A \to \cO$. If 
$e_1,e_2 \in A_{\0}$ are non-zero idempotents then, considering $e_1 A e_2$ as a $(e_1 A e_1,e_2 A e_2)$-bisupermodule, we have the isomorphism of $(e_2 A e_2,e_1 A e_1)$-bisupermodules
\begin{align*}
e_2 A e_1  \iso (e_1 A e_2)^*,\ 
x  \mapsto (y \mapsto \tr(xy)).
\end{align*}
\end{Lemma}

\begin{proof}
One can quickly check this is a homomorphism of $(e_2 A e_2,e_1 A e_1)$-bisupermodules. That it is an isomorphism in the case $e_1=e_2=1_A$ follows immediately from the fact $\tr$ is non-degenerate. Now, the kernel of the induced epimorphism $A \to (e_1 A e_2)^*$ is $A(1-e_1)+(1-e_2)A$ and $A/[A(1-e_1)+(1-e_2)A] \simeq e_2 A e_1$, as desired.
\end{proof}

\subsection{Wreath superproducts}\label{sec:wreath}

Throughout this subsection it is assumed that both $-1$ and $2$ have square roots in $\cO$ and that $A$ is a superalgebra.

Let $V$ be a superspace and $d\in\NN$. The symmetric group $\Si_d$ acts on $V^{\otimes d}$ via
\begin{align*}
{}^{w}(v_1\otimes\dots\otimes v_d):=(-1)^{[w;v_1,\dots,v_d]}v_{w^{-1}(1)}\otimes\dots\otimes v_{w^{-1}(d)},
\end{align*}
where
$$
[w;v_1,\dots,v_d]:=\sum_{1\leq a<c\leq d,\, w(a)>w(c)}|v_a||v_c|.
$$
\iffalse{
Equivalently,
\begin{align*}
{}^{s_r}(v_1\otimes\dots\otimes v_d):=(-1)^{|v_r||v_{r+1}|}v_{s_r(1)}\otimes\dots\otimes v_{s_r(d)},
\end{align*}
for all $1 \leq r \leq d-1$, where $s_r$ denotes the simple transposition $(r,r+1) \in \Si_d$.
}\fi

Following \cite[\S2.2a]{KL} (where we have worked over $\F$), given a superalgebra $A$, we consider the {\em wreath superproduct} $A\swr \Si_d$ to be $A^{\otimes d}\otimes \cO\Si_d$ as superspaces, with $\cO \Si_d$ concentrated in parity $\0$. To describe the algebra structure we identify $A^{\otimes d}$ and $\cO \Si_d$ as subalgebras in $A^{\otimes d}\otimes \cO \Si_d$ in the obvious way and define
\begin{align*}
w\,(a_1\otimes\dots\otimes a_d)={}^w(a_1\otimes\dots\otimes a_d)\,w\qquad(w\in\Si_d,\ a_1,\dots, a_d\in A).
\end{align*}
Note that, $A\swr \Si_d$ is an $\Si_d$-graded crossed superproduct via
\begin{align}\label{algn:AwrSd_graded}
A\swr \Si_d = \bigoplus_{w \in \Si_d} A^{\otimes d}w.
\end{align}

Recall from Example~\ref{ExTn} the twisted group superalgebra $\cT_d$ with basis $\{\ct_w\mid w\in\Si_d\}$. Following \cite[$\S$5.1a]{KL},  
we define the {\em twisted wreath superproduct} $A\swr \cT_d$  as the free product $A^{\otimes d}\star \cT_d$ of superalgebras subject to the relations
\begin{align}\label{algn:A_swr_Td}
\ct_r\,(a_1\otimes\dots\otimes a_d) = (-1)^{\sum_{u\neq r,r+1}|a_u|}\left({}^{s_r}(a_1\otimes\dots\otimes a_d)\,\ct_r\right)
\end{align}
for all $a_1,\dots,a_d\in A$ and $1\leq r \leq d-1$.

Recall Example~\ref{ExCl} and the notation from (\ref{algn:vr_def}).

\begin{Proposition}\label{prop:Sergeev}
We have an isomorphism of superalgebras
\begin{align*}
(A\otimes \cC_1)\swr {\Si}_d & \iso (A\swr \cT_d)\otimes \cC_d \\
(a \otimes x)_u \mapsto(-1)^{u|a|}a_u \otimes x_u, & \quad s_r\mapsto\frac{1}{\sqrt{-2}}\ct_r\otimes(\cc_r-\cc_{r+1})
%\\a_r \otimes x_r \mapsto(-1)^{r|a|}(a \otimes x)_r, & \quad \ct_r \otimes 1 \mapsto-\frac{1}{\sqrt{-2}}s_r(\cc_r-\cc_{r+1})
\end{align*}
for all $a\in A$, $x\in\cC_1$, $1\leq u\leq d$ and $1\leq r< d$.
\end{Proposition}

\begin{proof}
This was proved in \cite[Proposition 5.1.3]{KL} when everything is defined over a field. However, the proof over $\cO$ is identical, since $2$ is still invertible in $\cO$.
\end{proof}

Since $(\cc_r - \cc_{r+1})/\sqrt{-2}$ is invertible in $\cC_d$, an immediate consequence of Proposition~\ref{prop:Sergeev}  is that $A\swr \cT_d$ is an $\Si_d$-graded crossed superproduct via
\begin{align}\label{AswrT_graded}
A\swr \cT_d = \bigoplus_{w \in \Si_d} A^{\otimes d}\ct_w.
\end{align}
Given the $\Si_d$-grading in (\ref{AswrT_graded}), recall the definition of $(A \swr \cT_d,A \swr \cT_d)_{\Si_d}$ from (\ref{EDiag}).

\begin{Lemma}\label{lem:M_swr_Td}
Let $\gM$ be an $(A,A)$-bisupermodule. We can extend $\gM^{\boxtimes d}$ to an $(A \swr \cT_d,A \swr \cT_d)_{\Si_d}$-supermodule, that we denote $\gM^{\boxtimes d}_{\Si_d}$, via
\begin{align}\label{algn:ext_M^n}
(\ct_r\otimes \ct_r^{-1}).(m_1\otimes\dots\otimes m_d)
:=(-1)^{|m_r|+|m_{r+1}|}\left({}^{s_r}\,(m_1\otimes \dots \otimes m_d)\right),
\end{align}
for all $m_1,\dots,m_d\in \gM$ and $1\leq r\leq d-1$.
\end{Lemma}

\begin{proof}
This is essentially proved in the proof of \cite[Proposition 5.1.5]{KL}. In that proof the bisupermodule $\gM$ is assumed to induce a Morita superequivalence and our superalgebras are defined over a field. However, neither of these two details alter the proof at all.
%\textcolor{red}{Do you think this is sufficient, Sasha?}
\end{proof}

For an $(A,A)$-bisupermodule $\gM$, we now define
\begin{align}\label{algn:sup_wreath_bimod}
\gM \swr \cT_d := \Ind_{(A \swr \cT_d,A \swr \cT_d)_{\Si_d}}^{A \swr \cT_d \otimes (A \swr \cT_d)^{\sop}}(\gM^{\boxtimes d}_{\Si_d}).
\end{align}
Note that
$$
A \swr \cT_d \otimes (A \swr \cT_d)^{\sop} = \bigoplus_{w \in \Si_d} (\ct_w \otimes 1)(A \swr \cT_d,A \swr \cT_d)_{\Si_d}.
$$
Therefore, using the bimodule notation, we can write
\begin{align}\label{algn:M_wr_Td_sum}
\gM \swr \cT_d = \bigoplus_{w \in \Si_d}\ct_w \gM^{\boxtimes d}.
\end{align}
In particular,
\begin{align}\label{M_wr_Td_left}
\gM \swr \cT_d \simeq (A \swr \cT_d) \otimes_{A^{\otimes d}} \gM^{\boxtimes d},
\end{align}
as $(A \swr \cT_d,A^{\otimes d})$-bisupermodules.

\begin{Lemma}\label{lem:wreath_tensor}
Let $\gM$ and $\gN$ be $(A,A)$-bisupermodules. We have:
\begin{enumerate}
\item $\gM \swr \cT_d\mid(\gM \oplus \gN) \swr \cT_d$.
\item
$
(\gM \swr \cT_d) \otimes_{A \swr \cT_d} (\gN \swr \cT_d) \simeq (\gM \otimes_A \gN) \swr \cT_d.
$
\end{enumerate}
\end{Lemma}

\begin{proof}
(i) Certainly $(\gM \oplus \gN)^{\boxtimes d} = \gM^{\boxtimes d} \oplus \gM'$, where
\begin{align*}
\gM':=\bigoplus_{\substack{\gM_i \in \{\gM,\gN\} \\ \text{at least one }\gM_i = \gN}} \gM_1 \boxtimes \dots \boxtimes \gM_d.
\end{align*}
Moreover, $\gM^{\boxtimes d}$ and $\gM'$ are both $(A \swr \cT_d,A \swr \cT_d)_{\Si_d}$-subsupermodules of $(\gM \oplus \gN)^{\boxtimes d}$. The claim now follows by inducing up to $A \swr \cT_d \otimes (A \swr \cT_d)^{\sop}$.

(ii) \iffalse{We first note using (\ref{AswrT_graded}) that 
$
\gN \swr \cT_d = \bigoplus_{w \in \Si_d} \ct_w \gN^{\boxtimes d}.
$
\textcolor{blue}{This might be too slick and requires some explanations. By definition, 
$\gN \swr \cT_d$ is $\Ind_{(A \swr \cT_d,A \swr \cT_d)_{\Si_d}}^{A \swr \cT_d \otimes (A \swr \cT_d)^{\sop}}(\gN^{\boxtimes d}_{\Si_d})$, and it does not have elements of the form $\ct_w(n_1\otimes\dots\otimes n_d)$. Rather it has elements of the form $x\otimes (n_1\otimes\dots\otimes n_d)$ where $x\in A \swr \cT_d \otimes (A \swr \cT_d)^{\sop}$!
}
In particular,
\begin{align*}
\gN \swr \cT_d \simeq (A \swr \cT_d) \otimes_{A^{\otimes d}} \gN^{\boxtimes d}
\end{align*}
as $(A \swr \cT_d,A^{\otimes d})$-bisupermodules. 
\textcolor{blue}{I would say the above isomorphism needs to be stated just after (3.29) either as a named formula or as a lemma.}
}\fi
We first note that
\begin{align}
%\begin{split}
\label{algn:ten_sup_wreath}
(\gM \swr \cT_d) \otimes_{A \swr \cT_d} (\gN \swr \cT_d) \simeq  (\gM \swr \cT_d) \otimes_{A^{\otimes d}} \gN^{\boxtimes d} 
=  \bigoplus_{w \in \Si_d} \ct_w \big(\gM^{\boxtimes d} \otimes_{A^{\otimes d}} \gN^{\boxtimes d}\big)
%\end{split}
\end{align}
as $(A^{\otimes d},A^{\otimes d})$-bisupermodules, where the first isomorphism follows from (\ref{M_wr_Td_left}) and the second from (\ref{algn:M_wr_Td_sum}). In particular, the canonical homomorphism of $(A^{\otimes d},A^{\otimes d})$-bisupermodules
\begin{align*}
\gM^{\boxtimes d} \otimes_{A^{\otimes d}} \gN^{\boxtimes d} & \to (\gM \swr \cT_d) \otimes_{A \swr \cT_d} (\gN \swr \cT_d)
\end{align*}
is injective. Moreover, if we identify $\gM^{\boxtimes d} \otimes_{A^{\otimes d}} \gN^{\boxtimes d}$ with its image under the above map, it also follows from (\ref{algn:ten_sup_wreath}) that $(\gM \swr \cT_d) \otimes_{A \swr \cT_d} (\gN \swr \cT_d)$ is $\Si_d$-graded. Note that $\gM^{\boxtimes d} \otimes_{A^{\otimes d}} \gN^{\boxtimes d}$ is even an $(A \swr \cT_d,A \swr \cT_d)_{\Si_d}$-subsupermodule of $(\gM \swr \cT_d) \otimes_{A \swr \cT_d} (\gN \swr \cT_d)$ with
\begin{align}
\begin{split}\label{algn:tr_action}
&(\ct_r \otimes \ct_r^{-1}).[(m_1 \otimes \dots \otimes m_d) \otimes (n_1 \otimes \dots \otimes n_d)]\\
= & [(\ct_r \otimes \ct_r^{-1}).(m_1 \otimes \dots \otimes m_d)] \otimes [(\ct_r \otimes \ct_r^{-1}).(n_1 \otimes \dots \otimes n_d)],
\end{split}
\end{align}
for all $1\leq r\leq d-1$, $m_1,\dots,m_d \in \gM$ and $n_1,\dots,n_d \in \gN$. Therefore, by (\ref{algn:ten_sup_wreath}),
\begin{align*}
(\gM \swr \cT_d) \otimes_{A \swr \cT_d} (\gN \swr \cT_d) \simeq \Ind_{(A \swr \cT_d,A \swr \cT_d)_{\Si_d}}^{A \swr \cT_d \otimes (A \swr \cT_d)^{\sop}} \big(\gM^{\boxtimes d} \otimes_{A^{\otimes d}} \gN^{\boxtimes d}\big),
\end{align*}
as $(A \swr \cT_d,A \swr \cT_d)$-bisupermodules. The claim will be complete once we have shown that
\begin{align*}
\varphi:\gM^{\boxtimes d} \otimes_{A^{\otimes d}} \gN^{\boxtimes d} &\to (\gM \otimes_A \gN)_{\Si_d}^{\boxtimes d}\\
(m_1 \otimes \dots \otimes m_d) \otimes (n_1 \otimes \dots \otimes n_d) & \mapsto (-1)^{\sum_{i>j}|m_i||n_j|} (m_1 \otimes n_1) \otimes \dots \otimes (m_d \otimes n_d)
\end{align*}
is an isomorphism of $(A \swr \cT_d,A \swr \cT_d)_{\Si_d}$-supermodules, where $(\gM \otimes_A \gN)_{\Si_d}^{\boxtimes d}$ is defined via Lemma \ref{lem:M_swr_Td}. That it is an isomorphism of $(A^{\otimes d},A^{\otimes d})$-bisupermodules follows immediately from Lemma \ref{lem:tensor_bisupmod} and induction on $d$. Now, for all $1\leq r\leq d-1$, $m_1,\dots,m_d \in \gM$ and $n_1,\dots,n_d \in \gN$,
\begin{align}
\begin{split}\label{algn:first_tr_action}
&(\ct_r \otimes \ct_r^{-1}).[(m_1 \otimes \dots \otimes m_d) \otimes (n_1 \otimes \dots \otimes n_d)]\\
= & [(\ct_r \otimes \ct_r^{-1}).(m_1 \otimes \dots \otimes m_d)] \otimes [(\ct_r \otimes \ct_r^{-1}).(n_1 \otimes \dots \otimes n_d)]\\
= & (-1)^{|m_r|+|m_{r+1}|+|n_r|+|n_{r+1}|} [{}^{s_r}(m_1\otimes \dots \otimes m_d)] \otimes [{}^{s_r}(n_1\otimes \dots \otimes n_d)] \\
= & (-1)^{C_1} (m_{s_r(1)}\otimes \dots \otimes m_{s_r(d)}) \otimes (n_{s_r(1)}\otimes \dots \otimes n_{s_r(d)}),
\end{split}
\end{align}
where
\begin{align*}
C_1 = |m_r|+|m_{r+1}|+|n_r|+|n_{r+1}|+|m_r||m_{r+1}|+|n_r||n_{r+1}|,
\end{align*}
while
\begin{align}
\begin{split}\label{algn:second_tr_action}
&(\ct_r \otimes \ct_r^{-1}).[(m_1 \otimes n_1) \otimes \dots \otimes (m_d \otimes n_d)]\\
= & (-1)^{|m_r \otimes n_r|+|m_{r+1} \otimes n_{r+1}|}\left({}^{s_r}[(m_1 \otimes n_1) \otimes \dots \otimes (m_d \otimes n_d)]\right)\\
= & (-1)^{|m_r|+|n_r|+|m_{r+1}|+|n_{r+1}|}\left({}^{s_r}[(m_1 \otimes n_1) \otimes \dots \otimes (m_d \otimes n_d)]\right) \\
= & (-1)^{C_2} (m_{s_r(1)} \otimes n_{s_r(1)}) \otimes \dots \otimes (m_{s_r(d)} \otimes n_{s_r(d)}),
\end{split}
\end{align}
where
\begin{align*}
C_2 = |m_r|+|n_r|+|m_{r+1}|+|n_{r+1}| + (|m_r|+|m_{r+1}|)(|n_r|+|n_{r+1}|).
\end{align*}
Note that
\begin{align*}
C_2 = C_1 + |m_r||n_{r+1}| + |m_{r+1}||n_r|.
\end{align*}
Finally, a direct calculation shows that
and also that
\begin{align}
\begin{split}\label{algn:C3_equ}
&\varphi\big((m_{s_r(1)} \otimes \dots \otimes m_{s_r(d)}) \otimes (n_{s_r(1)} \otimes \dots \otimes n_{s_r(d)})\big) \\
= & (-1)^{C_3 + \sum_{i>j}|m_i||n_j|} (m_{s_r(1)} \otimes n_{s_r(1)}) \otimes \dots \otimes (m_{s_r(d)} \otimes n_{s_r(d)}),
\end{split}
\end{align}
where $C_3 = |m_r||n_{r+1}| + |m_{r+1}||n_r|$. Putting together (\ref{algn:first_tr_action}), (\ref{algn:second_tr_action}) and (\ref{algn:C3_equ}), we have now shown that the action of $\ct_r \otimes \ct_r$ commutes with $\varphi$. The claim follows.
\end{proof}

\begin{Lemma}\label{lem:dual_sup_wreath}
Let $A$ be a superalgebra with superunit and $\gM$ an $(A,A)$-bisupermodule. Then $(\gM \swr \cT_d)^* \simeq \gM^* \swr \cT_d$ as $(A \swr \cT_d,A \swr \cT_d)$-bisupermodules.
\end{Lemma}

\begin{proof}
Let $u \in A_{\1} \cap A^\times$. We consider the wreath product
$$C_2\wr {\Si}_d=\{g_1^{\eps_1}\cdots g_d^{\eps_d}w\mid \eps_1,\dots,\eps_d\in\{0,1\},\, w\in \Si_d\}$$
where $g_r$ is the generator of $r^{\nth}$ $C_2$ factor in the base group $C_2^{\times d}$. With (\ref{AswrT_graded}) in mind we can give $A\swr \cT_d$ the structure of a $C_2\wr {\Si}_d$-graded crossed superproduct with base component $A_{\0}^{\otimes d}$ and graded components
$$
(A\swr \cT_d)_{g_1^{\eps_1}\cdots g_d^{\eps_d}w}=
A_{\0}^{\otimes d}u_1^{\eps_1}\dots u_d^{\eps_d} \ct_w,
$$
for all $\eps_1,\dots,\eps_d\in\{0,1\}$ and $w\in \Si_d$, where we are utilising the notation from (\ref{algn:vr_def}). We can now define $(A\swr \cT_d,A\swr \cT_d)_{C_2\wr\Si_d}$ as in (\ref{EDiag}). Note that $(A\swr \cT_d,A\swr \cT_d)_{C_2\wr\Si_d}$ is generated by $A_{\0}^{\otimes d} \otimes (A_{\0}^{\otimes d})^{\op}$ together with $u_i \otimes u_i^{-1}$, for $1\leq i\leq d$ and $\ct_w \otimes \ct_w^{-1}$, for all $w \in \Si_d$. We now extend the $A_{\0}^{\otimes d} \otimes (A_{\0}^{\otimes d})^{\op}$-module $\gM_{\0}^{\boxtimes d}$ to an $(A\swr \cT_d,A\swr \cT_d)_{C_2\wr\Si_d}$-module, that we denote $(\gM_{\0}^{\boxtimes d})_{C_2 \wr \Si_d}$, via
\begin{align}
\begin{split}\label{algn:def_Ui_Tw}
(u_i \otimes u_i^{-1})(m_1 \otimes \dots \otimes m_d) & = u_i(m_1 \otimes \dots \otimes m_d)u_i^{-1} \\
& = m_1 \otimes \dots \otimes u m_i u^{-1} \otimes \dots \otimes m_d \\
(\ct_w \otimes \ct_w^{-1})(m_1 \otimes \dots \otimes m_d) & = \ct_w (m_1 \otimes \dots \otimes m_d)\ct_w^{-1} \\
& = m_{w^{-1}(1)} \otimes \dots \otimes m_{w^{-1}(d)},
\end{split}
\end{align}
for all $1\leq i \leq d$, $w \in \Si_d$ and $m_1 \otimes \dots \otimes m_d \in \gM_{\0}^{\boxtimes d}$. Note that no relations need to be checked here as this is just $\gM_{\0}^{\boxtimes d}$ viewed as an $(A\swr \cT_d,A\swr \cT_d)_{C_2\wr\Si_d}$-submodule of $M^{\boxtimes d}_{\Si_d}$, as defined in (\ref{algn:ext_M^n}). In particular,
\begin{align*}
\gM^{\boxtimes d}_{\Si_d} = \bigoplus_{\eps_1,\dots,\eps_d\in\{0,1\}} u_1^{\eps_1}\dots u_d^{\eps_d} (\gM_{\0}^{\boxtimes d})_{C_2 \wr \Si_d}
\end{align*}
and so
\begin{align*}
\gM^{\boxtimes d}_{\Si_d} \simeq \Ind_{(A\swr \cT_d,A\swr \cT_d)_{C_2\wr\Si_d}}^{(A\swr \cT_d,A\swr \cT_d)_{\Si_d}}(\gM_{\0}^{\boxtimes d})_{C_2 \wr \Si_d}.
\end{align*}
Therefore, by (\ref{algn:sup_wreath_bimod}),
\begin{align*}
\gM \swr \cT_d \simeq \Ind_{(A\swr \cT_d,A\swr \cT_d)_{C_2\wr\Si_d}}^{A\swr \cT_d \otimes (A\swr \cT_d)^{\sop}}(\gM_{\0}^{\boxtimes d})_{C_2 \wr \Si_d}.
\end{align*}
In exactly the same way we can extend $(\gM_{\0}^*)^{\boxtimes d}$ to the $(A\swr \cT_d,A\swr \cT_d)_{C_2\wr\Si_d}$-module $((\gM_{\0}^*)^{\boxtimes d})_{C_2 \wr \Si_d}$ that satisfies
\begin{align}\label{algn:dual_isom}
\gM^* \swr \cT_d \simeq \Ind_{(A\swr \cT_d,A\swr \cT_d)_{C_2\wr\Si_d}}^{(A\swr \cT_d \otimes (A\swr \cT_d)^{\sop}}\big((\gM_{\0}^*)^{\boxtimes d}\big)_{C_2 \wr \Si_d}.
\end{align}
To complete the proof we construct an isomorphism
\begin{align*}
\Ind_{(A\swr \cT_d,A\swr \cT_d)_{C_2\wr\Si_d}}^{(A\swr \cT_d \otimes (A\swr \cT_d)^{\sop}}\big((\gM_{\0}^*)^{\boxtimes d}\big)_{C_2 \wr \Si_d} & \to \Big(\Ind_{(A\swr \cT_d,A\swr \cT_d)_{C_2\wr\Si_d}}^{(A\swr \cT_d \otimes (A\swr \cT_d)^{\sop}}(\gM_{\0}^{\boxtimes d})_{C_2 \wr \Si_d} \Big)^*.
\end{align*}
To do this we first set
\begin{align*}
a_g := u_1^{\eps_1} \dots u_d^{\eps_d}\ct_w \in (A \swr \cT_d)_g,
\end{align*}
for each $g = g_1^{\eps_1} \dots g_d^{\eps_d} w \in C_2 \wr \Si_d$. Next we identify $(\gM_{\0}^*)^{\boxtimes d}$ and $(\gM_{\0}^{\boxtimes d})^*$ via
\begin{align*}
(f_1 \otimes \dots \otimes f_d)(m_1 \otimes \dots \otimes m_d) := f_1(m_1)\cdot\ldots\cdot f_d(m_d),
\end{align*}
for all $f_1 \otimes \dots \otimes f_d \in (\gM_{\0}^*)^{\boxtimes d}$ and $m_1 \otimes \dots \otimes m_d \in \gM_{\0}^{\boxtimes d}$. It is trivial to check that this is an isomorphism of $(A \swr \cT_d,A \swr \cT_d)_{C_2 \wr \Si_d}$-supermodules, as there are no signs to check in (\ref{algn:def_Ui_Tw}).

%(This is a point where it is particularly convenient to think of bisupermodules in the manner described in the comments preceding the lemma.)

We now construct the isomorphism in (\ref{algn:dual_isom}) via
\begin{align*}
f \otimes a_g & \mapsto \bigg(a_h \otimes m \mapsto
\begin{cases}
f(a_ga_hm) & \text{if }h=g^{-1} \\
0 & \text{otherwise.}
\end{cases}\bigg),
\end{align*}
for all $g,h \in C_2 \wr \Si_d$, $f \in (\gM_{\0}^*)^{\boxtimes d}$ and $m \in \gM_{\0}^{\boxtimes d}$. Much like the proof of Lemma \ref{lem:ind_dual}, one can now readily check that this does indeed define an isomorphism.% Indeed this really is just a particular application of Lemma \ref{lem:ind_dual}. However, it is rather tricky to pick apart exactly what that lemma is saying in the case of bisupermodules, hence the more direct proof of the current lemma.
\end{proof}

\subsection{Split semisimple \texorpdfstring{$\K$}{}-superalgebras}\label{sec:KSS}

For much of this subsection we mirror the results of \cite[$\S$12.2]{Kbook}. There it is assumed that $\K$ is algebraically closed. We do not make that assumption here meaning we cannot directly refer to the results from \cite{Kbook}.

Let $A$ be a $\K$-superalgebra. An irreducible $A$-supermodule $\gM$ is of type $\Mtype$ if $|\gM|$ is an irreducible $|A|$-module, and type $\Qtype$ otherwise. $A$ is called {\em split}\, if $\dim_\K\End_{|A|}(|\gM|)=1$ for all irreducible $\gM$ of type $\Mtype$ and $\dim_\K\End_{|A|}(|\gQ|)=2$ for all irreducible $\gQ$ of type $\Qtype$. An $A$-supermodule $\gM$ is called {\em semisimple} if it is isomorphic to a direct sum of irreducible $A$-supermodules. The superalgebra $A$ is called {\em semisimple} if every $A$-supermodule is semisimple.

Recall the algebra $\hat A$ from $\S$\ref{SSMod}. Using the isomorphism of the categories $\underlinesmod{A}$ and $\mod{\hat A}$, $A$ being semisimple is equivalent to $\hat A$ being semisimple as an algebra. This, in turn, is equivalent to $\hat A = \hat A e_{\0} \oplus \hat A e_{\1}$ being semisimple as an $\hat A$-module. Now, $\hat A e_{\0}$ and $\hat A e_{\1}$ correspond to the $A$-supermodules $A$ and $\Pi A$ respectively. Therefore, $\hat A$ being semisimple corresponds to $A$ and $\Pi A$ being semisimple $A$-supermodules. Certainly $\Pi A$ is semisimple if and only if $A$ is. We have, therefore, shown that $A$ is semismple if and only if $A$ is semisimple as an $A$-supermodule.

For the next lemma we introduce the $\K$-linear map $\si_{\gM}:\gM \to \gM$, $m \mapsto (-1)^{|m|}m$, for an $A$-supermodule $\gM$. Suppose that $\gN$ is an $|A|$-submodule of $|\gM|$. Then $\si_{\gM}(\gN)$ is a submodule of $|\gM|$ isomorphic to $\gN^{\si}$. Moreover, $\gN$ is a subsupermodule of $\gM$ if and only if $\si_{\gM}(\gN) = \gN$.

\begin{Lemma}\label{lem:irred_sup_mod}
Let $A$ be a split $\K$-superalgebra.
\begin{enumerate}
\item Let $\gM$ be an irreducible $A$-supermodule of type $\Mtype$.
\begin{enumerate}
\item $|\gM|$ is an irreducible $|A|$-module. 
\item $\End_{A}(\gM) \cong \K$ as superalgebras.
\item $|\gM|\cong |\Pi\gM|$ but\, $\gM \not\simeq \Pi \gM$.
\end{enumerate}
\item Let $\gQ$ be an irreducible $A$-supermodule of type $\Qtype$.
\begin{enumerate}
\item $|\gQ| \cong \gN \oplus \gN^{\si}$, for some irreducible $|A|$-module $\gN$, with $\gN \ncong \gN^{\si}$.
\item $\End_{A}(\gQ) \cong \cQ_1(\K)$ as superalgebras.
\item $\gQ \simeq \Pi \gQ$.
\end{enumerate}
\item If $\gM$ and $\gN$ are irreducible $A$-supermodule, then precisely one of the following occurs:
\begin{enumerate}
\item $\gM \simeq \gN$,
\item $\gM$ and $\gN$ are both of type $\Mtype$ and $\gM \simeq \Pi\gN$,
\item $\Hom_{|A|}(|\gM|,|\gN|) = \{0\}$.
\end{enumerate}
\end{enumerate}
\end{Lemma}

\begin{proof}
(i) Since $\gM$ is of type $\Mtype$, parts (a) and (b) are clear by definitions. Part (c) just follows by the definition of $\Pi \gM$ and part (b).

(ii) We follow \cite[Lemma 12.2.1]{Kbook}. Let $\gN \subseteq |\gQ|$ be an irreducible $|A|$-submodule. Since, $|\gQ|$ is not irreducible, $\gN$ cannot be an $A$-subsupermodule of $\gQ$ and $\si_{\gQ}(\gN) \neq \gN$. However, $\gN + \si_{\gQ}(\gN) = \gN \oplus \si_{\gQ}(\gN)$ is $\si_{\gQ}$-stable. Since we have $\dim_{\K}\End_{|A|}(|\gQ|)=2$, $\gN$ and $\si_{\gQ}(\gN) \cong \gN^{\si}$ must be non-isomorphic, proving part (a).

We now define $J \in \End_{A}(\gQ)$ by $(-\Id_{\gN}) \oplus \Id_{\si_{\gQ}(\gN)}$. Certainly $J$ is an $A$-module isomorphism. Moreover, $J$ anti-commutes with $\si_{\gQ}$ and so $J$ is odd. In particular, $\gQ \simeq \Pi\gQ$, proving part (c). Since $\dim_{\K}\End_{A}(\gQ)=\dim_{\K}\End_{|A|}(|\gQ|)=2$, we can now construct the isomorphism $\End_{A}(\gQ) \cong \cQ_1(\K)$, $J \mapsto (1,-1)$.

(iii) Suppose both (a) and (b) hold. Then, $\gM \simeq \Pi \gM$, contradicting part (i). Certainly (c) cannot hold together with either (a) or (b). It remains to show that at least one of the three statements is true.

Suppose (a) and (c) fail to hold. So, $\gM \not\simeq \gN$ and $
|\Hom_{A}(\gM,\gN)|=
\Hom_{|A|}(|\gM|,|\gN|) \neq \{0\}$. Now, $\Hom_{A}(\gM,\gN)_{\0} = \{0\}$, since $\gM$ and $\gN$ are non-isomorphic irreducible $A$-supermodules. Therefore, $\Hom_{A}(\gM,\gN)_{\1} \neq \{0\}$ or equivalently $\Hom_{A}(\gM,\Pi \gN)_{\0} \neq \{0\}$. Since $\gM$ and $\Pi\gN$ are irreducible $A$-supermodules, we must have $\gM \simeq \Pi \gN$. Finally, if $\gN$ has type $\Qtype$, then $\gM \simeq \Pi \gN \simeq \gN$, a contradiction. Similarly, $\gM$ has type $\Mtype$ and (b) holds.
\end{proof}

We now take a brief moment to describe the irreducible supermodules for $\cM_{m|n}(\K)$ and $\cQ_t(\K)$.
Let $\gU_{m,n}$ be the standard column vector supermodule for $\cM_{m|n}(\K)$, where the first $m$ entries are considered even and the last $n$ odd. Clearly $\gU_{m,n}$ is an irreducible $\cM_{m|n}(\K)$-supermodule of type $\Mtype$.

Let $\gV_t := \gV_1 \oplus \gV_2$, where $\gV_1$ and $\gV_2$ are the standard column vector spaces for the two matrix factors in $\cQ_t(\K)$. We identify $\gV_1$ and $\gV_2$ through $\si_{\cQ_t(\K)}$ and give $\gV_t$ the structure of a $\cQ_t(\K)$-supermodule by setting
\begin{align*}
(\gV_t)_{\0} := \{(v,v) \in \gV_t\}, \qquad (\gV_t)_{\1} := \{(v,-v) \in \gV_t\}.
\end{align*}
Since $\gV_1$ and $\gV_2$ are non-isomorphic, irreducible $|\cQ_t(\K)|$-modules and $\si_{\gV_t}(\gV_1) = \gV_2$, $\gV_t$ is an irreducible $\cQ_t(\K)$-supermodule of type $\Qtype$.

We claim that, up to isomorphism, $\gU_{m,n}$ and $\Pi \gU_{m,n}$ are the only irreducible $\cM_{m|n}(\K)$-supermodules and $\gV_t$ is the only irreducible $\cQ_t(\K)$-supermodule. By construction, in either case, any irreducible supermodule must have a non-zero module homomorphism from one of the above irreducible supermodules already constructed. The claim now follows from Lemma \ref{lem:irred_sup_mod}(iii).

The following lemma is well known. For example it can be deduced from results in \cite{Jo} and \cite{NO}. However, since most texts do not deal in the generality we do here, we include a proof for the convenience of the reader.

\begin{Lemma}\label{lem:split_semi}
Let $A$ be a $\K$-superalgebra. The following are equivalent:
\begin{enumerate}
\item $A$ is split semisimple.
\item $A$ is a direct sum of $\K$-superalgebras each of the form $\cM_{m|n}(\K)$ or $\cQ_t(\K)$, for $m,n,t \in \N$, with $m+n,t>0$.
\end{enumerate}
Moreover, every irreducible $|A|$-module is isomorphic to the direct summand of an irreducible $A$-supermodule considered as an $|A|$-module.
\end{Lemma}

\begin{proof}
(i)$\implies$(ii) Suppose $A$ is split semisimple. In particular, the left regular supermodule $A$ decomposes into a direct sum of irreducible $A$-supermodules. By Lemma \ref{lem:irred_sup_mod}, we can write
\begin{align*}
A \simeq \bigoplus_{i=1}^a \left( \gM_i^{\oplus m_i} \oplus (\Pi \gM_i)^{\oplus m_i'} \right) \oplus \bigoplus_{i=1}^b \gQ_i^{\oplus q_i},
\end{align*}
for some $a, b, m_i, m_i', q_i \in \N$, where the $\gM_i$'s and $\Pi \gM_i$'s form a complete list of representatives for the isomorphism classes of irreducible $A$-supermodules of type $\Mtype$ and the $\gQ_i$'s form a complete list of representatives for the isomorphism classes of irreducible $A$-supermodules of type $\Qtype$. In particular, by Lemma \ref{lem:irred_sup_mod}, we have 
\begin{align*}
\End_{A}(A) & \cong \bigoplus_{i=1}^a \End_{A}\big( \gM_i^{\oplus m_i} \oplus (\Pi \gM_i)^{\oplus m_i'} \big) \oplus \bigoplus_{i=1}^b \End_{A}(\gQ_j^{\oplus q_i}) \\
& \cong \bigoplus_{i=1}^a \cM_{m_i|m_i'}(\K) \oplus \bigoplus_{i=1}^b \cQ_{q_j}(\K).
\end{align*}
Now, $\End_{A}(A) \cong A^{\op}$ as superalgebras, where the isomorphism is given by right multiplication. The claim now follows by the isomorphisms of superalgebras:
\begin{align*}
\cM_{m_i|m_i'}(\K)^{\op} & \cong \cM_{m_i|m_i'}(\K) && \cQ_{q_i}(\K)^{\op} \cong \cQ_{q_i}(\K) \\
x & \mapsto x^T && \qquad (x,y) \mapsto (x^T,y^T).
\end{align*}
(ii)$\implies$(i) It is enough to prove the statement for $A = \cM_{m|n}(\K)$ and $\cQ_t(\K)$, for $m,n,t \in \N$, with $m+n,t>0$.

By the comments at the beginning of this subsection, to show $A$ is semisimple, we need only decompose $A$ into a direct sum of irreducible supermodules. Recall the irreducible $\cM_{m|n}(\K)$-supermodule $\gU_{m,n}$ and the irreducible $\cQ_t(\K)$-supermodule $\gV_t$ from the comments preceding the lemma. Now,
$
\cM_{m|n}(\K) \simeq (\gU_{m,n})^{\oplus m} \oplus (\Pi \gU_{m,n})^{\oplus n},
$
as $\cM_{m|n}(\K)$-supermodules, and
$
\cQ_t(\K) \simeq \gV_t^{\oplus t},
$
as $\cQ_t(\K)$-supermodules, proving semisimplicity.
Finally,
$$
\dim_{\K}\End_{|\cM_{m|n}(\K)|}(|\gU_{m,n}|) = \dim_{\K}\End_{|\cM_{m|n}(\K)|}(|\Pi \gU_{m,n}|) = 1
$$
and
$
\dim_{\K}\End_{|\cQ_t(\K)|}(\gV_t) = 2,
$
so all our irreducible supermodules are split, completing the claim.

We now prove the final part of the lemma. Since we have decomposed $A$ into a direct sum of irreducible $A$-supermodules, by Lemma \ref{lem:irred_sup_mod}(i)(a),(ii)(a), we can also decomposed $|A|$ into a direct sum of irreducible $|A|$-modules. Since every irreducible $|A|$-module is a homomorphic image of $|A|$, the claim follows.
\end{proof}

\subsection{(Super)characters of split semisimple superalgebras}\label{sec:KSS1}
Throughout the subsection, let $A$ be a split semisimple $\K$-superalgebra. In particular, $|A|$ is a split semisimple $\K$-algebra, and we have the terminology of irreducible characters of $|A|$ from \S\ref{sec:SSA}. By definition, these are elements of the Grothendieck group $\Grot(|A|)$. 

For simplicity we will write $\Grot(A)$ instead of $\Grot(|A|)$.

Recalling the operation $\gM\mapsto \gM^\si$ introduced in \S\ref{SSMod}, we denote by $$-^{\si}:\Grot(A) \to \Grot(A)$$ the corresponding automorphism of the Grothendieck group. 

The  set of irreducible $A$-supermodules (up to isomorphism $\simeq$) can be written as  
$$\{X_\la,\Pi X_\la\mid\la\in \La^A_\0\}\sqcup \{X_\la\mid\la\in \La^A_\1\},
$$ 
where the first set contains type $\Mtype$ irreducible $A$-supermodules, the second set contains type $\Qtype$ irreducible $A$-supermodules, and $\La^A_\0,\La^A_\1$ are just labeling sets. 
Later on, particularly when dealing with characters of double covers of the symmetric groups, it will be natural to call elements of $\La^A_{\0}$ {\em even} and elements of $\La^A_{\1}$ {\em odd}. 
We put $\La^A:=\La^A_\0\sqcup \La^A_\1$. For the corresponding classes in the Grothendieck group $\Grot(A)$, we use the following notation:
\begin{align*}
\xi_\la&:=[|X_\la|]=[|\Pi X_\la|]\in \Grot(A) &(\la\in\La^A_\0),
\\
\xi_\la&:=[|X_\la|]\in \Grot(A) &(\la\in\La^A_\1).
\end{align*}
We call the classes $\xi_\la$ the {\em irreducible supercharacters} of $A$. We denote the set of the irreducible supercharacters of $A$ by $\Irr_{\operatorname{super}}(A)$. 
Thus 
$$
\Irr_{\operatorname{super}}(A)=\{\xi_\la\mid\la\in\La^A\}.
$$
These need to be distinguished from the irreducible characters of $|A|$ which we denote simply by $\Irr(A)$ so that $\Grot(A)=\Z\Irr(A)$.

%Recall the terminology of irreducible characters of $|A|$ from \S\ref{sec:SSA}. 
Let $\la \in \La^A_{\1}$. By Lemma \ref{lem:irred_sup_mod}(ii)(a), $\xi_\la = \xi_\la^+ + \xi_\la^-$, for some irreducible characters $\xi_\la^+,\xi_\la^- \in \Irr(A)$, with $\xi_\la^+ \neq \xi_\la^-$ and $(\xi_\la^\pm)^{\si} = \xi_\la^\mp$. Moreover, by Lemma \ref{lem:irred_sup_mod} and the final part of Lemma \ref{lem:split_semi},
\begin{equation}\label{E221222}
\Irr(A)=\{\xi_\la\mid\la \in \La^A_{\0}\} \sqcup \{\xi_\la^+,\xi_\la^-\mid\la \in \La^A_{\1}\}.
\end{equation}
is a complete irredundant list of irreducible characters of $|A|$.
%be an irreducible character of $A$.
Adopting the language of \S\ref{SSMod}, if $\la \in \La^A_{\0}$, we say $\xi_\la$ is {\em self-associate} and if $\la \in \La^A_{\1}$, we say $\xi_\la^+,\xi_\la^-$ are {\em non-self-associate} and call them an {\em associate pair}.  % We define $\Irr_{\si}(A)$ to be the set of equivalence classes of $\Irr(A)$, under the equivalence relation generated by $-^{\si}$. Occasionally, we will also use this terminology of (non-)self-associate to refer to an $|A|$-module $\gM$ that is (or isn't) isomorphic to $\gM^{\si}$.
In this latter case it is usually not going to be important to specify which character is $\xi_\la^+$ and which is $\xi_\la^-$. However, when it is important, we will make this choice clear.

When $A$ is a $\K$-split semisimple ($\cO$-)superalgebra, we denote
\begin{equation}\label{EIrrNotation}
\Irr_{\operatorname{super}}(A):=\Irr_{\operatorname{super}}(\K A)
\quad\text{and}\quad
\Irr(A):=\Irr(\K A)=\Irr(|\K A|). 
\end{equation}
In particular, in this case we have $\Grot(\K A)=\Z\Irr(A).$

Let $B$ and $C$ be $\K$-split semisimple superalgebras and $\gN$ a $(B,C)$-bisupermodule. 
As in \S\ref{sec:SSA}, $|\K\gN| \otimes_{|\K C|} ?$ induces a ($\Z$-)linear function $\Z\Irr(C) \to \Z\Irr(B)$ that, by an abuse of notation, we denote $\gN \otimes_C ?$.
If $C$ is a subsuperalgebra of $B$, %that is also $\K$-split semisimple,
then we denote by
\begin{align*}
\downarrow^B_C: \Z\Irr(B) \to \Z\Irr(C) \qquad \text{and} \qquad \uparrow_C^B: \Z\Irr(C) \to \Z\Irr(B)
\end{align*}
the linear functions induced by the functors $\Res^{|\K B|}_{|\K C|}$ and $\Ind_{|\K C|}^{|\K B|}$ respectively. Similar notation applies for superalgebras over $\K$---for example if $B$ and $C$ are split semisimple $\K$-superalgebras and $\gN$ a $(B,C)$-bisupermodule, then $|\gN| \otimes_{|C|} ?$ induces a ($\Z$-)linear function $\Z\Irr(C) \to \Z\Irr(B)$ that, by an abuse of notation, we denote $\gN \otimes_C ?$.

\begin{Lemma}\label{lem:res_chars}
Let $A$ be a split semisimple $\K$-superalgebra with superunit $u$. Then $A_{\0}$ is a split semisimple $\K$-algebra. Moreover, if $\la \in \La^A_{\0}$, then $\tilde{\xi}_\la := \xi_\la \downarrow_{A_{\0}}^A = \tilde{\xi}_\la^+ + \tilde{\xi}_\la^-$, for distinct $\tilde{\xi}_\la^+, \tilde{\xi}_\la^- \in \Irr(A_{\0})$, with ${}^u\tilde{\xi}_\la^+ = \tilde{\xi}_\la^-$. If $\la \in \La^A_{\1}$, then $\xi_\la^+ \downarrow_{A_{\0}}^A = \xi_\la^- \downarrow_{A_{\0}}^A = \tilde{\xi}_\la$, for some $\tilde{\xi}_\la \in \Irr(A_{\0})$. Furthermore,
\begin{align*}
\Irr(A_\0)=\{ \tilde{\xi}_\la^+, \tilde{\xi}_\la^- \mid  \la \in \La^A_{\0}\} \sqcup \{ \tilde{\xi}_\la \mid \la \in \La^A_{\1}\}
\end{align*}
is a complete, irredundant list of irreducible characters of $A_{\0}$.
\end{Lemma}

\begin{proof}
By Lemma \ref{lem:split_semi}, it is enough to prove the result for $A = \cM_{m|n}(\K)$ and $\cQ_t(\K)$, for $m,n,t \in \N$, with $m+n,t>0$. Recall the irreducible $\cM_{m|n}(\K)$-supermodule $\gU_{m,n}$ and the irreducible $\cQ_t(\K)$-supermodule $\gV_t = \gV_1 \oplus \gV_2$ from the comments preceding Lemma \ref{lem:split_semi}.

We first deal with $A \cong \cM_{m|n}(\K)$. Note that $\gU_{m,n}$ is of type $\Mtype$ meaning we are in the $\La^A_{\0}$ case. If $m\neq n$, then every element of $\cM_{m|n}(\K)_{\1}$ has determinant zero. This contradicts $A$ having a superunit and so we must have $m=n$. Therefore, $A_{\0} \cong \cM_{m \times m}(\K) \oplus \cM_{m \times m}(\K)$. Moreover, the two matrix factors gets swapped by conjugation by
\begin{align*}
u:=\begin{pmatrix} 0 & I_m \\ I_m & 0 \end{pmatrix} \in \cM_{m|n}(\K)^\times \cap \cM_{m|n}(\K)_{\1}.
\end{align*}
Certainly $\Res^{|A|}_{A_{\0}}|\gU_{m,m}|$ decomposes as a direct sum of two irreducible $|A|$-modules, one for each isomorphism class, proving all the claims.

Next we consider $A \cong \cQ_t(\K)$. This time $\gV_t$ is of type $\Qtype$ meaning we are in the $\La^A_{\1}$ case. Now, $A_{\0} \cong \cM_{t \times t}(\K)$ and $\Res^{|A|}_{A_{\0}}|\gV_1| \cong \Res^{|A|}_{A_{\0}}|\gV_2|$ is, up to isomorphism, the unique irreducible $A_{\0}$-module, proving all the claims.
\end{proof}

From now on we will adopt the labeling of $\Irr(A_{\0})$ as in Lemma \ref{lem:res_chars}. As with $\xi_\la^+,\xi_\la^- \in \Irr(|A|)$ when $\la$ is odd, most of the time we will not be careful about distinguishing between $\tilde{\xi}_{\la}^+$ and $\tilde{\xi}_{\la}^-$, for $\la$ even.

\begin{Remark}\label{rem:ind_chars}
As a consequence of Lemma \ref{lem:res_chars} and the fact that $\Ind_{A_{\0}}^{|A|}$ and $\Res^{|A|}_{A_{\0}}$ are adjoint functors, we can describe $\uparrow_{A_{\0}}^A:\Z\Irr(A_{\0}) \to \Z\Irr(A)$. If $\la \in \La^A_{\0}$, then $\tilde{\xi}_\la^+ \uparrow_{A_{\0}}^A = \tilde{\xi}_\la^- \uparrow_{A_{\0}}^A = \xi_\la$. If $\la \in \La^A_{\1}$, then $\tilde{\xi}_\la \uparrow_{A_{\0}}^A = \xi_\la^+ + \xi_\la^-$.
\end{Remark}

Let $A$ and $B$ be split semisimple $\K$-superalgebras, $\gU$ an $A$-supermodule and $\gV$ a $B$-supermodule such that $|\gU|$ has character $\chi \in \Z\Irr(A)$ and $|\gV|$ has character $\psi \in \Z\Irr(B)$. We write $\chi \boxtimes \psi \in \Z\Irr(A \otimes B)$ for the character of the $|A \otimes B|$-module $|\gU \boxtimes \gV|$.

\begin{Lemma}\label{lem:char_A_ten_B}
Let $A$ and $B$ be split semisimple $\K$-superalgebras. \begin{enumerate}
\item $A \otimes B$ is a split semisimple $\K$-superalgebra.
\item If $\la \in \La^A$ and $\mu \in \La^B$, then $\xi_\la \boxtimes \xi_\mu \in \Irr_{\operatorname{super}}(A \otimes B)$ unless $\la \in \La^A_{\1}$ and $\mu \in \La^B_{\1}$. In the latter case $\xi_\la \boxtimes \xi_\mu$ is the sum of two copies of the same irreducible supercharacter of $A \otimes B$.
\end{enumerate}
In all cases we write $\xi_{\la,\mu}$ for the unique irreducible constituent of $\xi_\la \boxtimes \xi_\mu$.
\begin{enumerate}
\item[(iii)]
$
\Irr_{\operatorname{super}}(A \otimes B)=\{\xi_{\la,\mu}\mid\la \in \La^A, \mu \in \La^B\}
$
is a complete irredundant set of irreducible supercharacters of $A \otimes B$. Moreover, $\xi_{\la,\mu}$ corresponds to an irreducible supermodule of type $\Mtype$ if $\xi_\la$ and $\xi_\mu$ have the same type and type $\Qtype$ if $\xi_\la$ and $\xi_\mu$ have opposite type.
\end{enumerate}
\end{Lemma}

\begin{proof}
All claims, except that our set in (iii) is irredundant, follow from \cite[Lemma 12.2.13]{Kbook}. There it is assumed that $\K$ is algebraically closed. However, the proof runs through in exactly the same manner once we have established Lemma \ref{lem:split_semi}.

For the irredundancy claim, if $\xi_{\la_1,\mu_1} = \xi_{\la_2,\mu_2}$, then, by restricting to $A$ and to $B$ we can show that $\la_1 = \la_2$ and $\mu_1 = \mu_2$.
\end{proof}

\begin{Remark} \label{RCircledast}
{\rm 
Occasionally, instead of $\xi_{\la,\mu}$ we use the notation $\xi_\la\circledast\xi_\mu$. This corresponds to the operation `$\circledast$' on supermodules as in \cite[\S12.2]{Kbook}.
}
\end{Remark}

Using the above lemma we make the following identifications
$$
\La^{A \otimes B}=\La^A \times \La^B, \quad
\La^{A \otimes B}_{\0} = \La^A_{\0} \times \La^B_{\0} \sqcup \La^A_{\1} \times \La^B_{\1}, \quad
\La^{A \otimes B}_{\1} = \La^A_{\0} \times \La^B_{\1} \sqcup \La^A_{\1} \times \La^B_{\0},
$$
so that 
%In particular,
$$
\Irr(A \otimes B)=\{\xi_{\la,\mu} \mid (\la,\mu) \in 
\La^{A \otimes B}_{\0}
%\La^A_{\0} \times \La^B_{\0} \sqcup \La^A_{\1} \times \La^B_{\1}
\} \sqcup \{\xi_{\la,\mu}^+,  \xi_{\la,\mu}^-\mid (\la,\mu) \in 
\La^{A \otimes B}_{\1}
%\La^A_{\0} \times \La^B_{\1} \sqcup \La^A_{\1} \times \La^B_{\0}
\}
$$
is a complete, irredundant list of irreducible characters of $|A \otimes B|$ and
\begin{align*}
&\xi_{\la,\mu}^{\si} = \xi_{\la,\mu} &\text{ if } (\la,\mu)\in \La^{A \otimes B}_{\0},
%\La^A_{\0} \times \La^B_{\0} \sqcup \La^A_{\1} \times \La^B_{\1}, 
\\
&(\xi_{\la,\mu}^{\pm})^{\si} = \xi_{\la,\mu}^{\mp}&\text{ if } (\la,\mu)\in 
\La^{A \otimes B}_{\1}.
%\La^A_{\0} \times \La^B_{\1} \sqcup \La^A_{\1} \times \La^B_{\0}.
\end{align*}

More generally, if $A_1,\dots,A_n$ are split semisimple $\K$-superalgebras, we can identify 
\begin{align*}
\La^{A_1 \otimes \dots\otimes A_n}
&=\La^{A_1} \times \dots\times \La^{A_n},
\\
\La^{A_1 \otimes \dots\otimes A_n}_\0
&=\{(\la^1,\dots,\la^n)\in \La^{A_1} \times \dots\times \La^{A_n}\mid \text{number of the odd $\la^i$ is even}\},
\\
\La^{A_1 \otimes \dots\otimes A_n}_\1
&=\{(\la^1,\dots,\la^n)\in \La^{A_1} \times \dots\times \La^{A_n}\mid \text{number of the odd $\la^i$ is odd}\},
\end{align*}
so that 
\begin{align}
\Irr_{\operatorname{super}}(A_1 \otimes \dots\otimes A_n)=\,\,&\{\xi_{\la^1, \dots,\la^n}\mid 
(\la^1,\dots,\la^n)\in \La^{A_1} \times \dots\times \La^{A_n}\},
\\
\begin{split}
\label{EIrrLabelN}
\Irr(A_1 \otimes \dots\otimes A_n)=\,\,&\{\xi_{\la^1, \dots,\la^n} \mid (\la^1,\dots,\la^n) \in 
\La^{A_1 \otimes \dots\otimes A_n}_\0
\} 
\\
&\sqcup \{\xi_{\la^1, \dots,\la^n}^+,\xi_{\la^1, \dots,\la^n}^- \mid (\la^1,\dots,\la^n) \in 
\La^{A_1 \otimes \dots\otimes A_n}_\1
\}.
\end{split}
\end{align}

\begin{Lemma}\label{lem:A_ten_B}
Let $A$ and $B$ be split semisimple $\K$-superalgebras with superunits. Then $A \otimes B$ is also a split semisimple $\K$-superalgebra with superunit. Furthermore, we can label the elements of $\Irr(A_{\0})$, $\Irr(B)$ and $\Irr(A \otimes B)$ such that
\begin{enumerate}
\item
\begin{align*}
&\xi_{\la,\mu}\downarrow^{A \otimes B}_{A_{\0} \otimes B} = (\tilde{\xi}_\la^+ \boxtimes \xi_\mu) + (\tilde{\xi}_\la^- \boxtimes \xi_\mu)& \text{if $\la\in\La^A_\0$ and $\mu\in\La^B_\0$;}
\\
&\xi_{\la,\mu}^\pm \downarrow^{A \otimes B}_{A_{\0} 
\otimes B}
=
(\tilde{\xi}_\la^+ \boxtimes \xi_\mu^\pm) + (\tilde{\xi}_\la^- \boxtimes \xi_\mu^\mp)& \text{if $\la\in\La^A_\0$ and $\mu\in\La^B_\1$;}
\\
&\xi_{\la,\mu}^\pm \downarrow^{A \otimes B}_{A_{\0} 
\otimes B}
=
\tilde{\xi}_\la \boxtimes \xi_\mu& \text{if $\la\in\La^A_\1$ and $\mu\in\La^B_\0$;}
\\
&\xi_{\la,\mu} \downarrow^{A \otimes B}_{A_{\0} 
\otimes B}
=
\tilde{\xi}_\la \boxtimes \xi_\mu^+ + \tilde{\xi}_\la \boxtimes \xi_\mu^-& \text{if $\la\in\La^A_\1$ and $\mu\in\La^B_\1$.}
\end{align*}

\item
\begin{align*}
&(\tilde{\xi}_\la^{\pm} \boxtimes \xi_\mu) \uparrow^{A \otimes B}_{A_{\0} \otimes B} = \xi_{\la,\mu} 
& \text{if $\la\in\La^A_\0$ and $\mu\in\La^B_\0$;}
\\
&(\tilde{\xi}_\la^{\pm} \boxtimes \xi_\mu^\pm) \uparrow^{A \otimes B}_{A_{\0} \otimes B} = \xi_{\la,\mu}^+ 
& \text{if $\la\in\La^A_\0$ and $\mu\in\La^B_\1$;}
\\
&(\tilde{\xi}_\la^{\pm} \boxtimes \xi_\mu^\mp) \uparrow^{A \otimes B}_{A_{\0} \otimes B} = \xi_{\la,\mu}^- 
& \text{if $\la\in\La^A_\0$ and $\mu\in\La^B_\1$;}
\\
&(\tilde{\xi}_\la \boxtimes \xi_\mu) \uparrow^{A \otimes B}_{A_{\0} \otimes B} = \xi_{\la,\mu}^+ + \xi_{\la,\mu}^- 
& \text{if $\la\in\La^A_\1$ and $\mu\in\La^B_\0$;}
\\
&(\tilde{\xi}_\la \boxtimes \xi_\mu^\pm) \uparrow^{A \otimes B}_{A_{\0} \otimes B} = \xi_{\la,\mu}
& \text{if $\la\in\La^A_\1$ and $\mu\in\La^B_\1$.}
\end{align*}
\end{enumerate}
We, of course, have the corresponding equalities for $\downarrow^{A \otimes B}_{A \otimes B_{\0}}$ and $\uparrow^{A \otimes B}_{A \otimes B_{\0}}$.
\end{Lemma}

\begin{proof}
We already know from Lemma \ref{lem:char_A_ten_B} that $A \otimes B$ is split semisimple. Let $u_A \in A^\times \cap A_{\1}$ and $u_B \in B^\times \cap B_{\1}$. Clearly $u_A \otimes 1$ is a superunit of $A \otimes B$, proving the first part of the lemma.

For parts (i) and (ii) we prove the hardest case, that is $\la \in \La^A_{\0}$ and $\mu \in \La^B_{\1}$. The other cases are similar but easier as any choices made about the labeling of characters are unimportant.

Let $\gU$ be an irreducible $A$-supermodule with irreducible supercharacter $\xi_\la \in \Irr_{\operatorname{super}}(A)$ and $\gV$ an irreducible $B$-supermodule with 
irreducible supercharacter $\xi_\mu \in \Irr_{\operatorname{super}}(B)$. By Lemma \ref{lem:char_A_ten_B}, $\gU \boxtimes \gV$ is an irreducible $(A \otimes B)$-supermodule. By Remark \ref{rem:ind_chars}, $\Ind_{A_{\0}}^A \gU_{\0} \simeq \gU$. Therefore, $\Ind_{A_{\0} \otimes B}^{A \otimes B}(\gU_{\0} \boxtimes \gV) \simeq \gU \boxtimes \gV$.

Note that we made a choice when we selected $\gU$, as we could have chosen $\Pi \gU \not\simeq \gU$. Let's say we picked $\gU$ such that $\gU_{\0}$ has character $\tilde{\xi}_\la^+$. In the previous paragraph we showed that
$$
(\tilde{\xi}_\la^+ \boxtimes \xi_\mu)\uparrow_{A_{\0} \otimes B}^{A \otimes B} = \xi_{\la,\mu} = \xi_{\la,\mu}^+ + \xi_{\la,\mu}^-.
$$
We can now choose the labeling of the appropriate elements of $\Irr(A \otimes B)$ such that
$$
(\tilde{\xi}_\la^+ \boxtimes \xi_\mu^+)\uparrow_{A_{\0} \otimes B}^{A \otimes B} = \xi_{\la,\mu}^+, \qquad (\tilde{\xi}_\la^+ \boxtimes \xi_\mu^-)\uparrow_{A_{\0} \otimes B}^{A \otimes B} = \xi_{\la,\mu}^-.
$$
We can think of this last equation as a definition, once the labeling of $\Irr(A_{\0})$ and $\Irr(B)$ have been fixed. Now,
\begin{align*}
&(\tilde{\xi}_\la^- \boxtimes \xi_\mu^+)\uparrow_{A_{\0} \otimes B}^{A \otimes B} = \left({}^{u_A}(\tilde{\xi}_\la^- \boxtimes \xi_\mu^+)\right)\uparrow_{A_{\0} \otimes B}^{A \otimes B} = (\tilde{\xi}_\la^+ \boxtimes \xi_\mu^-)\uparrow_{A_{\0} \otimes B}^{A \otimes B} = \xi_{\la,\mu}^-, \\
&(\tilde{\xi}_\la^- \boxtimes \xi_\mu^-)\uparrow_{A_{\0} \otimes B}^{A \otimes B} = \left({}^{u_A}(\tilde{\xi}_\la^- \boxtimes \xi_\mu^-)\right)\uparrow_{A_{\0} \otimes B}^{A \otimes B} = (\tilde{\xi}_\la^+ \boxtimes \xi_\mu^+)\uparrow_{A_{\0} \otimes B}^{A \otimes B} = \xi_{\la,\mu}^+
\end{align*}
and part (ii) is proved. Part (i) follows form the fact that $\Ind_{A_{\0} \otimes B}^{A \otimes B}$ and $\Res_{A_{\0} \otimes B}^{A \otimes B}$ are an adjoint pair.
\end{proof}

Let $A$ be a split semisimple $\K$-superalgebra and $\la \in \La^A$. We set
\begin{equation}\label{EEps}
\eps_\la:=
\begin{cases}
1&\text{if }\la \in \La^A_{\0}, \\
\sqrt{2}&\text{if }\la \in \La^A_{\1}.
\end{cases}
\end{equation}
More generally, if $A_1,\dots,A_n$ are split semisimple  $\K$-superalgebras then we have identified $\La^{A_1\otimes \dots\otimes A_n}=\La^{A_1}\times\dots\times\La^{A_n}$ so that 
$(\la^1,\dots,\la^n)\in \La^{A_1\otimes \dots\otimes A_n}_\0$ if and only if the number of the odd $\la^i$ is even. Now, set 
%the notation $\eps_{(\la^1,\dots,\la^n)}$ makes sense. In fact, we have
\begin{equation}\label{EEpsN}
\eps_{\la^1,\dots,\la^n}:=\eps_{(\la^1,\dots,\la^n)}=
\begin{cases}
1&\text{if $(\la^1,\dots,\la^n)\in \La^{A_1\otimes \dots\otimes A_n}_\0$}, \\
\sqrt{2}&\text{if $(\la^1,\dots,\la^n)\in \La^{A_1\otimes \dots\otimes A_n}_\1$}.
\end{cases}
\end{equation}
Note that $\eps_{\la^1,\dots,\la^n}$ does not depend on the order of $\la^1,\dots,\la^n$ (unlike $\xi_{\la^1,\dots,\la^n}$).

\begin{Lemma}\label{lem:chars_bisupmod}
Let $A$ and $B$ be split semisimple $\K$-superalgebras with superunits and let $\gM$ be an $(A,B)$-bisupermodule. Then:
\begin{enumerate}
\item
$
\gM \otimes_B \xi^{\si} = (\gM \otimes_B \xi)^{\si},
$ 
for all $\xi \in \Irr(B)$.
\item We can write, for each $\mu \in \La^B$,
\begin{align*}
\gM \otimes_B \xi_\mu = \sum_{\la \in \La^A} a_{\mu,\la} \xi_\la,
\end{align*}
for some $a_{\mu,\la} \in \N$. Similarly, we can write, for each $\la \in \La^A$,
\begin{align*}
\gM^* \otimes_A \xi_\la = \sum_{\La^B} b_{\la,\mu} \xi_\mu,
\end{align*}
for some $b_{\la,\mu} \in \N$. Moreover, $\eps_\la^2 a_{\mu,\la} = \eps_\mu^2 b_{\la,\mu}$, for all $\la \in \La^A$ and $\mu \in \La^B$.
\end{enumerate}
\end{Lemma}

\begin{proof}
(i) Let ${}_{\sigma_A}\gM$ be the $(A,B)$-bisupermodule equal to $\gM$ as a superspace but with $(A,B)$-bimodule structure given by
$
a.m.b:=\sigma_A(a)mb,
$ 
for all $a\in A$, $b\in B$ and $m\in \gM$. Define $\gM_{\sigma_B}$ analogously. We have the $(A,B)$-bisupermodule isomorphism
$
{}_{\sigma_A}\gM \iso \gM_{\sigma_B},\ 
m \mapsto (-1)^{|m|}m.
$
The claim follows as
\begin{align*}
\gM \otimes_B \xi^{\si} = \gM_{\sigma_B} \otimes_B \xi = {}_{\sigma_A}\gM \otimes_B \xi = (\gM \otimes_B \xi)^{\si}.
\end{align*}

(ii) The expressions involving $\gM$ and $\gM^*$ just follow from part (i).

For the second part we simply run through all the possibilities of $\la$ and $\mu$ being even/odd. For example, let $\la$ be even and $\mu$ odd. Part (i) implies that $\xi_\la$ must appear with coefficient $a_{\mu,\la}/2$ in both $\gM \otimes_B \xi_\mu^+$ and $\gM \otimes_B \xi_\mu^-$. Therefore, $\xi_\mu^+$ and $\xi_\mu^-$ must both appear with coefficient $a_{\mu,\la}/2$ in $\gM^* \otimes_A \xi_\la$. Therefore, $b_{\la,\mu} = a_{\mu,\la}/2$, as desired. The other cases are proved similarly.
\end{proof}

Using the notation of the above lemma, we say $\xi_\la$ appears as an {\em irreducible superconstituent} (or just {\em superconstituent}) with coefficient $a_{\mu,\la}$ in $\gM \otimes_B \xi_\mu$.

\begin{Lemma}\label{lem:AboxM}
Let $A$, $B$ and $C$ be split semisimple $\K$-superalgebras with superunit, $\gM$ a $(B,C)$-bisupermodule and $\la \in \La^A$, $\mu \in \La^C$. 
If\, $\gM \otimes_{C} \xi_\mu = \sum_{\nu \in \La^{B}} a_{\mu,\nu} \xi_\nu$ then
\begin{align*}
(A \boxtimes \gM) \otimes_{A \otimes C} \xi_{\la,\mu} = \sum_{\nu \in \La^{B}} \frac{\eps_\nu \eps_{(\la,\mu)}}{\eps_\mu \eps_{(\la,\nu)}}a_{\mu,\nu}\xi_{\la,\nu}.
\end{align*}
\end{Lemma}

\begin{proof}
This is just a case of running through all the possibilities of $\la$, $\mu$ and $\nu$ being even/odd. For example, if $\la$ and $\mu$ are both odd and $\nu$ is even, then, by Lemma \ref{lem:chars_bisupmod}(i), $\xi_\nu$ appears with coefficient $a_{\mu,\nu}/2$ in both $\gM \otimes_{C} \xi_\mu^+$ and $\gM \otimes_{C} \xi_\mu^-$. Therefore, $\tilde{\xi}_\la \boxtimes \xi_\nu$ appears with coefficient $a_{\mu,\nu}/2$ in both
$
(A_{\0} \boxtimes \gM) \otimes_{A_{\0} \otimes C} (\tilde{\xi}_\la \boxtimes \xi_\mu^+)$ 
 and 
$(A_{\0} \boxtimes \gM) \otimes_{A_{\0} \otimes C} (\tilde{\xi}_\la \boxtimes \xi_\mu^-).
$
Using Lemma \ref{lem:A_ten_B}(ii),
\begin{align*}
(A \boxtimes B) \otimes_{A_{\0} \otimes B} (\tilde{\xi}_\la \boxtimes \xi_\nu) = \xi_{\la,\nu}^+ + \xi_{\la,\nu}^- = \xi_{\la,\nu}
\end{align*}
appears as a superconstituent with coefficient $a_{\mu,\nu}/2$ in
\begin{align*}
& (A \boxtimes B) \otimes_{A_{\0} \otimes B} (A_{\0} \boxtimes \gM) \otimes_{A_{\0} \otimes C} (\tilde{\xi}_\la \boxtimes \xi_\mu^+)
= (A \boxtimes \gM) \otimes_{A_{\0} \otimes C} (\tilde{\xi}_\la \boxtimes \xi_\mu^+)\\
=\, & (A \boxtimes \gM) \otimes_{A \otimes C} (A \boxtimes C)\otimes_{A_{\0} \otimes C} (\tilde{\xi}_\la \boxtimes \xi_\mu^+) = (A \boxtimes \gM) \otimes_{A \otimes C} \xi_{\la,\mu},
\end{align*}
where, above, we are applying Lemma \ref{lem:tensor_bisupmod} twice and Lemma \ref{lem:A_ten_B}(ii) once.
\end{proof}

\subsection{Super group algebras}\label{sec:super_group_alg}

Throughout this subsection $G$ will denote a finite group with an index $2$ subgroup $G_{\0}$. 
%For $\cR \in \{\K, \cO\}$, w
We give $\cR G$ the structure of an $\cR$-superalgebra with superunit via
\begin{align*}
(\cR G)_{\0} := \langle g \mid  g \in G_{\0} \rangle_{\cR}\quad\text{and}\quad (\cR G)_{\1} := \langle g \mid  g \in G \setminus G_{\0} \rangle_{\cR}.
\end{align*}
In particular, for $g\in G$, we have $|g|=\0$ of $g\in G_\0$ and $|g|=\1$ if $g\not\in G_\0$. 

For an absolutely indecomposable $\cR G$-supermodule $\gM$, when we speak of a {\em vertex}\, of $\gM$, we always mean a vertex of $|\gM|$. The same applies to relative projectivity with respect to a subgroup and other standard notions of finite group theory.

\begin{Lemma}\label{lem:root_unity}
If\, $\K$ contains a primitive $|G|^{\nth}$ root of unity and $e$ is a non-zero idempotent in $Z(\K G) \cap \K G_{\0}$, then $\K Ge$ is a split semisimple $\K$-superalgebra with superunit.
\end{Lemma}

\begin{proof}
We first prove $\K G$ is split semisimple. Since $\K$ contains a primitive $|G|^{\nth}$ root of unity, it is well known that $\K G_{\0}$ and $|\K G|$ are split semisimple $\K$-algebras. We make use of this fact at several points during the proof.

Let $\gU$ be an irreducible $\K G_{\0}$-module and set $\gM := \Ind_{G_{\0}}^G \gU$. We claim that $\gM$ is an irreducible $\K G$-supermodule.
If $\gU$ is not $G$-stable, then, by standard Clifford theory, $|\gM|$ is an irreducible $|\K G|$-module. Therefore, $\gM$ is an irreducible $\K G$-supermodule of type $\Mtype$ and $\End_{|\K G|}(|\gM|) = 1$. On the other hand, if $\gU$ is $G$-stable, then, again by standard Clifford theory, $|\gM| = \gM_1 \oplus \gM_2$, for two non-isomorphic, irreducible $|\K G|$-submodules $\gM_1, \gM_2 \subseteq |\gM|$. Furthermore, $\gM_1$ and $\gM_2$ are related by tensoring with the unique non-trivial, linear character of $G$ with kernel $G_{\0}$. In other words, $\gM_1^{\si} \cong \gM_2$. Therefore, neither $\gM_1$ nor $\gM_2$ are $\si_{\gM}$-invariant and $\gM$ must be an irreducible $\K G$-supermodule. Moreover, $\End_{|\K G|}(|\gM|) = 2$ and $\gM$ is of type $\Qtype$.

In fact, every irreducible $\K G$-supermodule is of the form $\Ind_{G_{\0}}^G \gU$, for some irreducible $\K G_{\0}$-module $\gU$. Indeed, if $\gM$ is an irreducible $\K G$-supermodule, then $\gM\simeq \Ind_{G_{\0}}^G \gM_{\0}$, similar to Lemma~\ref{lem_red_to_A0}(i).  We have now shown that $\K G$ is split.

Since we can decompose $\K G_{\0}$ into a direct sum of irreducible $\K G_{\0}$-modules, we can decompose $\K G \simeq \Ind_{G_{\0}}^G \K G_{\0}$ into a direct sum of irreducible $\K G$-supermodules, proving semisimplicity.

The statement for $\K Ge$ now follows from Lemma \ref{lem:split_semi}, since truncating by $e$ translates into deleting some of the factors in statement (ii) of said lemma.
\end{proof}

\begin{Remark} %\label{}%{\rm \cite{}}%{\bf ()}
We note that it is not true that a $\K$-superalgebra $A$ is split semisimple if and only if the algebra $|A|$ is split semisimple. For example, let $A=M_2(\R)$ with superstructure determined by $\si_A$, which is given by conjugation with $
\left(
\begin{matrix}
0 & 1  \\
-1 & 0
\end{matrix}
\right)
$.
Then $A_\0\cong \C$ and it follows easily from Lemma~\ref{lem:split_semi} that $A$ is not split over $\R$. On the other hand $|A|=M_2(\R)$ is split over $\R$.
\end{Remark}

For the next lemma, let $H$ be another finite group with an index $2$ subgroup $H_{\0}$.

\begin{Lemma}\label{lem:sup_vert}
Let $\gM$ be an absolutely indecomposable $\cO$-free $(\cO G,\cO H)$-bisupermodule. If  $\gM$  has vertex $P \leq G \times H$ then $\gM^*$ has vertex $P^*:=\{(y,x)\in H \times G\mid  (x,y) \in P\}$.
\end{Lemma}

\begin{proof}
Since $|\gM^*|=|\gM|^*$, this is just a result about the usual (non-super) dual. Taking duals commutes with induction. Therefore, $\gM^*$ is relatively $P$-projective as a right $\cO(G \times H)$-module. Equivalently, it is relatively $P^*$-projective as a left $\cO(H \times G)$-module. That it has vertex $P^*$ is just a consequence of the fact that $|\gM^*|^* \cong |\gM|$, since $\gM$ is $\cO$-free.
\end{proof}

For any subgroup $N \leq G$ we will use $N_{\0}$ to denote $N \cap G_{\0}$ and $\cO N$ will inherit its superalgebra structure from that of $\cO G$. In particular, if $N \nleq G_{\0}$, then $\cO N$ is also a $\K$-split semisimple superalgebra with superunit.
% In this context we can view $\uparrow_N^G: \mod{\cO N} \to \mod{\cO G}$ as a functor between the categories of supermodules by idenitfying it with $\cO G \otimes_{\cO N} ?$. We can do an analogous identification for $\downarrow^G_N: \mod{\cO G} \to \mod{\cO N}$.

We now state and prove a super version of the well-known Green correspondence.

\begin{Theorem}\label{thm:super_Green}
Let $P$ be a $p$-subgroup of $G$ and $N_G(P) \leq N \leq G$ be such that $N$ is not contained in $G_0$. Then there exists a correspondence between the absolutely indecomposable $\cO G$-supermodules with vertex $P$ and the absolutely indecomposable $\cO N$-supermodules with vertex $P$. Suppose $\gU$ is an absolutely indecomposable $\cO G$-supermodule with vertex $P$ corresponding to an absolutely indecomposable $\cO N$-supermodule $\gV$ with vertex $P$. Then
\begin{align}\label{algn:Green_eqn}
\Ind_N^G \gV \simeq \gU \oplus \gM_G\qquad\text{and}\qquad  \Res^G_N \gU \simeq \gV \oplus \gM_N,
\end{align}
for some $\cO G$-supermodule $\gM_G$ and some $\cO N$-supermodule $\gM_N$. Moreover, as an $\cO G$-module, $\gM_G$ is a direct sum of modules each with vertex contained in $P \cap {}^g P$, for some $g \in G\setminus N$ and, as an $\cO N$-module, $\gM_N$ is a direct sum of modules each with vertex contained in $N \cap {}^g P$, for some $g \in G\setminus N$.

Furthermore, given $\gU$ (resp. $\gV$), the decompositions of supermodules in (\ref{algn:Green_eqn}) uniquely determine $\gV$ (resp. $\gU$) up to isomorphism, subject to all the above conditions.
\end{Theorem}

\begin{proof}
All of this is the well-known Green correspondence except that (\ref{algn:Green_eqn}) holds as supermodules (and not just modules) and that $\gU$ and $\gV$ uniquely determine one another as supermodules (again, not just as modules), up to isomorphism.

Let $\gV$ be an absolutely indecomposable $\cO N$-supermodule with vertex $P$. Then, by Lemma \ref{lem_red_to_A0}(i), $\gV \simeq \Ind_{N_{\0}}^N\gV_{\0}$ and $\gV_{\0}$ is necessarily an indecomposable $\cO N_{\0}$-module with vertex $P$ or ${}^x P$, for some $x \in N \setminus N_{\0}$. (We are using the fact that $p>2$ here.) Now, $N_{G_{\0}}(P) \leq N_{\0}$ and so, by conjugating everything by $x$, we also have that $N_{G_{\0}}({}^x P) \leq N_{\0}$. We can, therefore, consider the Green correspondent $\gW$ of $\gV_{\0}$ in $G_{\0}$.

Note that, by Lemma \ref{lem_red_to_A0}(ii), $\gV_{\0}$ is not $N$-stable. Therefore, $\gW$ is not $N$-stable, which is equivalent to not being $G$-stable, and $\Ind_{G_{\0}}^G\gW$ is an absolutely indecomposable $\cO G$-supermodule, with vertex $P$.

Now, $\Ind_{N_{\0}}^{G_{\0}}\gV_{\0} \cong \gW \oplus \gZ_G$, where $\gZ_G$ is a direct sum of $\cO G_{\0}$-modules each with vertex contained in $P \cap {}^g P$ (or, in the case that $\gV_{\0}$ has vertex ${}^x P$, each with vertex contained in ${}^xP \cap {}^{gx} P$), for some $g \in G_{\0} \setminus N_{\0}$. Consequently, $\Ind_{G_{\0}}^G\gZ_G$ is a direct sum of modules each with vertex contained in $P \cap {}^g P$, for some $g \in G_{\0} \setminus N_{\0}$.

Setting $\gU:= \Ind_{G_{\0}}^G\gW$ now gives
\begin{align*}
\Ind_N^G\gV \simeq \Ind_{N_{\0}}^G\gV_{\0} \simeq (\Ind_{G_{\0}}^G\gW) \oplus (\Ind_{G_{\0}}^G\gZ_G) = \gU \oplus \Ind_{G_{\0}}^G\gZ_G,
\end{align*}
as $\cO G$-supermodules, as desired.

We now show uniqueness of $\gU$. Say $\Ind_N^G\gV \simeq \gU \oplus \gM_G$, as in (\ref{algn:Green_eqn}), where $\gU$ and $\gM_G$ have the desired properties. Then, taking the even part of both sides, $\Ind_{N_{\0}}^{G_{\0}}\gV_{\0} \cong \gU_{\0} \oplus (\gM_G)_{\0}$. Since $\gV$ has vertex $P$, $\gV_{\0}$ has vertex $P$ or ${}^xP$, for some $x \in N \setminus N_{\0}$. Similarly, $\gU_{\0}$ has vertex $P$ or ${}^yP$, for some $y \in G \setminus G_{\0}$. However, $\gU_{\0}\mid \Ind_{N_{\0}}^{G_{\0}}\gV_{\0}$ and so $\gV_{\0}$ and $\gU_{\0}$ both have vertices $P$ or ${}^x P$. Therefore, $\gU_{\0}$ is the Green correspondent of $\gV_{\0}$ in $G_{\0}$ (as before we have $N_{G_{\0}}(P),N_{G_{\0}}({}^x P) \leq N_{\0}$) and $\gV \simeq \Ind_{G_{\0}}^G\gV_{\0}$ is uniquely determined as an $\cO G$-supermodule.

Showing that the second equation in (\ref{algn:Green_eqn}) holds and that $\gU$ uniquely determines $\gV$, as an $\cO N$-supermodule, is completely analogous to the reverse argument.
\end{proof}

In the context of Theorem \ref{thm:super_Green} we say $\gU$ and $\gV$ are {\em super Green correspondents}. We will sometimes say $\gU$ is the {\em super Green correspondent of $\gV$ in $G$} or {\em $\gV$ is the super Green correspondent of $\gU$ in $N$}.

For the following lemma recall that, for a block $\cO Gb$ with defect group $D$, we can consider $\cO Gb$ as an $\cO(G \times G)$-supermodule with vertex $\Delta D$.

\begin{Lemma}\label{lem:Brauer_Green}
Let $b \in \cO G_{\0}$ be a block idempotent in $\cO G$ such that $\cO Gb$ has defect group $D$ and $N_G(D) \leq N \leq G$ with $N \nleq G_{\0}$. 
Let $f \in \cO N$ be such that $\cO Nf$ is the Brauer correspondent of $\cO Gb$ in $N$. Then $f \in \cO N_{\0}$. Moreover, there is a unique absolutely indecomposable summand $\gU$ of $b \cO Gf$, as an $\cO(G \times N)$-supermodule, with vertex $\Delta D$ and all other summands have vertices strictly contained in $\Delta D$. Moreover, $\gU$ is the super Green correspondent of both $\cO Gb$ and $\cO Nf$ in $G \times N$. In particular, $\cO Nf$ is the super Green correspondent of $\cO Gb$ in $N \times N$.
\end{Lemma}

\begin{proof}
Since $D \leq N_{\0}$, $\cO N\si_{\cO N}(f)$ also has defect group $D$ and
$$
\Br_D(b) = \Br_D(\si_{\cO G}(b)) = \Br_D(\si_{\cO G}(f)) = \Br_D(\si_{\cO N}(f)).
$$
Therefore, $\cO Nf$ and $\cO N\si_{\cO N}(f)$ are both the Brauer correspondent of $\cO Gb$ in $N$. In particular, $\si_{\cO N}(f) = f$ and so $f \in \cO N_{\0}$.

It is well known that $\cO Gb$ and $\cO Nf$ are (non-super) Green correspondents.
We set $\gU$ to be the super Green correspondent of $\cO Nf$ in $G \times N$. In particular, $\gU\mid \Ind_{N \times N}^{G \times N} \cO Nf \simeq \cO Gf$ and all other summands have vertex strictly contained in $\Delta D$. Due to the first paragraph, $\gU$ must be at least the (non-super) Green correspondent of $\cO Gb$ in $G \times N$. Therefore, $\gU\mid \Res^{G \times G}_{G \times N}\cO Gb \simeq b\cO G$, as an $\cO(G \times N)$-module and so $b\gU = \gU$.

We have now shown that $\gU$ is the unique indecomposable summand of $b\cO Gf$ with vertex $\Delta D$ and all other summands have vertex strictly contained in $\Delta D$. In particular, $\gU$ is the super Green correspondent of $\cO Gb$ in $G \times N$. The proof is now complete.
\end{proof}

If $K \leq G$, $g \in G$ and $\gM$ is an $\cO K$-supermodule, then $g\gM$ is the $\cO({}^g K)$-supermodule defined in the obvious way, that is $|gm|=|g|+|m|$, for all $m \in \gM$. In particular, if $H \leq G$, $k \in K$ and $h \in H$, then we have a canonical isomorphism of $\cO H$-supermodules
\begin{align*}
\Ind_{H \cap {}^g K}^H \Res^{{}^g K}_{H \cap {}^g K} \,g\gM  \iso \Ind_{H \cap {}^{hgk} K}^H\Res^{{}^{hgk} K}_{H \cap {}^{hgk} K}\, hgk\gM, \ 
h' \otimes gm  \mapsto h'h^{-1} \otimes hgk(k^{-1}m).
\end{align*}
This isomorphism ensures the right-hand side of the following `super Mackey decomposition formula' is well-defined.

\begin{Theorem}\label{thm:super_Mackey}
Let $H,K \leq G$ and $\gM$ an $\cO K$-supermodule. We have the following isomorphisms of $\cO H$-supermodules
\begin{align*}
\Res^G_H\Ind_K^G\gM \simeq \bigoplus_{g \in H \backslash G/K} \Ind_{H \cap \hspace{.2mm}{}^g\hspace{-.4mm}K}^H\Res^{{}^g \hspace{-.4mm}K}_{H \cap \hspace{.2mm}{}^g\hspace{-.4mm}K}\,g\gM.
\end{align*}
\end{Theorem}

\begin{proof}
In the proof of the original (non-super) version of the Mackey decomposition formula the isomorphism is given by $g \otimes m \mapsto gm \in g\gM$, for all $g \in H\backslash G/K$. This clearly respects superstructure.
\end{proof}

\subsection{Vertices of super tensor products}\label{sec:vert_sup_ten_prod}

Throughout this subsection $G$ and $H$ will denote finite groups with subgroups $G_{\0}\leq G$ and $H_\0\leq H$ of index at most $2$. For $g\in G$, we write $|g|=\0$ if $g\in G_\0$ and $|g|=\1$ otherwise. Similarly for $h\in H$. Note that, unlike \S\ref{sec:super_group_alg}, we do not exclude the cases  $G=G_\0$ and $H=H_\0$. 

We also suppose that $G_\0$ and $H_\0$ both contain the canonical central element $z$ of order $2$ (note we do not distinguish between the $z$ in $G$ and that in $H$ as these elements will be identified anyway). We set $G \times_z H$ to be the free product $G \star H$ of groups subject to this identification of $z$'s and the relation $[g,h]=z^{|g||h|}$, for all $g \in G$ and $h \in H$. We have a natural surjection $\pi: G \times_z H\to G/\langle z \rangle\times H/\langle z \rangle$ with kernel $\{1,z\}$. 
We now also have the subgroup $(G \times_z H)_\0:=\pi^{-1}(\{(g\langle z \rangle,h\langle z \rangle)\mid |g|=|h|\})$ of index at most $2$. The main source of examples for this construction will become clear in \S\ref{SSSpSubg}.

As in \S\ref{sec:super_group_alg}, we have the super group algebras $\cO G$, $\cO H$ and $\cO(G \times_z H)$, except that now it is possible that some of these superalgebras could be purely even.  

Let $e_z:=(1-z)/2$. Since $e_z$ is an even central idempotent, $\cO Ge_z$, $\cO He_z$ and $\cO(G \times_z H)e_z$ inherit superalgebra structure from $\cO G$, $\cO H$ and $\cO(G \times_z H)$, respectively. Moreover, we have the superalgebra isomorphism
$$
\cO (G \times_z H)e_z \cong \cO G e_z \otimes \cO H e_z.
$$
So, given an $\cO G e_z$-supermodule $\gM$ and  an $\cO He_z$-supermodule $\gN$, we can form the $\cO (G \times_z H)e_z$-supermodule $\gM \boxtimes \gN$.

Since $p$ is odd, if $P$ is a $p$-subgroup of $G$ and $Q$ a $p$-subgroup of $H$, then $P \leq G_{\0}$, $Q \leq H_{\0}$ and $P$ and $Q$ commute with one another in $G \times_z H$. Furthermore, since the subgroup of $G \times_z H$ generated by $G_{\0}$ and $H_{\0}$ is isomorphic to $(G_{\0} \times H_{\0})/\langle (z,z) \rangle$, the subgroup of $G \times_z H$ generated by $P$ and $Q$ is isomorphic to $P \times Q$. We will usually simply denote said subgroup by $P \times Q$.

\begin{Lemma}\label{lem:direct_prod_vert}
Adopting the above notation, suppose $\gM$ (resp. $\gN$) is $\cO$-free and absolutely indecomposable with vertex $P$ and source $\gU$ (resp. vertex $Q$ and source $\gV$). Then $\gM \boxtimes \gN$ is an $\cO$-free absolutely indecomposable $\cO(G \times_z H)e_z$-supermodule with vertex $P \times Q \leq G \times_z H$ and source $\gU \boxtimes \gV$.
\end{Lemma}

%\textcolor{green}{I believe the way I've now done this proof is optimal. I guess  strictly speaking it didn't really follow immediately from \cite{Ku2} in the commuting case, as there's still the issue that $G \times_z H$ is a not a direct product of groups.}

\begin{proof}
That $\gM \boxtimes \gN$ is $\cO$-free, is clear.

We now assume that $G_{\0}$ and $H_{\0}$ are proper subgroups of $G$ and $H$ respectively. (We deal with the possibilities of either $G_{\0} = G$ or $H_{\0} = H$ at the end of the proof.)

Set $L$ to be the subgroup of $G \times_z H$ generated by $G_{\0}$ and $H_{\0}$, so $L \cong (G_{\0} \times H_{\0})/\langle (z,z) \rangle$. Consider the $\cO(G_{\0} \times H_{\0})$-module $\gM_{\0} \boxtimes \gN_{\0}$. Since $(z,z)$ acts as the identity on this module, we can view $\gM_{\0} \boxtimes \gN_{\0}$ as the $\cO L$-module $\gM_{\0} \boxtimes \gN_{\0}$. Moreover,
\begin{align}\label{algn:M_box_N_res_L}
\Res^{G \times_z H}_L\gM \boxtimes \gN \simeq \bigoplus_{\eps_g,\eps_h \in \{0,1\}} u_g^{\eps_g}u_h^{\eps_h}(\gM_{\0} \boxtimes \gN_{\0}),
\end{align}
as $\cO L$-modules, where $u_g \in G_{\1}$ and $u_h \in H_{\1}$, and
\begin{align}\label{algn:tensor_M0N0}
\gM \boxtimes \gN \simeq \Ind_L^{G \times_z H}\gM_{\0} \boxtimes \gN_{\0},
\end{align}
as $\cO(G \times_z H)$-supermodules. By Lemma \ref{lem_red_to_A0}(i), $\gM_{\0}$ is an $\cO$-free indecomposable $\cO G_{\0}$-module and $\gN_{\0}$ is an $\cO$-free indecomposable $\cO H_{\0}$-module. Therefore, $\gM_{\0} \boxtimes \gN_{\0}$ is $\cO$-free and, by \cite[Proposition 1.1]{Ku2}, is indecomposable as an $\cO(G_{\0} \times H_{\0})$-module and hence as an $\cO L$-module. (Note that in \cite{Ku2} the algebras are defined over an algebraically closed field. However, that proof runs through for algebras defined over $\cO$, as long as the modules are $\cO$-free.)

Lemma \ref{lem_red_to_A0}(ii) tells us that $\gM_{\0} \ncong u_g \gM_{\0}$ as $\cO G_{\0}$-modules and $\gN_{\0} \ncong u_h \gN_{\0}$ as $\cO H_{\0}$-modules. Therefore, all four indecomposable summands in (\ref{algn:M_box_N_res_L}) are pairwise non-isomorphic and it follows from (\ref{algn:tensor_M0N0}) and \cite[\S5, Propositon 2]{Ward} that $\gM \boxtimes \gN$ is indecomposable as an $\cO(G \times_z H)$-module.

Next, we note that $\gM_{\0}$ has vertex and source a $G$-conjugate of the pair $(P,\gU)$. Similarly, $\gN_{\0}$ has vertex and source an $H$-conjugate of the pair $(Q,\gV)$. Therefore, by \cite[Proposition 1.2]{Ku2}, $\gM_{\0} \boxtimes \gN_{\0}$ has vertex and source a $(G \times H)$-conjugate of the pair $(P \times Q,\gU \boxtimes \gV)$, as an $\cO(G_{\0} \times H_{\0})$-module. (Again, the reference to \cite{Ku2} holds without complication for $\cO$-free modules.) Since $p$ is odd and $\langle (z,z) \rangle$ has order $2$, $\gM_{\0} \boxtimes \gN_{\0}$ has vertex and source a $(G \times_z H)$-conjugate of the pair $(P \times Q,\gU \boxtimes \gV)$, as an $\cO L$-module. That $\gM \boxtimes \gN$ has vertex $P \times Q$ and source $\gU \boxtimes \gV$ now follows from (\ref{algn:M_box_N_res_L}) and (\ref{algn:tensor_M0N0}).

Finally, we observe that, if $G= G_\0$ or $H =H_\0$, the proof runs through in a very similar manner except that there will be only one or two summands in (\ref{algn:M_box_N_res_L}).
\end{proof}

\section{Double covers of symmetric groups}\label{sec:spin_blocks}

Throughout this section it is assumed that $\K$ (and hence $\cO$) contains a primitive $(2n!)^{\nth}$ root of unity. To be able to utilize our results from $\S$\ref{sec:wreath}, we also assume that $-1$ and $2$ have square roots in $\K$ (and hence $\cO$). We note that this last condition is automatic if $n \geq 4$ and so is not a terribly onerous assumption.
\color{black}

\subsection{The groups $\tSi_n$ and $\tAi_n$}\label{sec:gen_tSi}
The double covers of ${\Si}_n$ are given by
\begin{align}
\begin{split}\label{double_pres}
{\tSi}_n^+ := \langle z,\ t_1,\dots, t_{n-1} \mid  z^2=1,\ t_i z = z t_i,\  t_i^2=1,\\
t_i t_j = z t_j t_i \text{ if } |i-j| > 1,\ (t_i t_{i+1})^3=1 
%\text{ for } 0\leq  i\leq n - 2 
\rangle,\\
{\tSi}_n^- := \langle z,\ t_1,\dots, t_{n-1} \mid  z^2=1,\ t_i z = z t_i,\  t_i^2=z,\\
t_i t_j = z t_j t_i \text{ if } |i-j| > 1,\ (t_i t_{i+1})^3=1  \rangle.
\end{split}
\end{align}
%However, we will rarely use the notation of ${\tSi}_n^+$ and ${\tSi}_n^-$. 
We will use ${\tSi}_n$ to simultaneously represent both double covers.  
The natural group homomorphism ${\tSi}_n \to {\Si}_n$, $z\mapsto 1$, $t_i \mapsto s_i$ is denoted by $\pi_n$ (recall that we denote by $s_i$ the transposition $(i,i+1) \in \Si_n$).  We will sometimes consider $\tSi_n$ acting on $[n]$ through $\pi_n$. We denote the double cover of the alternating group $\pi_n^{-1}(\Ai_n)$ by ${\tAi}_n$. We define $|-|:\tSi_n \to \Z/2$ to be the unique group homomorphism with kernel $\tAi_n$. If $n>1$ we consider $\cR {\tSi}_n$ as a superalgebra corresponding to $({\tSi}_n)_\0=\tAi_n$, see \S\ref{sec:super_group_alg}. In the exceptional case $n=1$, we consider $\cR {\tSi}_n$ as a purely even superalgebra.

We set 
$$e_z := (1 - z)/2 \in \cR {\tSi}_n.$$
We will use the adjective `spin' when referring to objects associated to $\cR {\tSi}_n e_z$ rather than the whole of $\cR {\tSi}_n$. For example, the spin characters of ${\tSi}_n$ will refer to the ordinary characters $\chi$ of ${\tSi}_n$ satisfying $\chi(z) = -\chi(1)$, while the spin blocks of $\cO {\tSi}_n$ refers to the blocks of $\cO {\tSi}_n e_z$. Note that $\cR {\tSi}_n e_z$ inherits the structure of superalgebra from $\cR {\tSi}_n$ since the idempotent $e_z$ is even. 

If $n>1$ then $\K \tSi_n e_z$ is a split semisimple $\K$-superalgebra with superunit. If $n=1$, then $\K \tSi_n e_z \cong \K$ is split semisimple but is without a superunit. We can, therefore, still utilize many of the results from $\S$\ref{sec:KSS1}. However, several of the subsequent proofs later on the article require us to consider this case separately.

Let $G \leq \tSi_n$ with $z \in G$. Following the notation from $\S$\ref{sec:super_group_alg}, we set 
\begin{equation}\label{EG0}
G_{\0} := G \cap \tAi_n\quad \text{and} \quad G_{\1} := G \setminus \tAi_n
\end{equation}
 and treat $\cR Ge_z$ as a superalgebra in the appropriate way. (Note that, unlike in $\S$\ref{sec:super_group_alg}, we are yet not assuming that $G_{\0}$ is a proper subgroup of $G$.) If $G \nleq \tAi_n$, then, due to Lemma \ref{lem:root_unity}, $\K Ge_z$ is a split semisimple $\K$-superalgebra with superunit.

Recall the twisted group superalgebra $\cT_n$ from Example~\ref{ExTn}. 
We have the following isomorphisms of superalgebras
\begin{equation}\label{algn:Sn_Tn_isom}
\begin{split}
\cO \tSi_n^+ e_z & \iso \cT_n, \ t_i e_z  \mapsto \ct_i,\ z e_z  \mapsto -1
\\
\cO \tSi_n^- e_z & \iso \cT_n,\ t_i e_z  \mapsto (-1)^i\sqrt{-1}\ct_i,\  z e_z  \mapsto -1.
\end{split}
\end{equation}

\subsection{Special subgroups}
\label{SSSpSubg}
Consider the subgroups $G \leq \tSi_m$ and $H \leq \tSi_n$, both containing the canonical central element $z$ (note we do not distinguish between the $z$ in $G$ and that in $H$ as these elements will be identified anyway). As in $\S$\ref{sec:vert_sup_ten_prod}, we set $G \times_z H$ to be the free product $G \star H$ of groups subject to this identification of $z$'s and the relation $[g,h]=z^{|g||h|}$, for all $g \in G$ and $h \in H$. Equivalently, $G \times_z H$ is isomorphic to the subgroup of $\tSi_{m+n}$ generated by $G$ and $H$, where we view $H \leq \tSi_n$ as a subgroup of $\tSi_{m+n}$ via $z \mapsto z$, $t_i \mapsto t_{m+i}$.

%We set $e_{z_-} = (1-z)/2 \in \cO {\tSi}_n$ and the isomorphism $S(n)\scong \F {\tSi}_n e_{z_-}$.
%\newline
%\newline
If $\sX\subseteq [n]$, then $\Si_\sX$ will signify the subgroup of ${\Si}_n$ consisting of all elements that fix $[n]\backslash \sX$ pointwise and $\tSi_\sX$ will denote $\pi_n^{-1}(\Si_X) \leq \tSi_n$. Moreover, for disjoint $\sX_1, \dots, \sX_k \subseteq [n]$, $\tSi_{\sX_1,\dots,\sX_k}$ is the subgroup of $\tSi_n$ generated by the $\tSi_{\sX_i}$'s. In particular,
\begin{align*}
\tSi_{\sX_1,\dots,\sX_k} \cong \tSi_{\sX_1} \times_z \dots \times_z \tSi_{\sX_k}
\end{align*}
and
\begin{align}\label{EZCommAlg}
\cO\tSi_{\sX_1,\dots,\sX_k}e_z \cong \cO\tSi_{\sX_1}e_z \otimes \dots \otimes \cO\tSi_{\sX_k}e_z
\cong \cT_{|\sX_1|} \otimes \dots \otimes \cT_{|\sX_k|}
\end{align}
as superalgebras. We also define the corresponding subgroups of $\tAi_n$:
$$\tAi_\sX:=\tSi_\sX \cap \tAi_n\quad\text{and}\quad \tAi_{\sX_1,\dots,\sX_k}:=\tSi_{\sX_1,\dots,\sX_k} \cap \tAi_n.
$$
In the important special case where $n=n_1+\dots+n_k$ and 
$\sY_l=[n_1+\dots+n_l]\setminus[n_1+\dots+n_{l-1}]$ for $l=1,\dots,k$, we have the {\em standard Young subgroups}
\begin{equation}\label{EYoungSubgr}
\tSi_{n_1,\dots,n_k}:=\tSi_{\sY_1,\dots,\sY_k} \cong \tSi_{n_1} \times_z \dots \times_z \tSi_{n_k}\leq \tSi_n\quad\text{and}\quad
\tAi_{n_1,\dots,n_k}=\tSi_{n_1,\dots,n_k}\cap\tAi_{n}\leq \tAi_n.
\end{equation}
If some of the $n_l$'s are equal to each other, we use the exponential notation, as for example for the standard Young subgroup $\tSi_{r+kp,p^{d-k}}=\tSi_{r+kp,p,\dots,p}$ defined in (\ref{Hi}) below.

\subsection{Irreducible characters of $\tSi_n$ and $\tAi_n$}
\label{SSIrrChar}

For the following theorem, originally due to Schur \cite{Sch}, we adopt the notation of $\S$\ref{sec:KSS1}. 

\begin{Theorem} \label{TSchur}
We can identify $\La^{\K \tSi_n e_z}$ with $\Par_0(n)$, where $\la \in \Par_0(n)$ is even/odd (in the sense of $\S$\ref{sec:KSS1}) if $|\la|-h(\la)$ is even/odd.
\end{Theorem}

To label the irreducible spin characters of $\tSi_n$, we now use (\ref{E221222}), but note that for the moment, we are not being precise about distinguishing between $\xi_\la^+$ and $\xi_\la^-$ when $\la$ is odd. 
To label the irreducible spin characters of $\tAi_n$, if $n=1$, then $\tilde{\xi}_{(1)}$ will denote the unique element of $\Irr(\K \tAi_1 e_z)$. For $n>1$, we use the notation of Lemma \ref{lem:res_chars}. Again, for the moment, we are not being precise about distinguishing between $\tilde{\xi}_\la^+$ and $\tilde{\xi}_\la^-$ when $\la$ is even.

If $m,n>1$, we can label the irreducible characters of
\begin{align*}
\K (\tSi_m \times_z \tSi_n)e_z \cong \K \tSi_m e_z \otimes \K \tSi_n e_z
\end{align*}
using Lemmas \ref{lem:char_A_ten_B} and \ref{lem:A_ten_B}. If $m=1$, then Lemma \ref{lem:A_ten_B} no longer applies. However, to unify our $m=1$ and $m>1$ theory, we set
$$
\xi_{(1),\la}^{(\pm)} := \xi_{(1)} \boxtimes \xi_{\la}^{(\pm)} \in \Irr(\K(\tSi_1 \times_z \tSi_n)e_z),
$$
for any $\la \in \Par_0(n)$. If $n=1$ we proceed similarly to the case $m=1$.

%\subsubsection{Shifted tableau and branching rules}

%\textcolor{red}{We can possibly move all the following combinatorics to the `Combinatorics' Section. However, I quite like having it separate here as it feels completely self-contained in that it's sole purpose is to state Lemma \ref{thm:stembridge}.}

\subsection{Shifted tableau and branching rules}
If $\mu \in \Par_0(m)$ and $\la \in \Par_0(n)$ are such that $\mu \subseteq \la$, we consider the shifted skew diagram $\sh[\la \setminus \mu]:=\sh[\la] \setminus \sh[\mu]$ of $\la \setminus \mu$.
We set $\mathbb{N}'$ to be the totally ordered set $\{1' < 1 < 2' < 2 < \dots\}$. The letters $1',2',3',\dots$ are said to be {\em marked}.
A {\em shifted tableau} $\T$ of shape $\la \setminus \mu$ is an assignment $\T:\sh[\la \setminus \mu] \to \mathbb{N}'$ satisfying
\begin{enumerate}
\item $\T(i,j) \leq \T(i+1,j)$, $\T(i,j) \leq \T(i,j+1)$, for all permissible $i,j$.
\item Each column has at most one $k$, for each $k \in \NN$.
\item Each row has at most one $k'$, for each $k \in \NN$.
\end{enumerate}
We say $\T$ has {\em content} $\nu = (\nu_1, \nu_2, \dots)$, where
$
\nu_k = |\{(i,j)|\T(i,j) = k \text{ or }k'\}|.
$
The {\em word} $w(\T)$ of $\T$ is the sequence obtained by reading the rows of $\T$ from left to right starting at the bottom row and working upwards. If $w(\T) = w_1, w_2, \dots w_{n-m}$ and $i \in \NN$, we define
\begin{align*}
m_i(j) = \text{ multiplicity of $i$ among }w_{n-m-j+1}, \dots, w_{n-m},
\end{align*}
for $0\leq j \leq n-m$ and
\begin{align*}
m_i(n-m+j) = m_i(n-m) + \text{ multiplicity of $i'$ among }w_1, \dots, w_j,
\end{align*}
for $0< j \leq n-m$.

The word $w$ is said to satisfy the lattice property if, whenever $m_i(j) = m_{i-1}(j)$, for permissible $i$ and $j$ with $j < 2(n-m)$, we have
\begin{align*}
w_{n-m-j} \neq i, i' &&&\text{ if } 0\leq j < n-m, \\
w_{j-n+m+1} \neq i-1, i'& &&\text{ if } n-m\leq j < 2(n-m).
\end{align*}

We write $|w|$ to denote the word $w$ but with all the $k'$'s replaced by $k$'s, for all $k \in \NN$.

If $\nu \in \Par_0(n-m)$, then we define $\mathfrak{f}_\nu(\la \setminus \mu)$ to be the number of shifted tableaux $\T:\sh[\la \setminus \mu] \to \mathbb{N}'$ of content $\nu$ such that
\begin{enumerate}
\item $w(\T)$ satisfies the lattice property.
\item The leftmost $i$ of $|w(\T)|$ is unmarked in $w$, for all $1\leq i \leq h(\nu)$.
\end{enumerate}

Recall the notation (\ref{EEps}),(\ref{EEpsN}).

\begin{Theorem}\label{thm:stembridge}{\rm \cite[Theorems 8.1,\,8.3]{Stem}}\,\,
Let $m,n \in \NN$, with $m < n$, $\mu \in \Par_0(m)$ and $\nu \in \Par_0(n-m)$. %We set $\sU:=[m]$ and $\sV:=[n]\setminus[m]$. 
Then
\begin{align*}
\xi_{\mu ,\nu} \uparrow_{\tSi_{m,n-m}}^{\tSi_n} = \sum_{\la \in \Par_0(n)}\frac{\eps_{\mu,\nu}}{\eps_\la}\,2^{(h(\mu) + h(\nu) - h(\la))/2}\,\mathfrak{f}_\nu(\la \setminus \mu)\,\xi_\la.
\end{align*}
Furthermore, if $(\mu,\nu)$ and $\la$ are both odd, then, unless $\la = \mu \sqcup \nu$, the coefficients of $\xi_\la^+$ and $\xi_\la^-$ are equal in both $\xi_{\mu,\nu}^+ \uparrow_{\tSi_{m,n-m}}^{\tSi_n}$ and $\xi_{\mu,\nu}^- \uparrow_{\tSi_{m,n-m}}^{\tSi_n}$.
\end{Theorem}

\begin{proof}
In \cite{Stem} this is stated for the character on the left-hand side being irreducible and not an irreducible supercharacter (as is the case here). This is the reason our right-hand side appears to be multiplied by $\eps_{(\mu,\nu)}^2$. 
There is also an exception mentioned in \cite{Stem}, namely when $\la$ is odd and $\la = \mu \sqcup \nu$. In this case the appropriate coefficient is $1$. As noted in that proof, we have $\mathfrak{f}_\nu(\la \setminus \mu) = 1$. Therefore, since $\eps_{\mu,\nu} = \eps_\la$ and $h(\la) = h(\mu) + h(\nu)$, our formula also gives coefficient~$1$.
\end{proof}

\subsection{Blocks of \texorpdfstring{$\tSi_n$ and $\tAi_n$}{}}\label{sec:blocks}

In this section we describe the spin blocks of $\tSi_n$ and $\tAi_n$, their defect groups and Brauer correspondents. Parts (i),(ii) of the following theorem are \cite[Theorem 1.1]{Hu}, while parts (iii),(iv) are \cite[Theorem A, Corollary 26]{Ca}.

\begin{Theorem}\label{thm:DCblocks}
Let $n\in\NN$. 
\begin{enumerate}
\item The spin blocks of $\cO {\tSi}_n$ are labelled by pairs $(\rho,d)$, where $\rho$ is a ${\bar p}$-core and $d\in \N$, with $n=|\rho|+dp$. If $d=0$ and $\rho$ is odd then $(\rho,0)$ labels two spin blocks of $\cO {\tSi}_n$ with corresponding block idempotents $\ei_{\rho,0}^+,\ei_{\rho,0}^- \in \cO {\tSi}_n e_z$. In all cases we denote the block idempotent (or sum of two block idempotents) corresponding to $(\rho,d)$ by $\ei_{\rho,d}\in \cO {\tSi}_n e_z$.

\item The character $\xi_\la^{(\pm)}$ lies in the (sum of) block(s) $\cO {\tSi}_n \ei_{\rho,d}$, where $\rho$ is the ${\bar p}$-core of $\la$. If $d=0$ and $\rho$ is odd, we choose the labeling so that $\xi_\rho^{\pm}$ lies in the block $\cO {\tSi}_n \ei_{\rho,0}^{\pm}$.
\end{enumerate}
For the remainder of this theorem we fix a ${\bar p}$-core $\rho$ and $d\in \N$ with $n=|\rho|+dp$, and 
set $\sR:=[|\rho|]$ and $\sJ :=[n]\setminus \sR$.
\begin{enumerate}
\item[(iii)] If $d=0$, each block of $\cO {\tSi}_n \ei_{\rho,d}$ has trivial defect group. If $d>0$, any Sylow $p$-subgroup $\Di $ of $\tSi_\sJ \leq \tSi_n$ is a defect group of $\cO {\tSi}_n \ei_{\rho,d}$.
\item[(iv)] If $d=0$, then each block of $\cO{\tSi}_n \ei_{\rho,d}$ is its own Brauer correspondent. If $d>0$, then, setting $\Di $ to be as in (iii), the Brauer correspondent of $\cO{\tSi}_n \ei_{\rho,d}$ in $N_{{\tSi}_n}(\Di )$ is $\cO N_{{\tSi}_n}(\Di ) \ei_{\rho,0}$. Here, we view $\ei_{\rho,0}$ as an element of $\cO N_{\tSi_n}(\Di )$ via $\ei_{\rho,0}\in \cO\tSi_\sR\into \cO N_{\tSi_n}(\Di )$. If $\rho = \varnothing$, we interpret $\ei_{\varnothing,0}$ as $e_z$.
\end{enumerate}
\end{Theorem}

\begin{Remark}\label{rem:superblocks}
The $\cO {\tSi}_n \ei_{\rho,d}$'s are precisely the superblocks of $\cO {\tSi}_n e_z$. That is, the finest decomposition of $\cO {\tSi}_n e_z$ into a direct sum of two-sided ideals that are invariant under the involution $\sigma_{\cO {\tSi}_n}$. This follows from the distribution of characters given in Theorem \ref{thm:DCblocks}(ii). In particular, we always have $\ei_{\rho,d} \in \cO \tAi_n e_z$.
\end{Remark}

The spin blocks of $\cO {\tAi}_n$ are described in~\cite[Proposition 3.16]{Ke}:

\begin{Theorem}\label{thm:An_blocks}
For each $n>1$, ${\bar p}$-core $\rho$ and $d\in \N$ with $n=|\rho|+dp$, $\cO {\tAi}_n \ei_{\rho,d}$ is a single block of $\cO {\tAi}_n$ unless $d=0$ and $\rho$ is even. In this latter case $\cO{\tAi}_n \ei_{\rho,d}$ is a direct sum of two blocks of $\cO {\tAi}_n$. If $n=1$, then $\cO \tAi_1 \ei_{(1),0} = \cO \tSi_1 \ei_{(1),0} = \cO e_z$ is a single block of $\cO \tAi_1$.

If $d=0$, once again, the defect group is trivial. In all other cases the defect group of $\cO {\tAi}_n \ei_{\rho,d}$ is the same as that of $\cO {\tSi}_n \ei_{\rho,d}$.
\end{Theorem}

%The Brauer correspondents of blocks of spin covers are described in~\cite[Theorem A, Corollary 26]{Ca}:

%\begin{Theorem}\label{thm:brauer}

%\end{Theorem}

We set 
$$\Blo^{\rho,d}:=\cO \tSi_n \ei_{\rho,d}\quad\text{and}\quad \Blo^{\rho,d}_{\0}:=\cO \tAi_n \ei_{\rho,d}
$$ 
(it is easy to see that $\cO \tAi_n \ei_{\rho,d}$ is indeed the purely even subalgebra of the superalgebra $\Blo^{\rho,d}$). We will often assume $d>0$ to ensure that blocks and superblocks coincide. However, the corresponding statements for trivial defect blocks are completely elementary. 

Fix a $\bar p$-core $\rho$ and a non-negative integer $d$. Let $r:=|\rho|$ and $n:=r+dp$. As in Theorem~\ref{thm:DCblocks}, we set $\sR:=[r]$ and $\sJ:=[n]\setminus \sR$. For $k=1,\dots,d$ we also set $
\sJ_k:=[r+kp]\setminus[r+(k-1)p].
$ 
In other words, we have:
\begin{equation}\label{EUVI}
\underbrace{1,\dots,r}_{\sR}\,\overbrace{\underbrace{r+1,\dots,r+p}_{\sJ_1}\,
\,\underbrace{r+p+1,\dots,r+2p}_{\sJ_2}\,
\dots\,\underbrace{r+p(d-1)+1,\dots,r+dp}_{\sJ_d}}^\sJ
\end{equation}

With this notation, the following now follows easily from Theorem~\ref{thm:DCblocks}, cf. \cite[Proposition 5.2.13]{KL}:

\begin{Lemma}\label{lem:Brauer_corr} 
%{\rm \cite[Proposition 5.2.13]{KL}}
Let $\Di $ be the defect group of the block $\Blo^{\rho,d}$. 
Then $\Di $ is abelian if and only if $d<p$. In this case we can choose $\Di =\Di_1\times \dots \times \Di_d$, where each $\Di_k$ is a Sylow $p$-subgroup of $\tSi_{\sJ_k}$. %, for $\sJ_i:=[|\rho|+ip]\setminus[|\rho|+(i-1)p].$
\end{Lemma}

Let $d<p$. 
We continue with the notation (\ref{EUVI}).
Let $\iota:\Si_d \to \Si_\sJ$ be the natural permutation action of $\Si_d$ on the $\sJ_k$'s. More precisely,
\begin{align*}
\iota(w)\cdot (r+(k-1)p+t) = r+(w(k)-1)p+t,
\end{align*}
for all $1\leq k\leq d$, $1\leq t\leq p$ and $w\in \Si_d$. We choose a lift $T_w$ of $\iota(w)$ to $\tSi_\sJ$, for each $w \in \Si_d$ and set $T_k := T_{s_k}$, for all $1\leq k\leq d-1$.

Note that every Sylow $p$-subgroup of $\Si_{\sJ_k}$ lifts uniquely and isomorphically to a Sylow $p$-subgroup of $\tSi_{\sJ_k}$. We may, therefore, choose the $\Di_k$'s in Lemma \ref{lem:Brauer_corr} such that $T_w \Di_k T_w^{-1} = \Di_{w(k)}$, for all $1\leq k\leq d$ and $w \in \Si_d$. Indeed, one can first construct Sylow $p$-subgroups of each of the $\Si_{\sJ_k}$'s that get permuted by the $\iota(w)$'s and then lift to $\tSi_{\sJ_k}$. By the construction of the $T_w$'s we also have $T_w \tSi_{\sJ_k} T_w^{-1} = \tSi_{\sJ_{w(k)}}$, for all $1\leq k\leq d$ and $w \in \Si_d$.

Note that $\cO \tSi_{\sJ_1}e_{z} \cong \cO \tSi_p e_{z}$ via $
t_{r+m} e_z
\mapsto t_m e_z $. We now fix isomorphisms between the $\cO \tSi_{\sJ_k}e_{z}$'s. For each $1 < k\leq d$, we identify $\cO \tSi_{\sJ_k}e_{z}$ with $\cO \tSi_{\sJ_1}e_{z} \cong \cO \tSi_p e_{z}$ via the isomorphism
\begin{align}\label{Sp_ident}
\cO \tSi_{\sJ_1}e_{z} \to \cO\tSi_{\sJ_k}e_{z},\quad
a \mapsto (-1)^{k|a|}T_{(1,k)}aT_{(1,k)}^{-1}.
\end{align}
(Note this does not depend on our choice of $T_{(1,k)}$, since $T_{(1,k)}$ is uniquely determined up to multiplication by $z$.) Through these isomorphisms and using (\ref{EZCommAlg}), we can identify the superalgebra
\begin{align*}%\label{algn:V1_Vd_ident}
\cO \tSi_{\sJ_1, \dots, \sJ_d} e_z \cong \cO \tSi_{\sJ_1}e_z \otimes \dots \otimes \cO \tSi_{\sJ_d}e_z
\end{align*}
with $(\cO \tSi_p e_z)^{\otimes d}$.

\begin{Lemma}\label{lem:kappa}
There exist $\kappa_k \in \cO^\times$, for $1\leq k \leq d-1$, satisfying the following properties:
\begin{enumerate}
\item 
$
(\kappa_k T_k e_z)(a_1\otimes \dots \otimes a_d)(\kappa_k^{-1} T_k^{-1} e_z) = (-1)^{\sum_{j\neq k,k+1}|a_j|}({}^{s_k}(a_1\otimes \dots \otimes a_d)),
$
for all $1\leq k\leq d-1$ and $a_1 \otimes \dots \otimes a_d\in (\cO \tSi_p e_z)^{\otimes d}$.
\item
$
(\kappa_k T_k e_z)^2 = e_z,
$ 
for all $1\leq k\leq d-1$.
\item
$
(\kappa_k T_k e_z)(\kappa_l T_l e_z) = -(\kappa_l T_l e_z)(\kappa_k T_k e_z),
$ 
for all $1\leq k,l\leq d-1$ with $|k-l|>1$.
\item
$
(\kappa_k T_k e_z)(\kappa_{k+1} T_{k+1} e_z)(\kappa_k T_k e_z) = (\kappa_{k+1} T_{k+1} e_z)(\kappa_k T_k e_z)(\kappa_{k+1} T_{k+1} e_z),
$
for all $1\leq k\leq d-2$.
\end{enumerate}
\end{Lemma}

\begin{proof}
This is essentially contained in the proof of \cite[Proposition 5.2.13]{KL} but we outline the main points.

Property (i) does not depend on the $\kappa_k$'s. Indeed, if $a_1 \otimes \dots \otimes a_d \in (\cO \tSi_p e_z)^{\otimes d}$, we have
\begin{align*}
&(\kappa_k T_k e_z)(a_1\otimes \dots \otimes a_d)(\kappa_k^{-1} T_k^{-1} e_z) = T_k(a_1\otimes \dots \otimes a_d)T_k^{-1} \\
= &  T_{(1,k)}T_{(1,k+1)}T_{(1,k)} (a_1\otimes \dots \otimes a_d) T_{(1,k)}^{-1}T_{(1,k+1)}^{-1}T_{(1,k)}^{-1}.
\end{align*}
Property (i) now follows from a direct calculation using (\ref{Sp_ident}).

Property (iii) also does not depend on the $\kappa_k$'s and follows immediately from (\ref{EZCommAlg}).

Since $T_k$ is a lift of $\iota(s_k)$, for each $1 \leq k\leq d-1$, either $T_k^2 = 1$ or $z$. Furthermore, for each $1\leq k,l \leq d-1$, $T_k$ is conjuagte in $\tSi_\sJ$ to either $T_l$ or $zT_l$. Therefore, property (ii) holds with either all $\kappa_k = 1$ or all $\kappa_k = \sqrt{-1}$.

Finally, since the $\pi_n(T_k)$'s satisfy the braid relations, property (iv) holds up to a sign, for each $1\leq k\leq d-2$. We may therefore replace some of the $\kappa_k$'s with $-\kappa_k$ to ensure that property (iv) holds. Note that this last reassignment will not stop property (ii) being satisfied.
\end{proof}

\subsection{Special subgroups and their blocks}\label{sec:notation}
Throughout the subsection we fix a  ${\bar p}$-core $\rho$ and an  integer $d$ satisfying $0\leq d< p$ (the assumption $d<p$ is to guarantee that we are in the abelian defect group case, see Lemma~\ref{lem:Brauer_corr}, although it is not always needed.) 

We adopt the notation (\ref{EUVI}), in particular $r:=|\rho|$ and $n:=r+dp$. 
For $0\leq k\leq d$, we define the following subgroups of~$\tSi_n$:
\begin{align}
\label{Gi}
\Gi_k&:= \tSi_{\sR\cup \sJ_1\cup\dots\cup \sJ_k}\cong \tSi_{r+kp},\\
\label{Hi}
\Hi_k&:= \tSi_{\sR\cup \sJ_1\cup\dots\cup \sJ_k, \sJ_{k+1}, \dots, \sJ_d}\cong \tSi_{r+kp,p^{d-k}},
%^{\times_z (d-i)},
\\
\label{Li}
\Li_k&:= \tSi_{\sR,\sJ_1,\dots,\sJ_k}\cong 
\tSi_{r,p^k},\\
\label{Ni}
\Ni_k&:= N_{\Gi_k}(\tSi_{\sJ_1,\dots,\sJ_k}),
\end{align}
and set $\Gi:=\Gi_d=\tSi_n$, $\Hi:=\Hi_{d-1}=\tSi_{n-p,p}$ (if $d>0$), $\Li :=\Li_d=\tSi_{r,p^d}$ and $\Ni :=\Ni_d$. Note that $\Ni $ must permute  $\sJ_1,\dots,\sJ_d$. In particular, $\Ni  \leq \tSi_n$ is the subgroup generated by $\Li=\tSi_{\sR,\sJ_1,\dots,\sJ_d}$ and the $T_w$'s from $\S$\ref{sec:blocks}. 

Recalling the notation (\ref{EG0}), we will also use the subgroups $\Gi_\0=\tAi_n$, $\Ni_\0=\Ni\cap \tAi_n$, etc. (not to be confused with $\Gi_0$, $\Ni_0$, etc.)

We make the following identifications using (\ref{EZCommAlg}):
\begin{align}
\begin{split}\label{algn:group_tensor_isom}
\cO \Hi_k e_z & \cong \cO \tSi_{\sR \cup \sJ_1\cup\dots\cup \sJ_k}e_z \otimes \cO \tSi_{\sJ_{k+1}}e_z \otimes \dots \otimes \cO \tSi_{\sJ_d}e_z \\
\cO \Li_k e_z & \cong \cO \tSi_\sR e_z \otimes \cO \tSi_{\sJ_1}e_z \otimes \dots \otimes \cO \tSi_{\sJ_k}e_z,
\end{split}
\end{align}
for all $0 \leq k \leq d$. Furthermore, for $2 \leq k \leq d$, we identify $\cO\tSi_{\sJ_k}e_z$ with $\cO\tSi_{\sJ_1}e_z \cong \cO\tSi_p e_z$ via (\ref{Sp_ident}). Through these identifications we set $\ei_{\varnothing,1}^{(k)} \in \cO \tSi_{\sJ_k}e_z$ to be the image of $\ei_{\varnothing,1} \in \cO \tSi_p e_z$. We now define central idempotents
\begin{align}
\label{EBI}
\bi_k&:= \ei_{\rho,k} \in \cO \Gi_k e_z, \\ 
\label{ECI}
\ci_k&:= \ei_{\rho,k} \otimes \ei_{\varnothing,1}^{(k+1)} \otimes \dots \otimes \ei_{\varnothing,1}^{(d)} \in \cO \Hi_k e_z, \\
\label{EFI}
\fid_k&:= \ei_{\rho,0} \otimes \ei_{\varnothing,1}^{(1)} \otimes \dots \otimes \ei_{\varnothing,1}^{(k)} \in \cO \Li_k e_z,
\end{align}
for all $0\leq k\leq d$ and set 
\begin{equation}
\label{EBCF}
\bi:=\bi_d = \ci_d \in \cO \Gi e_z, \quad \ci:=\ci_{d-1} \in \cO \Hi e_z,\quad  \fid:=\fid_d=\ci_0 \in \cO \Li e_z.
\end{equation}
For $0\leq k\leq l\leq d$, noting that the idempotents $\ci_m$ commute, we define  the idempotent 
\begin{equation}
\label{ECLong}
\ci_{k,l}:=\ci_k\ci_{k+1}\cdots\ci_l,
\end{equation}
which we consider as an elements of $\cO \Gi e_z$. 
Occasionally, we will also need  
\begin{align}
\Hi_k'&:= \tSi_{\sR\cup \sJ_1\cup\dots \sJ_k, \sJ_{k+1}, \dots, \sJ_{d-1}}\cong \tSi_{r+kp,p^{d-k-1}} \leq \Gi_{d-1}, 
\\
\label{ECI'}
\ci_k'&:= \ei_{\rho,k} \otimes \ei_{\varnothing,1}^{(k+1)} \otimes \dots \otimes \ei_{\varnothing,1}^{(d-1)} \in \cO \Hi_k' e_z,
\\
\label{EC'Long}
\ci_{k,l}'&:=\ci_k'\ci_{k+1}'\cdots\ci_l'
\end{align}
for $0\leq k\leq l\leq d-1$. We note that, by Remark \ref{rem:superblocks}, all the $\bi_k$'s,  $\fid_k$'s, $\ci_k$'s, $\ci_k'$'s, etc. are in $\cO \Gi_{\0}$.

Using Lemma \ref{lem:kappa} we can now precisely express $\cO \Ni \fid$ in terms of a twisted wreath product. More precisely, since we know $\Ni $ is generated by $\Li $ and the $T_w$'s, the isomorphism 
$
\cO \Li  e_z \cong \cO \tSi_\sR e_z \otimes (\cO\tSi_p e_z)^{\otimes d},
$ 
from (\ref{algn:group_tensor_isom}) extends to an isomorphism
\begin{align*}
\cO \Ni  e_z  \cong \cO \tSi_\sR e_z \otimes (\cO\tSi_p e_z \swr \cT_d),\ 
T_k e_z  \mapsto \kappa_k^{-1}(1 \otimes \ct_k),
\end{align*}
where the $\kappa_k$'s are those from Lemma \ref{lem:kappa}. In particular, we have the isomorphisms of superalgebras
\begin{align}
\begin{split}\label{algn:ONf_ident}
\cO \Li \fid & \cong \cO \tSi_\sR \ei_{\rho,0} \otimes (\cO\tSi_p \ei_{\varnothing,1})^{\otimes d} \cong \Blo^{\rho,0} \otimes (\Blo^{\varnothing,1})^{\otimes d} \\
\cO \Ni \fid & \cong \cO \tSi_\sR \ei_{\rho,0} \otimes (\cO\tSi_p \ei_{\varnothing,1} \swr \cT_d) \cong \Blo^{\rho,0} \otimes (\Blo^{\varnothing,1} \swr \cT_d).
\end{split}
\end{align}

As in Lemma \ref{lem:Brauer_corr} and the subsequent comments, for $k=1,\dots,d$, we set $\Di_k$ to be a Sylow $p$-subgroup of $\tSi_{\sJ_k}$ such that $T_w \Di_k T_w^{-1} = \Di_{w(k)}$, for all $k=1,\dots,d$ and $w \in \Si_d$, and define 
\begin{equation}\label{ED}
\Di  := \Di_1 \times \dots \times \Di_d \leq \tSi_\sJ.
\end{equation}

\begin{Lemma}\label{lem:CGD} We have:
\begin{enumerate}
\item 
$
C_\Gi (\Di ) = \tSi_\sR  \times_z C_{\tSi_{\sJ_1}}(\Di_1) \times_z \dots \times_z C_{\tSi_{\sJ_d}}(\Di_d).
$
\item $N_\Gi (\Di )$ is generated by
$
\tSi_\sR  \times_z N_{\tSi_{\sJ_1}}(\Di_1) \times_z \dots \times_z N_{\tSi_{\sJ_d}}(\Di_d)
$ 
and  $\{T_w\mid w\in \Si_d\}$. In particular, $N_\Gi (\Di ) \leq \Ni $.
\item $N_{\Gi  \times \Li }(\Delta \Di ) \leq \Li  \times \Li $ and $N_{\Gi  \times \Hi }(\Delta \Di ) \leq \Hi  \times \Hi $.
\end{enumerate}
\end{Lemma}

\begin{proof}
(i) Since $\Di_k \leq \tAi_{\sJ_k}$, for each $k$, we have that $\tSi_\sR $ and all the $C_{\tSi_{\sJ_k}}(\Di_k)$'s centralize $\Di $. The reverse inclusion is easily seen by passing to $\Si_n$, as centralizers decompose according to an element's cycle decomposition.

(ii) Since $\Di  \leq \Gi_{\0}$, it is easy to check that $\tSi_\sR $, all the $N_{\tSi_{\sJ_k}}(\Di_k)$'s and all the $T_w$'s normalize $\Di $. To show the reverse inclusion we again pass to $\Si_n$. By inspecting cycle types, if $g \in N_\Gi (\Di )$, then it must permute the $\Di_k$'s and consequently the $\sJ_k$'s. The claim follows.

(iii) Suppose $(g,h) \in N_{\Gi  \times \Li }(\Delta \Di )$ and $1 \leq k \leq d$. Then ${}^{(g,h)}\Delta \Di_k \leq \Delta \Di $. However, since $h \in \Li $, we must have ${}^{(g,h)}\Delta \Di_k = \Delta \Di_k$. Therefore,
\begin{align*}
g \in N_\Gi (\Di_k) \leq \tSi_{\sR  \cup \bigcup_{m \neq k} \sJ_m,\sJ_k}
\end{align*}
where the containment follows by passing to $\Si_n$ and considering cycle types. Taking all such $k$'s simultaneously gives that $g \in \Li $, as desired.

The second inclusion is proved similarly.
\end{proof}

\begin{Lemma}\label{lem:c_f_blocks}
Let $0<d<p$. 
\begin{enumerate}
\item $\cO \Gi \bi$ is a block with defect group $\Di $.
\item $\cO \Hi \ci$ is a block with defect group $\Di $.
\item $\cO \Li \fid$ and $\cO \Ni \fid$ are both blocks with defect group $\Di $.
\end{enumerate}
\end{Lemma}

\begin{proof}
(i) This follows immediately from Theorem \ref{thm:DCblocks}(i),(iii).

(ii) We first assume $d=1$ and $\rho$ is odd and set $\Hi_{\Ai} := \tSi_\sR  \times_z \tAi_{\sJ_1} \leq \Hi $. Now, $\ei_{\rho,d-1} = \ei_{\rho,0} = \ei_{\rho,0}^+ + \ei_{\rho,0}^-$ and $\ei_{\varnothing,1}^{(d)}=\ei_{\varnothing,1}^{(1)}$. Furthermore, by Theorem \ref{thm:DCblocks}(ii),(iii), $\sigma_{\cO \tSi_\sR }(\ei_{\rho,0}^+) = \ei_{\rho,0}^-$ and $\cO \tSi_\sR \ei_{\rho,0}^+$ is a block with trivial defect group and by Theorem \ref{thm:An_blocks} $\cO \tAi_{\sJ_1}\ei_{\varnothing,1}^{(1)}$ is a block with defect group $\Di_1=\Di $. Moreover, as $\tSi_\sR $ and $\tAi_{\sJ_1}$ commute and $|\tSi_\sR  \cap \tAi_{\sJ_d}|=2$, $\cO \Hi_{\Ai}(\ei_{\rho,0}^+ \otimes \ei_{\varnothing,1}^{(1)})$ is a block with defect group $\Di $.

If $g \in \tSi_{\sJ_1} \setminus \tAi_{\sJ_1}$, then ${}^g \ei_{\rho,0}^+ = \ei_{\rho,0}^-$. Therefore, Lemma \ref{lem:unstable_block} now gives that
\begin{align*}
\Tr_{\Hi_{\Ai}}^\Hi (\ei_{\rho,0}^+ \otimes \ei_{\varnothing,1}^{(1)}) = \ei_{\rho,0} \otimes \ei_{\varnothing,1}^{(1)}=\ci
\end{align*}
is a block idempotent of $\cO \Hi $ and $\cO \Hi \ci$ has defect group $\Di $.

We now assume that $d>1$ or $\rho$ is even. Then $\cO \Gi_{d-1}\ei_{\rho,d-1}$ is a block with defect group $\Di_1 \times \dots \times \Di_{d-1}$ and $\cO \tSi_{\sJ_d}\ei_{\varnothing,1}^{(d)}$ a block with defect group $\Di_d$. The claim now follows from Remark \ref{rem:bisupmod} and Lemma \ref{lem:direct_prod_vert}.

(iii) We first prove $\fid$ is a block idempotent of $\cO \Li $. We have already proved the claim for $d=1$ in part (i). For $d>1$, note that, by induction on $d$, $\cO \Li_{d-1} \fid_{d-1}$ is a block with defect group $\Di_1 \times \dots \times \Di_{d-1}$ and so the claim holds, as in part (i), by Remark \ref{rem:bisupmod} and Lemma \ref{lem:direct_prod_vert}. %(Here, the induction is valid since, by Remark \ref{rem:induct_Rou}, $\rho$ is $(d-1)$-Rouquier.)

To show that $\cO \Ni \fid$ is a block we simply apply Lemmas \ref{lem:CGD}(i) and \ref{lem:vert_blocks}. That it has defect group $\Di $ is immediate since $p \nmid[\Ni :\Li ]$.
\end{proof}

\begin{Lemma}\label{lem:Brauer_G_N}
Let $0<d<p$. Then 
$\cO \Ni \fid$ (resp. $\cO \Ni_{\0}\fid$) is the Brauer correspondent of $\cO \Gi \bi$ (resp. $\cO \Gi_{\0}\bi$) in $\Ni $ (resp. $\Ni_{\0}$).
\end{Lemma}

\begin{proof}
By Lemma \ref{lem:c_f_blocks}, $\cO \Ni \fid$ and $\cO \Gi \bi$ both have defect group $\Di $ and, by Lemma \ref{lem:CGD}(ii), $N_\Gi (\Di ) \leq \Ni $. We prove that $\cO \Ni \fid$ and $\cO \Gi \bi$ are Brauer correspondents by showing that they have a common Brauer correspondent in $N_\Gi (\Di )$.

By Theorem \ref{thm:DCblocks}(iv), the Brauer correspondent of $\cO \Gi \bi$ in $N_\Gi (\Di )$ is $\cO N_\Gi (\Di )\ei_{\rho,0}$. Now,
\begin{align*}
\Br_\Di (\fid) & = \Br_{\Di_1 \times \dots \times \Di_d}(\ei_{\rho,0} \otimes \ei_{\varnothing,1}^{(1)} \otimes \dots \otimes \ei_{\varnothing,1}^{(d)}) \\
& = {\bar \ei}_{\rho,0} \otimes \Br_{\Di_1}(\ei_{\varnothing,1}^{(1)}) \otimes \dots \otimes \Br_{\Di_1}(\ei_{\varnothing,1}^{(d)}) 
\\
& = {\bar \ei}_{\rho,0} \otimes {\bar e}_z \otimes \dots \otimes {\bar e}_z = {\bar \ei}_{\rho,0},
\end{align*}
where the second equality follows from the decomposition of $C_\Gi (\Di )$ in Lemma \ref{lem:CGD}(i) and the third from Theorem \ref{thm:DCblocks}(iv) applied to each $\cO \tSi_{\sJ_k}\ei_{\varnothing,1}^{(k)}$. The claim for $\cO \Ni \fid$ and $\cO \Gi \bi$ follows.

For $\cO \Ni_{\0}\fid$ and $\cO \Gi_{\0}\bi$, we first note that, by Theorem \ref{thm:An_blocks}, $\bi$ is a block idempotent of $\cO \Gi_{\0}$ and certainly $N_{\Gi_{\0}}(\Di) \leq \Ni_{\0}$. Next, by Lemma \ref{lem:Brauer_Green}, we have that $\cO \Ni \fid$ and $\cO \Gi \bi$ are super Green correspondents. In particular,
$$
\cO \Ni \fid\mid\Res^{\Gi  \times \Gi }_{\Ni  \times \Ni }(\cO \Gi \bi)\quad\text{and}\quad
 \cO \Gi \bi\mid\Ind_{\Ni  \times \Ni }^{\Gi  \times \Gi }(\cO \Ni \fid)
$$
as superbimodules. Taking the even part of both expressions gives that
$$
\cO \Ni_{\0}\fid\mid\Res^{\Gi_{\0} \times \Gi_{\0}}_{\Ni_{\0} \times \Ni_{\0}}(\cO \Gi_{\0}\bi)\quad\text{and}\quad \cO \Gi_{\0}\bi\mid\Ind_{\Ni_{\0} \times \Ni_{\0}}^{\Gi_{\0} \times \Gi_{\0}}(\cO \Ni_{\0}\fid).
$$
The claim follows.
\end{proof}

\section{RoCK blocks}\label{sec:RoCK}

To keep in line with Section \ref{sec:spin_blocks}, for the remainder of this article, we assume that $\K$ contains a primitive $(2n!)^{\nth}$ root of unity. We will also assume that $n \geq 4$, so, by the comments at the beginning of Section~\ref{sec:spin_blocks}, it is automatic that $-1$ and $2$ have square roots in $\K$. This is harmless since, when working with RoCK blocks, the assumption $n \geq 4$ will hold  automatically as we normally assume $d \geq 1$ in the setup below. %Therefore, by the comments at the beginning of Section~\ref{sec:spin_blocks}, it is automatic that $-1$ and $2$ have square roots in $\K$.

\subsection{Rouquier cores and RoCK blocks}
Let $d\in\N$. 
We take the following definition from our \cite[$\S$4.1a]{KL}.

\begin{Definition}\label{def_Rou}
A {\em $d$-Rouquier ${\bar p}$-core} $\rho$ is a ${\bar p}$-core such that $\Ab_\rho$ has the following properties:
\begin{enumerate}
\item The $1^{\nst}$ runner has at least $d$ beads.
\item The $(i+1)^{\nth}$ runner has at least $d-1$ more beads than the $i^{\nth}$ runner, for $1\leq i\leq \ell-1$.
\end{enumerate}
If $\rho$ is a $d$-Rouquier $\bar p$-core, we refer to $\Blo^{\rho,d}$ as a {\em RoCK block} of {\em weight $d$}. 
\end{Definition}

If $d\geq 1$ and $\rho$ is $d$-Rouquier ${\bar p}$-core then it is automatic that the $i^{\nth}$ runner of $\Ab_\rho$ is empty, for $\ell< i< p$.

We note that, for each $d$, there exist infinitely many even and infinitely many odd $d$-Rouquier ${\bar p}$-cores, as one can add arbitrarily many beads onto the $\ell^{\nth}$ runner.

The following trivial observation will be very useful throughout this paper.

\begin{Remark}\label{rem:induct_Rou}
If $\rho$ is a $d$-Rouquier ${\bar p}$-core then $\rho$ is a $k$-Rouquier ${\bar p}$-core for all $k \leq d$.
\end{Remark}

The key properties of Rouquier ${\bar p}$-cores are contained in the following lemma, mostly taken from \cite[Lemma 4.1.2]{KL}.

\begin{Lemma}\label{ChuKesLem}
Let $d \in \NN$, $\rho$ a $d$-Rouquier $\bar p$-core and $\la\in\Par_0(\rho,d)$. Then $\Ab_\la$ is obtained from $\Ab_\rho$ by $d$ consecutive elementary slides down on runners $0,1,\dots,\ell$. Moreover, if $\mu\in\Par_0(\rho,d-1)$ such that $\mu\subseteq \la$, then there exists $0\leq i \leq \ell$ such that:
\begin{enumerate}
\item $\Ab_\la$ is obtained from $\Ab_\mu$ by an elementary slide down on runner $i$.
\item $\sh[\la\setminus\mu]$ is of the form:
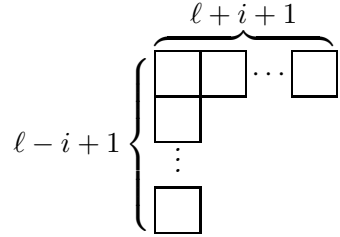
\begin{figure}[H]
    \centering
    \hspace{4.6em}$\overbrace{\hspace{6em}}^{\displaystyle \ell + i + 1}$ \\
    \begin{tabular}{r@{}l}
    \raisebox{-4.5ex}{$\ell - i + 1 \left\{\vphantom{\begin{array}{c}~\\[8.3ex] ~
    \end{array}}\right.$} &
    \begin{ytableau}
    ~               & & \none[\cdots] & \\
    ~               & \none & \none & \none \\
    \none[\vdots]   & \none & \none & \none \\
    ~               & \none & \none & \none \end{ytableau}
    \end{tabular}
\end{figure}
\item $\mu$ and $\la$ have the same parity if and only if $i=0$ and $\la = \mu \sqcup (p)$.
\end{enumerate}
\end{Lemma}

\begin{proof}
The statement concerning $\Ab_\la$ and $\Ab_\rho$ is just \cite[Lemma 4.1.1(i)]{KL}. Part (i) is stated in Lemma 4.1.2 from the same paper. There it is stated for $\la,\mu \in \Par_p$ ($p$-strict partitions) instead of $\Par_0$, so the above is just a specific case. Part (ii) follows as a consequence of Remark 4.1.3, again from the same paper. Note that the box diagrams look slightly different in \cite[Remark 4.1.3]{KL}, since it deals with non-shifted diagrams.

To prove part (iii) we note that, since the $i^{\nth}$ runner of $\Ab_\rho$ has at least $d$ beads on it, for all $1\leq i\leq \ell$, $\la$ must have the same length as $\mu$, unless $i=0$ and $\la = \mu \sqcup (p)$. Since $|\la|$ and $|\mu|$ certainly have opposite parity, the claim follows.
\end{proof}

\subsection{Induction and restriction of characters in RoCK blocks}\label{SSIRRoCK}
As the characters of $\Blo^{\varnothing,1}$ and $\Blo^{\varnothing,1}_{\0}$ will be so important throughout this paper we introduce the special notation for them. Note that
$\Par_0(\varnothing,1)=\{(p-j,j)\mid  j\in I\}$. In view of Theorem~\ref{TSchur}, the partition $(p)$ is even, while the partitions $(p-j,j)$, with $1\leq j\leq \ell$ are odd. 

Notationally, we often identify any such partition $(p-j,j)$ with $j$. For example, recalling the notation of \S\S\ref{sec:KSS1},\ref{SSIrrChar}, especially 
(\ref{EEps}) and (\ref{EEpsN}), we denote
\begin{align}
\label{EJ}
\begin{split}
\xi_j:=\xi_{(p-j,j)},\ \xi_j^{\pm}:=\xi_{(p-j,j)}^{\pm},\ \tilde{\xi}_j^{(\pm)}:=\tilde{\xi}_{(p-j,j)}^{(\pm)}, \ \xi_{\mu,j_1,\dots,j_k}:=\xi_{\mu,(p-j_1,j_1),\dots,(p-j_k,j_k)}
\\
\eps_j:=\eps_{(p-j,j)},\ \eps_{j_1,\dots,j_k}:=\eps_{(p-j_1,j_1),\dots,(p-j_k,j_k)},\ \eps_{\mu,j_1,\dots,j_k}:=\eps_{\mu,(p-j_1,j_1),\dots,(p-j_k,j_k)},
\end{split} 
\end{align}
whenever these make sense.  
In particular, we have 
\begin{align}
\label{ECharBEmptySuper}
\Irr_{\operatorname{super}}(\Blo^{\varnothing,1})&=\{\xi_0,\xi_1,\dots,\xi_{\ell}\},
\\
\label{ECharBEmpty1}
\Irr(\Blo^{\varnothing,1})&=\{\xi_0,\xi_1^{\pm},\dots,\xi_{\ell}^{\pm}\},
\\
\label{ECharBEmpty10}
\Irr(\Blo^{\varnothing,1}_{\0})&=\{\tilde{\xi}_0^{\pm},\tilde{\xi}_1,\dots,\tilde{\xi}_\ell\}. 
\end{align}
In Section \ref{sec:weight_one} we will be more precise about distinguishing between $\xi_j^+$ and $\xi_j^-$, when $j>0$.

Recall the notation used in Theorem~\ref{thm:stembridge}, especially (\ref{EEps}),(\ref{EEpsN}). Thus $\eps_\la=1$ if $\la$ is even (i.e. $|\la|-h(\la)$ is even), and $\eps_\la=\sqrt{2}$ if $\la$ is odd. In particular, $\eps_j=1$ if $j=0$ and $\eps_j=\sqrt{2}$ if $1\leq j\leq \ell$. 
Moreover, $\eps_{\mu,\la}=1$ if $(\mu,\la)$ is even (i.e. $\mu$ and $\la$ have the same parity), and $\eps_{\mu,\la}=\sqrt{2}$ if $(\mu,\la)$ is odd (i.e. $\mu$ and $\la$ have opposite  parities). 

Recall also the notation of \S\S\ref{sec:notation},\ref{sec:comb}.

\begin{Lemma}\label{lem:HtoG}
Let $d\geq 1$ and $\rho$ be a $d$-Rouquier $\bar p$-core. 
\begin{enumerate}
\item If $\mu \in \Par_0(\rho,d-1)$ and $j\in I$ then
\begin{align*}
\xi_{\mu,j}\uparrow_{\Hi ,\ci}^{\Gi ,\bi} = \sum_{\la\in \Par_0^{\leq \ell-j}(\mu)^+} \frac{\eps_{\mu,j}\eps_{\mu,\la}\eps_{j}}{\eps_\la} \,\xi_\la.
\end{align*}

\item If $\la \in \Par_0(\rho,d)$ then 
\begin{align*}
\xi_\la \downarrow^{\Gi ,\bi}_{\Hi ,\ci} = \sum_{j\in I}\sum_{\mu\in \Par_0^{\leq \ell-j}(\la)^-} \frac{\eps_\la \eps_{\mu,\la}\eps_{j}}{\eps_{\mu,j}} \,\xi_{\mu,j}.
\end{align*}
\end{enumerate}
Furthermore, if $(\mu,j)$ and $\la$ are both odd, then, unless $j=0$ and $\la = \mu \sqcup (p)$, the coefficients of $\xi_\la^+$ and $\xi_\la^-$ are equal in both $\xi_{\mu,j}^+ \uparrow_{\tSi_{\sR,\sJ}}^{\tSi_n}$ and $\xi_{\mu,j}^- \uparrow_{\tSi_{\sR,\sJ}}^{\tSi_n}$.
\end{Lemma}

\begin{proof}
Since, by Lemma \ref{lem:sup_alg_idempt_dual}, $(\bi\cO \Gi \ci)^* \simeq \ci\cO \Gi \bi$, we can apply Lemma \ref{lem:chars_bisupmod}(ii) to see that parts (i) and (ii) are equivalent. We will prove part (i).

Let $\mu \in \Par_0(\rho,d-1)$, $j\in I$ and suppose $\xi_\la$ appears with non-zero coefficient in $\xi_{\mu ,j}\uparrow_{\Hi ,\ci}^{\Gi ,\bi}$. Since $\xi_\la$ lies in the block $\cO \Gi \bi$, we have $\la \in \Par_0(\rho,d)$. Furthermore, Theorem~\ref{thm:stembridge} dictates that $\mu \subseteq \la$ and so, by Lemma~\ref{ChuKesLem}(i), $\Ab_\la$ is obtained from $\Ab_\mu$ by an elementary slide down on runner $i$, for some $i\in I$ and $\sh[\la \setminus \mu]$ is of the form

\begin{figure}[H]
    \centering
    \hspace{9.4em}$\overbrace{\hspace{6em}}^{\displaystyle \ell + i + 1}$ \\
    \begin{tabular}{r@{}l}
    \raisebox{-4.5ex}{$\mathfrak{h}_i:= \qquad \ell - i + 1 \left\{\vphantom{\begin{array}{c}~\\[8.3ex] ~
    \end{array}}\right.$} &
    \begin{ytableau}
    ~               & & \none[\cdots] & \\
    ~               & \none & \none & \none \\
    \none[\vdots]   & \none & \none & \none \\
    ~               & \none & \none & \none \end{ytableau}
    \end{tabular}
\end{figure}

Now
\begin{align*}
\mathfrak{f}_{(p-j,j)}(\mathfrak{h}_i) =
\begin{cases}
1 &\text{if }i+j \leq \ell,\\
0 &\text{if }i+j > \ell.
\end{cases}
\end{align*}
Indeed one can readily check that $\mathfrak{f}_{(p-j,j)}(\mathfrak{h}_i)=0$, if $i+j > \ell$ and that, if $i+j \leq \ell$, there is precisely one legal way of filling in $\mathfrak{h}_i$ with content $(p-j,j)$, namely

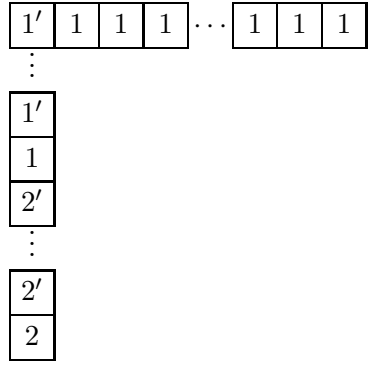
\begin{figure}[H]
    \centering
    %\hspace{9.2em}$\overbrace{\hspace{12.4em}}^{\displaystyle \ell + i + 1}$ \\
    \begin{tabular}{r@{}l}
    %\raisebox{-2.5ex}{$\mathfrak{h}_i:= \qquad \ell - i - j \left\{\vphantom{\begin{array}{c}~\\[4.2ex] ~\end{array}}\right.$} &
    \begin{ytableau}
    1'              & 1 & 1 & 1 & \none[\cdots] & 1 & 1 & 1 \\
    \none[\vdots]   & \none & \none & \none & \none & \none & \none & \none \\
    1'              & \none & \none & \none & \none & \none & \none & \none \\
    1               & \none & \none & \none & \none & \none & \none & \none \\
    2'              & \none & \none & \none & \none & \none & \none & \none \\
    \none[\vdots]   & \none & \none & \none & \none & \none & \none & \none \\
    2'              & \none & \none & \none & \none & \none & \none & \none \\
    2               & \none & \none & \none & \none & \none & \none & \none
    \end{ytableau}
    \end{tabular}
\end{figure}
Therefore, by Theorem~\ref{thm:stembridge}, we have
\begin{align*}
\xi_{\mu,j}\uparrow_{\Hi ,\ci}^{\Gi ,\bi} = \sum_{\la\in \Par_0^{\leq \ell-j}(\mu)^+} \frac{\eps_{\mu,j}}{\eps_\la} 2^{(h(\mu) + h((p-j,j)) - h(\la))/2} \xi_\la,
\end{align*}
for all $\mu \in \Par_0(\rho,d-1)$ and $j\in I$.

To get the coefficient into the desired form note that $2^{h((p-j,j))/2} = \sqrt{2}\eps_{j}$. Furthermore, $h(\la) = h(\mu)$ unless $\Ab_\la$ is obtained from $\Ab_\mu$ by adding a new bead on the runner $0$, in which case $h(\la) = h(\mu) + 1$. By Lemma \ref{ChuKesLem}(iii), this case happens precisely when $\eps_\la = \eps_\mu$. Therefore, $2^{(h(\mu) - h(\la))/2} = \eps_{\mu,\la}/\sqrt{2}$ and
\begin{align*}
2^{(h(\mu) + h((p-j,j)) - h(\la))/2} = \big(\eps_{\mu,\la}/\sqrt{2}\big)\big(\sqrt{2}\eps_{j}\big) = \eps_{\mu,\la}\eps_{j},
\end{align*}
as desired.

The final statement follows from the corresponding statement in Theorem~\ref{thm:stembridge} once we observe that, by Lemma \ref{ChuKesLem}(i), $\la = \mu \sqcup (p-j,j)$ if and only if $j=0$ and $\la = \mu \sqcup (p)$.
\end{proof}

\section{Weight one RoCK blocks}\label{sec:weight_one}
Recall that it is now assumed that $\K$ contains a primitive $(2n!)^{\nth}$ root of unity.

Throughout this section let $\rho$ be a $1$-Rouquier $\bar p$-core and $d=1$. We adopt the notation (\ref{EUVI}) and all the notation of $\S$\ref{sec:notation}. So 
\begin{align*}r=|\rho|,\ n = r+p,\ 
&\sR=[r],\ \sJ=[n]\setminus[r],\ 
\Gi  = \tSi_n,\ \Li  = \Ni  = \tSi_{\sR ,\sJ}\cong \tSi_{r,p},
\\
&\bi = \ei_{\rho,1} \in \cO \Gi ,\ \fid = \ei_{\rho,0} \otimes \ei_{\varnothing,1}^{(1)} \in \cO \Ni.
\end{align*}
We 
identify $\cO \Ni \fid$ with $\Blo^{\rho,0} \otimes \Blo^{\varnothing,1}$ via (\ref{algn:ONf_ident}). 
By Lemma \ref{lem:c_f_blocks}, 
$\Di=\Di_1$ 
is simultaneously a defect group of $\cO \tSi_{\sJ}\ei_{\varnothing,1}^{(1)} \cong \Blo^{\varnothing,1}$ and $\cO \Gi \bi = \Blo^{\rho,1}$. 

We recall the notation $\Irr$ and $\Irr_{\operatorname{super}}$ from (\ref{EIrrNotation}) and $\xi_j^{(\pm)},\tilde\xi_j^{(\pm)},\xi_{\rho,j}^{(\pm)}$, etc.  from (\ref{EJ}).

\subsection{Brauer trees of weight one RoCK blocks}

For $j\in I$, we denote by 
\begin{equation}
\label{ERho^I}
\rho^j\in\Par_0(\rho,1)
\end{equation}
 the partition whose abacus $\Ab_{\rho^j}$ is obtained from $\rho$ by sliding a bead down the $j^{\nth}$ runner. 

\begin{Lemma}\label{lem:weight1_equiv}
Let $\rho$ be a $1$-Rouquier $\bar p$-core and $d=1$.
\begin{enumerate}
\item If $\rho$ is even, there exists a Morita superequivalence between $\cO \Ni \fid$ and $\Blo^{\varnothing,1}$ with the corresponding bijection 
$$\Irr_{\operatorname{super}}(\cO \Ni \fid) \to \Irr_{\operatorname{super}}(\Blo^{\varnothing,1}),\ \xi_{\rho,j} \mapsto \xi_j.
$$

\item If $\rho$ is odd, there exists a Morita equivalence between $\cO \Ni \fid$ and $\Blo^{\varnothing,1}_{\0}$ with the corresponding bijection 
$$ \Irr(\cO \Ni \fid) \to \Irr(\Blo^{\varnothing,1}_{\0}),\ 
\left\{
\begin{array}{ll}
\xi_{\rho,j} \mapsto \tilde \xi_j&\hbox{if $1\leq j\leq \ell$,}\\
\xi_{\rho,0}^\pm\mapsto \tilde \xi_0^\pm.&\hbox{}
\end{array}
\right.
$$
for an appropriate choice of irreducible characters $(-)^\pm$. 
\end{enumerate}

The Morita equivalences in (i) and (ii) can be chosen to have trivial source.

%\textcolor{red}{Do we need super here? I'm not sure we do actually.}

\begin{enumerate}
\item[(iii)] The irreducible characters of $\Blo^{\varnothing,1}$ can be labelled such that $\Blo^{\varnothing,1}$ has Brauer tree:
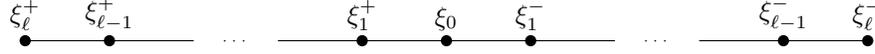
\begin{align*}
\begin{braid}\tikzset{baseline=1.5mm}
\coordinate (1) at (0,0);
\coordinate (2) at (3.5,0);
\coordinate (3) at (7,0);
\coordinate (4) at (10.5,0);
\coordinate (5) at (14,0);
\coordinate (6) at (17.5,0);
\coordinate (7) at (21,0);
\coordinate (8) at (24.5,0);
\coordinate (9) at (28,0);
\coordinate (10) at (31.5,0);
\coordinate (11) at (35,0);
\draw[thin] (1) -- (2);
\draw[thin] (2) -- (3);
\draw(8.75,0) node {$\cdots$};
\draw[thin] (4) -- (5);
\draw[thin] (5) -- (6);
\draw[thin] (6) -- (7);
\draw[thin] (7) -- (8);
\draw(26.25,0) node {$\cdots$};
\draw[thin] (9) -- (10);
\draw[thin] (10) -- (11);
\node at (1)[circle,fill,inner sep=1.5pt]{};
\node at (1) [above] {\small $\xi^+_{\ell}$};
\node at (2)[circle,fill,inner sep=1.5pt]{};
\node at (2) [above] {\small $\xi^+_{\ell-1}$};
\node at (5)[circle,fill,inner sep=1.5pt]{};
\node at (5) [above] {\small $\xi^+_1$};
\node at (6)[circle,fill,inner sep=1.5pt]{};
\node at (6) [above] {\small $\xi_0$};
\node at (7)[circle,fill,inner sep=1.5pt]{};
\node at (7) [above] {\small $\xi^-_1$};
\node at (10)[circle,fill,inner sep=1.5pt]{};
\node at (10) [above] {\small $\xi^-_{\ell-1}$};
\node at (11)[circle,fill,inner sep=1.5pt]{};
\node at (11) [above] {\small $\xi^-_{\ell}$};
\end{braid}
\end{align*}
The irreducible characters of $\Blo^{\varnothing,1}_{\0}$ can be labelled such that $\Blo^{\varnothing,1}_{\0}$ has Brauer tree:
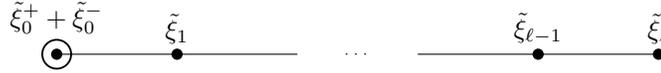
\begin{align*}
\begin{braid}\tikzset{baseline=1.5mm}
\coordinate (1) at (0,0);
\coordinate (A) at (0,0.4);
\coordinate (2) at (5,0);
\coordinate (3) at (10,0);
\coordinate (4) at (15,0);
\coordinate (5) at (20,0);
\coordinate (6) at (25,0);
\draw[thin] (1) -- (2);
\draw[thin] (2) -- (3);
\draw(12.5,0) node {$\cdots$};
\draw[thin] (4) -- (5);
\draw[thin] (5) -- (6);
\node at (1)[circle,fill,inner sep=1.5pt]{};
\node at (A) [above] {\small $\tilde{\xi}_0^+ + \tilde{\xi}_0^-$};
\node at (2)[circle,fill,inner sep=1.5pt]{};
\node at (2) [above] {\small $\tilde{\xi}_1$};
\node at (5)[circle,fill,inner sep=1.5pt]{};
\node at (5) [above] {\small $\tilde{\xi}_{\ell-1}$};
\node at (6)[circle,fill,inner sep=1.5pt]{};
\node at (6) [above] {\small $\tilde{\xi}_{\ell}$};
\draw (1) circle (6mm);
\end{braid}
\end{align*}
\item[(iv)] If $\rho$ is even, the irreducible characters of $\cO \Gi \bi$ can be labelled such that $\cO \Gi \bi$ has Brauer tree:
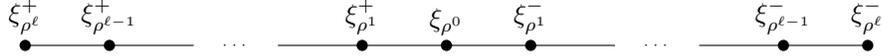
\begin{align*}
\begin{braid}\tikzset{baseline=1.5mm}
\coordinate (1) at (0,0);
\coordinate (2) at (3.5,0);
\coordinate (3) at (7,0);
\coordinate (4) at (10.5,0);
\coordinate (5) at (14,0);
\coordinate (6) at (17.5,0);
\coordinate (7) at (21,0);
\coordinate (8) at (24.5,0);
\coordinate (9) at (28,0);
\coordinate (10) at (31.5,0);
\coordinate (11) at (35,0);
\draw[thin] (1) -- (2);
\draw[thin] (2) -- (3);
\draw(8.75,0) node {$\cdots$};
\draw[thin] (4) -- (5);
\draw[thin] (5) -- (6);
\draw[thin] (6) -- (7);
\draw[thin] (7) -- (8);
\draw(26.25,0) node {$\cdots$};
\draw[thin] (9) -- (10);
\draw[thin] (10) -- (11);
\node at (1)[circle,fill,inner sep=1.5pt]{};
\node at (1) [above] {\small $\xi^+_{\rho^\ell}$};
\node at (2)[circle,fill,inner sep=1.5pt]{};
\node at (2) [above] {\small $\xi^+_{\rho^{\ell-1}}$};
\node at (5)[circle,fill,inner sep=1.5pt]{};
\node at (5) [above] {\small $\xi^+_{\rho^1}$};
\node at (6)[circle,fill,inner sep=1.5pt]{};
\node at (6) [above] {\small $\xi_{\rho^0}$};
\node at (7)[circle,fill,inner sep=1.5pt]{};
\node at (7) [above] {\small $\xi^-_{\rho^1}$};
\node at (10)[circle,fill,inner sep=1.5pt]{};
\node at (10) [above] {\small $\xi^-_{\rho^{\ell-1}}$};
\node at (11)[circle,fill,inner sep=1.5pt]{};
\node at (11) [above] {\small $\xi^-_{\rho^\ell}$};
\end{braid}
\end{align*}

\item[(v)] If $\rho$ is odd, the irreducible characters of $\cO \Gi \bi$ can be labelled such that $\cO \Gi \bi$ has Brauer tree:
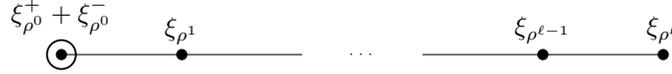
\begin{align*}
\begin{braid}\tikzset{baseline=1.5mm}
\coordinate (1) at (0,0);
\coordinate (A) at (0,0.4);
\coordinate (2) at (5,0);
\coordinate (3) at (10,0);
\coordinate (4) at (15,0);
\coordinate (5) at (20,0);
\coordinate (6) at (25,0);
\draw[thin] (1) -- (2);
\draw[thin] (2) -- (3);
\draw(12.5,0) node {$\cdots$};
\draw[thin] (4) -- (5);
\draw[thin] (5) -- (6);
\node at (1)[circle,fill,inner sep=1.5pt]{};
\node at (A) [above] {\small $\xi_{\rho^0}^+ + \xi_{\rho^0}^-$};
\node at (2)[circle,fill,inner sep=1.5pt]{};
\node at (2) [above] {\small $\xi_{\rho^1}$};
\node at (5)[circle,fill,inner sep=1.5pt]{};
\node at (5) [above] {\small $\xi_{\rho^{\ell-1}}$};
\node at (6)[circle,fill,inner sep=1.5pt]{};
\node at (6) [above] {\small $\xi_{\rho^\ell}$};
\draw (1) circle (6mm);
\end{braid}
\end{align*}
\end{enumerate}
\end{Lemma}

\begin{proof}
Suppose $\rho$ is even. If $r=1$, then $\cO \Ni \fid=\cO \tSi_{\sJ}\ei_{\varnothing,1}^{(1)} \cong \Blo^{\varnothing,1}$ and the  Morita superequivalence is clear. Let $r>1$.
By Theorems \ref{thm:DCblocks} and \ref{thm:An_blocks}, $|\Blo^{\rho,0}| \cong \cM_{2m \times 2m}(\cO)$ and $\Blo^{\rho,0}_{\0} \cong \cM_{m \times m}(\cO) \oplus \cM_{m \times m}(\cO)$, for some $m \in \NN$. In particular, any primitive idempotent in $e \in \Blo^{\rho,0}_{\0}$ remains primitive in $\Blo^{\rho,0}$. Viewing $e$ as an element of $\cO \Ni \fid \cong \Blo^{\rho,0} \otimes \Blo^{\varnothing,1}$, we have the  isomorphism of superalgebras
\begin{align*}
\Blo^{\varnothing,1}  \iso e\cO \Ni \fid e \cong e\Blo^{\rho,0}e \otimes \Blo^{\varnothing,1}, \ 
x  \mapsto e \otimes x.
\end{align*}
Furthermore, since $\Blo^{\rho,0}e\Blo^{\rho,0}=\Blo^{\rho,0}$, we have $(\cO \Ni \fid)e(\cO \Ni \fid) = \cO \Ni \fid$ and so, by Lemma \ref{lem:idmpt_Mor}, $e\cO \Ni \fid$ induces a Morita superequivalence between $\cO \Ni \fid$ and $\Blo^{\varnothing,1}$.

If $\rho$ is odd, then, by Theorem \ref{thm:DCblocks}, $|\Blo^{\rho,0}| \cong \cM_{m \times m}(\cO) \oplus \cM_{m \times m}(\cO)$, for some $m \in \NN$. Furthermore, $\sigma_{\cO\tSi_\sR }$ swaps these two factors. (In other words, $\Blo^{\rho,0} \cong \cQ_m(\cO)$.) Let $e \in \Blo^{\rho,0}$ be a primitive idempotent. The above shows that $e\sigma_{\cO\tSi_\sR }(e)=0$. Let $u \in \Blo^{\varnothing,1}_{\1} \cap (\Blo^{\varnothing,1})^{\times}$. Viewing $e$ as an element of $\cO \Ni \fid \cong \Blo^{\rho,0} \otimes \Blo^{\varnothing,1}$, we have that ${}^u e = \sigma_{\cO\tSi_\sR }(e)$. Therefore, $e\Blo^{\varnothing,1}e = e \otimes \Blo^{\varnothing,1}_{\0}$ and we have the isomorphism of algebras
\begin{align*}
\Blo^{\varnothing,1}_{\0} \iso e\cO \Ni \fid e \cong e\Blo^{\rho,0}e \otimes \Blo^{\varnothing,1}_{\0}, \
x \mapsto e \otimes x.
\end{align*}
Now,
$$
\Blo^{\rho,0}e\Blo^{\rho,0} + \Blo^{\rho,0}ueu^{-1}\Blo^{\rho,0} = \Blo^{\rho,0}(e + \sigma_{\cO \tSi_\sR }(e))\Blo^{\rho,0} = \Blo^{\rho,0}.
$$
Therefore, $(\cO \Ni \fid)e(\cO \Ni \fid) = \cO \Ni \fid$ and we have shown that $e\cO \Ni \fid$ induces a Morita equivalence between $\cO \Ni \fid$ and $\Blo^{\varnothing,1}_{\0}$.

The bijections of characters given in parts (i) and (ii) are both just a consequence of the fact that $e\cO \Ni \fid \otimes_{\cO \Ni \fid} ?$ is a summand of restriction $\Res^{\tSi_{\sR ,\sJ}}_{\tSi_{\sJ}}$. Therefore, by Lemma \ref{lem:A_ten_B}(ii), nothing other than $\xi_j^{(\pm)}$, when $\rho$ is even, or $\tilde{\xi}_j^{(\pm)}$, when $\rho$ is odd, can occur as an irreducible constituent of the image of $\xi_\rho \otimes \xi_j$.

To show that the relevant bimodule always has trivial source we note that, in both cases, $\cO \Ni \fid$ has trivial source as an $(\cO \Ni ,\cO \Ni )$-bimodule and so $e\cO \Ni \fid$ is a direct sum of trivial source bimodules when considered as an $(\cO \Di,\cO \Di)$-bimodule. The claim follows.

Parts (iii), (iv) and (v) all follow immediately from \cite[Theorem 4.4]{Mu}.
\end{proof}

\subsection{Morita superequivalences for weight one RoCK blocks}

For the following lemma we treat $\bi \cO \Gi \fid$ as an $(\cO \Gi \bi,\cO \Ni \fid)$-bisupermodule and $\fid \cO \Gi \bi$ as an $(\cO \Ni \fid,\cO \Gi \bi)$-bisupermodule.

\begin{Lemma}\label{Omega_bOGf_Morita}
Let $\rho$ be a $1$-Rouquier $\bar p$-core and $d=1$. 
\begin{enumerate}
\item As a bisupermodule, $\bi\cO \Gi \fid$ (resp. $\fid \cO \Gi \bi$) has a unique non-projective, absolutely indecomposable summand\, $\gU$ (resp. $\gU^*$). Furthermore,\, $\gU$ and\, $\gU^*$ induce a stable superequivalence of Morita type between $\cO \Ni \fid$ and\, $\cO \Gi \bi$.
\item The $(\cO \Gi \bi,\cO \Ni \fid)$-bisupermodule\, 
$
\Omega_{\cO \Gi \bi \otimes (\cO \Ni \fid)^{\sop}}^\ell(\bi\cO \Gi \fid)
$ 
induces a Morita superequivalence between $\cO \Ni \fid$ and\, $\cO \Gi \bi$. Moreover, we can choose the labeling of\, $\Irr(\cO \Ni \fid)$ and $\Irr(\cO \Gi \bi)$ such that the corresponding bijection of irreducible characters is given by
\begin{align*}
\xi_{\rho,0} \mapsto \xi_{\rho^0}, 
 \qquad 
 \xi_{\rho,i}^\pm \mapsto \xi_{\rho^j}^\pm\ \ (1\leq j \leq \ell),
\end{align*}
if $\rho$ is even, and
\begin{align*}
 \xi_{\rho,0}^\pm \mapsto \xi_{\rho^0}^\pm,
\qquad
\xi_{\rho,i} \mapsto \xi_{\rho^j}\ \ (1\leq j \leq \ell), 
\end{align*}
if $\rho$ is odd.
\item If $\rho$ is even then the bisupermodules $\bi\cO \Gi \fid$ and $\fid \cO \Gi \bi$ are absolutely indecomposable.
\end{enumerate}
\end{Lemma}

\begin{proof}
By Lemma \ref{lem:Brauer_G_N}, $\cO \Gi \bi$ and $\cO \Ni \fid$ are Brauer correspondents. Since $\fid \in \cO \Ni_{\0}$ and $\bi \in \cO \Gi_{\0}$, using Lemma \ref{lem:Brauer_Green}, we may set $\gU$ to be the common super Green correspondent of $\cO \Ni \fid$ and $\cO \Gi \bi$ in $\Gi  \times \Ni $. In particular, $\gU$ is isomorphic to the unique non-projective, absolutely indecomposable summand of $\bi\cO \Gi \fid$.

Now, by Lemma \ref{lem:sup_alg_idempt_dual}, $(\bi\cO \Gi \fid)^* \simeq \fid\cO \Gi \bi$ and so, with Lemma \ref{lem:sup_vert} in mind, $\gU^*$ is isomorphic to the the unique non-projective, absolutely indecomposable summand of $\fid\cO \Gi \bi$. In particular, $\gU^*$ is the super Green correspondent of both $\cO \Ni \fid$ and $\cO \Gi \bi$ in $\Ni  \times \Gi $.

By Theorem \ref{thm:super_Green} and the fact that $\gU \fid = \gU$, $\cO \Gi \bi$ is isomorphic to the unique non-projective, absolutely indecomposable summand of $\gU \otimes_{\cO \Ni \fid} \fid\cO \Gi \bi$. Since $\gU^*$ is isomorphic to the unique non-projective, absolutely indecomposable summand of $\fid \cO \Gi \bi$, Lemma \ref{lem:bimod_vert}(ii) gives that $\cO \Gi \bi$ is isomorphic to the unique non-projective, absolutely indecomposable summand of $\gU \otimes_{\cO \Ni \fid} \gU^*$. So 
$ 
\gU \otimes_{\cO \Ni \fid} \gU^* \simeq \cO \Gi \bi \oplus \gP,
$ 
for some projective $(\cO \Gi \bi,\cO \Gi \bi)$-bisupermodule $\gP$. Similarly,
$
\gU^* \otimes_{\cO \Gi \bi} \gU \simeq \cO \Ni \fid \oplus \gQ,
$
for some projective $(\cO \Ni \fid,\cO \Ni \fid)$-bisupermodule $\gQ$. We have now shown part (i).

Until further notice we assume $\rho$ is even. We address the case where $\rho$ is odd at the end of this proof. By Lemma \ref{lem:weight1_equiv}(i),(iii),(iv), $\cO \Gi \bi$ and $\cO \Ni \fid$ are both Morita equivalent to $\Zag_{\ell}$, introduced in $\S$\ref{sec:Brauer_trees}. Moreover, we may assume the corresponding bijections of characters are given by
\begin{align}\label{algn:chars_ONf_Bl}
\Irr(\cO \Ni \fid) \to \Irr(\Zag_{\ell}), \quad 
\xi_{\rho,0} \mapsto \chi_0,\quad 
\xi_{\rho,j}^{\pm} \mapsto \chi_{j^{\pm}}\ \ (1 \leq j \leq \ell).
\end{align}
We use these Morita equivalences with the above assumed bijections on characters at multiple stages throughout the remainder of the proof.

Note that
\begin{align*}
\Br_\Di(\fid) = \Br_\Di(\bi) = {\bar \ei}_{\rho,0} \in Z(\F C_\Gi (\Di)),
\end{align*}
where the first equality follows from Lemma \ref{lem:Brauer_G_N} and the second from Theorem \ref{thm:DCblocks}(iv). Furthermore, ${\bar \ei}_{\rho,0}$ is primitive in $Z(\F C_\Gi (\Di))$, as, by Lemma \ref{lem:CGD}(i), we have 
$
\F C_\Gi (\Di){\bar \ei}_{\rho,0} \cong \F \tSi_\sR {\bar \ei}_{\rho,0} \otimes \F \Di,
$ 
where $\F \Di$ is totally even. Now, by Lemma \ref{lem:CGD}(ii),
\begin{align*}
N_\Gi (\Di,{\bar \ei}_{\rho,0})/C_\Gi (\Di) = N_\Ni (\Di,{\bar \ei}_{\rho,0})/C_\Ni (\Di) \cong N_{\tSi_{\sJ}}(\Di)/C_{\tSi_{\sJ}}(\Di) \cong C_{2\ell}.
\end{align*}
Therefore, by \cite[5.2,5.4]{Lin4} and \cite[Proposition 5.1]{Lin2}, there exists $0\leq m < 4\ell$ such that
\begin{align*}
\gU \simeq \Omega^m_{\cO \Gi \bi \otimes (\cO \Ni \fid)^{\sop}}(\gV),
\end{align*}
for some $(\cO \Gi \bi,\cO \Ni \fid)$-bimodule $\gV$ inducing a Morita equivalence between $\cO \Ni \fid$ and $\cO \Gi \bi$. Also
\begin{align*}
\Omega^{4\ell}_{\cO \Gi \bi \otimes (\cO \Ni \fid)^{\sop}}(\gV) \simeq \gV.
\end{align*}
(Note we need to know that $\cO \Ni \fid$ and $\cO \Gi \bi$ are already Morita equivalent to be able to apply \cite[Proposition 5.1]{Lin2} as this result is concerning stable and Morita auto-equivalences.) 
Via the discussion in $\S$\ref{sec:proj_sup_mod}, Lemma \ref{lem:HomAM_M*} tells us that $\gV$ and $\gV^*$ even induce a Morita superequivalence between $\cO \Ni \fid$ and $\cO \Gi \bi$. Therefore,
\begin{align}
\begin{split}\label{algn:V*otimesU}
\gV^* \otimes_{\cO \Gi \bi} \gU & \simeq \gV^* \otimes_{\cO \Gi \bi} \Omega^m_{\cO \Gi \bi \otimes (\cO \Ni \fid)^{\sop}}(\gV) \simeq \Omega^m_{\cO \Ni \fid \otimes (\cO \Ni \fid)^{\sop}}(\gV^* \otimes_{\cO \Gi \bi} \gV) \\
&\simeq \Omega^m_{\cO \Ni \fid \otimes (\cO \Ni \fid)^{\sop}}(\cO \Ni \fid).
\end{split}
\end{align}

By Lemma \ref{lem:weight1_equiv}(i),(iii),(iv) and using the bijection of characters in (\ref{algn:chars_ONf_Bl}), we may assume that
\begin{align}\label{algn:V_on_char}
\gV \otimes_{\cO \Ni \fid} \xi_{\rho,0} = \xi_{\rho^0},\quad
\gV \otimes_{\cO \Ni \fid} \xi_{\rho,j}^\pm = \xi_{\rho^j}^\pm\ \ (1 \leq j \leq \ell).
\end{align}
(One can think of this as fixing the labeling of the characters of $\cO \Gi \bi$ given that of $\cO \Ni \fid$.)

Now, by Lemmas \ref{lem:HtoG} and \ref{lem:chars_bisupmod}(i), we have for all $i\in I$
\begin{align}\label{algn:bOGf_char}
\bi\cO \Gi \fid \otimes_{\cO \Ni \fid} \xi_{\rho ,i}^{(\pm)} = \bi\cO \Gi \ci \otimes_{\cO \Hi \ci} \xi_{\rho,i}^{(\pm)} =
\sum_{j=0}^{\ell-i} \xi_{\rho^j}.
\end{align}
(Here and for the remainder of the proof, $\xi_{\rho,i}^{(\pm)}$ denotes $\xi_{\rho,i}^{\pm}$ if $i\neq 0$ and $\xi_{\rho,0}$ if $i=0$.) 
In particular,
\begin{align*}
\bi\cO \Gi \fid \otimes_{\cO \Ni \fid} \xi_{\rho,\ell}^+ = \bi\cO \Gi \fid \otimes_{\cO \Ni \fid} \xi_{\rho,\ell}^- = \xi_{\rho^0}.
\end{align*}
As $\gU$ is isomorphic to the unique non-projective, absolutely indecomposable summand of $\bi\cO \Gi \fid$ and $\xi_{\rho^0} \notin \N \prj(\cO \Gi \bi)$, Lemma \ref{lem:bimod_vert}(i) now implies that
\begin{align*}
\gU \otimes_{\cO \Ni \fid} \xi_{\rho,\ell}^+ = \xi_{\rho^0}.
\end{align*}
Therefore, using (\ref{algn:V*otimesU}) and (\ref{algn:V_on_char}),
\begin{align*}
\Omega^m_{\cO \Ni \fid \otimes (\cO \Ni \fid)^{\sop}}(\cO \Ni \fid) \otimes_{\cO \Ni \fid} \xi_{\rho,\ell}^+ = \gV^* \otimes_{\cO \Gi \bi} \gU \otimes_{\cO \Ni \fid} \xi_{\rho,\ell}^+ = \xi_{\rho,0}.
\end{align*}
Since, $\cO \Ni \fid$ is Morita equivalent to $\Zag_{\ell}$ with the bijection of characters given in (\ref{algn:chars_ONf_Bl}), Lemma \ref{lem:Zig_Zag}(i) now implies that $m= \ell$ or $3\ell$. We can, in fact, take either $m = \ell$ or $3\ell$, since implicit in our application of \cite[Proposition 5.1]{Lin2} is that, we have a bijection of irreducible Brauer characters, $\IBr(\cO \Ni \fid) \to \IBr(\cO \Gi \bi)$, induced by some Morita equivalence. Then $\gV$ and $m$ are unique with respect to the condition that $\gV$ must induce this same bijection $\IBr(\cO \Ni \fid) \to \IBr(\cO \Gi \bi)$. By the structure of the Brauer tree, there are clearly two possible such bijections and hence we can take $m = \ell$ or $3\ell$. We set $m = 3\ell$. Now,
\begin{align*}
\gV \simeq \Omega_{\cO \Gi \bi \otimes (\cO \Ni \fid)^{\sop}}^{4\ell}(\gV) \simeq \Omega_{\cO \Gi \bi \otimes (\cO \Ni \fid)^{\sop}}^\ell(\gU) \simeq \Omega_{\cO \Gi \bi \otimes (\cO \Ni \fid)^{\sop}}^\ell(\bi\cO \Gi \fid).
\end{align*}
proving part (ii).

To show that $\bi\cO \Gi \fid$ is absolutely indecomposable it is enough to show that $\gU \simeq \bi\cO \Gi \fid$ which, in turn, follows once we have proved that
\begin{align*}
\gU \otimes_{\cO \Ni \fid} \xi_{\rho,i}^{(\pm)} = \bi\cO \Gi \fid \otimes_{\cO \Ni \fid} \xi_{\rho,i}^{(\pm)} =
\sum_{j=0}^{\ell-i} \xi_{\rho^j},
\end{align*}
for all $i\in I$. By (\ref{algn:V_on_char}), this is equivalent to
\begin{align}\label{algn:V*_ten_U}
\gV^* \otimes_{\cO \Gi \bi} \gU \otimes_{\cO \Ni \fid} \xi_{\rho ,i}^{(\pm)} = \sum_{j=0}^{\ell-i} \xi_{\rho,j},
\end{align}
for all $i\in I$. Certainly $\gU \oplus \gR \simeq \bi\cO \Gi \fid$, for some projective $(\cO \Gi \bi,\cO \Ni \fid)$-bisupermodule $\gR$. Therefore, by Lemma \ref{lem:bimod_vert}(i), (\ref{algn:bOGf_char}) and (\ref{algn:V_on_char}), we have 
\begin{align*}
\gV^* \otimes_{\cO \Gi \bi} \gU \otimes_{\cO \Ni \fid} \xi_{\rho,i}^{(\pm)}  \leq_{{}_{\prj(\cO \Ni \fid)}} \gV^* \otimes_{\cO \Gi \bi} \bi\cO \Gi \fid \otimes_{\cO \Ni \fid} \xi_{\rho,i}^{(\pm)}  = \sum_{j=0}^{\ell-i} \xi_{\rho,j},
\end{align*}
for all $i\in I$. However, by (\ref{algn:V*otimesU}) and Lemma \ref{lem:Zig_Zag}(i), we also know that
\begin{align*}
\gV^* \otimes_{\cO \Gi \bi} \gU \otimes_{\cO \Ni \fid} \xi_{\rho,i}^{(\pm)}  = \Omega^{3\ell}_{\cO \Ni \fid \otimes (\cO \Ni \fid)^{\sop}}(\cO \Ni \fid) \otimes_{\cO \Ni \fid} \xi_{\rho,i}^{(\pm)} 
\ \geq_{{}_{\prj(\cO \Ni \fid)}}\ \xi_{\rho,\ell-i}^+,
 % \quad \text{and} \quad 
 \xi_{\rho,\ell-i}^-,
\end{align*}
for all $i\in I$. Therefore, claim (\ref{algn:V*_ten_U}) will follow once we have proved that, for each $i\in I$, it is not possible to subtract a non-zero $\chi \in \N \prj(\cO \Ni \fid)$ from $\sum_{j=0}^{\ell-i} \xi_{\rho,j}$ and satisfy
\begin{align*}
\sum_{j=0}^{\ell-i} \xi_{\rho,j} - \chi\ \geq_{{}_{\prj(\cO \Ni \fid)}}\  \xi_{\rho,\ell-i}^+, \xi_{\rho,\ell-i}^-.
\end{align*}
For a contradiction, suppose such a $\chi$ does exist. Now, by inspecting the Brauer tree of $\Zag_{\ell}$ and using the bijection of characters in (\ref{algn:chars_ONf_Bl}), $\N \prj(\cO \Ni \fid)$ is precisely the $\N$-linear combination of characters of the form
\begin{align}\label{algn:proj_char_N}
\xi_{\rho,k}^{(\pm)} + \xi_{\rho,k+1}^{(\pm)},
\end{align}
for $0\leq k\leq \ell-1$. Say $\xi_{\rho,k}^+ + \xi_{\rho,k+1}^+$ appears with positive coefficient in $\chi$ when expressed as a linear combination of the characters in (\ref{algn:proj_char_N}), for some $0 \leq k \leq \ell - i - 1$. (The argument for $\xi_{\rho,k}^- + \xi_{\rho,k+1}^-$ is completely analogous.) Then we must have
\begin{align*}
\sum_{j=0}^{\ell-i} \xi_{\rho,j} - (\xi_{\rho,k}^+ + \xi_{\rho,k+1}^+) - \xi_{\rho,\ell-i}^+ \in \N \prj(\cO \Ni \fid)
\end{align*}
and
\begin{align*}
\sum_{j=0}^{\ell-i} \xi_{\rho,j} - (\xi_{\rho,k}^+ + \xi_{\rho,k+1}^+) - \xi_{\rho,\ell-i}^- \in \N \prj(\cO \Ni \fid).
\end{align*}
The first expression immediately rules out $k = \ell - i - 1$, as otherwise we have $\xi_{\rho,\ell-i}^+$ appearing with negative coefficient. In all other cases, one of the two above expressions is the sum of the characters corresponding to two disconnected paths in the Brauer trees, both of odd length. (If $k \equiv \ell - i$ modulo $2$ it is the first expression and if $k \not\equiv \ell - i$ modulo $2$ it is the second expression.) A sum of such characters can never be in $\N \prj(\cO \Ni \fid)$. We have now reached our desired contradiction and proved part (iii).

We now outline the proof for when $\rho$ is odd. (Note that, in this case, we are not yet claiming that $\bi\cO \Gi \fid$ is indecomposable.) This time, by Lemma \ref{lem:weight1_equiv}(ii),(iii),(v), $\cO \Gi \bi$ and $\cO \Ni \fid$ are both Morita equivalent to $\Zig_{\ell}$, introduced in $\S$\ref{sec:Brauer_trees}. We have
\begin{align*}
\Br_\Di(\fid) = \Br_\Di(\bi) = {\bar \ei}_{\rho,0} = {\bar \ei}_{\rho,0}^+ + {\bar \ei}_{\rho,0}^- \in \F C_\Gi (\Di)
\end{align*}
and
\begin{align*}
N_\Gi (\Di,{\bar \ei}_{\rho,0}^+)/C_\Gi (\Di) = N_\Ni (\Di,{\bar \ei}_{\rho,0}^+)/C_\Ni (\Di) \cong N_{\tAi_{\sJ}}(\Di)/C_{\tAi_{\sJ}}(\Di) \cong C_{\ell}.
\end{align*}
Once again, we use \cite[5.2,5.4]{Lin4} and \cite[Proposition 5.1]{Lin2} to prove that there exists $0\leq m < 2\ell$ such that 
$
\gU \simeq \Omega^m_{\cO \Gi \bi \otimes (\cO \Ni \fid)^{\sop}}(\gV),
$ 
for some $(\cO \Gi \bi,\cO \Ni \fid)$-bimodule $\gV$ inducing a Morita equivalence between $\cO \Ni \fid$ and $\cO \Gi \bi$. Also, 
$
\Omega^{2\ell}_{\cO \Gi \bi \otimes (\cO \Ni \fid)^{\sop}}(\gV) \simeq \gV.
$ 
As in the even case, we can choose the labeling of $\Irr(\cO \Ni \fid)$ and $\Irr(\cO \Gi \bi)$ to ensure that $\gV$ induces the desired bijection $\Irr(\cO \Ni \fid) \to \Irr(\cO \Gi \bi)$. We use Lemma \ref{lem:Zig_Zag}(ii) to prove that $m = \ell$. Therefore,
\begin{align*}
\gV \simeq \Omega_{\cO \Gi \bi \otimes (\cO \Ni \fid)^{\sop}}^{2\ell}(\gV) \simeq \Omega_{\cO \Gi \bi \otimes (\cO \Ni \fid)^{\sop}}^\ell(\gU) \simeq \Omega_{\cO \Gi \bi \otimes (\cO \Ni \fid)^{\sop}}^\ell(\bi\cO \Gi \fid),
\end{align*}
as desired.
\end{proof}

\subsection{The bisupermodule \texorpdfstring{$\bM$}{}}
\label{SSbM}

Throughout the remainder of the article we set $\bM$ to be the $(\Blo^{\varnothing,1},\Blo^{\varnothing,1})$-bisupermodule
\begin{equation}\label{EM}
\bM := \Omega_{\Blo^{\varnothing,1} \otimes (\Blo^{\varnothing,1})^{\sop}}^{\ell}(\Blo^{\varnothing,1}).
\end{equation}
It is well known that $\bM$ is absolutely indecomposable. For example, one can apply \cite[5.2,5.4]{Lin4} and \cite[Propostion 5.1]{Lin2} to show that $\Omega_{\Blo^{\varnothing,1} \otimes (\Blo^{\varnothing,1})^{\sop}}^{4\ell}(\Blo^{\varnothing,1}) \simeq \Blo^{\varnothing,1}$ and use that $\Omega$ commutes with direct sums. We will also need the fact that
\begin{align}\label{algn:M0_Heller}
\bM_{\0} \simeq \Omega_{\Blo^{\varnothing,1}_{\0} \otimes (\Blo^{\varnothing,1}_{\0})^{\sop}}^{\ell}(\Blo^{\varnothing,1}_{\0}),
\end{align}
as $(\Blo^{\varnothing,1}_{\0}, \Blo^{\varnothing,1}_{\0})$-bimodules and that $\bM_{\0}$ is indecomposable, again as a $(\Blo^{\varnothing,1}_{\0}, \Blo^{\varnothing,1}_{\0})$-bimodule. First note that, as a consequence of Corollary~\ref{COm0},
\begin{align}\label{algn:M0_Heller2}
\bM_{\0} \simeq \Omega_{(\Blo^{\varnothing,1} \otimes (\Blo^{\varnothing,1})^{\sop})_{\0}}^\ell(\Blo^{\varnothing,1}_{\0}),
\end{align}
as $(\Blo^{\varnothing,1} \otimes (\Blo^{\varnothing,1})^{\sop})_{\0}$-modules. Next, using the notation of (\ref{algn:marcus_not}), we have that $p\nmid[\tAi_p \times \tAi_p:(\tSi_p \times \tSi_p)_{\tSi_p/\tAi_p}]=2$. Therefore, $\Res^{(\tSi_p \times \tSi_p)_{\tSi_p/\tAi_p}}_{\tAi_p \times \tAi_p}P$ is projective, for any projective $(\Blo^{\varnothing,1} \otimes (\Blo^{\varnothing,1})^{\sop})_{\0}$-module $P$. Moreover, for any indecomposable $(\Blo^{\varnothing,1} \otimes (\Blo^{\varnothing,1})^{\sop})_{\0}$-module $\gN$, if $\Res^{(\tSi_p \times \tSi_p)_{\tSi_p/\tAi_p}}_{\tAi_p \times \tAi_p}(\gN)$ has any non-zero, projective, indecomposable summand, then $\gN$ is projective. Now (\ref{algn:M0_Heller}) just follows from (\ref{algn:M0_Heller2}). That $\bM_{\0}$ is indecomposable as an $(\Blo^{\varnothing,1}_{\0}, \Blo^{\varnothing,1}_{\0})$-bimodule is now just proved in the same way the analogous statement for $\bM$ was proved above.

\begin{Proposition}\label{prop:define_M}
The bisupermodule $\bM$ satisfies the following properties:
\begin{enumerate}
\item $\bM$ is absolutely indecomposable with vertex $\Delta \Di$ and source $\Omega_{\cO \Delta \Di}^{\ell}(\cO)$. In particular, $\bM$ has endopermutation source.
\item The $(\cO \Ni \fid,\cO \Ni \fid)$-bisupermodule $\Blo^{\rho,0} \boxtimes \bM$ is absolutely indecomposable with vertex $\Delta \Di$. Moroever, $\bi\cO \Gi  \fid \otimes_{\cO \Ni \fid} (\Blo^{\rho,0} \boxtimes \bM)$ has a unique non-projective, absolutely indecomposable summand $\gV$, and this summand induces a Morita superequivalence between $\cO \Ni \fid$ and $\cO \Gi \bi$. Furthermore, for all $i \in I$, we have 
$
\gV \otimes_{\cO \Ni \fid} \xi_{\rho,i} = \xi_{\rho^i}.
$
\item $\bM$ and $\bM^*$ induce a stable auto-superequivalence of Morita type of $\Blo^{\varnothing,1}$.
\item For all $i\in I$, we have 
$$
\bM \otimes_{\Blo^{\varnothing,1}} \xi_i = \bM^* \otimes_{\Blo^{\varnothing,1}} \xi_i=\eps_{i}^2\sum_{j=0}^{\ell-i}  \xi_j.
$$
\end{enumerate}
\end{Proposition}

\begin{proof}
(i) Since $\Blo^{\varnothing,1}$ has vertex $\Delta \Di$ and trivial source, $\bM$ also has vertex $\Delta \Di$ with source isomorphic to $\Omega_{\cO \Delta \Di}^{\ell}(\cO)$. It now follows from \cite[Proposition 7.3.4]{Lin6} that $\bM$ has endopermutation source.

(ii) We first prove that $\Blo^{\rho,0} \boxtimes \bM$ is absolutely indecomposable with vertex $\Delta \Di$. For $r=1$, $\cO \Ni \fid \cong \Blo^{\varnothing,1}$ and we just apply part (i).

For $\rho$ even with $r>1$, by Theorem \ref{thm:DCblocks}, $\Blo^{\rho,0}$ is an absolutely indecomposable $(\Blo^{\rho,0},\Blo^{\rho,0})$-bisupermodule and so the claim follows from Remark \ref{rem:bisupmod} and Lemma \ref{lem:direct_prod_vert}.

For $\rho$ odd, we look more carefully at the proof of Lemma \ref{lem:weight1_equiv}(ii). If we set $e \in \Blo^{\rho,0}$ to be a primitive idempotent, then $e\sigma_{\cO\tSi_\sR }(e)=0$ and $e\cO \Ni \fid \otimes_{\cO \Ni \fid} ?$ induces a Morita equivalence between $\cO \Ni \fid$ and $\Blo^{\varnothing,1}_{\0}$. Therefore, to show indecomposability it is enough to show that $e(\Blo^{\rho,0} \boxtimes \bM)e$ is absolutely indecomposable as a $(\Blo^{\varnothing,1}_{\0},\Blo^{\varnothing,1}_{\0})$-bimodule. Now, similarly to the proof of Lemma \ref{lem:weight1_equiv}(ii),
\begin{align*}
\bM_{\0} \to e(\Blo^{\rho,0} \boxtimes \bM)e, \ 
m  \mapsto e \otimes m
\end{align*}
is an isomorphism of $(\Blo^{\varnothing,1}_\0,\Blo^{\varnothing,1}_{\0})$-bimodules and $\bM_{\0}$ is an indecomposable $(\Blo^{\varnothing,1}_{\0},\Blo^{\varnothing,1}_{\0})$-bimodule, as noted in the comments preceding the Proposition.

We now prove that $\Blo^{\rho,0} \boxtimes \bM$ has vertex $\Delta \Di$. By Theorem \ref{thm:An_blocks}, $\Blo^{\rho,0}_{\0}$ has trivial defect and, by part (i), $\bM$ has vertex $\Delta \Di$. Therefore, by Remark \ref{rem:bisupmod} and Lemma \ref{lem:direct_prod_vert}, $\Blo^{\rho,0}_{\0} \boxtimes \bM$ has vertex $\Delta \Di$, as a $(\Blo^{\rho,0}_{\0} \otimes \Blo^{\varnothing,1},\Blo^{\rho,0}_{\0} \otimes \Blo^{\varnothing,1})$-bimodule. The claim follows.

By Lemma \ref{Omega_bOGf_Morita}(i), $\bi\cO \Gi  \fid$ induces a stable superequivalence of Morita type between $\cO \Ni \fid$ and $\cO \Gi \bi$. Therefore, $\bi\cO \Gi  \fid  \otimes_{\cO \Ni \fid} (\Blo^{\rho,0} \boxtimes \bM)$ has a unique non-projective, absolutely indecomposable summand $\gV$. Furthermore,
\begin{align*}
\bi\cO \Gi  \fid  \otimes_{\cO \Ni \fid} (\Blo^{\rho,0} \boxtimes \bM) \simeq & \bi\cO \Gi  \fid  \otimes_{\cO \Ni \fid} (\Blo^{\rho,0} \boxtimes \Omega_{\Blo^{\varnothing,1} \otimes (\Blo^{\varnothing,1})^{\sop}}^{\ell}(\Blo^{\varnothing,1})) \\
\simeq & \bi\cO \Gi  \fid  \otimes_{\cO \Ni \fid} \Omega_{\cO \Ni \fid \otimes (\cO \Ni \fid)^{\sop}}^{\ell}(\Blo^{\rho,0} \boxtimes \Blo^{\varnothing,1}) \\
\simeq & \bi\cO \Gi  \fid  \otimes_{\cO \Ni \fid} \Omega_{\cO \Ni \fid \otimes (\cO \Ni \fid)^{\sop}}^{\ell}(\cO \Ni \fid) \\
\simeq & \Omega_{\cO \Gi \bi \otimes (\cO \Ni \fid)^{\sop}}^{\ell}(\bi\cO \Gi  \fid  \otimes_{\cO \Ni \fid} \cO \Ni \fid) \oplus \gP \\
\simeq & \Omega_{\cO \Gi \bi \otimes (\cO \Ni \fid)^{\sop}}^{\ell}(\bi\cO \Gi  \fid) \oplus \gP,
\end{align*}
for some projective $(\cO \Gi \bi, \cO \Ni \fid)$-bisupermodule $\gP$. Here, the second isomorphism follows from Lemma \ref{lem:super_Heller} (unless $r=1$, in which case it is actually an equality), as 
$
\Omega_{\Blo^{\rho,0} \otimes (\Blo^{\rho,0})^{\sop}}(\Blo^{\rho,0}) = \{0\},
$ 
and the fourth isomorphism from the fact that $\bi\cO \Gi  \fid$ induces a stable superequivalence between $\cO \Ni \fid$ and $\cO \Gi \bi$. So, $\gV \simeq \Omega_{\cO \Gi \bi \otimes (\cO \Ni \fid)^{\sop}}^{\ell}(\bi\cO \Gi  \fid)$ and all the remaining results now follow from Lemma \ref{Omega_bOGf_Morita}(ii).

(iii) We now assume $\rho$ is even with $r>1$. Due to the comment following Definition \ref{def_Rou}, this is always possible. By Lemma \ref{Omega_bOGf_Morita}(i),(iii), $\bi\cO \Gi \fid$ and $\fid\cO \Gi \bi$ are both absolutely indecomposable and induce a stable superequivalence of Morita type between $\cO \Ni \fid$ and $\cO \Gi \bi$. Therefore,
\begin{align*}
\fid\cO \Gi  \bi \otimes_{\cO \Gi \bi} \bi\cO \Gi  \fid  \otimes_{\cO \Ni \fid} (\Blo^{\rho,0} \boxtimes \bM)
\end{align*}
has a unique non-projective, absolutely indecomposable summand isomorphic to $\Blo^{\rho,0} \boxtimes \bM$. However, $\gV$ from part (ii) is the unique non-projective, absolutely indecomposable summand of $\bi\cO \Gi  \fid  \otimes_{\cO \Ni \fid} (\Blo^{\rho,0} \boxtimes \bM)$ and so
\begin{align}\label{algn:BM_fGbV}
\Blo^{\rho,0} \boxtimes \bM \simeq \fid\cO \Gi  \bi \otimes_{\cO \Gi \bi} \gV.
\end{align}
Now,
$$
\Blo^{\rho,0} \boxtimes \bM^* \simeq (\Blo^{\rho,0})^* \boxtimes \bM^* \simeq (\Blo^{\rho,0} \boxtimes \bM)^* \simeq \gV^* \otimes_{\cO \Gi \bi} \bi\cO \Gi \fid,
$$
where the first isomorphism follows from Lemma \ref{lem:sup_alg_idempt_dual}, the second from Lemma \ref{lem:dual_prod} and the third from Lemmas \ref{lem:dual_comp} and \ref{lem:sup_alg_idempt_dual}. Furthermore, since $\gV$ induces a Morita superequivalence between $\cO \Ni \fid$ and $\cO \Gi \bi$, Lemmas \ref{lem:HomAM_M*} and \ref{Omega_bOGf_Morita}(i),(iii) give that $\Blo^{\rho,0} \boxtimes \bM$ and $\Blo^{\rho,0} \boxtimes \bM^*$ induce a stable auto-superequivalence of Morita type of $\cO \Ni \fid$.

We now set $e$ to be a primitive idempotent in $\Blo^{\rho,0}_{\0}$, as in the proof of Lemma \ref{lem:weight1_equiv}(i), where we showed that $e\cO \Ni \fid$ and $\fid\cO \Ni e$ induce a Morita superequivalence between $\cO \Ni \fid$ and $\Blo^{\varnothing,1}$. We now have that
$$
e \cO \Ni \fid \otimes_{\cO \Ni \fid} (\Blo^{\rho,0} \boxtimes \bM) \otimes_{\cO \Ni \fid} \fid \cO \Ni e \simeq e(\Blo^{\rho,0} \boxtimes \bM)e \simeq \bM
$$
and
$$
e \cO \Ni \fid \otimes_{\cO \Ni \fid} (\Blo^{\rho,0} \boxtimes \bM^*) \otimes_{\cO \Ni \fid} \fid \cO \Ni e \simeq e(\Blo^{\rho,0} \boxtimes \bM^*)e \simeq \bM^*
$$
induce a stable auto-superequivalence of Morita type of $\Blo^{\varnothing,1}$, as desired.

(iv) We continue to assume that $\rho$ is even with $r>1$. We have already seen in part (iii) that $\fid \cO \Gi  \bi \otimes_{\cO \Gi \bi} \gV \simeq \Blo^{\rho,0} \boxtimes \bM$. Now,
\begin{align*}
\fid \cO \Gi  \bi \otimes_{\cO \Gi \bi} \gV \otimes_{\cO \Ni \fid} \xi_{\rho,i} = \fid \cO \Gi  \bi \otimes_{\cO \Gi \bi} \xi_{\rho^i} = \eps_{i}^2\sum_{j=0}^{\ell-i}  \xi_{\rho,j},
\end{align*}
for all $i\in I$, where the first equality holds by part (ii) and the second due to Lemma \ref{lem:HtoG}(ii). (Note that $\eps_{\rho^i} = \eps_{i}$, see Lemma \ref{lem:weight1_equiv}(iii),(iv).) Therefore,
\begin{align*}
(\Blo^{\varnothing,1} \boxtimes \bM) \otimes_{\cO \Ni \fid} \xi_{\rho,i} = \eps_{i}^2\sum_{j=0}^{\ell-i}  \xi_{\rho,j},
\end{align*}
for all $i\in I$, and the claim follows from Lemma \ref{lem:AboxM}.

The claim for $\bM^*$ follows from Lemma \ref{lem:chars_bisupmod}(ii).
\end{proof}

\section{The bisupermodules $\bX$ and $\bY$}\label{sec:X_Y}

Throughout this section we set $d$ to be an integer with $1 \leq d < p$ and $\rho$ a $d$-Rouquier ${\bar p}$-core. 
%In particular, we are no longer assuming $d=1$, as in Section \ref{sec:weight_one}.  
We adopt all the notation of $\S$\ref{sec:notation}, in particular, $r=|\rho|$, $n=r+dp$, 
\begin{align*}
&\Gi=\tSi_n,\quad 
\Li= \tSi_{\sR,\sJ_1,\dots,\sJ_d},
\quad \Ni= N_{\Gi}(\tSi_{\sJ_1,\dots,\sJ_d}),\quad   
\Hi= \tSi_{\sR\cup \sJ_1\cup\dots\cup \sJ_{d-1}, \sJ_{d}},
\\
&\bi=\ei_{\rho,d}\in\cO\Gi,\quad  \fid= \ei_{\rho,0} \otimes \ei_{\varnothing,1}^{(1)} \otimes \dots \otimes \ei_{\varnothing,1}^{(d)} \in \cO \Li e_z,
\quad
\ci= \ei_{\rho,d-1} \otimes \ei_{\varnothing,1}^{(d)} \in \cO \Hi e_z,
\end{align*}
and the defect group 
$$\Di=\Di_1\times\dots\times \Di_d$$ is chosen as in (\ref{ED}). 
We identify $\cO \Li \fid$ with $\Blo^{\rho,0} \otimes (\Blo^{\varnothing,1})^{\otimes d}$ and $\cO \Ni \fid$ with $\Blo^{\rho,0} \otimes (\Blo^{\varnothing,1} \swr \cT_d)$ as in (\ref{algn:ONf_ident}) via Lemma \ref{lem:kappa}. We continue with all our assumptions on $\cO$ from Section~\ref{sec:spin_blocks}.

\subsection{$\Di$-small subgroups}
\label{SSDSmall}
Recall that $\Gi $ acts on $[n]$ via $\pi_n$, see \S\ref{sec:gen_tSi}. In particular, $\sR $ is the set of fixed points of $\Di $. % acting on $[n]$ via $\pi_n$. 
For $Q \leq \Di $, we write $Q <_\sfs\Di$ if the set of fixed points of $Q$ on $[n]$ strictly contains $\sR $. In other words, $Q <_\sfs\Di$ if and only $Q\leq \Di_1\times \dots \times \hat{\Di }_k \times \dots \times \Di_d$, for some $1\leq k\leq d$. (Here, $\hat{\Di }_k$ means that $\Di_k$ is omitted from the direct product.)

Recall the notation (\ref{EDePhi}). 

\begin{Lemma}\label{lem:conj_def}
We have: 
\begin{enumerate}
\item Let $g\in \Gi  \setminus \Ni $. Then $\Di  \cap {}^g \Di   <_\sfs\Di$. In particular, for any $(g_1,g_2)\in (\Gi \times \Gi ) \setminus (\Ni  \times \Ni )$, we have $\Delta \Di  \cap {}^{(g_1,g_2)} \Delta \Di  = \Delta Q$, for some $Q <_\sfs\Di$.
\item Let $(g,h) \in (\Gi  \times \Ni ) \setminus (\Ni  \times \Ni )$. Then $(\Di  \times \Di ) \cap {}^{(g,h)}\Delta \Di = \Delta_\varphi Q$, for some $Q <_\sfs\Di$ and $\varphi:Q \to \Di $, with $\varphi(Q) <_\sfs\Di$.
\end{enumerate}
\end{Lemma}

\begin{proof}
(i) Let $g\in \Gi  \setminus \Ni $. Then, by Lemma \ref{lem:CGD}(ii), we have $g \notin N_\Gi (\Di )$, so  ${}^{g^{-1}} \Di_k \not\subseteq \Di $ for some $1\leq k\leq d$. 
%Without loss of generality we assume ${}^{g^{-1}} \Di_1 \not\subseteq \Di $. %, where $h=(1\dots p)$.
We claim that ${}^g \Di  \cap \Di $ fixes $\sJ_c$ pointwise. Say $\al \in {}^g \Di  \cap \Di $ does not fix $\sJ_k$ pointwise. Then, since $\al \in \Di $, we can write $\al = h_1 h$, for some $h_1 \in \Di_k \setminus \{1\}$ and $h \in \Di_1 \times \dots\times \hat \Di_k\times\dots \times \Di_d$. As $h_1 h \in {}^g \Di $, we have ${}^{g^{-1}} h_1 {}^{g^{-1}} h \in \Di $. Therefore, ${}^{g^{-1}} h_1 \in \Di $, as $\pi_n({}^{g^{-1}} h_1)$ and $\pi_n({}^{g^{-1}} h)$ have disjoint cycle decomposition. This contradicts the containment ${}^{g^{-1}} \Di_k \not\subseteq \Di $.

For the second part we set 
$
Q:= \{x \in \Di \mid (x,x)\in \Delta \Di  \cap {}^{(g_1,g_2)} \Delta \Di \}.
$ 
Certainly $\Delta Q = \Delta \Di  \cap {}^{(g_1,g_2)}\Delta \Di $. To prove that $Q <_\sfs\Di$, we apply the first part to the first or second coordinate depending on whether $g_1 \notin \Ni $ or $g_2 \notin \Ni $.

(ii) Certainly $(\Di  \times \Di ) \cap {}^{(g,h)}\Delta \Di = \Delta_\varphi Q$, for some $Q \leq \Di $ and $\varphi:Q \to \Di $. We just need to show that $Q,\varphi(Q) <_\sfs\Di$. Since $g \notin \Ni $, it follows from part (i) that $Q \leq \Di  \cap {}^g \Di  <_\sfs\Di$. Now suppose $(x,y) \in (\Di  \times \Di ) \cap {}^{(g,h)}\Delta \Di  = \Delta_\varphi Q$. Then $({}^{g^{-1}} x,{}^{h^{-1}} y) \in \Delta \Di $. In particular, ${}^{g^{-1}} x={}^{h^{-1}} y$ and so ${}^{hg^{-1}} x = y$. So, since $hg^{-1} \notin \Ni $, $y \in \Di  \cap {}^{hg^{-1}} \Di <_\sfs\Di$, by part (i). As this holds for all such $y$, we have $\varphi(Q) <_\sfs\Di$, as desired.
\end{proof}

In the remainder of this section we will often encounter the following situation. Let $\Di\leq K_1,K_2\leq \Gi $, with $z \in K_1,K_2 \nleq \Gi_{\0}$, and let $\gU,\gV$ be $(\cO K_1,\cO K_2)$-bisupermodules. 

If $\gV$ is isomorphic to an absolutely indecomposable summand of $\gU$ with vertex $\Delta \Di $ and all other indecomposable summands of $\gU$, as an $(\cO K_1,\cO K_2)$-bimodule, have vertex contained in some $\Delta Q$, with $Q <_\sfs\Di$, then we write 
\begin{equation}\label{E|_D}
\gV\,|_\Di \, \gU.
\end{equation}
If instead, all other indecomposable summands of $\gU$ as an $(\cO K_1,\cO K_2)$-bimodule, have vertex contained in some $\Delta_\varphi Q$, with $Q,\phi(Q) <_\sfs\Di$ for some $\varphi: Q \to \Di $, then we write 
\begin{equation}\label{E|^D}
\gV\,|^\Di \, \gU.
\end{equation}
Moreover, we refer to such $\Delta_\varphi Q$ as {\em $\Di $-small}.

%\textcolor{red}{I'm happy to change this name `$\Di $-small'. I'm not terribly happy with it.}

\subsection{The bisupermodules \texorpdfstring{$\bM_\Li $}{} and \texorpdfstring{$\bM_\Ni $}{}}\label{sec:ML_MN}

Recall the $(\Blo^{\varnothing,1},\Blo^{\varnothing,1})$-bisupermodule~$\bM$ from (\ref{EM}). 
We have the $((\Blo^{\varnothing,1})^{\otimes d},(\Blo^{\varnothing,1})^{\otimes d})$-bisupermodule $\bM^{\boxtimes d}$.  For $1\leq k\leq d$, we set $(\Blo^{\varnothing,1})^{(k)} := \cO \tSi_{\sJ_k}\ei_{\varnothing,1}^{(k)}$ and identify each $(\Blo^{\varnothing,1})^{(k)}$ with $\Blo^{\varnothing,1}$ via (\ref{Sp_ident}), as in $\S$\ref{sec:notation}. We denote by $\bM^{(k)}$ the $k^{\nth}$ factor in $\bM^{\boxtimes d}$. That is, $\bM^{(k)}$ is a $((\Blo^{\varnothing,1})^{(k)},(\Blo^{\varnothing,1})^{(k)})$-bisupermodule that we identify with $\bM$.
We denote by $\bM^{\boxtimes d}_{\Si_d}$ the $(\Blo^{\varnothing,1} \swr \cT_d,\Blo^{\varnothing,1} \swr \cT_d)_{\Si_d}$-supermodule from Lemma \ref{lem:M_swr_Td}.

Recalling the identification $\cO \Li \fid=\Blo^{\rho,0} \otimes (\Blo^{\varnothing,1})^{\otimes d}$,  
define 
the $(\cO \Li \fid,\cO \Li \fid)$-bisupermodule
\begin{align*}
\bM_\Li  := \Blo^{\rho,0} \boxtimes \bM^{\boxtimes d}.
\end{align*}
We have also identified 
$\cO \Ni \fid$ with $\Blo^{\rho,0} \otimes (\Blo^{\varnothing,1} \swr \cT_d)$. 
Recalling (\ref{algn:sup_wreath_bimod}), we define the $(\cO \Ni \fid,\cO \Ni \fid)$-bisupermodule
\begin{align*}
\bM_\Ni  := \Blo^{\rho,0} \boxtimes (\bM \swr \cT_d)
=\Blo^{\rho,0} \boxtimes \Big(\Ind_{(\Blo^{\varnothing,1} \swr \cT_d,\Blo^{\varnothing,1} \swr \cT_d)_{\Si_d}}^{(\Blo^{\varnothing,1} \swr \cT_d) \otimes (\Blo^{\varnothing,1} \swr \cT_d)^{\sop}}\bM^{\boxtimes d}_{\Si_d}\Big).
\end{align*}

Recall the subgroup $(\Ni  \times \Ni )_{\Ni /\Li }\leq \Ni  \times \Ni$ from (\ref{algn:marcus_not}). To simplify the notation, we denote 
$$
(\Ni  \times \Ni )_{\Si_d }:=(\Ni  \times \Ni )_{\Ni /\Li }.
$$
Note that the $T_w$'s, introduced just before Lemma \ref{lem:kappa}, are group elements. Therefore, through our identification $\cO \Ni \fid=\Blo^{\rho,0} \otimes (\Blo^{\varnothing,1} \swr \cT_d)$, we have 
$$
(\Blo^{\varnothing,1} \swr \cT_d,\Blo^{\varnothing,1} \swr \cT_d)_{\Si_d} = \cO (\Ni  \times \Ni )_{\Si_d }(\fid \otimes \fid),
$$
In particular,
\begin{align}\label{algn:M_N_ext}
\bM_\Ni   
\simeq \Ind_{(\Ni  \times \Ni )_{\Si_d}}^{\Ni  \times \Ni }(\bM_\Li )_{\Si_d}.
\end{align}
where 
$$
(\bM_\Li )_{\Si_d} := \Blo^{\rho,0} \boxtimes \bM^{\boxtimes d}_{\Si_d}.
$$
We will sometimes consider $(\bM_\Li )_{\Si_d}$ as an $\cO((\Ni  \cap \Hi ) \times (\Ni  \cap \Hi ))_{(\Ni  \cap \Hi )/\Li }$-module via the inclusion $(\Ni  \cap \Hi ) \hookrightarrow \Ni $. 
In this case we set
\begin{align*}
((\Ni  \cap \Hi ) \times (\Ni  \cap \Hi ))_{\Si_{d-1}}:=((\Ni  \cap \Hi ) \times (\Ni  \cap \Hi ))_{(\Ni  \cap \Hi )/\Li }.
\end{align*}

\begin{Lemma}\label{ML_MN_indecomp}
We have:
\begin{enumerate}
\item[{\rm (i)}] $\bM_\Li $ is an absolutely indecomposable $(\cO \Li \fid,\cO \Li \fid)$-bisupermodule with vertex $\Delta \Di $
\item[{\rm (ii)}] $\bM_\Ni $ is an absolutely indecomposable $(\cO \Ni \fid,\cO \Li \fid)$-bisupermodule with vertex $\Delta \Di $. In particular, $\bM_\Ni $ is an absolutely indecomposable $(\cO \Ni \fid,\cO \Ni \fid)$-bisupermodule with vertex $\Delta \Di $. 
\item[{\rm (iii)}]  
$
\Res^{\Ni  \times \Ni }_{\Ni  \times \Li }\,\bM_\Ni  \simeq \Ind_{\Li  \times \Li }^{\Ni  \times \Li }\,\bM_\Li .
$
\end{enumerate}
  \end{Lemma}

\begin{proof}
(i) is proved via induction. For $d=1$ this is just Proposition \ref{prop:define_M}(ii). The inductive step is now proved using Remark \ref{rem:bisupmod} and Lemma \ref{lem:direct_prod_vert}. Note that this induction is valid due to Remark \ref{rem:induct_Rou}.

(ii),(iii) To show that $\bM_\Ni $ is indecomposable (as an $(\cO \Ni \fid,\cO \Li \fid)$-bisupermodule and therefore as an $(\cO \Ni \fid,\cO \Ni \fid)$-bimodule) and that
\begin{align*}
\Res^{\Ni  \times \Ni }_{\Ni  \times \Li }\,\bM_\Ni  \simeq \Ind_{\Li  \times \Li }^{\Ni  \times \Li }\,\bM_\Li ,
\end{align*}
we just apply Lemmas \ref{lem:vert_blocks} and \ref{lem:CGD}(i). (Note that Lemma \ref{lem:vert_blocks} is proved using the Mackey formula. However, we can use Theorem \ref{thm:super_Mackey} instead to obtain that the above does indeed holds as bisupermodules.) That $\bM_\Ni $ has vertex $\Delta \Di $ follows from the fact that $\bM_\Li $ does and that $p \nmid[\Ni :\Li ]$.
\end{proof}

Let $1\leq k \leq d$. We define $(\cO \Li_k\fid_k,\cO \Li_k\fid_k)$-bisupermodule 
$$\bM_{\Li_k} := \Blo^{\rho,0} \boxtimes \bM^{\boxtimes k}$$
and the $(\cO \Ni_k\fid_k,\cO \Ni_k\fid_k)$-bisupermodule
$$\bM_{\Ni_k}=\Blo^{\rho,0} \boxtimes (\bM \swr \cT_k).$$ That is, with Remark \ref{rem:induct_Rou} in mind, we do the same constructions as for $\bM_{\Li}=\bM_{\Li_d}$ and $\bM_{\Ni_d}=\bM_{\Ni}$, but with $d$ replaced by $k$. 
 Analogously to (\ref{algn:M_N_ext}), we have 
\begin{align*}
\bM_{\Ni_k} \simeq \Ind_{(\Ni_k \times \Ni_k)_{\Si_k}}^{\Ni_k \times \Ni_k}(\bM_{\Li_k})_{\Si_k},
\end{align*}
where 
$(\Ni_k \times \Ni_k)_{\Si_k} := (\Ni_k \times \Ni_k)_{\Ni_k/\Li_k}$
and $
(\bM_\Li )_{\Si_k} := \Blo^{\rho,0} \boxtimes \bM^{\boxtimes k}_{\Si_k}.
$

\begin{Lemma} \label{L300323}
The following supermodules are absolutely indecomposable with vertex $\Delta \Di $:
\begin{enumerate}
\item[{\rm (i)}] the $\cO(\Hi \times\Hi)$-supermodule $\cO \Gi_{d-1}\bi_{d-1} \boxtimes \bM^{(d)}$; 

\item[{\rm (ii)}] the $\cO((\Ni  \cap \Hi )\times(\Ni  \cap \Hi ))$-supermodule
$\cO \Ni_{d-1}\fid_{d-1} \boxtimes \bM^{(d)}$;
 
\item[{\rm (iii)}] the $\cO(\Li \times\Li)$-supermodule $\cO \Li_{d-1}\fid_{d-1} \boxtimes \bM^{(d)}$.

\end{enumerate}
\end{Lemma}
\begin{proof}
We prove (i). Parts (ii) and (iii) are proved similarly. 
If $d=1$, this is contained in Proposition \ref{prop:define_M}(ii). Let $d>1$. By Lemma \ref{lem:c_f_blocks}(i) and Remark \ref{rem:induct_Rou}, $\cO \Gi_{d-1}\bi_{d-1}$ is indecomposable with vertex $\Delta(\Di_1 \times \dots \times \Di_{d-1})$. Moreover, by Proposition \ref{prop:define_M}(i), $\bM^{(d)}$ is indecomposable with vertex $\Delta \Di_d$. Remark \ref{rem:bisupmod} and Lemma \ref{lem:direct_prod_vert} complete the proof.
\end{proof}

\subsection{The bisupermodules $\bX$ and $\bY$}\label{sec:def_X_Y}

In this subsection, we will define the bisupermodule $\bX$ that will ultimately induce a Morita superequivalence between $\cO \Ni \fid$ and $\cO \Gi \bi$. We will also introduce a related bisupermodule $\bY$, that will aid with the inductive arguments in Section~\ref{sec:main}. Recall the super Green correspondence of  Theorem~\ref{thm:super_Green}.

By Lemma \ref{ML_MN_indecomp}(ii), the $\cO(\Ni\times\Ni)$-supermodule $\bM_\Ni $ is absolutely indecomposable with vertex $\Delta \Di $, and Lemma \ref{lem:CGD}(ii) gives that $N_{\Gi\times\Ni} (\Delta \Di ) \leq \Ni\times\Ni $. We now define  $\bX$ to be the super Green correspondent of $\bM_\Ni $ in $\Gi \times \Ni $. 

By Lemma~\ref{L300323}(i), the $\cO(\Hi \times\Hi)$-supermodule $\cO \Gi_{d-1}\bi_{d-1} \boxtimes \bM^{(d)}$ is absolutely indecomposable with vertex $\Delta \Di$, and by  
Lemma \ref{lem:CGD}(iii), we have $N_{\Gi  \times \Hi }(\Delta \Di ) \leq \Hi  \times \Hi $. We now define $\bY$ to be the super Green correspondent of the $\cO(\Hi  \times \Hi )$-supermodule $\cO \Gi_{d-1}\bi_{d-1} \boxtimes \bM^{(d)}$ in $\Gi  \times \Hi $.

We need a technical lemma before continuing.

\begin{Lemma}\label{lem:Y_tech}
The $(\cO \Ni ,\cO(\Ni  \cap \Hi ))$-bisupermodule 
$%\begin{align*}
\cO \Ni  \otimes_{\cO (\Ni \cap \Hi )} (\cO \Ni_{d-1}\fid_{d-1} \boxtimes \bM^{(d)})
$ %\end{align*}
is absolutely indecomposable with vertex $\Delta \Di $. In particular, it is the super Green correspondent of $\bY$ in $\Ni  \times (\Ni  \cap \Hi )$.
\end{Lemma}

\begin{proof} 
Note that 
$$
\cO \Ni  \otimes_{\cO(\Ni  \cap \Hi )} (\cO \Ni_{d-1}\fid_{d-1} \boxtimes \bM^{(d)}) \simeq \cO \Gi_{d-1}\bi_{d-1} \boxtimes \bM^{(d)}.
$$
So for $d=1$ the statement is immediate. 
From now on we assume that $d>1$. We have 
\begin{align*}
%\begin{split}\label{algn:Mackey_N_d-1}
& \Res^{\Ni  \times (\Ni  \cap \Hi )}_{\Ni  \times \Li }\big(\cO \Ni  \otimes_{\cO (\Ni \cap \Hi )} (\cO \Ni_{d-1}\fid_{d-1} \boxtimes \bM^{(d)})\big) \\
\simeq\,\, & \Res^{\Ni  \times (\Ni  \cap \Hi )}_{\Ni  \times \Li } \Ind_{(\Ni  \cap \Hi ) \times (\Ni  \cap \Hi )}^{\Ni  \times (\Ni  \cap \Hi )}(\cO \Ni_{d-1}\fid_{d-1} \boxtimes \bM^{(d)}) \\
\simeq\,\, & \Ind_{(\Ni  \cap \Hi ) \times \Li }^{\Ni  \times \Li } \Res^{(\Ni  \cap \Hi ) \times (\Ni  \cap \Hi )}_{(\Ni  \cap \Hi ) \times \Li }(\cO \Ni_{d-1}\fid_{d-1} \boxtimes \bM^{(d)}) \\
\simeq\,\, & \Ind_{(\Ni  \cap \Hi ) \times \Li }^{\Ni  \times \Li } \Ind_{\Li  \times \Li }^{(\Ni  \cap \Hi ) \times \Li } (\cO \Li_{d-1}\fid_{d-1} \boxtimes \bM^{(d)}) \\
\simeq\,\, & \Ind_{\Li  \times \Li }^{\Ni  \times \Li } (\cO \Li_{d-1}\fid_{d-1} \boxtimes \bM^{(d)}),
%\end{split}
\end{align*}
where the second isomorphism follows from Theorem \ref{thm:super_Mackey} and the third from Lemma \ref{lem:vert_blocks}(ii). It now follows from Lemmas~\ref{L300323}, \ref{lem:vert_blocks}(i) and \ref{lem:CGD}(i) that
\begin{align*}
\cO \Ni  \otimes_{\cO (\Ni \cap \Hi )} (\cO \Ni_{d-1}\fid_{d-1} \boxtimes \bM^{(d)})
\end{align*}
is absolutely indecomposable as an $(\cO \Ni , \cO (\Ni  \cap \Hi ))$-bisupermodule with vertex $\Delta \Di $. (We have even shown it is absolutely indecomposable as an $(\cO \Ni , \cO \Li )$-bisupermodule.)

Next note that, by Lemma \ref{lem:CGD}(ii),(iii), $N_{\Gi  \times \Hi }(\Delta \Di ) \leq (\Ni  \cap \Hi ) \times (\Ni  \cap \Hi )$ and so, taking into account Lemma~\ref{L300323}, it makes sense to consider the super Green correspondents of $\cO \Gi_{d-1}\bi_{d-1} \boxtimes \bM^{(d)}$ and $\cO \Ni  \otimes_{\cO (\Ni \cap \Hi )} (\cO \Ni_{d-1}\fid_{d-1} \boxtimes \bM^{(d)})$ in $(\Ni  \cap \Hi ) \times (\Ni  \cap \Hi )$.

By Lemma \ref{lem:Brauer_G_N} and Remark \ref{rem:induct_Rou}, $\cO \Gi_{d-1}\bi_{d-1}$ and $\cO \Ni_{d-1}\fid_{d-1}$ are Brauer correspondents. Hence, by Lemma \ref{lem:Brauer_Green}, they are super Green correspondents. Therefore,
\begin{align*}
\cO \Gi_{d-1}\bi_{d-1} \boxtimes \bM^{(d)} \mid\,\,\, & (\cO \Gi_{d-1} \otimes_{\cO \Ni_{d-1}} \cO \Ni_{d-1}\fid_{d-1} \otimes_{\cO \Ni_{d-1}} \cO \Gi_{d-1}) \boxtimes \bM^{(d)} \\
\simeq\,\, & \cO \Hi  \otimes_{\cO(\Ni  \cap \Hi )} (\cO \Ni_{d-1}\fid_{d-1} \boxtimes \bM^{(d)}) \otimes_{\cO(\Ni  \cap \Hi )} \cO \Hi ,
\end{align*}
where the isomorphism is two applications of Lemma \ref{lem:tensor_bisupmod}.

We have now shown that $\cO \Gi_{d-1}\bi_{d-1} \boxtimes \bM^{(d)}$ and $\cO \Ni  \otimes_{\cO(\Ni  \cap \Hi )}(\cO \Ni_{d-1}\fid_{d-1} \boxtimes \bM^{(d)})$ both have super Green correspondent $\cO \Ni_{d-1}\fid_{d-1} \boxtimes \bM^{(d)}$ in $(\Ni  \cap \Hi ) \times (\Ni  \cap \Hi )$. In particular, $\bY$ is isomorphic to the unique absolutely indecomposable summand of
\begin{align*}
& \cO \Gi  \otimes_{\cO \Hi } \cO \Hi  \otimes_{\cO(\Ni  \cap \Hi )} (\cO \Ni_{d-1}\fid_{d-1} \boxtimes \bM^{(d)}) \otimes_{\cO(\Ni  \cap \Hi )} \cO \Hi  \\
\simeq\,\, & \cO \Gi  \otimes_{\cO(\Ni  \cap \Hi )} (\cO \Ni_{d-1}\fid_{d-1} \boxtimes \bM^{(d)}) \otimes_{\cO(\Ni  \cap \Hi )} \cO \Hi  \\
\simeq\,\, & \cO \Gi  \otimes_{\cO \Ni } \cO \Ni  \otimes_{\cO(\Ni  \cap \Hi )} (\cO \Ni_{d-1}\fid_{d-1} \boxtimes \bM^{(d)}) \otimes_{\cO(\Ni  \cap \Hi )} \cO \Hi 
\end{align*}
with vertex $\Delta \Di $, as desired.
\end{proof}

\begin{Lemma}\label{lem:X_as_left_mod}
We have:
\begin{enumerate}
\item $\bX$ is an $(\cO \Gi \bi,\cO \Ni \fid)$-bisupermodule.
\item $\bY$ is an $(\cO \Gi \bi,\cO \Hi \ci)$-bisupermodule.
\end{enumerate}
\end{Lemma}

\begin{proof}
(i) Since $\bX\mid\cO \Gi  \otimes_{\cO \Ni } \bM_\Ni $, the right action of $\cO \Ni \fid$ is clear.

By Lemma \ref{lem:Brauer_G_N}, $\cO \Gi \bi$ and $\cO \Ni \fid$ are Brauer correspondents. Since $\bi \in \cO \Gi_{\0}$ and $\fid \in \cO \Ni_{\0}$, we may apply Lemma \ref{lem:Brauer_Green} and we set $\gU = \bi\gU \fid$ to be the common Green correspondent of $\cO \Gi \bi$ and $\cO \Ni \fid$ in $\Gi  \times \Ni $. So, $\cO \Gi  \otimes_{\cO \Ni } \cO \Ni \fid$ is a direct sum of $\gU$, which has vertex $\Delta \Di $, and $\gV$, a direct sum of indecomposable bimodules each with vertex properly contained in $\Delta \Di $. Now,
\begin{align*}
\bX\mid \cO \Gi  \otimes_{\cO \Ni } \bM_\Ni  & \simeq \cO \Gi  \otimes_{\cO \Ni } \cO \Ni \fid \otimes_{\cO \Ni } \bM_\Ni   \simeq (\gU \oplus \gV) \otimes_{\cO \Ni } \bM_\Ni  \\
& \simeq (\gU \otimes_{\cO \Ni } \bM_\Ni ) \oplus (\gV \otimes_{\cO \Ni } \bM_\Ni ).
\end{align*}
Lemma \ref{lem:bimod_vert}(ii) implies that $\gV \otimes_{\cO \Ni \fid} \bM_\Ni $ is a direct sum of bimodules each with vertex of strictly smaller order than $\Delta \Di $. Therefore, $\bX\mid\gU \otimes_{\cO \Ni } \bM_\Ni $. The claim follows.

(ii) Since $\bY\mid\cO \Gi  \otimes_{\cO \Hi } (\cO \Gi_{d-1}\bi_{d-1} \boxtimes \bM^{(d)})$, the right $\cO \Hi \ci$ action is clear.

For the left action we first note that, by Lemma \ref{lem:Y_tech}, $\bY$ is actually the unique indecomposable summand of
\begin{align*}
& \cO \Gi  \otimes_{\cO \Ni } (\cO \Ni  \otimes_{\cO(\Ni  \cap \Hi )} (\cO \Ni_{d-1}\fid_{d-1} \boxtimes \bM^{(d)})) \otimes_{\cO(\Ni  \cap \Hi )} \cO \Hi  \\
\simeq\,\, & \cO \Gi  \otimes_{\cO \Ni } (\cO \Ni \fid \otimes_{\cO(\Ni  \cap \Hi )} (\cO \Ni_{d-1}\fid_{d-1} \boxtimes \bM^{(d)})) \otimes_{\cO(\Ni  \cap \Hi )} \cO \Hi 
\end{align*}
with vertex $\Delta \Di $. As is part (i), this implies
\begin{align*}
\bY\mid\gU \otimes_{\cO \Ni } \cO \Ni  \otimes_{\cO(\Ni  \cap \Hi )} (\cO \Ni_{d-1}\fid_{d-1} \boxtimes \bM^{(d)}) \otimes_{\cO(\Ni  \cap \Hi )} \cO \Hi .
\end{align*}
The claim follows.
\end{proof}

\begin{Lemma}\label{lem:GXL_indec}
$\bX$ is absolutely indecomposable as an $(\cO \Gi \bi,\cO \Li \fid)$-bisupermodule. In particular, $\Res^{\Gi  \times \Ni }_{\Gi  \times \Li }(\bX)$ is the super Green correspondent of $\bM_\Li $ in $\Gi  \times \Li $.
\end{Lemma}

\begin{proof}
By Lemma \ref{lem:CGD}(iii), we have $N_{\Gi  \times \Li }(\Delta \Di ) \leq \Li  \times \Li $ and so one can consider the super Green correspondent of $\bM_\Li $ in $\Gi  \times \Li $.

Certainly every summand of $\Res^{\Gi  \times \Ni }_{\Gi  \times \Li }(\bX)$ has vertex conjugate to $\Delta \Di $ in $\Gi  \times \Ni $. Now, by definition,
\begin{align*}
\Res^{\Gi  \times \Ni }_{\Gi  \times \Li }\bX\mid\Res^{\Gi  \times \Ni }_{\Gi  \times \Li } \Ind_{\Ni  \times \Ni }^{\Gi  \times \Ni }\bM_\Ni  & \simeq \Ind_{\Ni  \times \Li }^{\Gi  \times \Li } \Res^{\Ni  \times \Ni }_{\Ni  \times \Li }\bM_\Ni  
 \simeq \Ind_{\Ni  \times \Li }^{\Gi  \times \Li } \Ind_{\Li  \times \Li }^{\Ni  \times \Li }\bM_\Li   \simeq \Ind_{\Li  \times \Li }^{\Gi  \times \Li }\bM_\Li ,
\end{align*}
where the first isomorphism is due to Theorem \ref{thm:super_Mackey} and the second to Lemma \ref{ML_MN_indecomp}. However, by Theorem \ref{thm:super_Green}, $\Ind_{\Li  \times \Li }^{\Gi  \times \Li }\bM_\Li $ has a unique indecomposable summand whose vertex is not strictly contained in $\Delta \Di $. The claim follows.
\end{proof}

We will use the restriction $\Res^{\Gi  \times \Ni }_{\Gi  \times \Li }\bX$ extensively, so to simplify the notation, we set 
$${}_\Gi  \bX_\Li :=\Res^{\Gi  \times \Ni }_{\Gi  \times \Li }\bX.
$$

Recall the notation `$|_\Di $' from (\ref{E|_D}).

\begin{Lemma}\label{lem:vert_eqns_X_Y}
We have:
\begin{enumerate}
\item $\bX\,\,|_\Di \,\,(\cO \Gi  \otimes_{\cO \Ni }\bM_\Ni )$.
\item ${}_\Gi  \bX_\Li \,\,|_\Di \,\,(\cO \Gi  \otimes_{\cO \Li }\bM_\Li )$.
\item $\bY\,\,|_\Di \,\,\big(\cO \Gi  \otimes_{\cO \Hi } (\cO \Gi_{d-1}\bi_{d-1} \boxtimes \bM^{(d)})\big)$.
\end{enumerate}
\end{Lemma}

\begin{proof}
(i) This follows %immediately 
from the definition of $\bX$, Theorem \ref{thm:super_Green} and Lemma \ref{lem:conj_def}(i).

(ii) Similar to part (i), this follows from Lemma \ref{lem:GXL_indec}, Theorem \ref{thm:super_Green} and Lemma \ref{lem:conj_def}(i).

(iii) By Lemma \ref{lem:Y_tech}, $\bY$ and $\cO \Ni  \otimes_{\cO(\Ni  \cap \Hi )} (\cO \Ni_{d-1}\fid_{d-1} \boxtimes \bM^{(d)})$ are super Green correspondents. In particular, by Theorem \ref{thm:super_Green} and Lemma \ref{lem:conj_def}(i),
\begin{align}
\begin{split}\label{algn:OG_N_d-1_OH}
\bY \,\,|_\Di \,& \big( \cO \Gi  \otimes_{\cO \Ni } (\cO \Ni  \otimes_{\cO(\Ni  \cap \Hi )} (\cO \Ni_{d-1}\fid_{d-1} \boxtimes \bM^{(d)})) \otimes_{\cO(\Ni  \cap \Hi )} \cO \Hi \big)\\
\simeq\,\, & \cO \Gi  \otimes_{\cO(\Ni  \cap \Hi )} (\cO \Ni_{d-1}\fid_{d-1} \boxtimes \bM^{(d)}) \otimes_{\cO(\Ni  \cap \Hi )} \cO \Hi  \\
\simeq\,\, & \cO \Gi  \otimes_{\cO \Hi } (\cO \Hi  \otimes_{\cO(\Ni  \cap \Hi )} (\cO \Ni_{d-1}\fid_{d-1} \boxtimes \bM^{(d)}) \otimes_{\cO(\Ni  \cap \Hi )} \cO \Hi ).
\end{split}
\end{align}
However, as noted in the proof of Lemma \ref{lem:Y_tech}, $\cO \Gi_{d-1}\bi_{d-1} \boxtimes \bM^{(d)}$ and $\cO \Ni_{d-1}\fid_{d-1} \boxtimes \bM^{(d)}$ are also super Green correspondents. In particular,
\begin{align*}
\cO \Gi_{d-1}\bi_{d-1} \boxtimes \bM^{(d)}\mid\cO \Hi  \otimes_{\cO(\Ni  \cap \Hi )} (\cO \Ni_{d-1}\fid_{d-1} \boxtimes \bM^{(d)})) \otimes_{\cO(\Ni  \cap \Hi )} \cO \Hi .
\end{align*}
As we already know that $\bY\mid\cO \Gi  \otimes_{\cO \Hi }(\cO \Gi_{d-1}\bi_{d-1} \boxtimes \bM^{(d)})$, the claim now follows from~(\ref{algn:OG_N_d-1_OH}).
\end{proof}

Let $1\leq k \leq d$. With Remark \ref{rem:induct_Rou} in mind, we define $\bX_k$ and ${}_{\Gi_k}\bX_{\Li_k}$ in the same way as $\bX$ and ${}_\Gi  \bX_\Li $ were defined. (Of course, by Lemma \ref{lem:X_as_left_mod}(i), $\bX_k$ is an $(\cO \Gi_k \bi_k,\cO \Ni_k \fid_k)$-bimodule.)

\subsection{Dualizing the special bisuprmodules}
The special bisupermodules defined so far can be dualized in the sense of \S\ref{sec:dual} to get the bisupermodules $\bM_\Li^*:=(\bM_\Li)^*$, $\bM_\Ni^*:= (\bM_\Ni)^*$, $\bX^*$, $\bY^*$, ${}_\Li  \bX_\Gi^*:=({}_\Gi \bX_\Li)^*$, and similarly $\bM_{\Li_k}^*$, $\bM_{\Ni_k}^*$ ${}_{\Li_k} \bX_{\Gi_k}^*$ for $1 \leq k \leq d$.

\begin{Lemma}\label{lem:dual_ML_MN}
We have:
\begin{enumerate}
\item $\bM_\Li^* \simeq \Blo^{\rho,0} \boxtimes (\bM^*)^{\boxtimes d}$ and $\bM_\Li^*$ has vertex $\Delta \Di $.
\item $\bM_\Ni^* \simeq \Blo^{\rho,0} \boxtimes (\bM^* \swr \cT_d)$ and $\bM_\Ni^*$ has vertex $\Delta \Di $.
\item $\bX^*$ is the super Green correspondent of $\bM_\Ni^*$ in $\Ni  \times \Gi $ and ${}_\Li  \bX_\Gi^*$ is the super Green correspondent of $\bM_\Li^*$ in $\Li  \times \Gi $.
\end{enumerate}
\end{Lemma}

\begin{proof}
(i) By Lemma \ref{lem:sup_alg_idempt_dual} $(\Blo^{\rho,0})^* \simeq \Blo^{\rho,0}$, and so the first claim just follows from Lemma \ref{lem:dual_prod}. That $\bM_\Li $ and $\bM_\Li^*$ have the same vertex follows from Lemma \ref{lem:sup_vert}.

(ii) This is proved in exactly the same way as (i) once we note that $\bM^* \swr \cT_d \simeq (\bM \swr \cT_d)^*$ via Lemma \ref{lem:dual_sup_wreath}.

(iii) By Lemma \ref{lem:ind_dual}, we have 
$%\begin{align*}
 \bX^*\mid(\cO \Gi  \otimes_{\cO \Ni } \bM_\Ni )^* \simeq \bM_\Ni^* \otimes_{\cO \Ni } \cO \Gi .
$ %\end{align*}
That $\bX^*$ has vertex $\Delta \Di $ is, again, due to Lemma \ref{lem:sup_vert}. With Lemma \ref{lem:GXL_indec} in mind, the statement for ${}_\Li  \bX_\Gi^*$ is proved in an identical fashion.
\end{proof}

Recall the notation `$|^\Di $' from (\ref{E|^D}).

\begin{Lemma}\label{lem:X*_with_X}
We have:
\begin{enumerate}
\item $\cO \Ni \fid\,\,|^\Di \,(\bX^* \otimes_{\cO \Gi } \bX)$.
\item $\cO \Gi \bi\mid \bX \otimes_{\cO \Ni } \bX^*$.
\end{enumerate}
\end{Lemma}

\begin{proof}
We first note that
\begin{align*}
\bM_\Ni  \otimes_{\cO \Ni \fid} \bM_\Ni^* &\simeq \big(\Blo^{\rho,0} \boxtimes (\bM \swr \cT_d)\big) \otimes_{\cO \Ni \fid} \big(\Blo^{\rho,0} \boxtimes (\bM^* \swr \cT_d)\big) \\
&\simeq (\Blo^{\rho,0} \otimes_{\Blo^{\rho,0}} \Blo^{\rho,0}) \boxtimes \big((\bM \swr \cT_d) \otimes_{\Blo^{\varnothing,1}\swr \cT_d} (\bM^* \swr \cT_d)\big) \\
&\simeq \Blo^{\rho,0} \boxtimes \big((\bM \otimes_{\Blo^{\varnothing,1}} \bM^*) \swr \cT_d\big),
\end{align*}
where the second isomorphism follows from Lemma \ref{lem:tensor_bisupmod} and the third from Lemma \ref{lem:wreath_tensor}(ii). By Proposition~\ref{prop:define_M}(iii) and Lemma~\ref{lem:wreath_tensor}(i), this last bisupermodule has a direct summand isomorphic to $\cO \Ni \fid$.

Dualizing Lemma \ref{ML_MN_indecomp}, using Lemma \ref{lem:dual_ML_MN}(i),(ii) and Lemma \ref{lem:vert_blocks}(ii), we obtain
\begin{align*}
\Res^{\Ni  \times \Ni }_{\Ni  \times \Li }\bM_\Ni^* \simeq \Ind_{\Li  \times \Li }^{\Ni  \times \Li }\bM_\Li^*.
\end{align*}
Therefore,
\begin{align*}
\Res^{\Ni  \times \Ni }_{\Ni  \times \Li } (\bM_\Ni  \otimes_{\cO \Ni } \bM_\Ni^*) & \simeq \bM_\Ni  \otimes_{\cO \Ni } \Res^{\Ni  \times \Ni }_{\Ni  \times \Li }\bM_\Ni^*  \simeq \bM_\Ni  \otimes_{\cO \Ni } \Ind_{\Li  \times \Li }^{\Ni  \times \Li }\bM_\Li^* \\
& \simeq (\Res^{\Ni  \times \Ni }_{\Ni  \times \Li }\bM_\Ni ) \otimes_{\cO \Li } \bM_\Li^*  \simeq (\Ind_{\Li  \times \Li }^{\Ni  \times \Li }\bM_\Li ) \otimes_{\cO \Li } \bM_\Li^* \\
& \simeq \cO \Ni  \otimes_{\cO \Li } \bM_\Li  \otimes_{\cO \Li } \bM_\Li^*,
\end{align*}
where the fourth isomorphism is Lemma \ref{ML_MN_indecomp}(iii). Now, by Lemma \ref{lem:tensor_bisupmod} as well as Remark \ref{rem:bisupmod} and Lemma \ref{lem:direct_prod_vert}, we have $\cO \Li \fid \,\,|_\Di \, (\bM_\Li  \otimes_{\cO \Li } \bM_\Li^*)$. Therefore,
$$
\cO \Ni \fid \,\,|_\Di  \,\big(\Res^{\Ni  \times \Ni }_{\Ni  \times \Li }(\bM_\Li  \otimes_{\cO \Li } \bM_\Li^*)\big),
$$
as $(\cO \Ni ,\cO \Li )$-bisupermodules. Putting this together with the first paragraph gives that $\cO \Ni \fid \,\,|_\Di \, (\bM_\Ni  \otimes_{\cO \Ni \fid} \bM_\Ni^*)$. Similarly, we have $\cO \Ni \fid \,\,|_\Di \, (\bM_\Ni^* \otimes_{\cO \Ni \fid} \bM_\Ni )$.

(i) First note that $\bM_\Ni^* \otimes \bM_\Ni $, and consequently $\cO \Ni \fid$, is a direct summand of
\begin{align}\label{algn:MN_OG_MN}
\bM_\Ni^* \otimes_{\cO \Ni }\cO \Gi  \otimes_{\cO \Gi } \cO \Gi  \otimes_{\cO \Ni } \bM_\Ni  \simeq \bM_\Ni^* \otimes_{\cO \Ni } \cO \Gi  \otimes_{\cO \Ni } \bM_\Ni .
\end{align}

Now, by Lemma \ref{lem:vert_eqns_X_Y}(i), we have $\bX\,\, |_\Di \, (\cO \Gi  \otimes_{\cO \Ni }\bM_\Ni )$. Applying Lemmas \ref{lem:dual_comp} and \ref{lem:sup_vert}, we deduce that $\bX^* \,\,|_\Di \, (\bM_\Ni^* \otimes_{\cO \Ni } \cO \Gi )$. Therefore, $\bX^* \otimes_{\cO \Gi } \bX\mid \bM_\Ni^* \otimes_{\cO \Ni } \cO \Gi  \otimes_{\cO \Ni } \bM_\Ni $ and, by Lemma \ref{lem:bimod_vert}(ii), every direct summand of $\bM_\Ni^* \otimes_{\cO \Ni } \cO \Gi  \otimes_{\cO \Ni } \bM_\Ni $, as an $(\cO \Ni ,\cO \Ni )$-bimodule, with vertex $\Delta \Di $, must also be a summand of $\bX^* \otimes_{\cO \Gi b} \bX$. In particular, $\cO \Ni \fid\mid \bX^* \otimes_{\cO \Gi } \bX$, as an $(\cO \Ni ,\cO \Ni )$-bisupermodule.

All that remains to show is that all other summands of $\bX^* \otimes_{\cO \Gi } \bX$ have $\Di $-small vertex in the sense of \S\ref{SSDSmall}. We can, therefore, forget about superstructure for the time being. Given (\ref{algn:MN_OG_MN}) and the comments at the beginning of the proof, we need only show that
\begin{align}\label{algn:def_U}
\bM_\Ni^* \otimes_{\cO \Ni } \cO \Gi  \otimes_{\cO \Ni } \bM_\Ni  \simeq (\bM_\Ni^* \otimes_{\cO \Ni } \bM_\Ni ) \oplus \gU,
\end{align}
where, as an $(\cO \Ni ,\cO \Ni )$-bimodule, $\gU$ is a direct sum of bimodules each with $\Di $-small vertex.

We have already seen that $\bX\,\, |_\Di \, (\cO \Gi  \otimes_{\cO \Ni } \bM_\Ni )$. Therefore,
\begin{align}\label{algn:def_V}
\cO \Ni \fid \otimes_{\cO \Ni } \cO \Gi  \otimes_{\cO \Ni } \bM_\Ni  \simeq \fid(\Res^{\Gi  \times \Ni }_{\Ni  \times \Ni }\bX)\fid \oplus \gV,
\end{align}
where $\gV$ is a direct sum of $(\cO \Ni \fid,\cO \Ni \fid)$-bimodules each with vertex contained in some ${}^{(g,1)}\Delta Q$, with $g \in \Gi $ and $Q <_\sfs\Di$. Furthermore, Theorem \ref{thm:super_Green} implies that
\begin{align}\label{algn:MN_W}
\Res^{\Gi  \times \Ni }_{\Ni  \times \Ni }\bX \simeq \bM_\Ni  \oplus \gW,
\end{align}
where $\gW$ is a direct sum of $(\cO \Ni ,\cO \Ni )$-bimodules each with vertex contained in $(\Ni  \times \Ni ) \cap {}^{(g,h)}\Delta \Di$, for some $(g,h) \in (\Gi  \times \Ni ) \setminus (\Ni  \times \Ni )$.

Now $\cO \Ni \fid$ has defect group $\Di $ and so all the bimodules in (\ref{algn:def_V}) can be chosen to have vertex contain in $\Di  \times \Di $ (possibly after conjugating by an element of $\Ni  \times \Ni $).

In particular, $\gV$ is a direct sum of $(\cO \Ni \fid,\cO \Ni \fid)$-bimodules each with vertex contained in some $(\Di  \times \Di )\cap{}^{(g,h)}\Delta Q$, with $(g,h) \in \Gi  \times \Ni $ and $Q <_\sfs\Di$. By looking at cycle types via $\pi_n$, certainly $\Di  \cap {}^g Q,\Di  \cap {}^h Q <_\sfs\Di$ and so $\gV$ is a direct sum of $(\cO \Ni \fid,\cO \Ni \fid)$-bimodules each with $\Di $-small vertex.

Similarly, $\fid\gW \fid$ is a direct sum of $(\cO \Ni \fid,\cO \Ni \fid)$-bimodules each with vertex contained in $(\Di  \times \Di ) \cap {}^{(g,h)}\Delta \Di$, for some $(g,h) \in (\Gi  \times \Ni ) \setminus (\Ni  \times \Ni )$. By Lemma \ref{lem:conj_def}(ii), $\fid\gW \fid$ is a direct sum of $(\cO \Ni \fid,\cO \Ni \fid)$-bimodules each with $\Di $-small vertex.

Together, (\ref{algn:def_V}) and (\ref{algn:MN_W}) now imply that
\begin{align*}
\bM_\Ni \,\,|^\Di \,(\cO \Ni \fid \otimes_{\cO \Ni } \cO \Gi  \otimes_{\cO \Ni } \bM_\Ni ).
\end{align*}
By applying $\bM_\Ni^* \otimes_{\cO \Ni }?$ to both sides and utilizing Lemma \ref{lem:bimod_vert}(ii), we can show that $\gU$ in (\ref{algn:def_U}) is a direct sum of $(\cO \Ni ,\cO \Ni )$-bimodules each with vertex of the form $\Delta_\varphi Q$, for some $Q \leq \Di $ and $\varphi: Q \to \Di $, with $\varphi(Q) <_ \sfs\Di$.

In an entirely analogous way we can prove that
\begin{align*}
\bM_\Ni \,\,|^\Di \,(\bM_\Ni  \otimes_{\cO \Ni } \cO \Gi  \otimes_{\cO \Ni } \cO \Ni \fid).
\end{align*}
Once more, applying $? \otimes_{\cO \Ni } \bM_\Ni^*$ to both sides, we can have that $\gU$ in (\ref{algn:def_U}) is a direct sum of $(\cO \Ni ,\cO \Ni )$-bimodules each with vertex of the form $\Delta_\varphi Q$, for some $Q <_\sfs\Di$ and $\varphi: Q \to \Di $.

Since we can conjugate each $Q$ independently from $\varphi(Q)$, taking the last two paragraphs together gives that $\gU$ is a direct sum of $(\cO \Ni ,\cO \Ni )$-bimodules each with $\Di $-small vertex, as desired.

(ii) As in part (i) we have that $\bX\,\, |_\Di \, (\cO \Gi  \otimes_{\cO \Ni }\bM_\Ni )$ and $\bX^*\,\, |_\Di \, (\bM_\Ni^* \otimes_{\cO \Ni } \cO \Gi )$. Therefore,
\begin{align}\label{G_MN_G}
\bX \otimes_{\cO \Ni \fid} \bX^*\mid\cO \Gi  \otimes_{\cO \Ni }\bM_\Ni  \otimes_{\cO \Ni \fid} \bM_\Ni^* \otimes_{\cO \Ni } \cO \Gi 
\end{align}
and, by Lemma \ref{lem:bimod_vert}(ii), all other summands, as an $(\cO \Gi ,\cO \Gi )$-bimodule, have vertex of order strictly smaller than that of $\Delta \Di $. Moreover, by the comments at the beginning of the proof, the right-hand side of (\ref{G_MN_G}) has a summand isomorphic to $\cO \Gi  \otimes_{\cO \Ni } \cO \Ni \fid \otimes_{\cO \Ni } \cO \Gi $. Since $\cO \Gi \bi$ and $\cO \Ni \fid$ are super Green correspondents by Lemmas \ref{lem:Brauer_G_N} and \ref{lem:Brauer_Green}, we have $\cO \Gi \bi\mid\cO \Gi  \otimes_{\cO \Ni } \cO \Ni \fid \otimes_{\cO \Ni } \cO \Gi $. Since $\cO \Gi \bi$ has vertex $\Delta \Di $, the claim follows.
\end{proof}

\subsection{Relating  $\bX_d$ and $\bX_{d-1}$}
We now establish a lemma that will be key throughout inductive arguments in Section \ref{sec:main}, as it allows us to relate $\bX=\bX_d$ and $\bX_{d-1}$ using $\bY$.

\begin{Lemma}\label{lem:tensor_isom}
Let $d>1$.
\begin{enumerate}
\item As $(\cO \Gi \bi,\cO (\Ni  \cap \Hi )\fid)$-bisupermodules and as $(\cO \Gi \bi,\cO \Li \fid)$-bisupermodules, 
\begin{align*}
\bX\,\,|_\Di \, \big(\bY \otimes_{\cO \Hi } (\bX_{d-1} \boxtimes (\Blo^{\varnothing,1})^{(d)})\big).
\end{align*}
\item If $\bX_{d-1}$ induces a Morita superequivalence between $\cO \Ni_{d-1}\fid_{d-1}$ and $\cO \Gi_{d-1}\bi_{d-1}$, then as $(\cO \Gi \bi,\cO (\Ni  \cap \Hi )\fid)$-bisupermodules and as $(\cO \Gi \bi,\cO \Li \fid)$-bisupermodules, 
\begin{align*}
\bY \otimes_{\cO \Hi } \big(\bX_{d-1} \boxtimes (\Blo^{\varnothing,1})^{(d)}\big) \simeq \bX.
\end{align*}

\end{enumerate}
\end{Lemma}

\begin{proof} 
(i) 
Note by Lemma \ref{lem:tensor_bisupmod}, that 
$$
\cO \Gi  \otimes_{\cO \Hi }(\cO \Gi_{d-1} \bi_{d-1} \boxtimes \bM^{(d)}) \otimes_{\cO \Hi }\big(\bX_{d-1} \boxtimes (\Blo^{\varnothing,1})^{(d)}\big) 
\simeq  \cO \Gi  \otimes_{\cO \Hi }(\bX_{d-1} \boxtimes \bM^{(d)}).
$$
Since by definition, we have $\bY\mid \cO \Gi  \otimes_{\cO \Hi }(\cO \Gi_{d-1} \bi_{d-1} \boxtimes \bM^{(d)})$, it now follows that  
\begin{align}
\label{algn:Y_equ_1}
\bY \otimes_{\cO \Hi } \big(\bX_{d-1} \boxtimes (\Blo^{\varnothing,1})^{(d)}\big)\mid\,\,  \cO \Gi  \otimes_{\cO \Hi }(\bX_{d-1} \boxtimes \bM^{(d)}),
\end{align}
as $(\cO \Gi ,\cO(\Ni  \cap \Hi ))$-bisupermodules. 

Since by definition, we have $\bX_{d-1}\mid \cO \Gi_{d-1} \otimes_{\cO \Ni_{d-1}} \bM_{\Ni_{d-1}}$, it follows that 
\begin{align}
\begin{split}\label{algn:Y_equ_2}
\cO \Gi  \otimes_{\cO \Hi }(\bX_{d-1} \boxtimes \bM^{(d)})\mid \,\, & \cO \Gi  \otimes_{\cO \Hi }((\cO \Gi_{d-1} \otimes_{\cO \Ni_{d-1}} \bM_{\Ni_{d-1}}) \boxtimes \bM^{(d)}) \\
\simeq\, & \cO \Gi  \otimes_{\cO \Hi }\cO \Hi  \otimes_{\cO(\Ni  \cap \Hi )}(\bM_{\Ni_{d-1}} \boxtimes \bM^{(d)}) \\
\simeq\, & \cO \Gi  \otimes_{\cO(\Ni  \cap \Hi )} (\bM_{\Ni_{d-1}} \boxtimes \bM^{(d)}),
\end{split}
\end{align}
as $(\cO \Gi ,\cO(\Ni  \cap \Hi ))$-bisupermodules, where the first isomorphism follows from Lemma \ref{lem:tensor_bisupmod}. 

We also have 
\begin{align*}
\cO \Ni \otimes_{\cO(\Ni  \cap \Hi )}(\bM_{\Ni_{d-1}} \boxtimes \bM^{(d)}) \simeq\, & \cO \Ni \otimes_{\cO(\Ni  \cap \Hi )}\big(\Ind_{(\Ni_{d-1} \times \Ni_{d-1})_{\Si_{d-1}}}^{\Ni_{d-1} \times \Ni_{d-1}}(\bM_{\Li_{d-1}})_{\Si_{d-1}} \boxtimes \bM^{(d)}\big)\\
\simeq\, & \Ind_{((\Ni  \cap \Hi ) \times (\Ni  \cap \Hi ))_{\Si_{d-1}}}^{\Ni  \times (\Ni  \cap \Hi )} \big((\bM_{\Li_{d-1}})_{\Si_{d-1}} \boxtimes \bM^{(d)} \big) \\
\simeq\, & \Ind_{((\Ni  \cap \Hi ) \times (\Ni  \cap \Hi ))_{\Si_{d-1}}}^{\Ni  \times (\Ni  \cap \Hi )}(\bM_\Li )_{\Si_d} \\
\simeq\, & \Res^{\Ni  \times \Ni }_{\Ni \times (\Ni \cap \Hi )}\Ind_{(\Ni  \times \Ni )_{\Si_d}}^{\Ni  \times \Ni }(\bM_\Li )_{\Si_d} 
\simeq\,  \Res^{\Ni  \times \Ni }_{\Ni \times (\Ni \cap \Hi )}\bM_\Ni ,
\end{align*}
as $(\cO \Ni ,\cO(\Ni  \cap \Hi ))$-bisupermodules, where the fourth isomorphism follows from Theorem~\ref{thm:super_Mackey}. Therefore,
\begin{align}
\begin{split}\label{algn:Y_equ_3}
\cO \Gi  \otimes_{\cO(\Ni  \cap \Hi )} (\bM_{\Ni_{d-1}} \boxtimes \bM^{(d)}) & \simeq \Ind_{\Ni \times (\Ni \cap \Hi )}^{\Gi  \times (\Ni  \cap \Hi )}\Res^{\Ni  \times \Ni }_{\Ni \times (\Ni \cap \Hi )}\bM_\Ni   \simeq \Res^{\Gi  \times \Ni }_{\Gi  \times (\Ni  \cap \Hi )} \Ind_{\Ni  \times \Ni }^{\Gi  \times \Ni }\bM_\Ni ,
\end{split}
\end{align}
as $(\cO \Gi ,\cO(\Ni  \cap \Hi ))$-bisupermodules, where the second isomorphism is, once again, an application of Theorem \ref{thm:super_Mackey}. 

Putting (\ref{algn:Y_equ_1}), (\ref{algn:Y_equ_2}) and (\ref{algn:Y_equ_3}) together yields
\begin{align}\label{algn:Y_equ_123}
\bY \otimes_{\cO \Hi } \big(\bX_{d-1} \boxtimes (\Blo^{\varnothing,1})^{(d)}\big)\mid\Res^{\Gi  \times \Ni }_{\Gi  \times (\Ni  \cap \Hi )} \Ind_{\Ni  \times \Ni }^{\Gi  \times \Ni }\bM_\Ni .
\end{align}

By the definition of $\bX$ and Lemma \ref{lem:GXL_indec}, $\Res^{\Gi  \times \Ni }_{\Gi  \times (\Ni  \cap \Hi )}(\bX)$ is an indecomposable summand of the right-hand side of (\ref{algn:Y_equ_123}). We claim that
\begin{align}\label{algn:XNH|D}
\Res^{\Gi  \times \Ni }_{\Gi  \times (\Ni  \cap \Hi )}\bX\,\,\big{|}_\Di \, \big(\Res^{\Gi  \times \Ni }_{\Gi  \times (\Ni  \cap \Hi )}\Ind_{\Ni  \times \Ni }^{\Gi  \times \Ni }\bM_\Ni  \big).
\end{align}
We will actually prove the stronger statement
\begin{align}\label{algn:XL|D}
{}_\Gi  \bX_\Li \,\,\big{|}_\Di \, \big(\Res^{\Gi  \times \Ni }_{\Gi  \times \Li }\Ind_{\Ni  \times \Ni }^{\Gi  \times \Ni }\bM_\Ni  \big).
\end{align}
Indeed, as in the proof of Lemma \ref{lem:GXL_indec},
\begin{align*}
\Res^{\Gi  \times \Ni }_{\Gi  \times \Li } \Ind_{\Ni  \times \Ni }^{\Gi  \times \Ni }\bM_\Ni   \simeq \Ind_{\Ni  \times \Li }^{\Gi  \times \Li } \Res^{\Ni  \times \Ni }_{\Ni  \times \Li } \bM_\Ni  
 \simeq \Ind_{\Ni  \times \Li }^{\Gi  \times \Li } \Ind_{\Li  \times \Li }^{\Ni  \times \Li }\bM_\Li   \simeq \Ind_{\Li  \times \Li }^{\Gi  \times \Li }\bM_\Li ,
\end{align*}
where the first isomorphism is due to Theorem \ref{thm:super_Mackey} and the second to Lemma \ref{ML_MN_indecomp}. The claim (\ref{algn:XL|D}), and hence (\ref{algn:XNH|D}), now follows from the statement of Lemma \ref{lem:GXL_indec} and Lemma \ref{lem:conj_def}(i).

With (\ref{algn:Y_equ_123}) in mind, we now need only show that $\bY \otimes_{\cO \Hi } \big(\bX_{d-1} \boxtimes (\Blo^{\varnothing,1})^{(d)}\big)$ has some summand with vertex $\Delta \Di $, when treated as an $(\cO \Gi ,\cO \Li )$-bimodule.

By Lemma \ref{lem:GXL_indec} and also Remark \ref{rem:bisupmod} and Lemma \ref{lem:direct_prod_vert}, $\bX_{d-1} \boxtimes \bM^{(d)}$ is indecomposable, as an $(\cO \Gi ,\cO \Li )$-bimodule, with vertex $\Delta \Di $. Now, as we have seen in (\ref{algn:Y_equ_1}),
\begin{align*}
\bY \otimes_{\cO \Hi } \big(\bX_{d-1} \boxtimes (\Blo^{\varnothing,1})^{(d)}\big)\mid \cO \Gi  \otimes_{\cO \Hi } (\bX_{d-1} \boxtimes \bM^{(d)}),
\end{align*}
as $(\cO \Gi ,\cO \Li )$-bimodules. Hence, every summand of $\bY \otimes_{\cO \Hi } \big(\bX_{d-1} \boxtimes (\Blo^{\varnothing,1})^{(d)}\big)$ has vertex contained in $\Delta \Di $, when treated as an $(\cO \Gi ,\cO \Li )$-bimodule. If every summand had vertex strictly contained in $\Delta \Di $, then, by Lemmas \ref{lem:sup_vert}, \ref{lem:direct_prod_vert} and \ref{lem:bimod_vert}(ii), every summand of
\begin{align*}
\bY \otimes_{\cO \Hi } \big(\bX_{d-1} \boxtimes (\Blo^{\varnothing,1})^{(d)}\big) \otimes_{\cO \Li } \big(\bX_{d-1}^* \boxtimes (\Blo^{\varnothing,1})^{(d)}\big)
\end{align*}
has vertex of strictly smaller order than $\Delta \Di $. However, by Lemma \ref{lem:tensor_bisupmod}, Remark \ref{rem:induct_Rou} and Lemma \ref{lem:X*_with_X}(ii), $\bY$ is a summand, a contradiction.

(ii) If $\bX_{d-1}$ induces a Morita superequivalence between $\cO \Ni_{d-1}\fid_{d-1}$ and $\cO \Gi_{d-1}\bi_{d-1}$ then, by Lemma \ref{lem:MSE}, $\bX_{d-1} \boxtimes (\Blo^{\varnothing,1})^{(d)}$ induces a Morita superequivalence between 
$
\cO \Ni_{d-1}\fid_{d-1} \otimes (\Blo^{\varnothing,1})^{(d)} \cong \cO (\Ni  \cap \Hi )\fid$ and $\cO \Gi_{d-1}\bi_{d-1} \otimes (\Blo^{\varnothing,1})^{(d)} \cong \cO \Hi \ci.
$ 
In particular, $\bY \otimes_{\cO \Hi } \big(\bX_{d-1} \boxtimes (\Blo^{\varnothing,1})^{(d)}\big)$ is indecomposable as an $(\cO \Gi ,\cO (\Ni  \cap \Hi ))$-bimodule, since
\begin{align*}
\bY \otimes_{\cO \Hi } \big(\bX_{d-1} \boxtimes (\Blo^{\varnothing,1})^{(d)}\big) \otimes_{\cO(\Ni  \cap \Hi )} \big(\bX_{d-1} \boxtimes (\Blo^{\varnothing,1})^{(d)} \big)^* \simeq \bY
\end{align*}
is certainly indecomposable as an $(\cO \Gi ,\cO \Hi )$-module. Now (ii) follows (i).
\end{proof}

Recall the notation (\ref{ECI}),(\ref{ECLong}),(\ref{EC'Long}). 

\begin{Corollary}\label{cor:X_OGcM}
We have ${}_\Gi  \bX_\Li \,\,\big{|}_\Di \,\,  (\cO \Gi  \ci_{0,d} \otimes_{\cO \Li } \bM_\Li )$.
\end{Corollary}

\begin{proof}
We already know from Lemma \ref{lem:vert_eqns_X_Y}(ii) that ${}_\Gi  \bX_\Li  \,\,\big{|}_\Di \,\, (\cO \Gi  \otimes_{\cO \Li } \bM_\Li )$. We, therefore, need only show that ${}_\Gi  \bX_\Li \mid \cO \Gi  \ci_{0,d} \otimes_{\cO \Li } \bM_\Li $. The proof proceeds via induction on $d$. As always, our inductive argument is valid due to Remark \ref{rem:induct_Rou}.

When $d=1$, $\ci_0=\fid$ and $\ci_1=\bi$. The result follows via Lemma \ref{lem:X_as_left_mod}(i). Suppose the result is true for $d-1$. Then, by Lemma \ref{lem:tensor_isom}(i),
\begin{align*}
{}_\Gi  \bX_\Li \mid \bY \otimes_{\cO \Hi } \big((\cO \Gi_{d-1} \ci_{0,d-1}' \otimes_{\cO \Li_{d-1}} \bM_{\Li_{d-1}}) \boxtimes (\Blo^{\varnothing,1})^{(d)}\big),
\end{align*}
as an $(\cO \Gi ,\cO \Li )$-bisupermodule. 
%, where the $\ci_k'$'s are those defined in $\S$\ref{sec:notation}. 
This is in turn, by the definition of $\bY$ and Lemma \ref{lem:X_as_left_mod}(ii), isomorphic to a direct summand of
\begin{align*}
\cO \Gi \bi \otimes_{\cO \Hi }(\cO \Gi_{d-1}\bi_{d-1} \boxtimes \bM^{(d)}) \otimes_{\cO \Hi } & \big((\cO \Gi_{d-1} \ci_{0,d-1}' \otimes_{\cO \Li_{d-1}} \bM_{\Li_{d-1}}) \boxtimes (\Blo^{\varnothing,1})^{(d)}\big) \\
\simeq & \cO \Gi \bi\otimes_{\cO \Hi } \big((\cO \Gi_{d-1} \ci_{0,d-1}' \otimes_{\cO \Li_{d-1}} \bM_{\Li_{d-1}}) \boxtimes \bM^{(d)} \big) \\
\simeq & \cO \Gi \bi\otimes_{\cO \Hi } \cO \Hi  \ci_{0,d-1} \otimes_{\cO \Li } (\bM_{\Li_{d-1}} \boxtimes \bM^{(d)}) \\
\simeq & \cO \Gi  \bi \ci_{0,d-1} \otimes_{\cO \Li } \bM_\Li  = \cO \Gi  \ci_{0,d} \otimes_{\cO \Li } \bM_\Li ,
\end{align*}
where the first and second isomorphisms follow from Lemma \ref{lem:tensor_bisupmod}.
\end{proof}

\section{Main theorem}\label{sec:main}

We continue with all our notation from Section~\ref{sec:X_Y}, in particular, $\rho$ a $d$-Rouquier ${\bar p}$-core.
Recall also the irreducible supercharacters $\xi_0,\xi_1,\dots,\xi_\ell$ of $\Blo^{\varnothing,1}$ from 
(\ref{ECharBEmptySuper}) and the irreducible characters  $\xi_0,\xi_1^{(\pm)},\dots\xi_\ell^{(\pm)}$ of $\Blo^{\varnothing,1}$ from (\ref{ECharBEmpty1}). 

Recall the notation (\ref{EJ}). We also refer to the tuple 
$(\mu,j_1,\dots,j_k)$ as even or odd according to whether 
$(\mu,(p-j_1,j_1),\dots,(p-j_k,j_k))$ is even or odd, i.e. according to whether the number of the odd partitions among  $\mu,(p-j_1,j_1),\dots,(p-j_k,j_k)$ is even or odd.

\subsection{Character calculations}

%The main tool needed throughout $\S$\ref{sec:main_arg} is that of characters. 
In this subsection we gather together all the preparatory results about characters needed to ultimately prove Theorem \ref{thm:main}.

\begin{Lemma}\label{lem:HtoH}
Let $\mu \in \Par_0(\rho,d-1)$ and $i\in I$. Then in $\Grot(\K H\ci)$ we have 
\begin{align*}
(\cO \Gi_{d-1}\bi_{d-1}\boxtimes \bM^{(d)}) \otimes_{\cO \Hi } \xi_{\mu,i}  = 
(\cO \Gi_{d-1}\bi_{d-1}\boxtimes (\bM^{(d)})^*) \otimes_{\cO \Hi } \xi_{\mu,i}
 = \sum_{j=0}^{\ell-i} \frac{\eps_{\mu, i}\eps_{i}\eps_{j}}{\eps_{\mu, j}} \xi_{\mu,j}.
\end{align*}
\end{Lemma}

\begin{proof}
If $r=1$ and $d=1$, then $\cO \Hi \ci \cong \Blo^{\varnothing,1}$ and the equalities follow from \ref{prop:define_M}(iv). Otherwise the equalities follow from Proposition \ref{prop:define_M}(iv) and Lemma \ref{lem:AboxM}.
\end{proof}

Throughout the rest of this subsection we construct several congruences of characters modulo characters of modules with specific vertices. Given that many of the bimodule isomorphisms we constructed in Section \ref{sec:def_X_Y} were only modulo bimodules with certain vertices (see Lemmas \ref{lem:vert_eqns_X_Y}, \ref{lem:X*_with_X}, \ref{lem:tensor_isom}(i) and Corollary \ref{cor:X_OGcM}) these congruences of characters are the best we can hope for at the moment. However, in $\S$\ref{sec:main_arg} we will eventually determine the character of ${}_\Gi  \bX_\Li $ (see Hypothesis \ref{hyp:ind_on_X} proved in Theorem \ref{thm:main}).

An indecomposable $\cO \Li \fid$-module (resp. $\cO \Ni \fid$-module, $\cO \Hi \ci$-module or $\cO \Gi \bi$-module) with vertex $Q <_\sfs\Di$ is said to have vertex of {\em non-maximal support}. We define $\Irr_{<_\sfs \Di}(\cO \Li \fid)\subseteq \Irr(\cO \Li \fid)$ (resp. $\Irr_{<_\sfs \Di}(\cO \Ni \fid)\subseteq \Irr(\cO \Ni \fid)$, $\Irr_{<_\sfs \Di}(\cO \Hi \ci)\subseteq \Irr(\cO \Hi \ci)$ and $\Irr_{<_\sfs \Di}(\cO \Gi \bi)\subseteq \Irr(\cO \Gi \bi)$) to be the set of characters of irreducible   $\K\Li \fid$-modules (resp. $\K \Ni \fid$-modules, $\K \Hi \ci$-modules and $\K \Gi \bi$-modules) with vertex of non-maximal support.

For any $Q \leq \Di $, we denote by $\Z\prj_Q(\cO \Li \fid)$ the set of all $\Z$-linear combinations of characters of relatively $Q$-projective $\cO \Li \fid$-modules.

It what follows we use the identification of $\cO \Li \fid$ with $\Blo^{\rho,0} \otimes (\Blo^{\varnothing,1})^{\otimes d}$ from (\ref{algn:ONf_ident}) and (\ref{EIrrLabelN}),(\ref{EJ})  to label $\Irr(\cO \Li \fid)$, i.e. 
$$
\Irr(\cO \Li \fid)=\{\xi_{\rho,j_1,\dots ,j_d}\mid \text{$(\rho, j_1, \dots, j_d)$ is even}\}\sqcup
\{\xi_{\rho,j_1,\dots ,j_d}^\pm\mid \text{$(\rho, j_1, \dots, j_d)$ is odd}\}.
$$
Note that for $1 \leq k \leq d$ and $0\leq j_k<\ell$ we have that $(\rho,j_1,\dots,j_k, \dots ,j_d)$ and $(\rho,j_1,\dots,j_k + 1, \dots,j_d)$ have the same parity if $j_k\neq 0$ and the opposite parities if $j_k=0$. 
%However, we are not going to be careful about how we label associate pairs (see the discussion just before Lemma \ref{lem:res_chars}).

\begin{Lemma}\label{lem:nmv_char}
Let $1 \leq k \leq d$ and $Q = \Di_1\times \dots \times \hat{\Di }_k \times \dots \times \Di_d \leq \Di $. Then $\Z\prj_Q(\cO \Li \fid)$ is precisely the $\Z$-linear span of all 
\begin{align*}
&\xi_{\rho,j_1,\dots,j_k, \dots ,j_d}^{(\pm)} + \xi_{\rho,j_1,\dots,j_k + 1, \dots,j_d}^{(\pm)}&&\text{with $0<j_k<\ell$},
\\
&\xi_{\rho,j_1,\dots,j_k, \dots ,j_d} + \xi_{\rho,j_1,\dots,j_k + 1, \dots,j_d}^{\pm} && \text{with $j_k=0$ and $(\rho, j_1, \dots, j_d)$ even},
\\
&\xi_{\rho,j_1,\dots,j_k, \dots ,j_d}^+ + \xi_{\rho,j_1,\dots,j_k, \dots ,j_d}^- +\xi_{\rho,j_1,\dots,j_k + 1, \dots,j_d}
&&\text{with $j_k=0$ and $(\rho, j_1, \dots, j_d)$ odd}.
\end{align*}
In particular, $\Z\Irr_{<_\sfs \Di}(\cO \Li \fid) \subseteq \ker(\varphi)$, where $\varphi$ is the $\Z$-linear map defined by
\begin{align*}
\varphi:\Z\Irr(\cO \Li \fid)\to \Z,\ 
\xi_{\rho,j_1,\dots,j_k, \dots ,j_d}^{(\pm)}\mapsto \frac{\eps_\rho \prod_{k=1}^d ((-1)^{j_k}\eps_{j_k})}{\eps_{\rho,j_1,\dots,j_d}} %(-1)^{\sum_{k=1}^d j_k}.
\end{align*}
\end{Lemma}

\begin{proof}
We first note, by inspecting the Brauer tree of $\Blo^{\varnothing,1}$ from Lemma \ref{lem:weight1_equiv}(iii), that $\Z\prj(\Blo^{\varnothing,1})$ is the $\Z$-linear span of $\xi_0+\xi_1^\pm$ and  $\xi_j^{\pm} + \xi_{j+1}^{\pm}$, for $0< j< \ell$, and that 
%It is immediate from Lemma \ref{lem:weight1_equiv}(iii), that 
every projective indecomposable $\Blo^{\varnothing,1}$-module is non-self-associate.

For $r=1$ and $d=1$, the first claim now follows as, in this case, $k=1$, $Q=1$, $\prj_Q(\cO \Li \fid) = \N\prj(\cO \Li \fid)$ and $\Blo^{\varnothing,1} \cong \cO \Li \fid$ with corresponding bijection $\Irr(\Blo^{\varnothing,1}) \to \Irr(\cO \Li \fid)$, $\xi_j^{(\pm)} \mapsto \xi_{(1),j}^{(\pm)}$. From now on we assume $r>1$ or $d>1$.

Recalling the notation (\ref{EUVI}), 
%\begin{align*}\sJ_t = [r+tp] \setminus [r+(t-1)p]\end{align*} and 
we set $\Li_{\hat{\Ai}_k} := \Li  \cap \tAi_{[n]\backslash \sJ_k}$, so $\Li_{\hat{\Ai}_k} \times_z \tSi_{\sJ_k}$ is an index $2$ subgroup of $\Li $. We also set
\begin{align*}
\hat{e}_{\varnothing,1}^{(k)}:= \ei_{\rho,0} \otimes \ei_{\varnothing,1}^{(1)} \otimes \dots \otimes \ei_{\varnothing,1}^{(k-1)} \otimes \ei_{\varnothing,1}^{(k+1)} \otimes \dots \otimes \ei_{\varnothing,1}^{(d)} \in \cO \Li_{\hat{\Ai}_k}.
\end{align*}
(Note this really is in $\cO \Li_{\hat{\Ai}_k}$ since, by Remark \ref{rem:superblocks}, $\ei_{\rho,0} \in \cO\tAi_\sR $ and $\ei_{\varnothing,1}^{(l)} \in \cO\tAi_{\sJ_l}$, for each $l\neq k$.)

As $p \nmid [\Li :\Li_{\hat{\Ai}_k} \times_z \tSi_{\sJ_k}]$, for any indecomposable, relatively $Q$-projective $\cO \Li \fid$-module $\gU$, we have that 
$
\gU\mid \Ind_{\Li_{\hat{\Ai}_k} \times_z \tSi_{\sJ_k}}^\Li \tilde{\gU},
$ 
for some indecomposable, relatively $Q$-projective $\cO(\Li_{\hat{\Ai}_k} \times_z \tSi_{\sJ_k})\fid$-module $\tilde{\gU}$. Now, 
$
\cO(\Li_{\hat{\Ai}_k} \times_z \tSi_{\sJ_k})\fid \cong \cO \Li_{\hat{\Ai}_k}\hat{e}_{\varnothing,1}^{(k)} \otimes (\Blo^{\varnothing,1})^{(k)}
$ 
and $Q \leq \Li_{\hat{\Ai}_k}$. Therefore, $\tilde{\gU} \cong \hat{\gU}_k \boxtimes \gU_k$, for some 
indecomposable $\cO \Li_{\hat{\Ai}_k}\hat{e}_{\varnothing,1}^{(k)}$-module $\hat{\gU}_k$ and projective, indecomposable $(\Blo^{\varnothing,1})^{(k)}$-module $\gU_k$.

Recall that $(\Blo^{\varnothing,1})^{(k)} \cong \Blo^{\varnothing,1}$. Therefore, due to comments in the first paragraph, $\gU_k$ is non-self-associate. It follows that $\gU_k$ is not $\tSi_{[n]\backslash \sJ_k}$-stable and consequently that $\hat{\gU}_k \boxtimes \gU_k$ is not $\Li $-stable. (Here, we are using the fact that $r>1$ or $d>1$ and so $\tAi_{[n]\backslash \sJ_k} \lneq \tSi_{[n]\backslash \sJ_k}$.) Therefore, $\gU$ is actually isomorphic to $\Ind_{\Li_{\hat{\Ai}_k} \times_z \tSi_{\sJ_k}}^\Li (\hat{\gU}_k \boxtimes \gU_k)$.

We have now shown that $\Z\prj_Q(\cO \Li \fid)$ is precisely the $\Z$-linear span of characters of the form 
$
(\hat{\chi}_k \boxtimes \chi_k) \uparrow_{\Li_{\hat{\Ai}_k} \times_z \tSi_{\sJ_k}}^\Li ,
$ 
where $\hat{\chi}_k \in \Irr(\cO \Li_{\hat{\Ai}_k} \hat{e}_{\varnothing,1}^{(k)})$ and $\chi_k \in \prj(\cO \tSi_{\sJ_k} \ei_{\varnothing,1}^{(k)})$. In particular, by the first paragraph of this proof, $\Z\prj_Q(\cO \Li \fid)$ is the $\Z$-linear span of characters of the form 
$
\big(\hat{\chi}_k \otimes (\xi_j^{(\pm)} + \xi_{j+1}^{\pm})\big) \uparrow_{\Li_{\hat{\Ai}_k} \times_z \tSi_{\sJ_k}}^\Li ,
$ 
where $\hat{\chi}_k \in \Irr(\cO \Li_{\hat{\Ai}_k} \hat{e}_k)$ and $0\leq j< \ell$. The first claim now follows from Lemma \ref{lem:A_ten_B}(iii) by considering all four possibilities of $\hat{\chi}_k$ having an even or odd label and $j=0$ or $j>0$.

For the second claim we note that an indecomposable $\cO \Li \fid$-module has non-maximal vertex if and only if it is relatively $Q$-projective, for some $Q$ of the form considered above. One can now simply check that, for all $k$, the given characters are in the kernel of $\varphi$.
\end{proof}

The next lemma is proved in much the same way as Lemma \ref{lem:nmv_char} and we leave the proof to the reader.

\begin{Lemma}\label{lem:nmv_char_H}
Let $\mu \in \Par_0(\rho,d-1)$. The following are in $\Z\Irr_{<_\sfs \Di}(\cO \Hi \ci)$:
\begin{enumerate}
\item $\xi_{\mu,j}^{(\pm)} + \xi_{\mu,j+1}^{(\pm)}$, for $0<j<\ell$.
\item $\xi_{\mu,0} + \xi_{\mu,1}^{\pm}$, for $\mu$ even.
\item $\xi_{\mu,0}^+ + \xi_{\mu,0}^- + \xi_{\mu,1}$, for $\mu$ odd.
\end{enumerate}
\end{Lemma}

\begin{Remark}
Note that Lemmas \ref{lem:nmv_char} and \ref{lem:nmv_char_H} are consistent with the comments preceding Lemma \ref{lem:nmv_char} concerning the labeling of associate pairs. For example, in the case that $\mu \in \Par_0(\rho,d-1)$ is even and $j>0$, we can deduce that
\begin{align*}
\xi_{\mu,j}^+ + \xi_{\mu,j+1}^- 
= & \sum_{\substack{1\leq i<j \\ j-i \text{ odd}}}(\xi_{\mu,i}^+ + \xi_{\mu,i+1}^+) - \sum_{\substack{1\leq i<j  \\ j-i \text{ even}}}(\xi_{\mu,i}^+ + \xi_{\mu,i+1}^+) -  (-1)^j(\xi_{\mu,0} + \xi_{\mu,1}^+)
\\
& + (-1)^j(\xi_{\mu, 0} + \xi_{\mu,1}^-) + \sum_{\substack{1\leq i\leq j  \\ j-i \text{ even}}}(\xi_{\mu,i}^- + \xi_{\mu ,i+1}^-) - \sum_{\substack{1\leq i\leq j  \\ j-i \text{ odd}}}(\xi_{\mu,i}^- + \xi_{\mu ,i+1}^-),
\end{align*}
which, by Lemma \ref{lem:nmv_char_H}, is in $\Z\Irr_{<_\sfs \Di}(\cO \Hi \ci)$. Therefore, our choice of labeling of associate pairs for Lemma \ref{lem:nmv_char_H}(i) was unimportant.
\end{Remark}

For the following lemma recall the notation (\ref{ECI}),(\ref{ECLong}),(\ref{EC'Long}).

\begin{Lemma}\label{lem:vert_prsv}
The maps
\begin{align*}
\bM_\Li^*\otimes_{\cO \Li } \ci_{0,d}\cO \Gi \otimes_{\cO \Gi } ?: \Z \Irr(\cO \Gi \bi) & \to \Z \Irr(\cO \Li \fid) \\
{}_\Li  \bX_\Gi^*\otimes_{\cO \Gi } ?: \Z \Irr(\cO \Gi \bi) & \to \Z \Irr(\cO \Li \fid) \\
\big((\bM_{\Li_{d-1}}^* \otimes_{\cO \Li_{d-1}} \ci_{0,d-1}' \cO \Gi_{d-1} ) \boxtimes (\Blo^{\varnothing,1})^{(d)}\big) \otimes_{\cO \Hi } ?: 
\Z\Irr(\cO \Hi \ci) & \to  \Z\Irr(\cO \Li \fid) \\
\cO \Li \fid \otimes_{\cO \Li } \cO \Ni  \otimes_{\cO \Ni } ?: \Z \Irr(\cO \Ni \fid) & \to \Z \Irr(\cO \Li \fid)
\end{align*}
restrict to maps
\begin{align*}
\bM_\Li^*\otimes_{\cO \Li } \ci_{0,d}\cO \Gi \otimes_{\cO \Gi } ?: \Z \Irr_{<_\sfs \Di}(\cO \Gi \bi) & \to \Z \Irr_{<_\sfs \Di}(\cO \Li \fid) \\
{}_\Li  \bX_\Gi^*\otimes_{\cO \Gi } ?: \Z \Irr_{<_\sfs \Di}(\cO \Gi \bi) & \to \Z \Irr_{<_\sfs \Di}(\cO \Li \fid) \\
\big((\bM_{\Li_{d-1}}^* \otimes_{\cO \Li_{d-1}} \ci_{0,d-1}' \cO \Gi_{d-1} ) \boxtimes (\Blo^{\varnothing,1})^{(d)}\big) \otimes_{\cO \Hi } ?: 
\Z\Irr_{<_\sfs \Di}(\cO \Hi \ci) & \to  \Z\Irr_{<_\sfs \Di}(\cO \Li \fid) \\
\cO \Li \fid \otimes_{\cO \Li } \cO \Ni  \otimes_{\cO \Ni } ?: \Z \Irr_{<_\sfs \Di}(\cO \Ni \fid) & \to \Z \Irr_{<_\sfs \Di}(\cO \Li \fid)
\end{align*}
(For the third map we are assuming that $d>1$.)
Furthermore, the induced maps,
\begin{align*}
\bM_\Li^*\otimes_{\cO \Li } \ci_{0,d}\cO \Gi \otimes_{\cO \Gi } ?: & \Z \Irr(\cO \Gi \bi)/\Z \Irr_{<_\sfs \Di}(\cO \Gi \bi) \to \Z \Irr(\cO \Li \fid)/\Z \Irr_{<_\sfs \Di}(\cO \Li \fid) \\
{}_\Li  \bX_\Gi^*\otimes_{\cO \Gi } ?: & \Z \Irr(\cO \Gi \bi)/\Z \Irr_{<_\sfs \Di}(\cO \Gi \bi) \to \Z \Irr(\cO \Li \fid)/\Z \Irr_{<_\sfs \Di}(\cO \Li \fid),
\end{align*}
coincide.

\end{Lemma}

\begin{proof}
We first show that each of our four bimodules is a direct sum of bimodules each with vertex conatined in $\Delta \Di $.
By Lemma \ref{lem:dual_ML_MN}(i), $\bM_\Li^*\otimes_{\cO \Li } \ci_{0,d}\cO \Gi $ is a direct sum of $(\cO \Li ,\cO \Gi )$-bimodules each with vertex contained in $\Delta \Di $ and, by Lemma \ref{lem:dual_ML_MN}(iii), ${}_\Li  \bX_\Gi^*$ has vertex $\Delta \Di $. By Remarks \ref{rem:induct_Rou}, \ref{rem:bisupmod} and Lemma \ref{lem:direct_prod_vert}, 
\begin{align*}
(\bM_{\Li_{d-1}}^* \otimes_{\cO \Li_{d-1}} \ci_{0,d-1}' \cO \Gi_{d-1} ) \boxtimes (\Blo^{\varnothing,1})^{(d)}
\end{align*}
is also a direct sum of $(\cO \Li ,\cO \Hi )$-bimodules each with vertex contained in $\Delta \Di $. 
Finally, $\cO \Li \fid$ is an $(\cO \Li \fid,\cO \Li \fid)$-bimodule with vertex $\Delta \Di $ and so $\cO \Li \fid \otimes_{\cO \Li } \cO \Ni $ is a direct sum of $(\cO \Li \fid,\cO \Ni \fid)$-bimodules each with vertex contained in $\Delta \Di $.

Lemma \ref{lem:bimod_vert}(i) now tells us that all four bimodules take a module with vertex $Q <_\sfs Q$ to a direct sum of modules each with vertex contained in $\Di  \cap {}^g Q$, for some $g \in \Gi $. We claim that $\Di  \cap {}^g Q <_\sfs Q$, if $Q <_\sfs\Di$. This will complete the proof of the first part of the lemma.
If $g \notin \Ni $, then $\Di  \cap {}^g Q \leq \Di  \cap {}^g \Di <_\sfs\Di$ by Lemma \ref{lem:conj_def}(i). If $g \in \Ni $, then the set of fixed points of ${}^g Q$, and hence $\Di  \cap {}^g Q$, on $[n]$ strictly contains $\sR $.

Now, dualizing Corollary \ref{cor:X_OGcM} using Lemma \ref{lem:sup_vert}, we have
\begin{align*}
{}_\Li  \bX_\Gi^*\,\,|_\Di \,(\cO \Gi  \ci_{0,d} \otimes_{\cO \Li } \bM_\Li )^* \simeq \bM_\Li^* \otimes_{\cO \Li } \ci_{0,d} \cO \Gi ,
\end{align*}
where the isomorphism follows from Lemmas \ref{lem:sup_alg_idempt_dual}, \ref{lem:dual_ML_MN}(i) and \ref{lem:dual_comp}. 
The second part of the Lemma now follows from Lemma \ref{lem:bimod_vert}(i).
\end{proof}

\begin{Lemma}\label{lem:HtoJ_nmv}
Let $\la \in \Par_0(\rho,d)$. The following congruence holds modulo $\Z\Irr_{<_\sfs \Di}(\cO \Hi \ci)$:
\begin{align*}
(\cO \Gi_{d-1}\bi_{d-1}\boxtimes (\bM^{(d)})^*)\otimes_{\cO \Hi } \big(\xi_\la\downarrow_{\Hi ,\ci}^{\Gi ,\bi}\big) \equiv \sum_{j\in I} 
\sum_{\mu\in \Par_0^j(\la)^-} 
\frac{\eps_{\mu ,\la} \eps_{\la,j}}{\eps_\mu \eps_{j}} \,\xi_{\mu,j}.
\end{align*}
\end{Lemma}

\begin{proof}
First note that, by Lemma \ref{lem:HtoG}(ii), every superconstituent of $\xi_\la \downarrow^{\Gi ,b}_{\Hi ,\ci}$ must be of the form $\xi_{\mu,i}$, for some $\mu \in \Par_0^j(\la)^-$ and $i,j \in I$. Moreover, the coefficient is
\begin{align*}
\begin{cases}
\frac{\eps_\la \eps_{\mu, \la} \eps_{i}}{\eps_{\mu,i}} & \text{if }j \leq \ell - i \\
0 & \text{otherwise.}
\end{cases}
\end{align*}
Let $j\in I$ and $\mu \in \Par_0^j(\la)^-$. By Lemma \ref{lem:HtoH}, we now have that, for $m \in I$, $\xi_{\mu ,m}$ appears as a superconstituent in
$
(\cO \Gi_{d-1}\bi_{d-1}\boxtimes (\bM^{(d)})^*)\otimes_{\cO \Hi } \big(\xi_{\la}\downarrow_{\Hi ,\ci}^{\Gi ,\bi}\big)
$
with coefficient
\begin{align*}
&\sum_{i=0}^{\min\{\ell-j,\ell-m\}}\frac{\eps_\la \eps_{\mu , \la} \eps_{i}}{\eps_{\mu ,i}}\frac{\eps_{\mu,i}\eps_{i}\eps_{m}}{\eps_{\mu,m}} 
&&= \ \ \  \sum_{i=0}^{\min\{\ell-j,\ell-m\}}\frac{\eps_\la \eps_{\mu ,\la} \eps_{i}^2 \eps_{m}}{\eps_{\mu ,m}}\\
= &
\begin{cases}
\sum_{i=0}^{\ell - j}\frac{\eps_\la \eps_{\mu , \la} \eps_{i}^2 \eps_{m}}{\eps_{\mu ,m}}& \text{if }0\leq m\leq j \\
\sum_{i=0}^{\ell - m}\frac{\eps_\la \eps_{\mu , \la} \eps_{i}^2 \eps_{m}}{\eps_{\mu ,m}}& \text{if }j< m\leq \ell
\end{cases}
&&= \ \ \ 
\begin{cases}
(2\ell - 2j + 1)\frac{\eps_\la \eps_{\mu , \la} \eps_{m}}{\eps_{\mu ,m}}& \text{if }0\leq m\leq j \\
(2\ell - 2m + 1)\frac{\eps_\la \eps_{\mu , \la} \eps_{m}}{\eps_{\mu ,m}}& \text{if }j< m\leq \ell.
\end{cases}
\end{align*}
We have now shown that
\begin{align}
\begin{split}\label{algn:char_equal}
&\,\big(\cO \Gi_{d-1}\bi_{d-1}\boxtimes (\bM^{(d)})^*\big)\otimes_{\cO \Hi } \big(\xi_\la\downarrow_{\Hi ,\ci}^{\Gi ,\bi}\big) \\
= & \sum_{j\in I}\hspace{-1mm}\sum_{\mu\in \Par_0^j(\la)^-} \hspace{-2.5mm}\Big( \hspace{-.5mm}\sum_{m=0}^j (2\ell - 2j + 1)\frac{\eps_\la \eps_{\mu ,\la} \eps_{m}}{\eps_{\mu,m}} \xi_{\mu,m} 
 + \hspace{-2.5mm} \sum_{m=j+1}^\ell (2\ell - 2m + 1)\frac{\eps_\la \eps_{\mu , \la} \eps_{m}}{\eps_{\mu,m}} \xi_{\mu,m} \Big).
\end{split}
\end{align}
If $\mu$ is even, we have the following congruence modulo $\Z\Irr_{<_\sfs \Di}(\cO \Hi \ci)$:
\begin{align}
\begin{split}\label{algn:mu_even}
\sum_{m=0}^j (2\ell - 2j + 1)\frac{\eps_\la \eps_{\mu , \la} \eps_{m}}{\eps_{\mu,m}} \xi_{\mu,m} 
=  \sum_{m=0}^j (2\ell - 2j + 1) \eps_\la^2 \xi_{\mu,m}
\equiv (2\ell - 2j + 1)\frac{\eps_\la^2}{\eps_{j}^2} \xi_{\mu,j},
\end{split}
\end{align}
where the congruence follows by subtracting
\begin{align*}
\textstyle\sum_{m=0}^{j-1} (2\ell - 2j +1) \frac{\eps_\la^2}{\eps_{j}^2} \Big(\frac{2}{\eps_{m}^2}\xi_{\mu,m} + \xi_{\mu,m+1}\Big),
\end{align*}
which, by Lemma \ref{lem:nmv_char_H}(i),(ii), is in $\Z\Irr_{<_\sfs \Di}(\cO \Hi \ci)$. (Note that, since $\mu$ is even, by Lemma \ref{ChuKesLem}(iii), $\la$ must be odd, unless $j=0$. Therefore, $\eps_\la^2/\eps_{j}^2$ is an integer.)

If $\mu$ is odd, we have the following congruence modulo $\Z\Irr_{<_\sfs \Di}(\cO \Hi \ci)$:
\begin{align}
\begin{split}\label{algn:mu_odd}
\sum_{m=0}^j (2\ell - 2j + 1)\frac{\eps_\la \eps_{\mu , \la} \eps_{m}}{\eps_{\mu ,m}} \xi_{\mu,m}  = \sum_{m=0}^j (2\ell - 2j + 1) \eps_{m}^2 \xi_{\mu,m} 
 \equiv (2\ell - 2j + 1) \xi_{\mu,j},
\end{split}
\end{align}
where the congruence follows by subtracting
\begin{align*}
\textstyle\sum_{m=0}^{j-1} (2\ell - 2j +1) (\xi_{\mu,m} + \xi_{\mu,m+1}),
\end{align*}
which, by Lemma \ref{lem:nmv_char_H}(i),(iii), is in $\Z\Irr_{<_\sfs \Di}(\cO \Hi \ci)$.

Next note that, for $j < \ell$, we have the following congruence modulo $\Z\Irr_{<_\sfs \Di}(\cO \Hi \ci)$:
\begin{align}
\begin{split}\label{algn:mu_both}
&\sum_{m=j+1}^\ell (2\ell - 2m + 1)\frac{\eps_\la \eps_{\mu , \la} \eps_{m}}{\eps_{\mu ,m}} \xi_{\mu,m} \\
= & \sum_{m=j+1}^\ell (2\ell - 2m + 1) \eps_\la \eps_{\mu ,\la} \eps_\mu \xi_{\mu ,m} \equiv (\ell - j)\eps_\la \eps_{\mu, \la} \eps_\mu \xi_{\mu ,j+1},
\end{split}
\end{align}
where the equality holds since $(p-m,m)$ is odd for $m>0$ and the congruence holds by subtracting
\begin{align*}
\textstyle\sum_{m=j+1}^{\ell-1} (\ell - m) \eps_\la \eps_{\mu , \la} \eps_\mu  (\xi_{\mu,m} + \xi_{\mu,m+1}),
\end{align*}
which, by Lemma \ref{lem:nmv_char_H}(i), is in $\Z\Irr_{<_\sfs \Di}(\cO \Hi \ci)$.

Denote by $\Par_0^j(\la)^-_\0$ (resp. $\Par_0^j(\la)^-_\1)$ the set of all even (resp. odd) partitions in $\Par_0^j(\la)^-$. 
Putting (\ref{algn:char_equal}), (\ref{algn:mu_even}), (\ref{algn:mu_odd}) and (\ref{algn:mu_both}) together, we now have that
\begin{align*}
(\cO \Gi_{d-1}\bi_{d-1}\boxtimes (\bM^{(d)})^*)\hspace{-.3mm}\otimes_{\cO \Hi } \big(\xi_\la\downarrow_{\Hi ,\ci}^{\Gi ,\bi}\big)\hspace{-1mm} 
\equiv &\sum_{j\in I}\hspace{-.5mm}\sum_{\mu\in \Par_0^j(\la)^-_\0} \hspace{-5mm}\big((2\ell - 2j + 1) \frac{\eps_\la^2}{\eps_{j}^2} \xi_{\mu,j} + (\ell - j) \eps_\la^2 \xi_{\mu,j+1}\big) \\
+ & \sum_{j\in I}\sum_{\mu\in \Par_0^j(\la)^-_\1 } \big((2\ell - 2j + 1) \xi_{\mu,j} + 2(\ell - j) \xi_{\mu,j+1}\big)
\end{align*}
modulo $\Z\Irr_{<_\sfs \Di}(\cO \Hi \ci)$. Finally, we subtract
\begin{align*}
\sum_{j\in I}\sum_{\mu\in \Par_0^j(\la)^-_\0 } (\ell - j) \eps_\la^2 \Big(\frac{2}{\eps_{j}^2}\xi_{\mu,j} + \xi_{\mu,j+1}\Big) 
+  \sum_{j\in I}\sum_{\mu\in \Par_0^j(\la)^-_\1 } 2(\ell - j) (\xi_{\mu,j} + \xi_{\mu,j+1}),
\end{align*}
which, by Lemma \ref{lem:nmv_char_H}, is in $\Z\Irr_{<_\sfs \Di}(\cO \Hi \ci)$, giving
\begin{align*}
&(\cO \Gi_{d-1}\bi_{d-1}\boxtimes (\bM^{(d)})^*)\otimes_{\cO \Hi } \big(\xi_\la\downarrow_{\Hi ,\ci}^{\Gi ,\bi}\big)
\equiv \sum_{j\in I}\sum_{\mu\in \Par_0^j(\la)^-_\0 } \frac{\eps_\la^2}{\eps_{j}^2}\xi_{\mu,j} + 
\sum_{j\in I}
\sum_{\mu\in \Par_0^j(\la)^-_\1} \xi_{\mu,j}
\end{align*}
modulo $\Z\Irr_{<_\sfs \Di}(\cO \Hi \ci)$. This completes the claim noting that, by Lemma \ref{ChuKesLem}(iii), $\la$ and $\mu$ must have opposite parity unless $j=0$.
\end{proof}

Recall the notation $\La(I,d)$ and (\ref{EBinom}),(\ref{EDUnderline}) from \S\ref{sec:comb}. 
For $k\in I$, we denote 
$$
\ude_k:=(0,\dots,0,1,0,\dots,0)\in\La(I,1)
$$
with $1$ in the $k$th position. We can add the compositions coordinate-wise. For example for $\ud=(d_0,\dots,d_\ell)\in\La(I,d)$ we have 
$$
\ud+\ude_k=(d_0,\dots,d_{k-1},d_k+1,d_{k+1},\dots,d_\ell)\in\La(I,d+1). 
$$
Similarly we have $\ud-\ude_k\in\La(I,d-1)$ makes sense if $d_k>0$.

For $\ud=(d_0,\dots,d_\ell)\in\La(I,d)$. 
We now introduce the special notation
\begin{align*}
\xi_{\rho,\ud}:=\xi_{\rho,0^{d_0}, \dots ,\ell^{d_\ell}}^{(\pm)}:=
\xi_{\rho,0,\dots,0,\, \dots ,\,\ell,\dots,\ell}^{(\pm)}\in\Irr(\K \Li \fid)
\ \text{and}\ 
\eps_{\rho,\ud}:=
\eps_{\rho,0^{d_0}, \dots ,\ell^{d_\ell}}:=
\eps_{\rho,0,\dots,0,\, \dots ,\,\ell,\dots,\ell}
\end{align*}
where $i$ is repeated $d_i$ times for all $i\in I$.

Similarly, for $\ud\in \La(I,d-1)$ and $k\in I$, we can make sense of $\xi_{\rho,\ud,k}^{(\pm)}$, $\xi_{\rho,k,\ud}^{(\pm)}$, $\eps_{\rho,\ud,k}$, etc:$$
\xi_{\rho,\ud,k}:=\xi_{\rho,0^{d_0}, \dots ,\ell^{d_\ell},k}^{(\pm)}\in\Irr(\K \Li \fid)\ \text{and}\ 
\eps_{\rho,\ud,k}:=\eps_{\rho,0^{d_0}, \dots ,\ell^{d_\ell},k}^{(\pm)}
$$
Note that $\eps_{\rho,\ud,k}=\eps_{\rho,\ud+\ude_k}$ but similar equality is not in general true for the $\xi$'s.

$\ud\in\La(I,d)$, we denote by 
$\xi_{\rho,\ud}^{\tSi_d}$
the sum of the irreducible character $\xi_{\rho,\ud}^{(\pm)}$ with all its $\Ni $-conjugates and their associates. 
Note that this is the sum of exactly 
$
\binom{d}{\ud} \eps_{\rho,\ud}^2
$
irreducible characters of $\K \Li \fid$. Similarly, for $\ud\in\La(I,d-1)$, we define 
$\xi_{\rho,\ud,k}^{\tSi_{d-1}}$ to be the sum of the irreducible character $\xi_{\rho,\ud ,k}^{(\pm)}$ with all its $\Ni_{d-1}$-conjugates and their associates. Note that it is the sum of exactly 
$
\binom{d-1}{\ud} \eps_{\rho,\ud+\ude_k}^2
$
irreducible characters of $\K \Li \fid$.

Recall the notation (\ref{EK(la)}) and other combinatorial notation from \S\ref{sec:comb}. 

\begin{Lemma}\label{lem:restrict_calc}
Let $\ud\in\La(I,d)$ and $\la\in\Par_0(\rho,\ud)$. 
Then 
%we have the following congruence modulo\, $\Z\Irr_{<_\sfs \Di}(\K \Li \fid)$:
\begin{align*}
{}_\Li  \bX_\Gi^*\otimes_{\cO \Gi } \xi_\la
\equiv \frac{\eps_{\la}\,2^{(|\la^{(0)}| - h(\la^{(0)}))/2}
}{\eps_{\rho,\ud}} \,K(\la)\,\xi_{\rho,\ud}^{\tSi_d}\pmod{\Z\Irr_{<_\sfs \Di}(\K \Li \fid)}.
\end{align*}
\end{Lemma}

\begin{proof}
Throughout this proof all congruences will assumed to be modulo $\Z\Irr_{<_\sfs \Di}(\K \Li \fid)$.

We note that, with the last part of Lemma \ref{lem:vert_prsv} in mind, it is enough to prove that
\begin{align}
\label{algn:ML_OG_ind}
 \bM_\Li^*\otimes_{\cO \Li } \ci_{0,d}\cO \Gi \otimes_{\cO \Gi } \xi_\la 
\equiv  \frac{\eps_{\la}2^{(|\la^{(0)}| - h(\la^{(0)}))/2}}{\eps_{\rho,\ud}} 
K(\la)\xi_{\rho,\ud}^{\tSi_d}.
\end{align}

We prove (\ref{algn:ML_OG_ind}) by induction on $d$. We first check that the statement holds for $d=1$. In this case $|\la^{(0)}| = h(\la^{(0)})$, $K_{\la^{(0)}}' = K_{\la^{(1)}} = \dots = K_{\la^{(\ell)}} = 1$ and $\ud=\ude_k$ for some $k\in I$. By Lemma \ref{ChuKesLem}(iii), $\eps_\la = \eps_{\rho ,k}$.
%, where $\la \in \Par_0^k(\rho)^+$. 
We also have $\ci_0 = \ci$ and $\ci_1 = \bi$. We, therefore, need to show that
\begin{align*}
\bM_\Li^* \otimes_{\cO \Li } \ci\cO \Gi  \bi\otimes_{\cO \Gi }\xi_\la \equiv \frac{\eps_\la}{\eps_{\rho ,k}}\xi_{\rho,k} = \xi_{\rho,k} .
\end{align*}
This now agrees with taking $d=1$ in Lemma \ref{lem:HtoJ_nmv}, as applying Lemma \ref{ChuKesLem}(iii) again gives
\begin{align*}
\frac{\eps_{\rho , \la} \eps_{\la ,k}}{\eps_\rho \eps_{k}} = \frac{\eps_{k} \eps_{\rho}}{\eps_\rho \eps_{k}} = 1.
\end{align*}
We now assume $d>1$ and that (\ref{algn:ML_OG_ind}) holds for all $\mu\in\Par(\rho,d-1)$. Note that, by Remark \ref{rem:induct_Rou}, the inductive steps we make below are valid.

%Recall the $\ci_k'$'s as defined in $\S$\ref{sec:notation}. 
Now,
\begin{align*}
&\bM_\Li^* \otimes_{\cO \Li } \ci_{0,d}\cO \Gi 
\\ \simeq\, & (\bM_\Li^* \otimes_{\cO \Li } \ci_{0,d-1} \cO \Hi ) \otimes_{\cO \Hi } \cO \Gi  \ci_d\\
\simeq\, & (\bM_{\Li_{d-1}}^* \boxtimes (\bM^{(d)})^*) \otimes_{\cO \Li } \big(\ci_{0,d-1}' \cO \Gi_{d-1} \boxtimes (\Blo^{\varnothing,1})^{(d)}\big) \otimes_{\cO \Hi } \cO \Gi  \ci_d\\
\simeq\, & \big((\bM_{\Li_{d-1}}^* \otimes_{\cO \Li_{d-1}} \ci_{0,d-1}' \cO \Gi_{d-1} ) \boxtimes (\bM^{(d)})^*\big) \otimes_{\cO \Hi } \cO \Gi  \ci_d \\
\simeq\, & \big((\bM_{\Li_{d-1}}^* \hspace{-1 mm}\otimes_{\cO \Li_{d-1}} \ci_{0,d-1}' \cO \Gi_{d-1} ) \boxtimes (\Blo^{\varnothing,1})^{(d)}\big) \hspace{-.5 mm}\otimes_{\cO \Hi } (\cO \Gi_{d-1} \ci_{d-1}' \hspace{-.3 mm}\boxtimes (\bM^{(d)})^*) \hspace{-.5 mm}\otimes_{\cO \Hi } \cO \Gi  \ci_d,
\end{align*}
where the second isomorphism follows from Lemmas \ref{lem:dual_ML_MN}(i) and \ref{lem:dual_prod} and the third and fourth from Lemma \ref{lem:tensor_bisupmod}. Since $\ci_d=\bi$, $\ci_{d-1}=\ci$ and $\ci_{d-1}' = \bi_{d-1}$, Lemma \ref{lem:HtoJ_nmv} now gives
\begin{align*}
&\bM_\Li^* \otimes_{\cO \Li } \ci_{0,d}\cO \Gi \otimes_{\cO \Gi } \xi_\la
\\
=\, & \big((\bM_{\Li_{d-1}}^* \hspace{-.4 mm}\otimes_{\cO \Li_{d-1}} \ci_{0,d-1}' \cO \Gi_{d-1} ) \boxtimes (\Blo^{\varnothing,1})^{(d)}\big) \hspace{-.5 mm}\otimes_{\cO \Hi } \hspace{-.25 mm}
 \big(\cO \Gi_{d-1} \bi_{d-1} \boxtimes (\bM^{(d)})^*\big) \hspace{-.5 mm}\otimes_{\cO \Hi } \big(\xi_\la\downarrow_{\Hi ,\ci}^{\Gi ,\bi}\big)\\
\equiv\, & \big((\bM_{\Li_{d-1}}^* \otimes_{\cO \Li_{d-1}} \ci_{0,d-1}' \cO \Gi_{d-1} ) \boxtimes (\Blo^{\varnothing,1})^{(d)}\big) \otimes_{\cO \Hi } \Big( \sum_{k\in I}\sum_{\mu\in \Par_0^k(\la)^-} \frac{\eps_{\mu , \la} \eps_{\la ,k}}{\eps_\mu \eps_{k}} \xi_{\mu ,k}\Big).
\end{align*}
Implicit in the above calculation is that, by Lemma \ref{lem:vert_prsv},
\begin{align*}
\big((\bM_{\Li_{d-1}}^* \otimes_{\cO \Li_{d-1}} \ci_{0,d-1}' \cO \Gi_{d-1} ) \boxtimes (\Blo^{\varnothing,1})^{(d)}\big) &\otimes_{\cO \Hi } ?
\end{align*}
maps $\Z\Irr_{<_\sfs \Di}(\K \Hi \ci)$ to $\Z\Irr_{<_\sfs \Di}(\K \Li \fid)$. Now, by the inductive hypothesis and Lemma \ref{lem:AboxM},
\begin{align}
\begin{split}\label{algn:ML_OG_la}
 &\bM_\Li^* \otimes_{\cO \Li } \ci_{0,d}\cO \Gi \otimes_{\cO \Gi } \xi_\la 
 \\
 \equiv 
&\sum_{k\in I}\sum_{\mu\in \Par_0^k(\la)^-} \frac{\eps_{\mu , \la} \eps_{\la ,k}}{\eps_\mu \eps_{k}} \frac{\eps_{\rho,\ud-\ude_k} \eps_{\mu,k}}{\eps_\mu \eps_{\rho,\ud-\ude_k,k}}
\frac{\eps_\mu}{\eps_{\rho,\ud-\ude_k}} 
K(\mu) 2^{(|\mu^{(0)}| - h(\mu^{(0)}))/2}\xi_{\rho,\ud-\ude_k ,k}^{\tSi_{d-1}}
\\
=&\sum_{k\in I}\sum_{\mu\in \Par_0^k(\la)^-} \frac{\eps_{\mu , \la} \eps_{\la ,k} \eps_{\mu,k}}{\eps_\mu \eps_{k}\eps_{\rho,\ud}} \
K(\mu) 2^{(|\mu^{(0)}| - h(\mu^{(0)}))/2}\xi_{\rho,\ud-\ude_k ,k}^{\tSi_{d-1}}.
\end{split}
\end{align}
%Of course, unless $i=k$, $\mu^{(i)} = \la^{(i)}$.

We now split this sum into two summands, one for $k=0$ and one for $k\neq 0$. For the $k=0$ summand, by Lemma \ref{ChuKesLem}(iii), we have  
$
2^{(|\la^{(0)}|-h(\la^{(0)}))/2} = 2^{(|\mu^{(0)}|-h(\mu^{(0)}))/2} \eps_{\mu , \la}.
$ 
Noting also that in this case $(p-0,0)$ is even, we get
\begin{align}\label{algn:k=0}
\begin{split}
 &
\sum_{\mu\in \Par_0^k(\la)^-} \frac{\eps_{\mu , \la} \eps_{\la ,0} \eps_{\mu,0}}{\eps_\mu \eps_{0}\eps_{\rho,\ud}} \
K(\mu) 2^{(|\mu^{(0)}| - h(\mu^{(0)}))/2}\xi_{\rho,\ud-\ude_0 ,0}^{\tSi_{d-1}}
\\
= & \sum_{\mu^{(0)} \in \Par_0(\la^{(0)})^{-1}} \frac{\eps_\la}{\eps_{\rho,\ud}}K_{\mu^{(0)}}' K_{\la^{(1)}}\dots K_{\la^{(\ell)}}
2^{(|\la^{(0)}| - h(\la^{(0)}))/2}\xi_{\rho,\ud-\ude_0 ,0}^{\tSi_{d-1}}\\
= & \frac{\eps_\la}{\eps_{\rho,\ud}}K(\la) 
2^{(|\la^{(0)}| - h(\la^{(0)}))/2}\xi_{\rho,\ud-\ude_0 ,0}^{\tSi_{d-1}},
\end{split}
\end{align}
where the second equality follows from the comments at the end of $\S$\ref{sec:comb} and the third from Lemma \ref{lem:count_tab}(ii).

For the $k>0$ summands, we have  
$
2^{(|\la^{(0)}|-h(\la^{(0)}))/2} = 2^{(|\mu^{(0)}|-h(\mu^{(0)}))/2}
$ 
and $(p-k,k)$ is always odd. Also, by Lemma \ref{ChuKesLem}(iii), $\mu$ and $\la$ always have opposite parity. Therefore, we get for the $k$th summand:
\begin{align}\label{algn:k>0}
\begin{split}
& \sum_{\mu\in \Par_0^k(\la)^-} \frac{\eps_{\mu , \la} \eps_{\la ,k} \eps_{\mu,k}}{\eps_\mu \eps_{k}\eps_{\rho,\ud}} \
K(\mu) 2^{(|\mu^{(0)}| - h(\mu^{(0)}))/2}\xi_{\rho,\ud-\ude_k ,k}^{\tSi_{d-1}}
\\
= & \sum_{\mu\in \Par_0^k(\la)^-} \frac{\eps_{\mu,k}}{\eps_{\rho,\ud}}K_{\la^{(0)}}' K_{\la^{(1)}} \dots K_{\mu^{(k)}} \dots K_{\la^{(\ell)}}
2^{(|\la^{(0)}| - h(\la^{(0)}))/2}
\xi_{\rho,\ud-\ude_k ,k}^{\tSi_{d-1}} 
\\
= & \sum_{\mu^{(k)}\in \Par(\la^{(k)})^{-1}} \frac{\eps_\la}{\eps_{\rho,\ud}}K_{\la^{(0)}}' K_{\la^{(1)}} \dots K_{\mu^{(k)}} \dots K_{\la^{(\ell)}}
2^{(|\la^{(0)}| - h(\la^{(0)}))/2}\xi_{\rho,\ud-\ude_k ,k}^{\tSi_{d-1}} 
\\
= & \frac{\eps_\la}{\eps_{\rho,\ud}}K(\la)
2^{(|\la^{(0)}| - h(\la^{(0)}))/2}\xi_{\rho,\ud-\ude_k ,k}^{\tSi_{d-1}},
\end{split}
\end{align}
where, again, the second equality follows from the comments at the end of $\S$\ref{sec:comb} and the third from Lemma \ref{lem:count_tab}(i).

Putting (\ref{algn:ML_OG_la}), (\ref{algn:k=0}) and (\ref{algn:k>0}) together, we have now shown that
\begin{align*}
&\bM_\Li^* \otimes_{\cO \Li } \ci_{0,d} \cO \Gi  \otimes_{\cO \Gi }\xi_\la 
\equiv\,\,& \sum_{k \in I} \frac{\eps_\la}{\eps_{\rho,\ud}} K(\la)\,
2^{(|\la^{(0)}| - h(\la^{(0)}))/2}\xi_{\rho,\ud-\ude_k ,k}^{\tSi_{d-1}} 
=\,\, &\frac{\eps_\la}{\eps_{\rho,\ud}} K(\la)\xi_{\rho,\ud}^{\tSi_d},
\end{align*}
as desired.
\end{proof}

\begin{Lemma}\label{lem:ith_runner}
Let $\mu \in \Par_0(\rho,d-1)$, $\la \in \Par_0(\rho,d)$ and $\ud=(d_0,\dots,d_\ell)\in\La(I,d)$.
\begin{enumerate}
\item Let $i \in I$. If $\xi_\la$ appears as a superconstituent with non-zero coefficient in 
$
\bY \otimes_{\cO \Hi } \xi_{\mu,i},
$ 
then $\la \in \Par_0^j(\mu)^+$, for some $j\in I$. %\textcolor{red}{(Note we are not yet saying that $i=j$.)}
\item Say $\xi_\la$ appears with non-zero coefficient in 
$
{}_\Gi  \bX_\Li  \otimes_{\cO \Li } \xi_{\rho,\ud}.
$ 
If $d_i>0$, for some $i\in I$, then $|\la^{(i)}|>0$.
\item Suppose $\bX_{d-1}$ induces a Morita superequivalence between $\cO \Ni_{d-1}\fid_{d-1}$ and $\cO \Gi_{d-1}\bi_{d-1}$. If $\xi_\la$ appears with non-zero coefficient in 
$
\bY \otimes_{\cO \Hi } \xi_{\mu,i},
$ 
for some $i \in I$, then $|\la^{(i)}|>0$.
\end{enumerate}
\end{Lemma}

\begin{proof}
(i) By its definition and Lemma \ref{lem:X_as_left_mod}(ii), $\bY$ is a direct summand of
\begin{align*}
\cO \Gi \bi \otimes_{\cO \Hi }(\cO \Gi_{d-1}\bi_{d-1} \boxtimes \bM^{(d)}).
\end{align*}
The statement now follows from Lemma \ref{lem:HtoH} and Lemma \ref{lem:HtoG}(i).

(ii)
By Corollary \ref{cor:X_OGcM} and Remark \ref{rem:induct_Rou}, ${}_{\Gi_1}\bX_{\Li_1}$ is the unique indecomposable summand of $\bi_1\cO \Gi_1 \otimes_{\cO \Li_1} \bM_{\Li_1}$ with vertex $\Delta \Di_1$ and all other summands are projective. Therefore, ${}_{\Gi_1}\bX_{\Li_1}$ is isomorphic to the $(\cO \Gi_1 \bi_1,\cO \Li_1 \fid_1)$-bisupermodule $\gV$ from Proposition \ref{prop:define_M}(ii). Now,
\begin{align}\label{algn:OHc1c0}
\cO \Hi_1 \ci_{0,1} = \bi_1 \cO \Gi_1 \fid_1 \otimes (\Blo^{\varnothing,1})^{(2)} \otimes \dots \otimes (\Blo^{\varnothing,1})^{(d)}.
\end{align}
Therefore,
\begin{align*}
\cO \Gi  \ci_{0,d} \otimes_{\cO \Li } \bM_\Li  \simeq & \cO \Gi  \ci_{0,d} \otimes_{\cO \Li } (\bM_{\Li_1} \boxtimes \bM^{(2)} \boxtimes \dots \boxtimes \bM^{(d)})\\
\simeq & \cO \Gi  \ci_{2,d} \otimes_{\cO \Hi_1} \cO \Hi_1 \ci_{0,1}  \otimes_{\cO \Li } (\bM_{\Li_1} \boxtimes \bM^{(2)} \boxtimes \dots \boxtimes \bM^{(d)})\\
\simeq & \cO \Gi  \ci_{2,d} \otimes_{\cO \Hi_1} \left( (\bi_1 \cO \Gi_1 \otimes_{\cO \Li_1} \bM_{\Li_1}) \boxtimes \bM^{(2)} \boxtimes \dots \boxtimes \bM^{(d)})\right),
\end{align*}
where the final isomorphism follows from (\ref{algn:OHc1c0}) and Lemma \ref{lem:tensor_bisupmod}.

Since each $\bM^{(k)}$ has vertex $\Delta \Di_k$, it follows from Remark \ref{rem:bisupmod} and Lemma \ref{lem:direct_prod_vert} that this last bimodule has
\begin{align}\label{algn:X_direct_sum}
\cO \Gi  \ci_{2,d} \otimes_{\cO \Hi_1} (\bX_1 \boxtimes \bM^{(2)} \boxtimes \dots \boxtimes \bM^{(d)})
\end{align}
as a direct summand and all other summands have vertex strictly smaller than $\Delta \Di $. In particular, by Corollary \ref{cor:X_OGcM}, ${}_\Gi  \bX_\Li $ is a direct summand of the bimodule in (\ref{algn:X_direct_sum}). We now move towards proving the claim.

Since $\bX$ is actually an $(\cO \Gi ,\cO \Ni )$-bimodule, we can assume that $\xi_i$ appears in the first position. (If it doesn't, then apply an appropriate element of $\Ni $ to ensure that it does.) In other words, $\xi_\la$ appears with non-zero coefficient in 
$
{}_\Gi  \bX_\Li  \otimes_{\cO \Li } \xi_{\rho,i,\ud-\ude_i}.
$ 
Therefore, by the comments following (\ref{algn:X_direct_sum}), $\xi_\la$ appears with non-zero coefficient in
\begin{align*}
\cO \Gi  \ci_{2,d} & \otimes_{\cO \Hi_1} \big(\bX_1 \boxtimes \bM^{(2)} \boxtimes \dots \boxtimes \bM^{(d)} \big) \otimes_{\cO \Li } \xi_{\rho,i,\ud-\ude_i}.
\end{align*}
Proposition \ref{prop:define_M}(ii) and Lemma \ref{lem:AboxM} now imply that $\xi_\la$ appears with non-zero coefficient in
\begin{align*}
\cO \Gi  \ci_{2,d} & \otimes_{\cO \Hi_1} \big(\cO \Gi_1 \boxtimes \bM^{(2)} \boxtimes \dots \boxtimes \bM^{(d)} \big) \otimes_{\cO \Li } \xi_{\rho^i,\ud-\ude_i},
\end{align*}
where we take the same notation of $\rho^i$ from Section \ref{sec:weight_one}. The claim now follows by repeatedly applying Lemmas \ref{prop:define_M}(iv) and \ref{lem:AboxM} and then repeated applying Lemma \ref{lem:HtoG} and Lemma \ref{lem:AboxM}.

(iii) As $\bX_{d-1}$ induces a Morita superequivalence between $\cO \Ni_{d-1}\fid_{d-1}$ and $\cO \Gi_{d-1}\bi_{d-1}$, ${}_{\Li_{d-1}}\bX_{\Gi_{d-1}}^* \otimes_{\cO \Gi_{d-1}} [\xi_\mu]$ must be non-zero. So, by Lemma \ref{lem:chars_bisupmod}(ii), there exist some $j_1, \dots, j_{d-1} \in I$ such that $[\xi_\mu]$ appears as a superconstituent with non-zero coefficient in 
$%\begin{align*}
{}_{\Gi_{d-1}}\bX_{\Li_{d-1}} \otimes_{\cO \Li_{d-1}} \xi_{\rho,{j_1} , \dots,j_{d-1}}.
$ %\end{align*}
Now, by Lemmas \ref{lem:tensor_isom}(ii) and \ref{lem:AboxM}, $\xi_\la$ appears with non-zero coefficient in
\begin{align*}
{}_\Gi  \bX_\Li  \otimes_{\cO \Li } \xi_{\rho,j_1, \dots , j_{d-1},i} 
=  \bY \otimes_{\cO \Hi } \big({}_{\Gi_{d-1}} \bX_{\Li_{d-1}} \boxtimes (\Blo^{\varnothing,1})^{(d)}\big)\otimes_{\cO \Li } \xi_{\rho,j_1, \dots,j_{d-1},i}.
\end{align*}
The claim now follows from part (ii).
\end{proof}

\subsection{Main argument}\label{sec:main_arg}

The remainder of this section is concerned with proving that $\bX$ induces a Morita equivalence between $\cO \Ni \fid$ and $\cO \Gi \bi$. In fact, we will prove the following pair of statements via induction on $m$.

\begin{Hypothesis}\label{hyp:ind_on_X}
Let $1 \leq m \leq d$.
\begin{enumerate}
\item $\bX_m$ induces a Morita superequivalence between $\cO \Ni_m\fid_m$ and $\cO \Gi_m\bi_m$.
\item We have the following character identities:
\begin{enumerate}
\item for all 
$\ud\in\La(I,m)$: 
\begin{align*}
{}_{\Gi_m}\bX_{\Li_m} \otimes_{\cO \Li_m} \xi_{\rho,\ud} = 
\sum_{\la \in \Par(\rho,\ud)}
\frac{\eps_{\rho,\ud}}{\eps_{\la}}\, 2^{(|\la^{(0)}| - h(\la^{(0)}))/2}
K(\la)\xi_\la.
\end{align*}

\item 
for all 
$\ud\in\La(I,m)$ and all $\la \in \Par_0(\rho,\ud)$:
\begin{align*}
{}_{\Li_m}\bX_{\Gi_m}^* \otimes_{\cO \Gi_m} \xi_\la = \frac{\eps_{\la}}{\eps_{\rho,\ud}} \,2^{(|\la^{(0)}| - h(\la^{(0)}))/2}
K(\la)\xi_{\rho,\ud}^{\tSi_m}.
\end{align*}
\end{enumerate}
\end{enumerate}
\end{Hypothesis}

\begin{Remark}
The two conditions in Hypothesis \ref{hyp:ind_on_X}(ii) are easily seen to be equivalent using Lemma \ref{lem:chars_bisupmod}(ii) and the fact that $\bX$ is actually an $(\cO \Gi ,\cO \Ni )$-bimodule.
\end{Remark}

Our main argument involves proving Hypothesis \ref{hyp:ind_on_X} by induction. The next two propositions do most of the work in the subsequent inductive argument.

\begin{Proposition}\label{prop:main_char_calc}
Let $d>1$. If Hypothesis \ref{hyp:ind_on_X} holds for $1 \leq m \leq d-1$, then \ref{hyp:ind_on_X}(ii) holds for $m = d$.
\end{Proposition}

\begin{proof}
We will assume throughout this proof that Hypothesis \ref{hyp:ind_on_X} holds for $1 \leq m \leq d-1$. In particular, we are assuming that $\bX_m$ induces a Morita superequivalence between $\cO \Ni_m \fid_m$ and $\cO \Gi_m \bi_m$, for all $1 \leq m \leq d$. We will implicitly use this fact at several points during the proof.

For $\la \in \Par_0(\rho,m)$, we set
\begin{align*}
\tilde K_{\la^{(i)}}: =
\begin{cases}
K_{\la^{(0)}}' &\text{ if } i=0\\
K_{\la^{(i)}} &\text{ if } i>0.
\end{cases}
\end{align*}
The majority of the remaining argument is dedicated to showing that
\begin{align}\label{L_to G_all_i}
{}_\Gi  \bX_\Li  \otimes_{\cO \Li } \xi_{\rho,i^d} = \sum_{\substack{\la \in \Par(\rho,d) \\ |\la^{(i)}| = d}}\frac{\eps_{\rho,i^d}}{\eps_{\la}} 2^{(|\la^{(0)}| - h(\la^{(0)}))/2}
\tilde K_{\la^{(i)}}\xi_\la,
\end{align}
for all $i \in I$. 
Until further notice we fix $i\in I$ and some $\mu \in \Par_0(\rho,d-1)$ with $|\mu^{(i)}|=d-1$ and $|\mu^{(j)}|=0$, for all $j\in I\setminus\{ i\}$. Let $k\in I\setminus\{ i\}$. By Lemmas \ref{lem:chars_bisupmod}(i) and \ref{lem:ith_runner}(i)(iii),
\begin{align}\label{algn:def_C}
\bY \otimes_{\cO \Hi } \xi_{\mu ,k} = C\xi_\la,
\end{align}
for some $C\in \N$, where $\la$ is the unique partition obtained from $\mu$ by sliding a bead down the $k^{\nth}$ runner. In other words, $\la^{(i)} = \mu^{(i)}$, $\la^{(k)} = (1)$ and $|\la^{(j)}| = 0$, for all $j \neq i,k$.

Using Lemma \ref{lem:HomAM_M*} we define $\chi \in \Irr_{\si}(\cO \Ni_{d-1} \fid_{d-1})$ via
\begin{align*}
\chi := \bX_{d-1}^* \otimes_{\cO \Gi_{d-1}} \xi_\mu.
\end{align*}
Then, recalling the notation from Remark~\ref{RCircledast}, by Lemma \ref{lem:AboxM}, we get 
\begin{align}\label{algn:chi_mu_k}
\chi \circledast \xi_k = (\bX_{d-1}^* \boxtimes (\Blo^{\varnothing,1})^{(d)}) \otimes_{\cO \Hi } \xi_{\mu,k}\in \Z \Irr(\cO (\Ni  \cap \Hi )\fid).
\end{align}
(Note there is no coefficient when we apply Lemma \ref{lem:AboxM} since $\bX_{d-1}^*$ induces a Morita superequivalence and $\chi$ and $\xi_\mu$ have the same parity.)

Since we are assuming Hypothesis \ref{hyp:ind_on_X}(ii)(b) for $m = d-1$, the construction of $\mu$ implies that 
\begin{align*}
\chi^{(\pm)}\downarrow^{\Ni_{d-1}}_{\Li_{d-1}} = {}_{\Li_{d-1}}\bX_{\Gi_{d-1}}^* \otimes \xi_\mu^{(\pm)}
\end{align*}
only has constituents of the form $\xi_{\rho,i^{d-1}}^{(\pm)}$. Therefore, the stabiliser of any constituent of $(\chi \circledast \xi_k)^{(\pm)}\downarrow^{\Ni  \cap \Hi }_\Li $ is contained in $\Ni  \cap \Hi $ and, by \cite[Theorem 6.11]{Is}, $\psi^{(\pm)} := (\chi \circledast \xi_k)^{(\pm)}\uparrow_{\Ni  \cap \Hi }^\Ni $ is irreducible. Moreover, as induction commutes with tensoring with the sign character, $\psi^{(\pm)}$ is self-associate if and only if $(\chi \circledast \xi_k)^{(\pm)}$ is. Therefore, $\psi = (\chi \circledast \xi_k)\uparrow_{\Ni  \cap \Hi }^\Ni $. So, by Lemma \ref{lem:tensor_isom}(ii), (\ref{algn:chi_mu_k}) and (\ref{algn:def_C}),
\begin{align}
\begin{split}\label{algn:X_tensor_psi}
\bX \otimes_{\cO \Ni } \psi = & \bX \otimes_{\cO \Ni } \cO \Ni  \otimes_{\cO(\Ni  \cap \Hi )} (\chi \circledast \xi_k) 
=  \bX \otimes_{\cO(\Ni  \cap \Hi )} (\chi \circledast \xi_k) \\
= & \bY \otimes_{\cO \Hi } \big(\bX_{d-1} \boxtimes (\Blo^{\varnothing,1})^{(d)}\big) \otimes_{\cO(\Ni  \cap \Hi )} (\chi \circledast \xi_k)
=  \bY \otimes_{\cO \Hi } \xi_{\mu,k} = C\xi_\la.
\end{split}
\end{align}
Now, by Lemmas \ref{lem:X*_with_X}(i) and \ref{lem:bimod_vert}(i),
\begin{align*}
\bX^* \otimes_{\cO \Gi } \bX \otimes_{\cO \Ni } \psi \equiv \psi
\end{align*}
modulo $\Z \Irr_{<_\sfs \Di}(\cO \Ni \fid)$. Therefore, utilising Lemma \ref{lem:vert_prsv}, we have
\begin{align}\label{algn:X*_tensor_X_psi}
{}_\Li  \bX_\Gi^* \otimes_{\cO \Gi } \bX \otimes_{\cO \Ni } \psi \equiv \cO \Li \fid \otimes_{\cO \Li } \cO \Ni  \otimes_{\cO \Ni } \psi = \psi\downarrow^\Ni_\Li 
\end{align}
modulo $\Z \Irr_{<_\sfs \Di}(\cO \Li \fid)$. By Lemma \ref{lem:restrict_calc},
\begin{align}\label{X*_tensor_xi_la}
{}_\Li  \bX_\Gi^* \otimes_{\cO \Gi } \xi_\la \equiv \frac{\eps_{\la}}{\eps_{\rho,i^{d-1},k}} 2^{(|\la^{(0)}| - h(\la^{(0)}))/2}
\tilde K_{\la^{(i)}}\xi_{\rho,i^{d-1},k}^{\tSi_d}
\end{align}
modulo $\Z \Irr_{<_\sfs \Di}(\cO \Li \fid)$. Applying ${}_\Li  \bX_\Gi^* \otimes_{\cO \Gi } ?$ to both sides of (\ref{algn:X_tensor_psi}) and substituting using (\ref{algn:X*_tensor_X_psi}) and (\ref{X*_tensor_xi_la}) gives
\begin{align}\label{algn:n_eq_1}
\psi\downarrow^\Ni_\Li  \equiv
C\frac{\eps_{\la}}{\eps_{\rho,i^{d-1},k}} 2^{(|\la^{(0)}| - h(\la^{(0)}))/2}
\tilde K_{\la^{(i)}}\xi_{\rho,i^{d-1},k}^{\tSi_d}
\end{align}
modulo $\Z \Irr_{<_\sfs \Di}(\cO \Li \fid)$.

Since we are assuming Hypothesis \ref{hyp:ind_on_X} holds for $m=d-1$,
\begin{align*}
\chi\downarrow^{\Ni_{d-1}}_{\Li_{d-1}} = {}_{\Li_{d-1}}\bX_{\Gi_{d-1}}^* \otimes_{\cO \Gi_{d-1}} \xi_\mu = \frac{\eps_{\mu}}{\eps_{\rho,i^{d-1}}} 2^{(|\mu^{(0)}| - h(\mu^{(0)}))/2} \tilde K_{\mu^{(i)}}\xi_{\rho,i^{d-1}}.
\end{align*}
Therefore, using Lemma \ref{lem:AboxM},
\begin{align*}
(\chi \circledast \xi_k)\downarrow^{\Ni  \cap \Hi }_\Li  = \frac{\eps_{\rho,i^{d-1}} \eps_{\mu,k}}{\eps_\mu \eps_{\rho,i^{d-1} ,k}} \frac{\eps_{\mu}}{\eps_{\rho,i^{d-1}}} 2^{(|\mu^{(0)}| - h(\mu^{(0)}))/2} \tilde K_{\mu^{(i)}}\xi_{\rho,i^{d-1},k}.
\end{align*}
(Here, when applying Lemma \ref{lem:AboxM}, we are again using the fact that $\chi$ and $\xi_\mu$ have the same parity.) Simplifying, we get
\begin{align*}
(\chi \circledast \xi_k)\downarrow^{\Ni  \cap \Hi }_\Li  = \frac{\eps_{\mu,k}}{\eps_{\rho,i^{d-1},k}} 2^{(|\mu^{(0)}| - h(\mu^{(0)}))/2} \tilde K_{\mu^{(i)}}\xi_{\rho,i^{d-1},k}.
\end{align*}
Therefore,
\begin{align}
\begin{split}\label{algn:n_eq_2}
\psi\downarrow^\Ni_\Li  \ =\ & (\chi \circledast \xi_k)\uparrow_{\Ni  \cap \Hi }^\Ni  \downarrow^\Ni_\Li  \ \ 
=  \sum_{g\in \Ni  / \Ni  \cap \Hi } {}^g((\chi \circledast \xi_k)\downarrow^{\Ni  \cap \Hi }_\Li ) \\
\ =\ & \frac{\eps_{\mu,k}}{\eps_{\rho,i^{d-1},k}} 2^{(|\mu^{(0)}| - h(\mu^{(0)}))/2} \tilde K_{\mu^{(i)}}\xi_{\rho,i^{d-1},k}^{\tSi_d}.
\end{split}
\end{align}
Taking the $\varphi$ from Lemma \ref{lem:nmv_char}, we have that $\varphi(\xi_{\rho,i^{d-1},k}^{\tSi_d})$ is either strictly positive or strictly negative. Therefore, we can equate (\ref{algn:n_eq_1}) and (\ref{algn:n_eq_2}) to obtain
\begin{align*}
C \eps_\la 2^{(|\la^{(0)}| - h(\la^{(0)}))/2} \tilde K_{\la^{(i)}} = \eps_{\mu,k} 2^{(|\mu^{(0)}| - h(\mu^{(0)}))/2} \tilde K_{\mu^{(i)}}.
\end{align*}
Note that $\la^{(i)} = \mu^{(i)}$ and so $\tilde K_{\la^{(i)}} = \tilde K_{\mu^{(i)}}$. In every case we also have $2^{(|\la^{(0)}| - h(\la^{(0)}))/2} = 2^{(|\mu^{(0)}| - h(\mu^{(0)}))/2}$. (Either $k=0$, in which case $\la^{(0)} = (1)$ and $\mu^{(0)} = \varnothing$ or $k \neq 0$, in which case $\la^{(0)} = \mu^{(0)}$.) So
\begin{align*}
C = \frac{\eps_{\mu,k}}{\eps_\la} = 1,
\end{align*}
where the second equality follows from Lemma \ref{ChuKesLem}(iii).

From now on $\mu$ will no longer denote a fixed partition. We have
\begin{align}
\begin{split}\label{algn:i^d-1_j}
& {}_\Gi  \bX_\Li  \otimes_{\cO \Li }\xi_{\rho,i^{d-1},k} = \bY \otimes_{\cO \Hi } \big(\bX_{d-1} \boxtimes (\Blo^{\varnothing,1})^{(d)}\big) \otimes_{\cO \Li } \xi_{\rho,i^{d-1},k}\\
& = \sum_{\substack{\mu \in \Par(\rho,d-1) \\ |\mu^{(i)}| = d-1}}\frac{\eps_\mu \eps_{\rho,i^{d-1},k}}{\eps_{\rho,i^{d-1}} \eps_{\mu ,k}}\frac{\eps_{\rho,i^{d-1}}}{\eps_{\mu}} 2^{(|\mu^{(0)}| - h(\mu^{(0)}))/2}
\tilde K_{\mu^{(i)}} \left(\bY \otimes_{\cO \Hi } \xi_{\mu,k}\right)\\
& = \sum_{\substack{\la \in \Par(\rho,d) \\ |\la^{(i)}| = d-1, |\la^{(k)}| = 1}}\frac{\eps_{\rho,i^{d-1},k}}{\eps_{\la}} 2^{(|\la^{(0)}| - h(\la^{(0)}))/2}
\tilde K_{\la^{(i)}} \xi_\la.
\end{split}
\end{align}
The first equality follows from Lemma \ref{lem:tensor_isom}(ii). The second equality uses Lemma \ref{lem:AboxM} and the fact that we are assuming that Hypothesis \ref{hyp:ind_on_X}(ii) holds for $m=d-1$. Finally, the third equality is a consequence of (\ref{algn:def_C}) and, once again, the fact that $\eps_{\mu,k} = \eps_\la$, $2^{(|\la^{(0)}| - h(\la^{(0)}))/2} = 2^{(|\mu^{(0)}| - h(\mu^{(0)}))/2}$ and that $\mu^{(i)} = \la^{(i)}$.

We now determine ${}_\Li  \bX_\Gi^* \otimes_{\cO \Gi } [\xi_\la]$, for any $\la \in \Par_0(\rho,d)$ with $|\la^{(i)}| = d-1$ and $|\la^{(k)}|=1$. By Lemmas \ref{lem:chars_bisupmod}(ii) and \ref{lem:ith_runner}(i),(iii),
\begin{align*}
\bY^* \otimes_{\cO \Gi } \xi_\la = \sum_{j,s \in \{i,k\} }\sum_{\mu\in \Par_0^j(\la)^-} C_{\mu,s} \xi_{\mu,s},
\end{align*}
for some $C_{\mu,s} \in \N$. Therefore, Lemmas \ref{lem:tensor_isom}(ii), \ref{lem:dual_comp}, \ref{lem:dual_prod} and \ref{lem:sup_alg_idempt_dual} imply
\begin{align*}
{}_\Li  \bX_\Gi^* \otimes_{\cO \Gi } \xi_\la = & \big({}_{\Li_{d-1}}\bX_{\Gi_{d-1}}^* \boxtimes (\Blo^{\varnothing,1})^{(d)}\big) \otimes_{\cO \Hi } \bY^* \otimes_{\cO \Gi } \xi_\la\\
= & \sum_{j,s \in \{i,k\} }\sum_{\mu\in \Par_0^j(\la)^-}C_{\mu,s} \big({}_{\Li_{d-1}}\bX_{\Gi_{d-1}}^* \boxtimes (\Blo^{\varnothing,1})^{(d)}\big) \otimes_{\cO \Hi } \xi_{\mu,s}.
\end{align*}
Using Lemma \ref{lem:AboxM}, the fact that we are assuming Hypothesis \ref{hyp:ind_on_X}(ii) for $m=d-1$ and that $\bX^*$ is actually an $(\cO \Ni ,\cO \Gi )$-bimodule yields
\begin{align*}
{}_\Li  \bX_\Gi^* \otimes_{\cO \Gi } \xi_\la = C_0 \xi_{\rho,i^d}^{\tSi_d} + C_1 \xi_{\rho,i^{d-1},k}^{\tSi_d} + C_2 \xi_{\rho,i^{d-2} ,k^2}^{\tSi_d},
\end{align*}
for some $C_0, C_1, C_2 \in \N$. By (\ref{algn:i^d-1_j}) and Lemma \ref{lem:chars_bisupmod}(ii) we know that
\begin{align*}
C_1 = \sum_{\substack{\la \in \Par(\rho,d) \\ |\la^{(i)}| = d-1, |\la^{(k)}| = 1}}\frac{\eps_{\la}}{\eps_{\rho,i^{d-1},k}} 2^{(|\la^{(0)}| - h(\la^{(0)}))/2}
\tilde K_{\la^{(i)}}
\end{align*}
and, by Lemma \ref{lem:restrict_calc},
\begin{align*}
 C_0 \xi_{\rho,i^d}^{\tSi_d} + C_1 \xi_{\rho,i^{d-1},k}^{\tSi_d} + C_2 \xi_{\rho,i^{d-2} ,k^2}^{\tSi_d}
\equiv  \sum_{\substack{\la \in \Par(\rho,d) \\ |\la^{(i)}| = d-1, |\la^{(k)}| = 1}}
\hspace{-2mm}
\frac{\eps_{\la}}{\eps_{\rho,i^{d-1},k}} 2^{(|\la^{(0)}| - h(\la^{(0)}))/2}
\tilde K_{\la^{(i)}}\xi_{\rho,i^{d-1},k}^{\tSi_d}
\end{align*}
modulo $\Z\Irr_{<_\sfs \Di}(\cO \Li \fid)$. Therefore,
\begin{align*}
C_0 \xi_{\rho,i^d}^{\tSi_d} + C_2 \xi_{\rho,i^{d-2},k^2}^{\tSi_d} \in \Z\Irr_{<_\sfs \Di}(\cO \Li \fid).
\end{align*}
Once again, taking the $\varphi$ from Lemma \ref{lem:nmv_char}, we have that $\varphi(\xi_{\rho,i^d}^{\tSi_d})$ and $\varphi(\xi_{\rho,i^{d-2},k^2}^{\tSi_d})$ are either both strictly positive or both strictly negative. Therefore, $C_0 = C_2 = 0$ and so
\begin{align}\label{algn:di_1j}
{}_\Li  \bX_\Gi^* \otimes_{\cO \Gi } \xi_\la = \frac{\eps_{\la}}{\eps_{\rho,i^{d-1},k}} 2^{(|\la^{(0)}| - h(\la^{(0)}))/2} \tilde K_{\la^{(i)}}\xi_{\rho,i^{d-1},k}^{\tSi_d}.
\end{align}
We are finally in a position to prove (\ref{L_to G_all_i}). We first have that
\begin{align*}
{}_\Gi  \bX_\Li  \otimes_{\cO \Li } \xi_{\rho,i^d} =  \bY \otimes_{\cO \Hi } \big(\bX_{d-1} \boxtimes (\Blo^{\varnothing,1})^{(d)}\big) \otimes_{\cO \Li } \xi_{\rho,i^d}
=  \sum_{\substack{\mu \in \Par_0(\rho,d-1) \\ |\mu^{(i)}| = d-1}} C_\mu (\bY \otimes_{\cO \Hi } \xi_{\mu,i}),
\end{align*}
for some $C_\mu \in \N$. Here, the first equality follows from Lemma \ref{lem:tensor_isom}(ii) and the second from Lemma \ref{lem:AboxM} and the fact that we are assuming Hypothesis \ref{hyp:ind_on_X}(ii) for $m=d-1$. (Note we could already explicitly write down the expression for the $C_\mu$'s but this is not necessary at this stage.) Therefore, by Lemma \ref{lem:ith_runner}(i),
\begin{align*}
{}_\Gi  \bX_\Li  \otimes_{\cO \Li } \xi_{\rho ,i^d} = \sum_{\substack{\la \in \Par_0(\rho,d) \\ |\la^{(i)}| = d}} C_\la \xi_\la + \sum_{j \neq i} \sum_{\substack{\la \in \Par_0(\rho,d) \\ |\la^{(i)}| = d-1, |\la^{(j)}| = 1}} C_\la \xi_\la,
\end{align*}
for some $C_\la \in \N$. However, (\ref{algn:di_1j}) and Lemma \ref{lem:chars_bisupmod}(ii) give that $C_\la = 0$ unless $|\la^{(i)}| = d$. In other words,
\begin{align}\label{algn:d_i's}
{}_\Gi  \bX_\Li  \otimes_{\cO \Li } \xi_{\rho,i^d} = \sum_{\substack{\la \in \Par_0(\rho,d) \\ |\la^{(i)}| = d}} C_\la \xi_\la.
\end{align}
We are now in a position to prove (\ref{L_to G_all_i}), that is, that each
\begin{align*}
C_\la = \frac{\eps_{\rho,i^d}}{\eps_{\la}} 2^{(|\la^{(0)}| - h(\la^{(0)}))/2}
\tilde K_{\la^{(i)}}.
\end{align*}
However, this is actually unnecessary and we go ahead and prove the more general statement of Hypothesis \ref{hyp:ind_on_X}(ii) for $m=d$.

Consider $\ud\in\La(I,d)$. We set
\begin{align}\label{X_otimes_xi's}
{}_\Gi  \bX_\Li  \otimes_{\cO \Li } \xi_{\rho,\ud} = \sum_{\la \in \Par_0(\rho,d)} \ka_\la \xi_\la,
\end{align}
for some $\ka_\la \in \N$. Suppose $\ka_\la \neq 0$, for some $\la \in \Par_0(\rho,d)$. We claim that $|\la^{(j)}| \geq d_j$, for each $j \in I$. For $d_j=d$, this is just (\ref{algn:d_i's}). If $d_j<d$, then
\begin{align*}
 {}_\Gi  \bX_\Li  \otimes_{\cO \Li } \xi_{\rho,\ud}
=  {}_\Gi  \bX_\Li  \otimes_{\cO \Li } \xi_{\rho,j^{d_j},\ud-d_j\ude_j} 
=  \bY \otimes_{\cO \Hi } \big(\bX_{d-1} \boxtimes (\Blo^{\varnothing,1})^{(d)}\big) \otimes_{\cO \Li } \xi_{\rho,j^{d_j},\ud-d_j\ude_j},
\end{align*}
where the first equality follows from the fact that $\bX$ is actually an $(\cO \Gi ,\cO \Ni )$-bimodule and the second from Lemma \ref{lem:tensor_isom}(ii). The claim now holds by the fact that we are assuming Hypothesis \ref{hyp:ind_on_X}(ii) for $m=d-1$ and Lemmas \ref{lem:AboxM} and \ref{lem:ith_runner}(i).

Since $d_1 + \dots + d_\ell = d$, in fact $|\la^{(j)}| = d_j$, for all $\la$ with $\ka_\la \neq 0$ and $j \in I$. Therefore, dualizing (\ref{X_otimes_xi's}) using \ref{lem:chars_bisupmod}(ii) and noting that $\bX^*$ is an $(\cO \Ni , \cO \Gi )$-bimodule, we have that for all $\la \in \Par_0(\rho,\ud)$,
\begin{align}\label{algn:LXG_la}
{}_\Li  \bX_\Gi^* \otimes_{\cO \Gi } \xi_\la = \al_\la \xi_{\rho,\ud}^{\tSi_d},
\end{align}
for some $\al_\la\in \N$. Also, by Lemma \ref{lem:restrict_calc},
\begin{align}
\begin{split}\label{algn:LXG_la_cong}
 {}_\Li  \bX_\Gi^* \otimes_{\cO \Gi } \xi_\la 
\equiv  \frac{\eps_{\la}}{\eps_{\rho,\ud}}\, 2^{(|\la^{(0)}| - h(\la^{(0)}))/2}
K(\la)\xi_{\rho,\ud}^{\tSi_d}
\end{split}
\end{align}
modulo $\Z\Irr_{<_\sfs \Di}(\cO \Li \fid)$. Once more, taking $\varphi$ from Lemma \ref{lem:nmv_char}, we have that $\varphi(\xi_{\rho,\ud}^{\tSi_d})$ is either strictly positive or strictly negative. Therefore, considering (\ref{algn:LXG_la}) and (\ref{algn:LXG_la_cong}) together gives 
$
\al_\la = \frac{\eps_{\la}}{\eps_{\rho,\ud}}\, 2^{(|\la^{(0)}| - h(\la^{(0)}))/2}
K(\la),
$
for all $\la \in \Par_0(\rho,\ud)$, and the result is proved.
\end{proof}

\begin{Proposition}\label{prop:prelim_main}
Let $d>1$. If Hypothesis \ref{hyp:ind_on_X} holds for all $1\leq m\leq d-1$, then Hypothesis \ref{hyp:ind_on_X}(i) holds for $m=d$.
\end{Proposition}

\begin{proof}
By Lemma \ref{ML_MN_indecomp}, $\bM_\Li $ has vertex $\Delta \Di $. In particular, by the Mackey decomposition formula, $\bM_\Li $ is projective as both a left and right $\cO \Li $-module. By Lemma \ref{lem:GXL_indec}, ${}_\Gi  \bX_\Li $ is a direct summand of $\cO \Gi  \otimes_{\cO \Li } \bM_\Li $ and so ${}_\Gi  \bX_\Li $ is projective as a left $\cO \Gi $-module and as a right $\cO \Li $-module. Moreover, since $p\nmid[\Ni :\Li ]$, $\bX$ is also projective as a right $\cO \Ni $-module.

By Lemma \ref{lem:X*_with_X}(ii), $\cO \Gi \bi$ is a direct summand of $\bX \otimes_{\cO \Ni } \otimes \bX^*$ as an $(\cO \Gi ,\cO \Gi )$-bimodule. As seen above, $\bX$ is projective as a right $\cO \Ni $-module and so $\bX^*$ is projective as a left $\cO \Ni $-module. It follows that every indecomposable summand of $\bX \otimes_{\cO \Ni } \otimes \bX^*$, as a left $\cO \Gi $-module, appears as a summand of $\bX \simeq \bX \otimes_{\cO \Ni } \cO \Ni $. In particular, $\bX$ is a progenerator as a left $\cO \Gi \bi$-module. As a consequence, $\End_{\cO \Gi }\bX$ is free as an $\cO$-module.

Similar to the above argument, using part (i) of Lemma \ref{lem:X*_with_X} rather than part (ii), we may prove that $\bX$ is also a progenerator as a right $\cO \Ni \fid$-module.

Consider the natural algebra homomorphism
\begin{align}\label{algn:main_hom}
(\cO \Ni \fid)^{\op} \to \End_{\cO \Gi } \bX
\end{align}
given by right multiplication. As $\bX$ is a progenerator as a right $\cO \Ni \fid$-module, this is actually a monomorphism. We claim that this homomorphism is an $\cO$-split monomorphism. Indeed, by Lemma \ref{lem:X*_with_X}(i), $\bX^* \otimes_{\cO \Gi } \bX \simeq \cO \Ni \fid \oplus \gU$, for some $(\cO \Ni \fid,\cO \Ni \fid)$-bisupermodule $\gU$. Therefore, we have the sequence of $\cO$-module homomorphisms
\begin{align*}
\End_{\cO \Gi } \bX & \xrightarrow{\varphi} \End_{\cO \Ni }(\bX^* \otimes_{\cO \Gi } \bX) \cong \End_{\cO \Ni }(\cO \Ni \fid \oplus \gU) \\
& \xrightarrow{\vartheta} \End_{\cO \Ni }(\cO \Ni \fid) \cong (\cO \Ni \fid)^{\op},
\end{align*}
where $\varphi$ is induced by $\bX^* \otimes_{\cO \Gi }?$ and $\vartheta$ is induced by the projection of $\cO \Ni \fid \oplus \gU$ onto $\cO \Ni \fid$. It is now a trivial task to prove that the resulting composition $\End_{\cO \Gi } \bX \to (\cO \Ni \fid)^{\sop}$ splits the homomorphism in (\ref{algn:main_hom}), as an $\cO$-module homomorphism.

Once we have shown that $\cO \Ni \fid$ and $\End_{\cO \Gi }\bX$ have the same $\cO$-rank, we will have that the monomorphism in (\ref{algn:main_hom}) is actually an isomorphism and consequently that $\bX$ induces a Morita equivalence between $\cO \Ni \fid$ and $\cO \Gi \bi$. The majority of the remainder of this proof is dedicated to proving this equality of ranks.

As $\End_{\cO \Gi }\bX$ is free as an $\cO$-module, we can prove the equality of $\cO$-ranks by showing that
\begin{align}\label{dim_K_formula}
\dim_{\K}(\End_{\K \Gi }\K \bX) = d!\dim_{\K}(\K \Li \fid) = \dim_{\K}(\K \Ni \fid).
\end{align}
The second equality is immediate since
\begin{align*}
\dim_{\K}(\K \Ni \fid) = \dim_{\K}(\K \Ni  \otimes_{\K \Li } \K \Li \fid) = [\Ni :\Li ]\dim_{\K}(\K \Li \fid) = d! \dim_{\K}(\K \Li \fid).
\end{align*}

We now calculate the degrees of the characters in $\Irr(\cO \Li \fid)$. \cite[Proposition 4.2(a)]{Stem} immediately gives that
\begin{align}
\begin{split}\label{char_dim_L}
\xi_{\rho,j_1, \dots ,j_d}^{(\pm)}(1) & = \frac{\eps_\rho \prod_{k=1}^d \eps_{j_k}}{\eps_{\rho,j_1,\dots,j_d}} \xi_\rho^{(\pm)}(1) \xi_{j_1}^{(\pm)}(1) \dots \xi_{j_d}^{(\pm)}(1)\\
& = \frac{\eps_\rho\, 2^{\frac{d - d_0}{2}}}{\eps_{\rho,j_1,\dots,j_d}} \xi_\rho^{(\pm)}(1) \xi_{j_1}^{(\pm)}(1) \dots \xi_{j_d}^{(\pm)}(1),
\end{split}
\end{align}
for any $j_1,\dots,j_d \in I$, where $0$ appears $d_0$ times amongst the $j_k$'s.

Since we are assuming that Hypothesis \ref{hyp:ind_on_X} holds for $1 \leq m \leq d-1$, we can apply Proposition \ref{prop:main_char_calc}, to obtain that the character of
\begin{align*}
\K \bX \cong \K \bX \otimes_{\K \Li \fid}\K \Li \fid,
\end{align*}
as a left $\K \Gi $-module, is
\begin{align*}
&\sum_{\chi \in \Irr(\cO \Li \fid)} \chi(1) (\bX \otimes_{\cO \Li } \chi) \\
= & \sum_{\ud\in\La(I,d)} \frac{\eps_\rho}{\eps_{\rho,\ud}} 2^{\frac{d - d_0}{2}} \xi_\rho^{(\pm)}(1) \xi_0^{(\pm)}(1)^{d_0} \dots \xi_\ell^{(\pm)}(1)^{d_\ell}  (\bX \otimes \xi_{\rho,\ud}^{\tSi_d})\\
= & \sum_{\ud\in\La(I,d)} \sum_{\la \in \Par_0(\rho,\ud)} \hspace{-1mm}\binom{d}{\ud} \frac{ \eps_\rho}{\eps_\la} 2^{\frac{d - d_0}{2}} 
\xi_\rho^{(\pm)}(1) \xi_0^{(\pm)}(1)^{d_0} \dots \xi_\ell^{(\pm)}(1)^{d_\ell} 
2^{(|\la^{(0)}| - h(\la^{(0)}))/2}K(\la)\xi_\la,
\end{align*}
where the first equality follows from (\ref{char_dim_L}) and the second is due to Proposition \ref{prop:main_char_calc}.

Therefore,
\begin{align*}
&\dim_{\K}(\End_{\K \Gi }\K \bX) = \\
&\sum_{\ud\in\La(I,d)}\sum_{\la \in \Par_0(\rho,\ud) } \eps_\la^2 \big[\binom{d}{\ud}\frac{\eps_\rho}{\eps_\la}2^{\frac{d - d_0}{2}}
\xi_\rho^{(\pm)}(1)\xi_0^{(\pm)}(1)^{d_0}\dots\xi^{(\pm)}_{\ell}(1)^{d_\ell} 
\,2^{(|\la^{(0)}|-h(\la^{(0)}))/2} K(\la) \big]^2\\
= & \sum_{\ud\in\La(I,d)} \binom{d}{\ud}^2 \eps_\rho^2 2^{d-d_0} \big[\xi_\rho^{(\pm)}(1)\xi_0^{(\pm)}(1)^{d_0}\dots\xi^{(\pm)}_{\ell}(1)^{d_\ell}\big]^2 d_0!d_1!\dots d_{\ell}! \\
= & d!\sum_{\ud\in\La(I,d)} \binom{d}{\ud} \eps_\rho^2 2^{d-d_0} \big[\xi_\rho^{(\pm)}(1)\xi_0^{(\pm)}(1)^{d_0}\dots\xi^{(\pm)}_{\ell}(1)^{d_\ell}\big]^2\\
= & d!\sum_{\chi \in \Irr(\K \Li \fid)}\chi(1)^2=d!\dim_{\K}(\K \Li \fid),
\end{align*}
where the first equality follows from the above expression for the character of $\K \bX$. (Note that the $\eps_\la^2$ comes from the fact that $[\xi_\la]$ is actually the sum of $\eps_\la^2$ irreducibles.) The second equality follows from Lemma \ref{lem:dim_sqd} and the fourth from (\ref{char_dim_L}). (Note that the $\eps_{\rho,j_1,\dots,j_d}$ from (\ref{char_dim_L}) disappears as $\xi_{\rho,{j_1}, \dots ,{j_d}}$ is actually the sum of $\eps_{\rho,j_1,\dots,j_d}^2$ irreducibles.) We have now proved the first equality in (\ref{dim_K_formula}) and the proof is complete.

That $\bX$ induces a Morita superequivalence (and not just a Morita equivalence) is now an immediate consequence of Lemma \ref{lem:HomAM_M*}.
\end{proof}

\begin{Theorem}\label{thm:main}
Hypothesis \ref{hyp:ind_on_X} holds for all $1 \leq m \leq d$. In particular, $\bX \otimes_{\cO \Ni \fid} ?$ induces a Morita superequivalence between $\cO \Ni \fid$ and $\cO \Gi \bi$.
\end{Theorem}

\begin{proof}
The proof will proceed via induction. The statement for $m=1$ is just Proposition \ref{prop:define_M}(ii). Note that, in this case, the coefficients in Hypothesis \ref{hyp:ind_on_X}(ii) really are all $1$, as in Proposition \ref{prop:define_M}(ii).

Let $1\leq k < d$. We assume the statement is true for all $1 \leq m < k$. We first apply Proposition \ref{prop:main_char_calc} with $d$ replaced by $k$. (This is valid due to Remark \ref{rem:induct_Rou}.) We immediately get that Hypothesis \ref{hyp:ind_on_X}(ii) holds for $m=k$.

Next, we apply Proposition \ref{prop:prelim_main} with $d$ replaced by $k$. (Again, this is valid via Remark \ref{rem:induct_Rou}.) This time we get that Hypothesis \ref{hyp:ind_on_X}(i) holds for $m=k$. This completes the proof.
\end{proof}

The following corollary is now just a consequence of Theorem \ref{thm:main} and Lemma \ref{lem:Mor_A0}.

\begin{Corollary}\label{cor:main}
$\bX_{\0} \otimes_{\cO \Ni_{\0}\fid} ?$ induces a Morita equivalence between $\cO \Ni_{\0}\fid$ and $\cO \Gi_{\0}\bi$.
\end{Corollary}

Of course, $\cO \Gi_{\0}\bi$ is a spin block of a double cover of an alternating group. Furthermore, by Lemma \ref{lem:Brauer_G_N}, $\cO \Ni_{\0}\fid$ is the Brauer correspondent of $\cO \Gi_{\0}\bi$ in $\Ni_{\0}$. We can, therefore, interpret Corollary \ref{cor:main} as a result concerning RoCK blocks for double covers of alternating groups.

\begin{Remark}
Note that, for the $\Gi_{\0}$ in Corollary \ref{cor:main}, $|\Gi_{\0}| = n!$. In order to be consistent with the convention from the introduction, we would like to know that the Corollary holds under the condition that $\K$ (and hence $\cO$) contains a primitive $(n!)^{\nth}$ root of unity and not necessarily a primitive $(2n!)^{\nth}$ root of unity. We claim that even Theorem \ref{thm:main} holds under these weaker assumptions on $\K$.

We have $n \geq 4$ and so we still automatically have that $-1$ and $2$ have square roots in $\K$. Next note we are able to define $\fid$ and $\bi$ since they live in $\cO \Gi_{\0}$ and $\cO \Ni_{\0}$ respectively. The $\bM$ from Proposition \ref{prop:define_M} is defined entirely using $\cO \tSi_p$. Since $2p! \leq n!$, this does not cause any issues. Similarly, $2r! \leq n!$ and so $\Blo^{\rho,0}$, and hence $\bM_\Ni $, can be constructed.

We can now construct $\bX$ as the super Green correspondence of $\bM_\Ni $ and we have an algebra homomorphism $(\cO \Ni \fid)^{\op} \to \End_{\cO \Gi } \bX$ given by the right action of $\cO \Ni \fid$ on $\bX$. That this is an isomorphism follows from the fact that the corresponding homomorphism obtained by extending $\cO$ is an isomorphism.
\end{Remark}

\section{Vertices and source of bimodules}\label{sec:vert_source}

We continue to adopt the notation from Section~\ref{sec:X_Y}, in particular, $\rho$ a $d$-Rouquier ${\bar p}$-core.
%In particular, recall the definitions of $\bM_\Li $ and $\bM_\Ni $ from $\S$\ref{sec:ML_MN} and those of $\bX$ and $\bY$ from $\S$\ref{sec:def_X_Y}.

\subsection{Splendid derived equivalences}

In this section we show that the Morita equivalences between $\cO \Ni \fid$ and $\cO \Gi \bi$ from Theorem \ref{thm:main} and between $\cO \Ni_{\0}\fid$ and $\cO \Gi_{\0}\bi$ from Corollary \ref{cor:main} give rise to a splendid derived equivalences. 

Throughout this section all isomorphisms and direct summands will assumed to be as bimodules (and not bisupermodules).

Let $K$ be a finite group and $e_K$ a block idempotent of $\cO K$ with corresponding defect group $Q\leq K$. Recall that a {\em source idempotent} of $\cO Ke_K$ is a primitive idempotent $i\in (\cO Ke_K)^Q$ such that $\Br_Q(i)\neq 0$. Let $J$ be another finite group and $e_J$ a block idempotent of $\cO J$ with corresponding defect group isomorphic to $Q$ (that we identify with $Q$) and $j\in (\cO Je_J)^Q$ a source idempotent of $\cO Je_J$. A {\em splendid derived equivalence} between $\cO Ke_K$ and $\cO Je_J$ is a derived equivalence induced by a complex $\cY$ of $(\cO Ke_K,\cO Je_J)$-bimodules such that in each degree we have a finite direct sum of summands of the $(\cO Ke_K,\cO Je_J)$-bimodules $\cO Ki \otimes_{\F R} j\cO J$, where $R$ runs over the subgroups of $Q$, cf. \cite[1.10]{Lin5}.

For our purposes it will be useful to think of source idempotents in the following way outlined in \cite{ALR}. An idempotent $i\in (\cO Ke_K)^Q$ is a source idempotent of $\cO Ke_K$ if and only if $\cO Ki$ is an indecomposable $\cO(K \times Q)$-module with vertex $\Delta Q$ (see \cite[Remark 3]{ALR}). Equally, an idempotent $i\in (\cO Ke_K)^Q$ is a source idempotent of $\cO Ke_K$ if and only if $i\cO K$ is an indecomposable $\cO(Q \times K)$-module with vertex $\Delta Q$.

We set $\bU$ to be the $\cO\Delta \Di $-module
\begin{align*}
\bU := \Omega_{\cO\Delta \Di_1}^\ell(\cO) \boxtimes \dots \boxtimes \Omega_{\cO\Delta \Di_d}^\ell(\cO).
\end{align*}

\begin{Lemma}\label{lem:ON_MN_ON}
As $(\cO \Ni ,\cO \Ni )$-bimodules,
$$\bM_\Ni \mid \cO \Ni  \otimes_{\cO \Di } (\Ind_{\Delta \Di }^{\Di  \times \Di }\bU) \otimes_{\cO \Di } \cO \Ni .$$
\end{Lemma}

\begin{proof}
We first note that, by the definition of $\bM_\Li $, Proposition \ref{prop:define_M}(i), Remark \ref{rem:bisupmod} and Lemma \ref{lem:direct_prod_vert},
$$\bM_\Li \mid\Ind_{\Delta \Di }^{\Li  \times \Li }\bU \simeq \cO \Li  \otimes_{\cO \Di } (\Ind_{\Delta \Di }^{\Di  \times \Di }\bU) \otimes_{\cO \Di } \cO \Li .$$
Therefore, by Lemma \ref{ML_MN_indecomp},
$$\Res^{\Ni  \times \Ni }_{\Ni  \times \Li }\bM_\Ni  \simeq \Ind_{\Li  \times \Li }^{\Ni  \times \Li }\bM_\Li \mid \cO \Ni  \otimes_{\cO \Di } (\Ind_{\Delta \Di }^{\Di  \times \Di }\bU) \otimes_{\cO \Di } \cO \Li .$$
Since $p \nmid [\Ni :\Li ]$, the claim follows.
\end{proof}

\begin{Lemma}\label{lem:OG_U_ON}
%\begin{enumerate}
%\item
As $(\cO \Gi ,\cO \Ni )$-bimodules,
$$\bX\mid \cO \Gi i \otimes_{\cO \Di } (\Ind_{\Delta \Di }^{\Di  \times \Di }\bU) \otimes_{\cO \Di } j\cO \Ni ,$$
for some source idempotents $i \in (\cO \Gi \bi)^\Di $ and $j \in (\cO \Ni \fid)^\Di $.
\end{Lemma}

\begin{proof}
By the definition of $\bX$ and Lemmas \ref{lem:ON_MN_ON} and \ref{lem:X_as_left_mod}(i),
$$\bX\mid \cO \Gi \bi \otimes_{\cO \Di } (\Ind_{\Delta \Di }^{\Di  \times \Di }\bU) \otimes_{\cO \Di } \cO \Ni \fid.$$
Since $\bX$ is indecomposable, it is isomorphic to a direct summand of
\begin{align}\label{algn:V_U_W}
\gV \otimes_{\cO \Di } (\Ind_{\Delta \Di }^{\Di  \times \Di }\bU) \otimes_{\cO \Di } \gW,
\end{align}
for some indecomposable summand $\gV$ of $\cO \Gi \bi$ as an $\cO(\Gi  \times \Di )$-module and some indecomposable summand $\gW$ of $\cO \Ni \fid$ as an $\cO(\Di  \times \Ni )$-module.

By Mackey's decomposition formula, $\gV$ is relatively $\Delta Q$-projective, where $\Delta Q={}^{(g,g)}\Delta \Di  \cap (\Gi  \times \Di )$, for some $g \in \Gi $. In particular, $Q \leq \Di $. Similarly, $\gW$ is relatively $\Delta R$-projective, where $\Delta R={}^{(h,h)}\Delta \Di \cap (\Di  \times \Gi )$, for some $h \in \Ni $. Again, $R \leq \Di $. If $g \notin N_\Gi (\Di )$ (resp. $h \notin N_\Hi (\Di )$), then $Q < \Di $ (resp. $R < \Di $). Therefore, by Lemma \ref{lem:bimod_vert}(ii), the bimodule in (\ref{algn:V_U_W}) is a direct sum of bimodules each with vertex of order strictly smaller than that of $\Delta \Di $. Since $\bX$ has vertex $\Delta \Di $, this is a contradiction. From now on we assume $g \in N_\Gi (\Di )$ and $h \in N_\Hi (\Di )$.

We now have that $Q = R = \Delta \Di $. In fact, $\gV$ and $\gW$ must actually both have vertex $\Delta \Di $. Indeed, if either vertex is strictly contained in $\Delta \Di $, then Lemma \ref{lem:bimod_vert}(ii) once again gives that the bimodule in (\ref{algn:V_U_W}) is a direct sum of bimodules each with vertex of order strictly smaller than that of $\Delta \Di $.

By the \cite{ALR} description of source idempotents, we have now shown that $\gV \cong \cO \Gi i$, for some source idempotent $i \in (\cO \Gi \bi)^D$ and $\gW \cong j\cO \Ni $, for some source idempotent $j \in (\cO \Ni \fid)^\Di $, as desired.
\end{proof}

For a finite group $K$ and an $\cO K$-module $\gZ$, an {\em endosplit $p$-permutation resolution} of $\gZ$ is a bounded complex $\cY$ of finitely generated $p$-permutation $\cO K$-modules with homology concentrated in degree $0$ isomorphic to $\gZ$
such that the complex $\cY \otimes_{\cO} \cY^*$ is a split complex of $\cO K$-modules, where $K$ acts diagonally on the tensor product.

For the following lemma, we treat $\bU$ as an $\cO \Di $-module via the obvious group isomorphism $\Di  \cong \Delta \Di $.

\begin{Lemma}\label{lem:U_endosplit}
$\bU$ has an $N_\Gi (\Di )$-stable endosplit $p$-permutation resolution.
\end{Lemma}

\begin{proof}
Until further notice we fix some $k$, with $1 \leq k \leq d$, and let $g$ be a generator of $\Di_k$. We first claim that $\Omega_{\cO \Di_k}^2(\cO) \cong \cO$. Indeed, once we have made the identification
\begin{align*}
\Omega_{\cO \Di_k}(\cO) \cong \big\{\sum_{x \in \Di_k} \al_x x \mid  \sum_{x \in \Di_k} \al_x = 0\big\} \subseteq \cO \Di_k,
\end{align*}
the kernel of the $\cO \Di_k$-module epimorphism given by
\begin{align*}
\cO \Di_k  \twoheadrightarrow \Omega_{\cO \Di_k}(\cO), \ 
x  \mapsto (1-g)x 
\end{align*}
is $\langle \sum_{x \in \Di_k} x\rangle_{\cO} \cong \cO$, as claimed.

We have now shown that
\begin{align*}
\Omega_{\cO \Di_k}^{\ell}(\cO) \cong
\begin{cases}
\cO & \text{if }\ell\text{ is even,}\\
\Omega_{\cO \Di_k}(\cO) & \text{if }\ell\text{ is odd.}
\end{cases}
\end{align*}
Therefore, if $\ell$ is even, $\bU \cong \cO$ and we can just take the complex with $\cO$ concentrated in degree $0$. From now on we assume $\ell$ is odd and consequently that $\Omega_{\cO \Di_k}^{\ell}(\cO) \cong \Omega_{\cO \Di_k}(\cO)$.

Certainly
\begin{align*}
\cY_k := \dots \to 0 \to \cO \Di_k \to \cO \to 0 \to \dots,
\end{align*}
where the $\cO \Di_k$ is in degree $0$ and the non-zero boundary map is given by $x \mapsto 1$, for all $x \in \Di_k$, has homology concentrated in degree $0$ isomorphic to $\Omega_{\cO \Di_k}(\cO)$. Moreover, by \cite[Proposition 7.11.8]{Lin6}, this is an endosplit $p$-permutation resolution of $\Omega_{\cO \Di_k}(\cO)$. It is also clear that this resolution is $N_{\tSi_{\sJ_k}}(\Di_k)$-stable.

We now drop our assumption that $k$ is fixed. By \cite[Lemma 7.4]{Ri}, $\cY:=\cY_1 \otimes \dots \otimes \cY_d$ is an endosplit $p$-permutation resolution of $\bU$.

All that remains to show is that $\cY$ is $N_\Gi (\Di )$-stable. For this we refer to Lemma \ref{lem:CGD}(ii). We note that, since $\Di  \leq \Gi_{\0}$, all the conjugation actions by elements of $N_\Gi (\Di )$ are particularly easy to write down. For example, $N_{\tSi_{\sJ_k}}(\Di_k)$ commutes with every $\Di_l$, with $l \neq k$, and, through our identification of all the $\cO \Di_k$'s via (\ref{Sp_ident}),
\begin{align}\label{algn:Tw_on_D}
T_w(\al_1 \otimes \dots \otimes \al_d)T_w^{-1} = \al_{w^{-1}(1)} \otimes \dots \otimes \al_{w^{-1}(d)},
\end{align}
for all $w \in \Si_d$ and $\al_k \in \cO \Di_k$.

Firstly, $\tSi_\sR $ commutes with $\Di $, so $\cY$ is $\tSi_\sR $-stable. That $\cY$ is $N_{\tSi_{\sJ_k}}(\Di_k)$-stable, for all $1 \leq k \leq d$, just follows from our construction. Therefore, the only non-trivial thing to show is that $\cY$ is stable under the action of each $T_w$. However, this is an immediate consequence of (\ref{algn:Tw_on_D}) and \cite[Lemma 4.1(b)]{Mar}
\end{proof}

\begin{Theorem}\label{thm:splendid}
$\cO \Ni \fid$ is splendidly derived equivalent to $\cO \Gi \bi$.
\end{Theorem}

\begin{proof}
With Lemmas \ref{lem:OG_U_ON}, \ref{lem:U_endosplit} and Theorem \ref{thm:main} in mind, this can just be concluded from \cite[Proposition 9.11.5(i)]{Lin6}, originally stated in \cite[Theorem 1.3]{Lin3}.
\end{proof}

We also have the analogous theorem for the RoCK blocks of double covers of alternating groups from Corollary \ref{cor:main}.

\begin{Theorem}\label{thm:splendid_0}
$\cO \Ni_{\0}\fid$ is splendidly derived equivalent to $\cO \Gi_{\0}\bi$.
\end{Theorem}

\begin{proof}
Once we have replaced Theorem \ref{thm:main} with Corollary \ref{cor:main}, the proof of this result is identical to that of Theorem \ref{thm:splendid} except that the analogue of Lemma \ref{lem:OG_U_ON} requires us to be slightly more careful than in the original lemma. More precisely, it is not clear that $\bX_{\0}$ has vertex $\Delta \Di $ and source $\bU$. However, since $p\nmid[\Gi  \times \Ni :\Gi_{\0} \times \Ni_{\0}]$ and $\Res^{\Gi  \times \Ni }_{\Gi_{\0} \times \Ni_{\0}}\bX \cong \bX_{\0} \oplus \bX_{\1}$, we must at least have one of $\bX_{\0}$ or $\bX_{\1}$ having vertex $\Delta \Di $ and source $\bU$.

The proof will be complete once we have proved that $\bX_{\1}$ induces a Morita equivalence between $\cO \Ni_{\0}\fid$ and $\cO \Gi_{\0}b$. Now, $\bX_{\1} \cong g \bX_{\0}$, for any $g \in \Gi_{\1}$. Therefore, as functors, $\bX_{\1} \otimes_{\cO \Ni_{\0}\fid} ?$ is isomorphic to $\bX_{\0} \otimes_{\cO \Ni_{\0}\fid} ?$ composed with the Morita auto-equivalence of $\cO \Gi_{\0}\bi$ given by conjugation by $g$. This completes the claim.
\end{proof}

We can now construct splendid derived equivalences from our RoCK blocks to their true Brauer correspondents. That is, their Brauer correspondent in the normalizer of the defect group.

\begin{Corollary}\label{cor:spld}
$\cO \Gi \bi$ is splendidly derived equivalent to its Brauer correspondent in $N_\Gi (\Di)$ and $\cO \Gi_{\0}\bi$ is splendidly derived equivalent to its Brauer correspondent in $N_{\Gi_{\0}}(\Di )$.
\end{Corollary}

\begin{proof}
In \cite[Proposition 5.2.33]{KL} a splendid derived equivalence between $\F \Ni {\bar \fid}$ (resp. $\F \Ni_{\0}{\bar \fid}$) and the Brauer correspondent of $\F \Gi {\bar \bi}$ (resp. $\F \Gi_{\0}{\bar \bi}$) in $N_\Gi (\Di )$ (resp. $N_{\Gi_{\0}}(\Di )$) was constructed. By \cite[Theorem 5.2]{Ri}, all these splendid derived equivalences lift to splendid derived equivalences for the appropriate blocks defined over $\cO$. The result now follows from Theorems \ref{thm:splendid}, \ref{thm:splendid_0} and \cite[Lemma 10.2.6]{Rou}.
\end{proof}

\subsection{Source algebra equivalences}

A {\em source algebra equivalence} between two blocks is a Morita equivalence induced by a bimodule with trivial source. As noted in the proof of Lemma \ref{lem:U_endosplit}, $\bU \cong \cO$, when $\ell$ is even, that is, when $p \equiv 1 \bmod 4$. We, therefore, have the following theorem:

\begin{Theorem}\label{thm:source}
If $p \equiv 1 \bmod 4$, $\cO \Ni \fid$ is source algebra equivalent to $\cO \Gi \bi$ and $\cO \Ni_{\0}\fid$ is source algebra equivalent to $\cO \Gi_{\0}\bi$.
\end{Theorem}

If $p \equiv 3 \bmod 4$, our bimodule $\bX$ no longer has trivial source. However, it is still unclear whether or not an appropriate trivial source bimodule can be found:

\begin{Question}\label{qstn:source}
If $p \equiv 3 \bmod 4$, is $\cO \Ni \fid$ source algebra equivalent to $\cO \Gi \bi$ and is $\cO \Ni_{\0}\fid$ source algebra equivalent to $\cO \Gi_{\0}\bi$?
\end{Question}

We compare Theorem \ref{thm:source} and Question \ref{qstn:source} with the analogous situation for RoCK blocks of symmetric groups in \cite[Theorem 2]{CK}. In that article, the analogue of the bimodule $\bX$ from this paper was taken to be the Green correspondent of the local block itself, rather than something more complicated like our bimodule $\bM_\Ni $. It, therefore, automatically had trivial source.

\end{document}